\documentclass[twoside, 12pt]{amsart}
\usepackage[T1]{fontenc}
\usepackage{amsfonts,amsmath,amsthm,amscd,amssymb,latexsym,amsbsy,caption}
\usepackage[noadjust]{cite}
\usepackage[a4paper, top=3cm, bottom=3cm, left=3cm, right=3cm]{geometry}

\usepackage{tikz}
\usetikzlibrary{positioning, trees, decorations.markings, patterns, arrows, arrows.meta}

\usepackage{diagmac2}

\usepackage{bm} % for bold of greek

\newtheorem{theorem}{Theorem}[section]
\newtheorem{lemma}[theorem]{Lemma}
\newtheorem{corollary}[theorem]{Corollary}
\newtheorem{proposition}[theorem]{Proposition}

\theoremstyle{definition}
\newtheorem{definition}[theorem]{Definition}
\newtheorem{problem}[theorem]{Problem}

\newtheorem{conjecture}{Conjecture}[section]

\theoremstyle{remark}
\newtheorem{example}[theorem]{Example}
\newtheorem{remark}[theorem]{Remark}

\numberwithin{equation}{section}
\allowdisplaybreaks[4]

\makeatletter
\renewcommand{\subsection}{\@startsection{subsection}{2}%
  {0pt}{.7\linespacing\@plus\linespacing}{.5\linespacing}%
  {\normalfont\bfseries}}

\renewcommand{\theenumi}{\ensuremath{(\roman{enumi})}}

\renewcommand{\p@enumii}{}

\@namedef{subjclassname@2020}{%
  \textup{2020} Mathematics Subject Classification}
\makeatother

\newcommand{\bfa}{\mathbf{a}}
\newcommand{\bfx}{\mathbf{x}}
\newcommand{\bfy}{\mathbf{y}}

\newcommand{\BB}{\mathbf{B}}
\newcommand{\RR}{\mathbb{R}}

\newcommand{\levdist}{1.2cm}

\DeclareMathOperator{\card}{card}
\DeclareMathOperator{\diam}{diam}
\DeclareMathOperator{\dist}{dist}
\DeclareMathOperator{\lev}{lev}

\def\<#1>{\langle #1 \rangle}

\newcommand{\acr}{\newline\indent}

\author{Oleksiy Dovgoshey}
\address{\textbf{Oleksiy Dovgoshey}\acr
Institute of Applied Mathematics and Mechanics of the NAS of Ukraine,\acr
Sloviansk, Ukraine and\acr
Department of Mathematics and Statistics, University of Turku, \acr
Turku, Finland}
\email{oleksiy.dovgoshey@gmail.com, oleksiy.dovgoshey@utu.fi}

\title{Totally bounded ultrametric spaces and locally finite trees}



\begin{document}

\begin{abstract}
We investigate the interrelations between the metric properties, order properties and combinatorial properties of the set of balls in totally bounded ultrametric space. In particular, the Gurvich---Vyalyi representation of finite, ultrametric spaces by monotone rooted trees is generalized to the case of totally bounded ultrametric spaces. It is shown that such spaces have isometric completions if and only if their labeled representing trees are isomorphic. We characterize up to isomorphism the representing trees of these spaces and, up to order isomorphism, the posets of open balls in such spaces.
\end{abstract}

\maketitle

\tableofcontents

\section{Introduction}

In 2001 at the Workshop on General Algebra the attention of experts on the theory of lattices was paid to the following problem of I.~M.~Gelfand: Using graph theory describe up to isometry all finite ultrametric spaces~\cite{Lem2001}. An appropriate representation of finite, ultrametric spaces by monotone rooted trees was proposed by V.~Gurvich and M.~Vyalyi in~\cite{GV2012DAM}. A simple geometric description of Gurvich---Vyalyi representing trees was found in \cite{PD2014JMS}. This description allows us effectively use the Gurvich---Vyalyi representation in various problems associated with finite ultrametric spaces. In particular, this leads to a graph-theoretic interpretation of the Gomory---Hu inequality \cite{DPT2015}. A characterization of finite ultrametric spaces which are as rigid as possible also was obtained \cite{DPT2017FPTA} on the basis of the Gurvich---Vyalyi representation. Some other extremal properties of finite ultrametric spaces and related them properties of monotone rooted trees have been found in~\cite{DP2020pNUAA}. The interconnections between the Gurvich---Vyalyi representation and the space of balls endowed with the Hausdorff metric are discussed in~\cite{Dov2019pNUAA} (see also \cite{Qiu2009pNUAA, Qiu2014pNUAA, DP2018pNUAA, Pet2018pNUAA, Pet2013PoI}).

It is well-known that the sets of leaves of phylogenetic equidistant trees with additive metric are ultrametric. The finite equidistant trees can be considered as finite subtrees of the so-called \(R\)-trees (see~\cite{Ber2019SMJ} for some interesting results related to \(R\)-trees and ultrametrics). A description of interrelations between finite subtrees of \(R\)-trees and finite, monotone rooted trees can be found in~\cite{Dov2020TaAoG}. The categorical equivalence of trees and ultrametric spaces was investigated in \cite{H04} and \cite{Lem2003AU}.

The paper is organized as follows.

In Sections~\ref{sec2} and \ref{sec3} we collect together some properties of ultrametric balls and, respectively, of totally bounded metric spaces.

In Proposition~\ref{p2.3} of Section~\ref{sec4.1} we consider a generalization of strong triangle inequality for arbitrary cycle of points of ultrametric space and give some applications of this generalization in Propositions~\ref{p4.4} and \ref{p4.5}.

New interrelations between complete multipartite graphs and diametrical graphs of ultrametric spaces are found in Section~\ref{sec4.2}. In particular, using complete multipartite graphs we found new necessary and sufficient conditions under which metric spaces are ultrametric, Theorem~\ref{t5.18}. In Theorems~\ref{t2.25} and \ref{t2.35} of Section~5.3 we characterize bounded ultrametric spaces which are weakly similar to unbounded ones via diametrical graphs. Characteristic properties of functions preserving totally bounded ultrametrics are proved in Theorem~\ref{t5.15} of Section~5.3.

The distance set of totally bounded ultrametric spaces and separable metric spaces are characterized in Theorem~\ref{t5.10} and Proposition~\ref{p5.16} of Section~5.2.

In section~\ref{sec6}, we prove that the ballean (the set of all open balls) of totally bounded ultrametric space is a subset of all closed balls of this space, Corollary~\ref{c2.41}, and find conditions under which the converse inclusion holds, Corollary~\ref{c6.7}. Proposition~\ref{p3.10} and Theorem~\ref{t7.12} of the section show that the balleans of any two dense subsets of arbitrary ultrametric space are order isomorphic with respect to the partial order generated by set inclusion.

The balleans endowed with the Hausdorff metric are described in Theorem~\ref{t7.8} of Section~\ref{sec8} for compact, ultrametric space. Using this theorem we show that every totally bounded ultrametric space admits an isometric embedding in a compact, ultrametric space with the dense set of isolated points, Theorem~\ref{p6.11}.

The new results related to the representing trees of ultrametric spaces are mainly concentrated in Sections~\ref{sec9}, \ref{sec10} and \ref{sec11}. In particular, in Section~\ref{sec11} we consider the labeled trees which represent the totally bounded ultrametric spaces that generalizes the Gurvich---Vyalyi representation of finite ultrametric spaces. A constructive characterization of such trees is given in Theorem~\ref{t7.3}. Moreover, the corresponding characterizations for rooted trees without labeling and for free trees are given in Proposition~\ref{p10.17} and Theorem~\ref{t9.10}, respectively.

Theorem~\ref{t9.20} from Section~\ref{sec10} presents also a description of partial orders generated by choice of a root in arbitrary tree. That allows us to characterize in Theorem~\ref{t10.17} all posets which are order isomorphic to posets of the open balls from totally bounded ultrametric spaces.

\section{Ultrametric balls}
\label{sec2}

In the paper the ultrametric balls are used as a starting point in the construction of representing trees corresponding to totally bounded ultrametric spaces. The main goal of this section is to recall some known properties of balls in ultrametric spaces.

Let us start from basic concepts.

A \textit{metric} on a set $X$ is a function $d\colon X\times X\rightarrow \RR^+$ such that for all \(x\), \(y\), \(z \in X\)
\begin{enumerate}
\item $d(x,y)=d(y,x)$,
\item $(d(x,y)=0)\Leftrightarrow (x=y)$,
\item \(d(x,y)\leq d(x,z) + d(z,y)\).
\end{enumerate}

A metric space \((X, d)\) is \emph{ultrametric} if the \emph{strong triangle inequality}
\[
d(x,y)\leq \max \{d(x,z),d(z,y)\}.
\]
holds for all \(x\), \(y\), \(z \in X\). In this case the function \(d\) is called \emph{an ultrametric} on \(X\).

\begin{definition}\label{d2.2}
Let \((X, d)\) and \((Y, \rho)\) be metric spaces. A mapping \(\Phi \colon X \to Y\) is called to be an \emph{isometric embedding} of \((X, d)\) in \((Y, \rho)\) if
\[
d(x,y) = \rho(\Phi(x), \Phi(y))
\]
holds for all \(x\), \(y \in X\). In the case when \(\Phi\) is bijective, we say that it is an \emph{isometry} of \((X, d)\) and \((Y, \rho)\). The metric spaces are \emph{isometric} if there is an isometry of these spaces.
\end{definition}

Let \(A\) be a subset of a metric space \((X, d)\), \(A \subseteq X\). The quantity
\begin{equation}\label{e2.1}
\diam A = \diam (A, d) =\sup\{d(x,y) \colon x, y\in A\}
\end{equation}
is the \emph{diameter} of \(A\). If the inequality \(\diam A < \infty\) holds, then we say that \(A\) is a \emph{bounded} subset of \(X\).

We define the \emph{distance set} \(D(X)\) of a metric space \((X,d)\) as
\[
D(X) = D(X, d) = \{d(x, y) \colon x, y \in X\}.
\]
It is clear that \(\diam X = \sup D(X)\).

If \((X, d)\) is an ultrametric space and \(A\) is a nonempty subset of \(X\), then, using the strong triangle inequality, we can easily prove the equality
\begin{equation}\label{e2.3}
\diam A = \sup\{d(x,a) \colon x \in A\}
\end{equation}
for every \(a \in A\) \cite{Dov2019pNUAA}.

Let \((X, d)\) be a metric space. An \emph{open ball} with a \emph{radius} \(r > 0\) and a \emph{center} \(c \in X\) is the set
\[
B_r(c) = \{x \in X \colon d(c, x) < r\}.
\]

Let \(c\) be a point of a metric space \((X, d)\) and let \(r \geqslant 0\). Similarly to open balls we define the \emph{closed ball} \(\overline{B}_r(c)\) with a radius \(r\) and a center \(c\) as
\begin{equation}\label{e2.4}
\overline{B}_r(c) = \{x \in X \colon d(c, x) \leqslant r\}.
\end{equation}
Thus, the equality \(\overline{B}_r(c) = \{c\}\) holds if \(r = 0\).

Now we are ready to introduce the first, basic for us, definition.

\begin{definition}\label{d2.1}
The ballean of a metric space \((X, d)\) is the set \(\BB_{X} = \BB_{(X, d)}\) of all open balls in \((X, d)\).
\end{definition}

In what follows, we will also denote by \(\overline{\mathbf{B}}_X = \overline{\mathbf{B}}_{(X, d)}\) the set of all closed balls in \((X, d)\). An arbitrary \(B \in \mathbf{B}_X \cup \overline{\mathbf{B}}_X\) is a \emph{ball} in \((X, d)\).

\begin{remark}\label{r2.2}
The term ballean, which was introduced in Definition~\ref{d2.1}, is also used to mean a set \(X\) endowed with a special family of subsets of \(X \times X\). See, for example, \cite{PP2020MS, PB2003, PZ2007}.
\end{remark}

\begin{remark}\label{r2.5}
Usually, when defining closed balls, only strictly positive values of the radius \(r\) are used (see, for example, Definition~5.1.1 in \cite{Sea2007}). Breaking this tradition in present paper can be partially justified by following technical advantages:
\begin{enumerate}
\item \label{r2.5:s1} Every chain of closed balls in spherically complete ultrametric space has a meet and, in particular, every maximal chain of such balls contains the smallest ball.
\item \label{r2.5:s2} For each metric space \((X, d)\), the mapping
\[
X \ni x \mapsto \{x\} \in \overline{\BB}_{X}
\]
is an isometric embedding of \((X, d)\) in \(\overline{\BB}_{X}\) if we endowed \(\overline{\BB}_{X}\) with the Hausdorff metric.
\end{enumerate}
\end{remark}

Let \((X, d)\) be a metric space. A subset \(A\) of \(X\) is \emph{open} if for every \(a \in A\) there is \(r > 0\) such that \(B_r(a) \subseteq A\). A set \(C \subseteq X\) is said to be \emph{closed} if its complement \(X \setminus C\) is open. Moreover, we say that a set \(Y \subseteq X\) is \emph{clopen} if it is closed and open simultaneously.

For every metric space \((X, d)\), the closed balls are closed subsets of \(X\) and the open balls are open subsets of \(X\). A much stronger result is true if \((X, d)\) is an ultrametric space.

Recall that a point \(p\) in a metric space \((X, d)\) is \emph{isolated} if there is \(\varepsilon > 0\) such that \(d(p, x) > \varepsilon\) for every \(x \in X \setminus \{p\}\). If \(p\) is not an isolated point of \(X\), then \(p\) is called an \emph{accumulation point} of \(X\).

\begin{proposition}\label{p2.2}
Let \((X, d)\) be ultrametric. Then the following statements hold:
\begin{enumerate}
\item\label{p2.2:s1} For every \(r \in (0, \infty)\) and every \(c \in X\) the balls \(B_r(c)\) and \(\overline{B}_r(c)\) are clopen subsets of \(X\).
\item\label{p2.2:s2} If \(r = 0\), then the ball \(\overline{B}_r(c)\) is clopen if and only if \(c\) is an isolated point of \(X\).
\end{enumerate}
\end{proposition}

\begin{proof}
Statement~\ref{p2.2:s1} follows directly from Proposition~18.4 \cite{Sch1985}. Statement~\ref{p2.2:s2} is easy to prove for arbitrary metric space.
\end{proof}

The following proposition claims that every point of an arbitrary ultrametric ball is a center of that ball.

\begin{proposition}\label{p2.4}
Let \((X, d)\) be an ultrametric space. Then for every ball \(B_r(c)\) (\(\overline{B}_r(c)\)) and every \(a \in B_r(c)\) (\(a \in \overline{B}_r(c)\)) we have
\[
B_r(c) = B_r(a) \quad (\overline{B}_r(c) = \overline{B}_r(a)).
\]
\end{proposition}

\begin{proof}
It is a simple modification of Proposition~18.4~\cite{Sch1985}.
\end{proof}

Using Proposition~\ref{p2.4} it is easy to prove that if \(B\) is an open ball in ultrametric space \((X, d)\) and if \(B_1\) is an open ball in \((B, d|_{B \times B})\), then \(B_1\) is also an open ball in \((X, d)\).

\begin{corollary}\label{c2.4}
Let \((X, d)\) be an ultrametric space. Then the inclusion
\begin{equation}\label{c2.4:e1}
\BB_{B} \subseteq \BB_{X}
\end{equation}
holds for every \(B \in \BB_{X}\)
\end{corollary}

\begin{remark}\label{r2.4}
It is interesting to note that Proposition~\ref{p2.4} gives us a characteristic property of the ultrametric spaces (see Proposition~1.7 \cite{Dov2019pNUAA}). On the other hand, even if inclusion~\eqref{c2.4:e1} holds for every \(B \in \BB_{X}\), then, in general, it does not imply the ultrametricity of \((X, d)\). Indeed, if \(X = \RR\) and \(d\) is the standard Euclidean metric on \(\RR\), then we evidently have \eqref{c2.4:e1} for every \(B \in \BB_{X}\), but \((X, d)\) is not ultrametric.
\end{remark}

\begin{lemma}\label{l2.6}
Let \(\overline{B}_r(c)\) be a closed ball in an ultrametric space \((X, d)\) and let
\[
r_1 = \diam \overline{B}_r(c).
\]
Then the equality \(\overline{B}_{r_1}(c) = \overline{B}_r(c)\) holds.
\end{lemma}

\begin{proof}
Using the definition of \(\overline{B}_r(c)\) and equality \eqref{e2.3} with \(A = \overline{B}_r(c)\) and \(a = c\), we obtain the inequality \(r_1 \leqslant r\). Consequently, \(\overline{B}_{r_1}(c) \subseteq \overline{B}_r(c)\) holds. To prove the converse inclusion it suffices to note that \(d(x, c) \leqslant \diam \overline{B}_r(c) = r_1\) holds for every \(x \in \overline{B}_r(c)\), that implies \(\overline{B}_r(c) \subseteq \overline{B}_{r_1}(c)\).
\end{proof}

Lemma~\ref{l2.6} and Proposition~\ref{p2.4} give us the next statement.

\begin{proposition}\label{p2.7}
Let \((X, d)\) be an ultrametric space and let \(B_1\), \(B_2 \in \overline{\mathbf{B}}_X\). Then $B_1 \subseteq B_2$ holds if and only if $B_1 \cap B_2 \neq \varnothing$ and
\[
\diam B_1 \leqslant \diam B_2.
\]
\end{proposition}

\begin{remark}\label{r2.6}
In Proposition~\ref{p2.7} we cannot replace the condition
\begin{itemize}
\item \(B_1\), \(B_2 \in \overline{\mathbf{B}}_X\)
\end{itemize}
for the condition
\begin{itemize}
\item \(B_1\), \(B_2 \in \mathbf{B}_X\)
\end{itemize}
as the following example shows.
\end{remark}

\begin{example}\label{ex2.6}
Let \(X\) be the intersection of the set \(\mathbb{Q}\) of all rational numbers with the interval \([0,1]\) of the real line \(\RR\), \(X = \{x \in \mathbb{Q} \colon 0 \leqslant x \leqslant 1\}\). Following Delhomm\'{e}, Laflamme, Pouzet and Sauer \cite{DLPS2008TaiA}, we define an ultrametric \(d \colon X \times X \to \RR^{+}\) as
\begin{equation}\label{ex2.6:e1}
d(x, y) = \begin{cases}
0 & \text{if } x = y,\\
\max\{x, y\} & \text{if } x \neq y.
\end{cases}
\end{equation}
Let us consider in \((X, d)\) the open balls \(B_{r_1}(0)\) and \(B_{r_2}(0)\) having the common center \(0\) and the radii \(r_1 = 1\) and \(r_2 = 2\). Then we evidently have \(\diam B_{r_1}(0) = \diam B_{r_2}(0) = 1\), but
\[
B_{r_1}(0) = \mathbb{Q} \cap [0, 1) \neq B_{r_2}(0) = \mathbb{Q} \cap [0, 1].
\]
\end{example}

The next proposition is a modification of Proposition~18.5 \cite{Sch1985}.

\begin{proposition}\label{p2.5}
Let \((X, d)\) be an ultrametric space and let \(B_1\) and \(B_2\) be different balls in \((X, d)\). Then we have either \(B_1 \subset B_2\) or \(B_2 \subset B_1\) whenever \(B_1 \cap B_2 \neq \varnothing\). If \(B_1\) and \(B_2\) are disjoint, \(B_1 \cap B_2 = \varnothing\), then the equality
\begin{equation}\label{p2.5:e1}
d(x_1, x_2) = \diam (B_1 \cup B_2)
\end{equation}
holds for every \(x_1 \in B_1\) and every \(x_2 \in B_2\).
\end{proposition}

\begin{proof}
By Proposition~18.5 \cite{Sch1985}, we have
\[
(B_1 \subset B_2 \text{ or } B_2 \subset B_1) \Leftrightarrow (B_1 \cap B_2 \neq \varnothing)
\]
and, moreover, the same proposition claims the equality
\begin{equation}\label{p2.5:e5}
d(x_1, x_2) = d(y_1, y_2)
\end{equation}
for all \(x_1\), \(y_1 \in B_1\) and all \(x_2\), \(y_2 \in B_2\) whenever \(B_1 \cap B_2 = \varnothing\).

Write
\begin{equation}\label{p2.5:e6}
\Delta(B_1, B_2) = \sup\{d(x_1, x_2) \colon x_1 \in B_1 \text{ and } x_2 \in B_2\}.
\end{equation}
It follows directly from~\eqref{e2.1} that
\begin{equation}\label{p2.5:e2}
\diam (B_1 \cup B_2) = \max \bigl\{\diam B_1, \diam B_2, \Delta(B_1, B_2)\bigr\}.
\end{equation}
Now, using \eqref{p2.5:e5} and \eqref{p2.5:e2} we see that \eqref{p2.5:e1} holds if
\begin{equation}\label{p2.5:e3}
\diam B_i \leqslant \Delta(B_1, B_2), \quad i = 1, 2.
\end{equation}
Without loss of generality, we may assume \(\diam B_1 \geqslant \diam B_2\). Now, if \eqref{p2.5:e3} does not hold, then there are \(z_1\), \(y_1 \in B_1\) such that
\begin{equation}\label{p2.5:e4}
d(z_1, y_1) > \Delta(B_1, B_2).
\end{equation}
Let \(r_1\) be the radius of \(B_1\). Using Proposition~\ref{p2.4} we can consider \(z_1\) as a center of the ball \(B_1\). By inequality~\eqref{p2.5:e4},
\[
r_1 \geqslant d(z_1, y_1) > d(z_1, x_2)
\]
holds for every \(x_2 \in B_2\). Thus, \(B_2 \subseteq B_1\), contrary to \(B_1 \cap B_2 = \varnothing\).
\end{proof}

\begin{remark}\label{r2.10}
Let \((X, d)\) be a metric space. For any nonempty subsets \(A\) and \(B\) of \(X\), the \emph{distance} between \(A\) and \(B\) is defined as
\begin{equation}\label{r2.10:e1}
\dist (A, B) = \inf \{d(x, y) \colon x \in A, y \in B\}.
\end{equation}
It was shown in above mentioned Proposition~18.5 \cite{Sch1985} that \(d(x_1, x_2) = \dist (B_1, B_2)\) holds for all \(x_1 \in B_1\) and \(x_2 \in B_2\) if \(B_1\) and \(B_2\) are disjoint balls and \(d\) is an ultrametric. Thus, from \eqref{p2.5:e1} it follows that
\begin{equation}\label{r2.10:e2}
\dist (B_1, B_2) = \diam (B_1 \cup B_2) = \Delta(B_1, B_2)
\end{equation}
whenever \(B_1\) and \(B_2\) are disjoint, ultrametric balls and \(\Delta(B_1, B_2)\) is defined by~\eqref{p2.5:e6}.
\end{remark}

\begin{remark}\label{r2.13}
A partial generalization of Propositions~\ref{p2.4}, \ref{p2.7} and \ref{p2.5}, for ``balls'' whose points are finite, nonempty subsets of \(X\), is given in Proposition~2.1~\cite{DDP2011pNUAA}.
\end{remark}

\begin{definition}\label{d2.9}
Let \((X, d)\) be a metric space and let \(A\) be a nonempty, bounded subset of \(X\). A ball \(B^* = B^*(A) \in \overline{\mathbf{B}}_X\) is the smallest ball containing \(A\) if \(B^* \supseteq A\) and the implication
\begin{equation}\label{d2.9:e1}
(B \supseteq A) \Rightarrow (B \supseteq B^*)
\end{equation}
is valid for every \(B \in \overline{\mathbf{B}}_X\).
\end{definition}

It is clear that, for every nonempty, bounded \(A \subseteq X\), the smallest ball \(B^*(A)\) is unique if it exists. The next proposition implies, in particular, the existence of such balls in every nonempty, ultrametric space \((X, d)\).

\begin{proposition}[{\cite{Dov2019pNUAA}}]\label{p2.12}
Let \((X, d)\) be an ultrametric space, let \(A\) be a bounded subset of \(X\) and let \(a \in A\). Then the equality \(B^*(A) = \overline{B}_r(a)\) holds with \(r = \diam A\).
\end{proposition}

The following corollary of Proposition~\ref{p2.12} shows that condition~\eqref{d2.9:e1} is equivalent to some weaker condition, if \((X, d)\) is ultrametric.

\begin{corollary}[{\cite{Dov2019pNUAA}}]\label{c2.12}
Let \((X, d)\) be an ultrametric space, \(B_1 \in \overline{\BB}_{X}\) and let \(A\) be a nonempty bounded subset of \(B_1\). Then \(B_{1} = B^{*}(A)\) holds if and only if
\[
(B \supseteq A) \Rightarrow (\diam B \geqslant \diam B_{1})
\]
is valid for every \(B \in \overline{\mathbf{B}}_X\).
\end{corollary}

Proposition~\ref{p2.12} can be also reformulated as follows.

\begin{corollary}\label{c2.18}
Let \((X, d)\) be an ultrametric space, let \(A\) be a nonempty bounded subset of \(X\) and let \(B_1 \in \overline{\BB}_{X}\). Then \(B_{1} = B^{*}(A)\) holds if and only if \(A \cap B_1 \neq \varnothing\) and \(\diam B_1 = \diam A\).
\end{corollary}

\section{Totally bounded metric spaces}
\label{sec3}

The object of our studies is totally bounded ultrametric spaces. In the present section we mainly collect together some known properties of totally bounded metric spaces.

An important subclass of totally bounded metric spaces is the class of compact metric spaces.

\begin{definition}[Borel---Lebesgue property]\label{d2.3}
Let \((X, d)\) be a metric space. A set \(A\) of \(X\) is \emph{compact} if every family \(\mathcal{F} \subseteq \mathbf{B}_X\) satisfying the inclusion
\[
A \subseteq \bigcup_{B \in \mathcal{F}} B
\]
contains a finite subfamily \(\mathcal{F}_0 \subseteq \mathcal{F}\), \(|\mathcal{F}_0| < \infty\), such that
\[
A \subseteq \bigcup_{B \in \mathcal{F}_0} B.
\]
\end{definition}

A standard definition of compactness usually formulated as: Every open cover of \(A\) in \(X\) has a finite subcover.

There exists a simple interdependence between the compactness of a set \(A \subseteq X\) and the so-called convergent sequences in \(A\). A sequence \((x_n)_{n \in \mathbb{N}}\) of points in a metric space \((X, d)\) is said to converge to a point \(a \in X\),
\begin{equation}\label{e2.8}
\lim_{n \to \infty} x_n = a,
\end{equation}
if, for every open ball \(B \ni a\), it is possible to find an integer \(n_0 \geqslant 1\) such that \(x_n \in B\) for every \(n \geqslant n_0\). Thus, \eqref{e2.8} holds if and only if
\[
\lim_{n \to \infty} d(x_n, a) = 0.
\]

The following classical theorem was proved by Frechet.

\begin{proposition}[Bolzano---Weierstrass property]\label{p2.9}
A subset \(A\) of a metric space is compact if and only if every infinite sequence of points of \(A\) contains a subsequence which converges to a point of \(A\).
\end{proposition}

This result and a long list of distinct criteria of compactness can be found, for example, in \cite[p.~206]{Sea2007}.

It is clear that a subset \(A\) of a metric space \((X, d)\) is \emph{bounded} if there is a ball \(B\) in \((X, d)\) such that \(A \subseteq B\).

Now we present a definition of \emph{total boundedness} in the spirit of Borel---Lebesgue property.

\begin{definition}\label{d2.10}
A subset \(A\) of a metric space \((X, d)\) is totally bounded if for every \(r > 0\) there is a finite set \(\{B_r(x_1), \ldots, B_r(x_n)\} \subseteq \mathbf{B}_X\) such that
\[
A \subseteq \bigcup_{i = 1}^{n} B_r(x_i).
\]
\end{definition}

It is easy to see that every totally bounded set is bounded and every compact set is totally bounded.

To formulate an analog of Bolzano---Weierstrass property for totally bounded metric spaces we recall that, for metric space \((X, d)\), a sequence \((x_n)_{n \in \mathbb{N}} \subseteq X\) is a \emph{Cauchy sequence} in \((X, d)\) if, for every \(r > 0\), there is an integer \(n_0 = n_0(r) \geqslant 1\) such that \(x_n \in B_r(x_{n_0})\) for every \(n \geqslant n_0\). It is easy to see that \((x_n)_{n \in \mathbb{N}}\) is a Cauchy sequence if and only if
\[
\lim_{\substack{m \to \infty\\ n \to \infty}} d(x_n, x_m) = 0
\]
holds.

\begin{remark}\label{r2.17}
Here and later for a set \(X\) and sequence \((x_n)_{n\in \mathbb{N}}\) we write \((x_n)_{n\in \mathbb{N}} \subseteq X\) if and only if \(x_n \in X\) holds for every \(n \in \mathbb{N}\).
\end{remark}

\begin{proposition}\label{p2.11}
A subset \(A\) of a metric space \((X, d)\) is totally bounded if and only if every infinite sequence of points of \(A\) contains a Cauchy subsequence.
\end{proposition}

See, for example, Theorem~7.8.2 \cite{Sea2007}.

\begin{corollary}\label{c2.17}
If \(A\) is a totally bounded subset of \((X, d)\) and \(C\) is a subset of \(A\), then \(C\) is also totally bounded.
\end{corollary}

Let \((X, d)\) be a metric space and let \(A \subseteq X\). A subset \(S\) of \(A\) is said to be \emph{dense} in \(A\) if for every \(a \in A\) there is a sequence \((s_n)_{n\in \mathbb{N}} \subseteq S\) such that
\[
a = \lim_{n \to \infty} s_n.
\]

By definition, a metric space \((X, d)\) is \emph{separable} if \(X\) contains an at most countable, dense (in \(X\)) subset. It is clear that every totally bounded metric space is separable.

\begin{definition}\label{d2.18}
A metric space \((X, d)\) is \emph{complete} if every Cauchy sequence in \((X, d)\) converges to a point of \(X\).
\end{definition}

It can be proved that for a given metric space \((X, d)\) there exists exactly one (up to isometry) complete metric space \((\widetilde{X}, \widetilde{d})\) such that \(\widetilde{X}\) contains a dense, isometric to \((X, d)\), subset (see, for example, Theorem~4.3.19 \cite{Eng1989}). The metric space \((\widetilde{X}, \widetilde{d})\) satisfying the above condition is called the \emph{completion of} \((X, d)\).

\begin{proposition}\label{p2.13}
The completion \((\widetilde{X}, \widetilde{d})\) of a metric space \((X, d)\) is compact if and only if \((X, d)\) is totally bounded.
\end{proposition}

For the proof see, for example, Corollary~4.3.30 in \cite{Eng1989}.

\begin{lemma}\label{l2.20}
The completion \((\widetilde{X}, \widetilde{d})\) is ultrametric for every ultrametric \((X, d)\).
\end{lemma}

\begin{proof}
Since \(\widetilde{d}\) is a metric on \(\widetilde{X}\), to prove that \((\widetilde{X}, \widetilde{d})\) is ultrametric it suffices to justify
\begin{equation}\label{l2.20:e2}
\widetilde{d} (x, y) \leqslant \max\{\widetilde{d} (x, z), \widetilde{d} (z, y)\}
\end{equation}
for all \(x\), \(y\), \(z \in \widetilde{X}\). By definition, the metric space \((\widetilde{X}, \widetilde{d})\) contains a dense subspace \((X', \widetilde{d}|_{X' \times X'})\) which is isometric to \((X, d)\). For simplicity we may assume \(X = X'\) and \(\widetilde{d}|_{X' \times X'} = d\).

Let \((x_n)_{n \in \mathbb{N}}\), \((y_n)_{n \in \mathbb{N}}\) and \((z_n)_{n \in \mathbb{N}}\) be sequences of points of \(X\) such that
\begin{equation}\label{l2.20:e3}
\lim_{n \to \infty} x_n = x, \quad \lim_{n \to \infty} y_n = y \quad \text{and} \quad \lim_{n \to \infty} z_n = z.
\end{equation}
Since \(\widetilde{d} \colon \widetilde{X} \times \widetilde{X} \to \RR^{+}\) is a continuous mapping and \(d \colon X \times X \to \RR^{+}\) is an ultrametric, \eqref{l2.20:e3} implies
\begin{align*}
\widetilde{d} (x, y) = \lim_{n \to \infty} d(x_n, y_n) &\leqslant \max\left\{\lim_{n \to \infty} d(x_n, z_n), \lim_{n \to \infty} d(z_n, x_n)\right\} \\
&= \max\{\widetilde{d} (x, z), \widetilde{d} (z, y)\}.
\end{align*}
Inequality~\eqref{l2.20:e2} follows.
\end{proof}

The following result is a modification of Lemma~3.19 from~\cite{BDKP2017AASFM}.

\begin{proposition}\label{p3.9}
Let \((X, d)\) and \((Y, \rho)\) be compact metric spaces and let \(\Phi \colon X \to Y\) be an isometric embedding. If there is an isometric embedding \(\Psi \colon Y \to X\), then \(\Phi\) is an isometry.
\end{proposition}

\begin{proof}
Suppose that there is an isometric embedding \(\Psi \colon Y \to X\). Let \(F \colon X \to X\) be the isometric self-embedding for which the following diagram
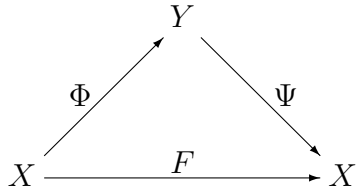
\begin{equation*}
\ctdiagram{
\ctv 0,0:{X}
\ctv 100,0:{X}
\ctv 50,50:{Y}
\ctel 0,0,50,50:{\Phi}
\ctet 0,0,100,0:{F}
\cter 50,50,100,0:{\Psi}
}
\end{equation*}
is commutative. The mapping \(\Phi \colon X \to Y\) is an isometry if and only if \(\Phi(X) = Y\) holds. Assume that \(\Phi(X) \neq  Y\). Then the set \(F(X)\) is a proper subset of \(X\) and, consequently, we can find a point \(p \in X \setminus F(X)\). Let us define a sequence \((p_i)_{i \in \mathbb{N}} \subseteq X\) as
\begin{equation}\label{p3.9:e1}
p_1 = p, \quad p_2 = F(p_1), \quad p_3 = F(p_2), \ldots
\end{equation}
an so on. Since \(F\) is continuous, the set \(F(X)\) is a compact subset of \(X\) \cite[Theorem~3.1.10]{Eng1989}. Hence, the number
\[
\varepsilon_0 = \inf_{x \in F(X)} d (p_1, x)
\]
is strictly positive, \(\varepsilon_0 > 0\). It follows from~\eqref{p3.9:e1} that \(p_i\) belongs to \(F(X)\) for every \(i \geqslant 2\). Consequently, we have the inequality \(d (p_1, p_i) \geqslant \varepsilon_0\) for every \(i \geqslant 2\). The function \(F\) preserves the distances. It implies
\[
d(p_j, p_i) = d(F(p_{j-1}), F(p_{i-1})) = d(p_{j-1}, p_{i-1}) = \ldots = d(p_{1}, p_{i-(j-1)}) \geqslant \varepsilon_0
\]
whenever \(i > j\). Thus, \(d(p_j, p_i) \geqslant \varepsilon_0 > 0\) holds if \(i \neq j\). In particular, the sequence \((p_i)_{i \in \mathbb{N}}\) has not any convergent subsequence, contrary to Proposition~\ref{p2.9}. This contradiction implies \(F(X) = X\).
\end{proof}

The following lemma will be useful for description of distance sets of totally bounded ultrametric spaces.

\begin{lemma}\label{l6.4}
For every compact metric space \((X, d)\) the distance set \(D(X)\) is a compact subset of \((\RR^{+}, \rho)\) with respect to the standard Euclidean metric \(\rho\).
\end{lemma}

\begin{proof}
Let \((X, d)\) be compact. The metric \(d \colon X \times X \to \RR^{+}\) is a continuous function on the Cartesian product \(X \times X\). Any Cartesian product of compact spaces is compact by Tychonoff theorem (\cite[Theorem~3.2.4]{Eng1989}). Furthermore, every continuous image of compact Hausdorff space is compact (\cite[Theorem~3.1.10]{Eng1989}). To complete the proof it suffices to note that \((X, d)\) evidently is a Hausdorff space and that Cartesian products of Hausdorff spaces are Hausdorff (\cite[Theorem~2.3.11]{Eng1989}).
\end{proof}

The following corollary of Lemma~\ref{l6.4} is a reformulation of Theorem~68 from \cite{Kap1977}.

\begin{corollary}\label{c3.10}
For every nonempty, compact metric space \((X, d)\) there are points \(x_1\), \(x_2 \in X\) such that \(d(x_1, x_2) = \diam X\).
\end{corollary}

We conclude the section with an important example of compact ultrametric space.

\begin{example}[The Cantor Set as the unit ball in \(\mathbb{Q}_2\)]\label{ex3.11}
Recall that, for every prime number \(p \geqslant 2\), the \(p\)-adic valuation of \(t \in \mathbb{Q}\) is defined as
\begin{equation}\label{e3.2}
|t|_p = \begin{cases}
p^{-\gamma} & \text{if } t \neq 0,\\
0 & \text{if } t = 0,
\end{cases}
\end{equation}
where \(\gamma = \gamma(t)\) is the unique integer number such that
\[
t = p^{\gamma} \frac{m}{n}
\]
and \(m\), \(n\) are coprime to \(p\). The \(p\)-adic valuation \(|\cdot|_p\) satisfies the non-Archimedean property
\[
|t+w|_p \leqslant \max \{|t|_p, |w|_p\}
\]
which implies that the mapping
\begin{equation}\label{e3.3}
d_p \colon \mathbb{Q} \times \mathbb{Q} \to \RR^{+}, \quad d_p(t, w) = |t-w|_p,
\end{equation}
is an ultrametric on \(\mathbb{Q}\) (formula (1.10), Chapter~I, \cite{Bachman1964}). The completion of the field \(\mathbb{Q}\) of rational numbers with respect to ultrametric \(d_p \colon \mathbb{Q} \times \mathbb{Q} \to \RR^{+}\) is the field of \(p\)-adic numbers and it is denoted by \(\mathbb{Q}_p\) (see, for example, Chapter II in \cite{Bachman1964} for details). For simplicity, we will always assume that the field \(\mathbb{Q}\) is a subfield of the field \(\mathbb{Q}_p\) and will preserve the symbol \(d_p\) for the metric on the completion \(\mathbb{Q}_p\) of \((\mathbb{Q}, d_p)\). Every \(p\)-adic number \(x\) has the unique canonical form
\begin{equation}\label{ex3.11:e2}
x = \sum_{i \in \mathbb{Z}} d_i(x) p^{i}
\end{equation}
where \(d_i \in \{0, 1, 2,\ldots, p-1\}\) for all \(i \in \mathbb{Z}\) and there is \(i_0 = i_0(x) \in \mathbb{Z}\) such that \(d_i = 0\) whenever \(i \leqslant i_0\) (Theorem~2.1, Chapter~II, \cite{Bachman1964}). In particular, every member of \(\mathbb{Q}_2\) can be written in form~\eqref{ex3.11:e2} with \(d_k \in \{0, 1\}\).

Let \(D = \{0, 1\}\) be a discrete two-point topological space. Following Engelking \cite{Eng1989}, we denote by \(D^{\aleph_0}\) the Cartesian product
\[
\prod_{i \in \mathbb{N}} D_i,
\]
endowed with the Tychonoff topology, where \(D_i = D\) for every \(i \in \mathbb{N}\), so every point \(x \in D^{\aleph_0}\) is a sequence \((\beta_i(x))_{i \in \mathbb{N}} \subseteq D\). Now to \(x \in D^{\aleph_0}\), we can associate an unique formal power series
\begin{equation}\label{ex3.11:e1}
\sum_{i=1}^{\infty} \beta_i(x) 2^{i-1} = \beta_1(x) 2^{0} + \beta_2(x) 2^{1} + \ldots + \beta_i(x) 2^{i-1} + \ldots
\end{equation}
with \(\beta_i(x) \in \{0, 1\}\) and, conversely, for every \(\sum_{i=1}^{\infty} \beta_i 2^{i-1}\) with \(\beta_i \in \{0, 1\}\), there is an unique \(x \in D^{\aleph_0}\) such that
\[
\sum_{i=1}^{\infty} \beta_i(x) 2^{i-1} = \sum_{i=1}^{\infty} \beta_i 2^{i-1}.
\]
The inequality \(|\beta_i(x) 2^{i-1}|_2 \leqslant 2^{1-i}\) implies the convergence of series~\eqref{ex3.11:e1} in \(\mathbb{Q}_2\). It is easy to prove that the image of \(D^{\aleph_0}\) under mapping
\begin{equation}\label{ex3.11:e3}
D^{\aleph_0} \ni x \mapsto \sum_{i=1}^{\infty} \beta_i (x) 2^{i-1} \in \mathbb{Q}_2
\end{equation}
coincides with the unit closed ball
\[
\overline{B}_1(0) = \{x \in \mathbb{Q}_2 \colon d_2(0, x) \leqslant 1\}.
\]
The ball \(\overline{B}_1(0)\) is a compact, ultrametric space (Theorem~5.1 \cite{Sch1985}). It follows from the definitions of Tychonoff topology and of the metric \(d_2\) that bijection~\eqref{ex3.11:e3} is continuous. Since every continuous bijection of compact spaces is a homeomorphism (\cite[p.~125, Theorem~3.1.13]{Eng1989}), the spaces \(\overline{B}_1(0)\) and \(D^{\aleph_0}\) are homeomorphic.
\end{example}

\begin{remark}
The unit closed ball \(\overline{B}_1(0)\) in the field \(\mathbb{Q}_2\) has the rich algebraic structure. It is the (Hensel's) ring \(\mathbb{Z}_2\) of \(2\)-adic integer numbers and it is also an integral domain for which \(\mathbb{Q}_2\) is the ``smallest'' field containing this domain. We also note that the ring \(\mathbb{Z}\) of all usual integer numbers is a dense subset of \(\mathbb{Z}_2\).
\end{remark}

\begin{remark}
Hensel's ring \(\mathbb{Z}_2\) is not the only known model for \(D^{\aleph_0}\). Among the many spaces homeomorphic to \(D^{\aleph_0}\), the most popular and well-studied is apparently the Cantor ternary set. It is a perfect subset of \(\RR\) which was defined in Cantor's paper \cite{Can1884AM}. The interesting information related to Cantor's set and Cantor's function can be found in \cite{Val2013} and \cite{DMRV2006EM}. Section~\ref{sec9} of the present paper contains a new description of all homeomorphic to \(D^{\aleph_0}\), ultrametric spaces on the language of representing trees (see Proposition~\ref{p9.10}).
\end{remark}

\section{To ultrametricity via cycles and complete multipartite graphs}
\label{sec4}

At the beginning of this section, we recall the concepts of a (simple) graph and cycle. A characterization of cycles of points of ultrametric space is given in Proposition~\ref{p2.3}. The first part of the section also contains several applications of this proposition in the proofs of properties of ultrametric spaces.

Theorem~\ref{t5.18} from the second part of the section gives us a characterization of ultrametric spaces with the help of complete multipartite graphs. In Theorem~\ref{t2.24}, it is shown that the diametrical graph of ultrametric space is nonempty if and only if it is a complete multipartite graph. Using this result we prove that, for every totally bonded ultrametric space, the diametrical graph is complete \(k\)-partite with some \(k \in \mathbb{N}\) (see Proposition~\ref{p2.36}).

\subsection{Cycles in ultrametric spaces}
\label{sec4.1}

A \textit{simple graph} is a pair $(V,E)$ consisting of a nonempty set $V$ and a set $E$ whose elements are unordered pairs \(\{u, v\}\) of different points \(u\), \(v \in V\). For a graph $G=(V,E)$, the sets $V=V(G)$ and $E=E(G)$ are called \textit{the set of vertices (or nodes)} and \textit{the set of edges}, respectively. We say that \(G\) is \emph{empty} if \(E(G) = \varnothing\). A graph \(G\) is \emph{finite} if \(V(G)\) is a finite set, \(|V(G)| < \infty\). If $\{x,y\} \in E(G)$, then the vertices $x$ and $y$ are \emph{adjacent}, and we say that \(x\) and \(y\) are \emph{ends} of \(\{x,y\}\).

A finite graph $C$ is a \textit{cycle} if $|V(C)|\geq 3$ and there exists an enumeration \(v_1\), \(v_2\), \(\ldots\), \(v_n\) of its vertices such that
\begin{equation*}
(\{v_i,v_j\}\in E(C))\Leftrightarrow (|i-j|=1\quad \mbox{or}\quad |i-j|=n-1).
\end{equation*}

Using the concept of cycle we can reformulate the condition of ultrametricity as follows.

\begin{proposition}\label{p2.3}
Let \((X, d)\) be a metric space. Then the following two conditions are equivalent:
\begin{enumerate}
\item \label{p2.3:s1} \((X, d)\) is ultrametric.
\item \label{p2.3:s2} For every cycle \(C\) with \(V(C) \subseteq X\) there exist at least two different edges \(\{x_1, y_1\}\), \(\{x_2, y_2\} \in E(C)\) such that
\begin{equation}\label{p2.3:e1}
d(x_1, y_1) = d(x_2, y_2) = \max\{d(x, y) \colon \{x, y\} \in E(C)\}.
\end{equation}
\end{enumerate}
\end{proposition}

\begin{proof}
\(\ref{p2.3:s1} \Rightarrow \ref{p2.3:s2}\). Let \(|E(C)|\) be the number of edges in a cycle \(C\), \(V(C) \subseteq X\). If \(|E(C)| = 3\), then \eqref{p2.3:e1} follows from the strong triangle inequality. Assume that \eqref{p2.3:e1} is true for \(|E(C)| \leqslant n\), but that there exists a cycle with \(|E(C)| = n + 1\) which contains exactly one edge \(\{x, y\}\) such that
\[
d(x, y) = \max\{d(u, v) \colon \{u, v\} \in E(C)\}.
\]
Let \(z\) be the vertex of cycle \(C\) which is adjacent to \(y\) but distinct from \(x\). In view of the uniqueness of the edge of maximum length, we have
\[
d(y, z) < d(x, y).
\]
Hence, because of the strong triangle inequality, this implies \(d(x, z) = d(x, y)\). Let \(C_1\) be the cycle for which
\[
V(C_1) = V(C) \setminus \{y\}, \quad E(C_1) = \bigl(E(C) \setminus \bigl\{\{x, y\}, \{y, z\}\bigr\}\bigr) \cup \{x, z\}.
\]
Then \(|E(C_1)| = n\) and \(\{x, z\}\) is the unique edge of maximum length, which contradicts the inductive hypothesis.

\(\ref{p2.3:s2} \Rightarrow \ref{p2.3:s1}\). For the proof of validity of \(\ref{p2.3:s2} \Rightarrow \ref{p2.3:s1}\) we note that \ref{p2.3:s2} implies the strong triangle inequality if \(|V(C)| = 3\).
\end{proof}

\begin{remark}\label{r4.2}
Proposition~\ref{p2.3} is, in fact, a simple modification of Lemma~1 from~\cite{DP2013SM}.
\end{remark}

Let us consider some useful corollaries of Proposition~\ref{p2.3}. The following result can be found in~\cite{Comicheo2018} (see formula (12), Theorem~1.6).

\begin{corollary}\label{c2.15}
Let \((X, d)\) be a nonempty ultrametric space and let \(A\) be a dense subset of \(X\). Then the distance sets of \((X, d)\) and of \((A, d|_{A \times A})\) are the same,
\begin{equation}\label{c2.15:e1}
D(A) = D(X).
\end{equation}
\end{corollary}

\begin{proof}
It is clear that \(D(A) \subseteq D(X)\) and \(0 \in D(A)\). Consequently, \eqref{c2.15:e1} holds if and only if for each pair of distinct \(x_1\), \(x_2 \in X\) there are \(a_1\), \(a_2 \in A\) satisfying
\begin{equation}\label{c2.15:e2}
d(x_1, x_2) = d(a_1, a_2).
\end{equation}
Since \(x_1 \neq x_2\) and \(A\) is dense in \(X\), we can find \(a_1\), \(a_2 \in A\) such that
\begin{equation}\label{c2.15:e3}
d(x_1, x_2) > d(x_i, a_i), \quad i = 1, 2.
\end{equation}
If \(x_1 = a_1\) or \(x_2 = a_2\), then \eqref{c2.15:e2} follows from the strong triangle inequality. Suppose that \(x_1 \neq a_1\) and \(x_2 \neq a_2\). Using \eqref{c2.15:e3} we can easily prove that the points \(a_1\), \(x_1\), \(x_2\), \(a_2\) are pairwise distinct. Let us consider the cycle \(C\) with
\[
V(C) = \{a_1, x_1, x_2, a_2\} \quad \text{and} \quad E(C) = \bigl\{\{a_1, x_1\}, \{x_1, x_2\}, \{x_2, a_2\}, \{a_2, a_1\}\bigr\}.
\]
Equality~\eqref{c2.15:e2} follows from~\eqref{c2.15:e3} by Proposition~\ref{p2.3}.
\end{proof}

Corollary~\ref{c2.15} and Lemma~\ref{l2.20} imply the following well-known result: ``No new values of the ultrametric after completion'' (see, for example, \cite[p.~4]{PerezGarcia2010} or \cite[Theorem~1.6, Statement~(12)]{Comicheo2018}).

\begin{proposition}\label{p4.4}
Let \((X, d)\) be an ultrametric space and let \((\widetilde{X}, \widetilde{d})\) be the completion of \((X, d)\). Then the equality
\begin{equation}\label{p4.4:e1}
D(X) = D(\widetilde{X})
\end{equation}
holds.
\end{proposition}

In the next known proposition (see, for example, \cite[p.~4]{PerezGarcia2010} or \cite[Theorem~1.6, Statement~(13)]{Comicheo2018}) we consider a comfortable ``ultrametric modification'' of the notion of Cauchy sequence.

\begin{proposition}\label{p4.5}
Let \((X, d)\) be an ultrametric space. A sequence \((x_n)_{n \in \mathbb{N}} \subseteq X\) is a Cauchy sequence if and only if the limit relation
\begin{equation}\label{p4.5:e1}
\lim_{n \to \infty} d(x_n, x_{n+1}) = 0
\end{equation}
holds.
\end{proposition}

\begin{proof}
Let \((x_n)_{n \in \mathbb{N}} \subseteq X\) be a Cauchy sequence. Then \eqref{p4.5:e1} follows directly from the definition.

Let \((x_n)_{n \in \mathbb{N}} \subseteq X\) and \eqref{p4.5:e1} hold. Then, for every \(r > 0\), there is \(n_0 = n_0(r)\in \mathbb{N}\) such that
\begin{equation}\label{p4.5:e2}
d(x_n, x_{n+1}) \leqslant r
\end{equation}
whenever \(n \geqslant n_0\). The sequence \((x_n)_{n \in \mathbb{N}}\) is a Cauchy sequence if \eqref{p4.5:e2} implies
\[
d(x_n, x_{n+k}) \leqslant r
\]
for every \(k \in \mathbb{N}\) and every \(n \geqslant n_0(r)\). Suppose contrary that there are \(n \geqslant n_0(r)\) and \(k \in \mathbb{N}\) such that
\begin{equation}\label{p4.5:e3}
d(x_n, x_{n+k}) > r.
\end{equation}
Since \eqref{p4.5:e2} holds for every \(n \geqslant n_0(r)\), we have the inequality \(k \geqslant 2\). Let us consider the cycle \(C\) with \(V(C) = \{x_n, x_{n+1}, \ldots, x_{n+k}\}\) and
\[
E(C) = \{\{x_n, x_{n+1}\}, \ldots, \{x_{n+k-1}, x_{n+k}\}, \{x_{n+k}, x_{n}\}\}.
\]
Using \eqref{p4.5:e3} and \eqref{p4.5:e2} we obtain
\[
d(x_n, x_{n+k}) > r \geqslant d(x, y)
\]
whenever \(\{x, y\} \in E(C)\) and \(\{x, y\} \neq \{x_{n+k}, x_{n}\}\). The last statement contradicts Proposition~\ref{p2.3}.
\end{proof}

\begin{example}\label{ex4.6}
The field \(\mathbb{Q}_p\) of \(p\)-adic numbers is a complete ultrametric space with respect to the \(p\)-adic norm \(d_p\). Consequently, the infinite series \(\sum_{k=1}^{\infty} x_k\) is convergent in \(\mathbb{Q}_p\) if the sequence \((S_n)_{n \in \mathbb{N}} \subseteq \mathbb{Q}_p\) of partial sums \(S_n = \sum_{k=1}^{n} x_k\) is a Cauchy sequence. Hence, by Proposition~\ref{p4.5}, \(\sum_{k=1}^{\infty} x_k\) is convergent if and only if
\[
\lim_{n \to \infty} d_p(S_{n}, S_{n-1}) = \lim_{n \to \infty} |S_n - S_{n-1}|_p = \lim_{n \to \infty} |x_{n}|_p = 0
\]
(see, for example, \cite[Chapter~II, Statement~(1.2)]{Monna1970}).
\end{example}

The above example shows that for every series \(\sum_{k=1}^{\infty} x_k\) in \(\mathbb{Q}_p\) and each sequence \((a_k)_{k \in \mathbb{N}} \subseteq \mathbb{Q}_p\) with
\[
0 < \liminf_{k \to \infty} |a_k|_p \leqslant \limsup_{k \to \infty} |a_k|_p < \infty,
\]
the convergence of \(\sum_{k=1}^{\infty} x_k\) and of \(\sum_{k=1}^{\infty} a_k x_k\) are equivalent. In particular, it makes no sense to talk about convergent but not absolutely convergent series in \(\mathbb{Q}_p\).

\subsection{Complete multipartite graphs and ultrametric spaces}
\label{sec4.2}

The notion of complete multipartite graph \(G\) is well-known when the vertex set \(V(G)\) is finite. Below we need this concept for graphs having the vertex sets of arbitrary cardinality.

\begin{definition}\label{d2.6}
Let \(G\) be a graph and let \(k \geqslant 2\) be a cardinal number. The graph \(G\) is complete \(k\)-partite if the vertex set \(V(G)\) can be partitioned into \(k\) nonvoind, disjoint subsets, or parts, in such a way that no edge has both ends in the same part and any two vertices in different parts are adjacent.
\end{definition}

We shall say that $G$ is a \emph{complete multipartite graph} if there is a cardinal number \(k\) such that $G$ is complete $k$-partite. It is easy to see \(G\) is complete multipartite if and only if the non-adjacency is an equivalence relation on \(V(G)\) \cite[p.~14]{Die2005}.

Our next definition is a modification of Definition~2.1 from~\cite{PD2014JMS}.

\begin{definition}\label{d5.2}
Let $(X,d)$ be a nonempty metric space. Denote by \(G_{X,d}\) a graph such that \(V(G_{X,d}) = X\) and, for \(u\), \(v \in V(G_{X,d})\),
\begin{equation}\label{d5.2:e1}
(\{u,v\}\in E(G_{X,d}))\Leftrightarrow (d(u,v)=\diam X \text{ and } u \neq v).
\end{equation}
We call $G_{X,d}$ the \emph{diametrical graph} of \((X, d)\).
\end{definition}

\begin{example}\label{ex2.24}
If \((X, d)\) is an unbounded metric space or \(|X| = 1\) holds, then the diametrical graph \(G_{X,d}\) is empty, \(E(G_{X,d}) = \varnothing\).
\end{example}

\begin{theorem}\label{t2.24}
Let \((X, d)\) be an ultrametric space with \(|X| \geqslant 2\). Then the following statements are equivalent:
\begin{enumerate}
\item\label{t2.24:s1} The diametrical graph \(G_{X,d}\) of \((X, d)\) is nonempty.
\item\label{t2.24:s2} The diametrical graph \(G_{X,d}\) is complete multipartite.
\item\label{t2.24:s3} There are some points \(x_1\), \(x_2 \in X\) such that \(d(x_1, x_2) = \diam X\).
\end{enumerate}
Furthermore, if \(G_{X, d}\) is complete multipartite, then every part of \(G_{X, d}\) is an open ball with a center in \(X\) and the radius equals \(\diam X\) and, conversely, every open ball \(B_r(c)\) with \(r = \diam X\) and \(c \in X\) is a part of \(G_{X, d}\).
\end{theorem}

\begin{proof}
The validity of \(\ref{t2.24:s1} \Leftrightarrow \ref{t2.24:s2}\) follows from Theorems~3.1 and 3.2 of paper~\cite{DDP2011pNUAA}. The equivalence \(\ref{t2.24:s1} \Leftrightarrow \ref{t2.24:s3}\) is obviously true.

Let \(G_{X, d}\) be a complete multipartite graph, let \(X_{1}\) be a part of \(G_{X, d}\) and let \(x_{1}\) be a point of \(X_{1}\). We claim that the equality
\begin{equation}\label{t2.24:e1}
X_{1} = B_r(x_{1})
\end{equation}
holds with \(r = \diam X\). Using Example~\ref{ex2.24}, we see that the double inequality \(0 < \diam X < \infty\) holds. Hence, the open ball \(B_r(x_1)\) is correctly defined.

Let \(x_{2}\) be a point of the set \(X \setminus X_1\). Since \(G_{X, d}\) is complete multipartite and \(x_{2} \notin X_{1}\), the membership
\begin{equation}\label{t2.24:e3}
\{x_1, x_2\} \in E(G_{X, d})
\end{equation}
is valid. From \eqref{t2.24:e3} it follows that
\[
d(x_1, x_2) = \diam X = r.
\]
Hence, \(x_2 \in X \setminus B_r(x_1)\). Thus, the inclusion
\begin{equation}\label{t2.24:e2}
X \setminus X_1 \subseteq X \setminus B_r(x_1)
\end{equation}
holds.

Similarly, we can prove the inclusion \(X \setminus B_r(x_1) \subseteq X \setminus X_1\). The last inclusion and \eqref{t2.24:e2} imply the equality
\[
X \setminus B_r(x_1) = X \setminus X_1.
\]
This is equivalent to \eqref{t2.24:e1}.

Let us consider now an open ball \(B_r(c)\) with \(r = \diam X\) and arbitrary \(c \in X\). Then there is a part \(X_2\) of \(G_{X, d}\) such that \(c \in X_2\). Arguing as in the proof of equality \eqref{t2.24:e1}, we obtain the equality \(X_2 = B_r(c)\).
\end{proof}

\begin{proposition}\label{p2.36}
Let \((X, d)\) be a totally bounded ultrametric space with \(|X| \geqslant 2\). Then there is an integer \(k \geqslant 2\) such that the diametrical graph \(G_{X, d}\) is complete \(k\)-partite.
\end{proposition}

\begin{proof}
First of all we claim that \(G_{X, d}\) is nonempty.

Let \((\widetilde{X}, \widetilde{d})\) be the completion of \((X, d)\). Then \((\widetilde{X}, \widetilde{d})\) is a compact metric space. By Corollary~\ref{c3.10} and Theorem~\ref{t2.24}, the diametrical graph of compact metric space is nonempty if this space contains at least two points. Consequently, \(G_{\widetilde{X}, \widetilde{d}}\) is a nonempty graph. Now the last statement and Corollary~\ref{c2.15} imply that \(G_{X, d}\) is also nonempty.

It is shown that the graph \(G_{X, d}\) is nonempty. Hence, by Theorem~\ref{t2.24}, this graph is complete multipartite. Consequently, there is a cardinal number \(k\) such that \(G_{X, d}\) is complete \(k\)-partite. If \(k\) is an infinite cardinal, then using Definitions~\ref{d2.6} and \ref{d5.2} we can find a sequence \((p_n)_{n \in \mathbb{N}} \subseteq X\) such that
\[
d(p_n, p_m) = \diam(X, d) > 0
\]
for all distinct \(n\), \(m \in \mathbb{N}\). Hence, the sequence \((p_n)_{n \in \mathbb{N}}\) does not contain any Cauchy subsequence, that contradicts Proposition~\ref{p2.11}.
\end{proof}

\begin{corollary}\label{c2.39}
Let \((X, d)\) be a totally bounded ultrametric space and let \(A\) be a subset of \(X\) such that \(|A| \geqslant 2\). Then the diametrical graph \(G_{A, d|_{A \times A}}\) is complete \(k\)-partite for an integer \(k \geqslant 2\).
\end{corollary}

\begin{proof}
It follows from Corollary~\ref{c2.17} and Proposition~\ref{p2.36}.
\end{proof}

\begin{example}\label{ex5.15}
Let \(\mathbb{Z}\) be the ring of all integer numbers and let
\[
\overline{B}_1(0) = \{x \in \mathbb{Q}_p \colon d_p(x, 0) \leqslant 1\}
\]
be the unit closed ball in the ultrametric space \((\mathbb{Q}_p, d_p)\) of \(p\)-adic numbers. Then \(\overline{B}_1(0)\) is a compact infinite subset of \((\mathbb{Q}_p, d_p)\) (Theorem~5.1, \cite{Sch1985}). Hence, by Proposition~\ref{p2.36}, the diametrical graph \(G_{\overline{B}_1(0), d_p|_{\overline{B}_1(0) \times \overline{B}_1(0)}}\) is complete \(k\)-partite with some integer \(k \geqslant 2\). Since the ball \(\overline{B}_1(0)\) can be written as disjoint union of open balls,
\begin{equation}\label{ex5.15:e1}
\overline{B}_1(0) = B_1(0) \cup B_1(1) \cup \ldots \cup B_1(p-1),
\end{equation}
(see, for example, Problem~50 in \cite{Gou1993}) the diametrical graph of \(\overline{B}_1(0)\) is complete \(p\)-partite with the parts \(B_1(i) \in \BB_{\mathbb{Q}_p}\), \(i=0\), \(1\), \(\ldots\), \(p-1\). The set \(\mathbb{Z}\) is a dense subset of \(\overline{B}_1(0)\) (see, for example, statement~\((ii)\), Proposition~3.3.4 \cite{Gou1993}). Using \eqref{ex5.15:e1} we see that the diametrical graph \(G_{\mathbb{Z}, d_p|_{\mathbb{Z} \times \mathbb{Z}}}\) is also complete \(p\)-partite with the parts
\[
B^{k} = B_1(k) \cap \mathbb{Z}, \quad k=0, 1, \ldots, p-1
\]
(see the proof of Proposition~\ref{p2.36} for details). It is interesting to note that the parts \(B^{0}\), \(\ldots\), \(B^{p-1}\) are exactly the same as the residue classes (modulo \(p\)), i.e., the equivalence
\[
(x \in B^{k}) \Leftrightarrow (x \equiv k \pmod p)
\]
is valid for all \(x \in \mathbb{Z}\) and \(k \in \{0, 1, \ldots, p-1\}\).
\end{example}

Using Corollary~\ref{c2.39} we also obtain the following.

\begin{corollary}\label{c5.15}
Let \((X, d)\) be a totally bounded ultrametric space and let
\begin{equation}\label{c5.15:e2}
B = B_r(c) = \{x \in X \colon d(x, c) < r\}
\end{equation}
be an open ball in \((X, d)\). Then the inequality
\begin{equation}\label{c5.15:e1}
\diam B < r
\end{equation}
holds.
\end{corollary}

\begin{proof}
Inequality~\eqref{c5.15:e1} evidently holds if \(|B| = 1\). If we have \(|B| \geqslant 2\), then the diametrical graph \(G_{B, d|_{B \times B}}\) is complete multipartite by Corollary~\ref{c2.39}. Let \(X_1\) and \(X_2\) be parts of \(G_{B, d|_{B \times B}}\) such that \(c \in X_1\) and \(X_2 \neq X_1\), and let \(p\) be an arbitrary point of \(X_2\). Then it follows directly from Definitions~\ref{d2.6} and \ref{d5.2} that
\begin{equation}\label{c5.15:e3}
\diam B = d(c, p).
\end{equation}
Now inequality \eqref{c5.15:e1} follows from \eqref{c5.15:e2} and \eqref{c5.15:e3}.
\end{proof}

\begin{remark}\label{r5.16}
Example~\ref{ex2.6} shows that strict inequality~\eqref{c5.15:e1} can be false for space \((X, d)\) even if \((X, d)\) is ultrametric and separable but not totally bounded.
\end{remark}

The notion of diametrical graph can be generalized by following way.

Let \((X, d)\) be a metric space with \(|X| \geqslant 2\) and let \(r \in (0, \infty]\). Denote by \(G_{X, d}^{r}\) a graph such that \(V(G_{X, d}^{r}) = X\) and, for \(u\), \(v \in V(G_{X, d}^{r})\),
\begin{equation}\label{e5.21}
\bigl(\{u, v\} \in E(G_{X, d}^{r})\bigr) \Leftrightarrow \bigl(d(u, v) \geqslant r\bigr).
\end{equation}

\begin{remark}\label{r5.18}
It is clear that \eqref{d5.2:e1} and \eqref{e5.21} are equivalent if \(r = \diam X\). Consequently, we have the equality \(G_{X, d}^{r} = G_{X, d}\) for \(r = \diam X\). In particular, the equality \(G_{X, d}^{\infty} = G_{X, d}\) holds if \((X, d)\) is unbounded.
\end{remark}

Now we can give a new characterization of ultrametric spaces.

\begin{theorem}\label{t5.18}
Let \((X, d)\) be a metric space with \(|X| \geqslant 2\). Then the following statements are equivalent:
\begin{enumerate}
\item \label{t5.18:s1} The metric space \((X, d)\) is ultrametric.
\item \label{t5.18:s2} \(G_{X, d}^{r}\) is either empty or complete multipartite for every \(r \in (0, \diam X]\).
\end{enumerate}
\end{theorem}

\begin{proof}
\(\ref{t5.18:s1} \Rightarrow \ref{t5.18:s2}\). Let \((X, d)\) be ultrametric, let \(r \in (0, \diam X]\) and let \(\Psi_r \colon \RR^{+} \to \RR^{+}\) be defined as
\begin{equation}\label{t5.18:e1}
\Psi_r(t) = \min\{r, t\}, \quad t \in \RR^{+}.
\end{equation}
It is easy to prove directly (or see Theorem~\ref{t5.21} in the next section of the paper) that the mapping \(\rho_r = \Psi_r \circ d\) is an ultrametric on \(X\). From \eqref{t5.18:e1} and \(r \in (0, \diam X] = (0, \diam(X, d)]\) it follows that
\begin{equation*}%\label{t5.18:e9}
\diam (X, \rho_r) = r.
\end{equation*}
The last equality and \eqref{e5.21} imply
\begin{equation}\label{t5.18:e4}
G_{X, \rho_r} = G_{X, d}^{r}.
\end{equation}
By Theorem~\ref{t2.24}, the diametrical graph \(G_{X, \rho_r}\) is either empty or complete multipartite. The validity of \(\ref{t5.18:s1} \Rightarrow \ref{t5.18:s2}\) follows.

\(\ref{t5.18:s2} \Rightarrow \ref{t5.18:s1}\). Let \ref{t5.18:s2} hold. Suppose that there are \(x_1\), \(x_2\), \(x_3\) satisfying
\begin{equation}\label{t5.18:e2}
d(x_1, x_2) > \max\{d(x_1, x_3), d(x_3, x_2)\}.
\end{equation}
Let us consider the graph \(G_{X, d}^{r}\) with \(r = d(x_1, x_2)\). It is clear that \(G_{X, d}^{r}\) is a nonempty graph. Inequality \eqref{t5.18:e2} implies that the points \(x_1\), \(x_2\), \(x_3\) are pairwise distinct. In the correspondence with \ref{t5.18:s2}, \(G_{X, d}^{r}\) is a complete multipartite graph. Let \(X_i\) be a part of \(G_{X, d}^{r}\) such that \(x_i \in X_i\) holds, \(i = 1\), \(2\), \(3\). By \eqref{e5.21}, we have \(\{x_1, x_2\} \in E(G_{X, d}^{r})\). Hence, \(X_1\) are \(X_2\) are distinct, \(X_1 \neq X_2\). If \(X_1 = X_3\) holds, then from \eqref{e5.21} it follows that
\begin{equation}\label{t5.18:e3}
d(x_2, x_3) \geqslant r = d(x_1, x_2),
\end{equation}
contrary to \eqref{t5.18:e2}. Thus, we have \(X_1 \neq X_3\). Similarly, we obtain \(X_2 \neq X_3\). Hence, \(X_1\), \(X_2\), \(X_3\) are distinct parts of \(G_{X, d}^{r}\). The last statement also implies \eqref{t5.18:e3}, that contradicts \eqref{t5.18:e2}. It is shown that the strong triangle inequality holds for all \(x_1\), \(x_2\), \(x_3 \in X\). The validity of \(\ref{t5.18:s2} \Rightarrow \ref{t5.18:s1}\) follows.
\end{proof}

\section{The structure of distance sets and ultrametric preserving functions}
\label{sec5}

In the first part of the section, we mainly recall some basic facts related to the theory of ordered sets. Proposition~\ref{p5.7}, which contains a characteristic property of spherically complete ultrametric spaces, is the only new result of this part.

Using a method for constructing ultrametric spaces given by Delhomm\'{e}, Laflamme, Pouzet and Sauer, we completely specify the distance sets of infinite totally bounded ultrametric spaces and, respectively, of separable ultrametric spaces in Theorems~\ref{t5.10} and \ref{p5.16} from the second part of the section.

The main new results of the third part of the section are Theorem~\ref{t5.46}, Theorem~\ref{t5.51} and Theorem~\ref{t5.34} describing interrelations between weak similarities of ultrametric spaces and ultrametric (pseudoultrametric) preserving functions. Moreover, in Theorems~\ref{t2.25} and \ref{t2.35} with the help of complete multipartite graphs we characterize bounded ultrametric spaces which are weakly similar to unbounded ones. A similar characterization of totally bounded ultrametric spaces is given in Theorem~\ref{t5.9}. Some characteristic properties of functions preserving totally bounded ultrametrics are found in Theorem~\ref{t5.15}.

\subsection{Basic facts about posets and isotone mappings}

Let us start from the definition of order relation. A \emph{binary relation} on a set \(Y\) is a subset of the Cartesian square \(Y^{2} = Y \times Y\) of all ordered pairs \(\<x, y>\), \(x\), \(y \in Y\).

A binary relation \(R \subseteq Y^{2}\) is \emph{reflexive} if \(\<y, y> \in R\) holds for every \(y \in Y\). As usual, we say that \(R \subseteq Y^{2}\) is a \emph{transitive} binary relation if the \emph{transitive law}
\[
(\<x, y> \in R \text{ and } \<y, z> \in R) \Rightarrow (\<x, z> \in R)
\]
is valid for all \(x\), \(y\), \(z \in Y\).

Recall that a reflexive and transitive binary relation \(\preccurlyeq\) on \(Y\) is a \emph{partial order} on \(Y\) if \(\preccurlyeq\) has the \emph{antisymmetric property}, the implication
\[
\bigl(\<x, y> \in \preccurlyeq \text{ and } \<y, x> \in \preccurlyeq \bigr) \Rightarrow (x = y)
\]
is valid for all \(x\), \(y \in Y\). In what follows we use the formula \(x \preccurlyeq y\) instead of \(\<x, y> \in \preccurlyeq\) and, in addition, \(x\prec y\) and \(x\parallel y\) mean that
\[
x \preccurlyeq y \quad \text{and} \quad x \neq y
\]
and, respectively, that
\[
\neg (x \preccurlyeq y) \quad \text{and} \quad \neg (y \preccurlyeq x).
\]

Let \(\preccurlyeq_Y\) be a partial order on a set \(Y\). A pair \((Y, {\preccurlyeq}_Y)\) is called to be a \emph{poset} (a partially ordered set). If \(Z\) is a subset of \(Y\) and \({\preccurlyeq}_Z\) is a partial order on \(Z\) given by \({\preccurlyeq}_Z = Z^{2} \cap {\preccurlyeq}_Y\), then we say that \(Z = (Z, {\preccurlyeq}_Z)\) is a subposet of \((Y, {\preccurlyeq}_Y)\). A poset \((Y, {\preccurlyeq}_Y)\) is \emph{totally ordered} if, for all points \(y_1\), \(y_2 \in Y\), we have
\[
y_1 \preccurlyeq_Y y_2 \quad \text{or} \quad y_2 \preccurlyeq_Y y_1.
\]

Let \((P, {\preccurlyeq}_P)\) be a poset. The totally ordered subposets of \((P, {\preccurlyeq}_P)\) are usually called the \emph{chains} in \(P\). A chain \(C \subseteq P\) is said to be \emph{maximal} (in \(P\)) if the implication
\[
(C \subseteq C_1) \Rightarrow (C = C_1)
\]
is valid for every chain \(C_1 \subseteq P\).

The following proposition implies the existence of maximal chains in every nonempty poset.

\begin{proposition}\label{l11.18}
Let \((P, {\preccurlyeq}_P)\) be a nonempty poset and let \(C_0 \subseteq P\) be a chain. Then there is a maximal chain \(M \subseteq P\) such that \(C_0 \subseteq M\).
\end{proposition}

The proof of Proposition~\ref{l11.18} can be obtained by standard application of Zorn's Lemma (see Proposition~2.55 in \cite{Sch2016}).

\begin{definition}\label{d2.8}
Let \((Q, {\preccurlyeq}_Q)\) and \((L, {\preccurlyeq}_L)\) be posets. A mapping \(f \colon Q \to L\) is \emph{isotone} if, for all \(q_1\), \(q_2 \in Q\), we have
\[
(q_1 \preccurlyeq_Q q_2) \Rightarrow (f(q_1) \preccurlyeq_L f(q_2)).
\]
If, in addition, the isotone mapping \(f \colon Q \to L\) is bijective and the inverse mapping \(f^{-1} \colon L \to Q\) is also isotone, then we say that \(f\) is an \emph{order isomorphism} and that \((Q, {\preccurlyeq}_Q)\) and \((L, {\preccurlyeq}_L)\) have the same \emph{order type}.

A mapping \(f \colon Q \to L\) is \emph{strictly isotone} if the implication
\begin{equation}\label{d2.8:e1}
(q_1 \prec_Q q_2) \Rightarrow (f(q_1) \prec_L f(q_2))
\end{equation}
is valid for all \(q_1\), \(q_2 \in Q\).
\end{definition}

It is clear that every strictly isotone mapping is isotone and every order isomorphism is strictly isotone.

The following lemma will be useful in the sequel.

\begin{lemma}\label{l6.2}
Let \((Q, {\preccurlyeq}_Q)\) and \((L, {\preccurlyeq}_L)\) be posets and let \(f \colon Q \to L\) be strictly isotone and surjective. If the implication
\[
(q_1 \parallel_Q q_2) \Rightarrow (f(q_1) \parallel_L f(q_2))
\]
is valid for all \(q_1\), \(q_2 \in Q\), then \(f\) is an order isomorphism.
\end{lemma}

For the proof see, for example, Theorem~9.3 in~\cite[p.~35]{Har2005}.

\begin{corollary}\label{c5.4}
Let \((Q, {\preccurlyeq}_Q)\) be totally ordered. Then, for every poset \((L, {\preccurlyeq}_L)\), each strictly isotone surjection \(f \colon Q \to L\) is an order isomorphism.
\end{corollary}

\begin{definition}
Let \((P, {\preccurlyeq}_P)\) be a poset and let \(S \subseteq P\). Then \(p \in P\) is called an \emph{upper bound} of \(S\) if the inequality \(s \preccurlyeq_P p\) holds for every \(s \in S\). The point \(l \in P\) is called the \emph{lowest upper bound} or \emph{join} of \(S\) if \(l\) is an upper bound of \(S\) and the inequality \(l \preccurlyeq_P p\) holds for every upper bound \(p\) of \(S\).
\end{definition}

The concepts of the \emph{lower bound} and the \emph{greatest lower bound} or \emph{meet} can be defined by duality.

For arbitrary subset \(A\) of arbitrary poset \((P, {\preccurlyeq}_P)\) we will denote the join and the meet of \(A\) by \(\vee A\) and by \(\wedge A\), respectively. Furthermore, we will write
\[
\sup A = \vee A \quad \text{and} \quad \inf A = \wedge A
\]
for the case when \(A\) is a subset of the poset \((\RR, \leqslant)\).

Let us recall now the concept of the largest element of poset.

\begin{definition}\label{d9.19}
Let \((P, \preccurlyeq_{P})\) be a partially ordered set. An element \(l \in P\) is called the \emph{largest element} of \(P\) if we have \(x \preccurlyeq_{P} l\) for every \(x \in P\). We will also say that \(m \in P\) is a \emph{maximal element} of \(P\) if, for every \(y \in P\), the inequality \(m \preccurlyeq_{P} y\) implies the equality \(m = y\).
\end{definition}

It is easy to see that if \(l\) is the largest element of \((P, \preccurlyeq_{P})\), then \(l\) is a maximal element of \((P, \preccurlyeq_{P})\), but not conversely in general.

\begin{example}\label{ex5.7}
Let \((X, d)\) be a nonempty metric space with a distance set \(D(X)\). Let us consider \((D(X), {\leqslant})\) as a subposet of \((\RR, {\leqslant})\). Then the join of \(D(X)\) exists if and only if \(X\) is bounded. For bounded \((X, d)\) the equality
\[
\diam X = \vee D(X)
\]
holds. Moreover, \(\diam X \in D(X)\) holds if and only if \((D(X), {\leqslant})\) contains a largest element \(l\) and in this case we have \(l = \diam X\).
\end{example}

\begin{example}\label{ex5.8}
Let \((P, {\preccurlyeq}_P)\) be a poset and let \(Ch = Ch_P\) be the set of all totally ordered subposets of \((P, {\preccurlyeq}_P)\). Let us define a partial order \({\preccurlyeq}_{Ch}\) on \(Ch\) as
\[
(C_1 \preccurlyeq_{Ch} C_2) \Leftrightarrow (C_1 \subseteq C_2).
\]
Then \((Ch, {\preccurlyeq}_{Ch})\) has the largest element if and only if \((P, {\preccurlyeq}_P)\) is a chain.
\end{example}

It is well-known that the completeness of the real line \(\RR\) with respect to the standard Euclidean metric is equivalent to the existence of \(\sup A\) and \(\inf A\) for every nonempty bounded \(A \subseteq \RR\). In the last section of the paper, we will use a similar interconnection to construct a compact ultrametric space corresponding to a poset of vertices of locally finite tree with monotone labeling.

\begin{example}\label{ex5.6}
Let \(X\) be a nonempty bounded ultrametric space, \(\overline{\BB}_X\) be the set of all closed balls of \(X\) and let \(\mathbf{A} \subseteq \overline{\BB}_X\) be nonempty. Let us define a partial order \(\preccurlyeq_X\) on \(\overline{\BB}_X\) as
\[
(B_1 \preccurlyeq_X B_2) \Leftrightarrow (B_1 \subseteq B_2).
\]
Definition~\ref{d2.9} and Proposition~\ref{p2.12} imply the existence of the join \(\vee \mathbf{A}\) and the equality
\[
\vee \mathbf{A} = B^* \left(\bigcup_{B \in \mathbf{A}} B\right),
\]
where \(B^* \left(\bigcup_{B \in \mathbf{A}} B\right)\) is the smallest ball containing \(\bigcup_{B \in \mathbf{A}} B\). It follows from Proposition~\ref{p2.12} that the existence of the meet \(\wedge \mathbf{A}\) is equivalent to the condition
\begin{equation}\label{ex5.6:e1}
\bigcap_{B \in \mathbf{A}} B \neq \varnothing.
\end{equation}
Moreover, if \eqref{ex5.6:e1} holds, then we have the equalities
\begin{equation}\label{ex5.6:e2}
\wedge \mathbf{A} = B^* \left(\bigcap_{B \in \mathbf{A}} B\right) = \bigcap_{B \in \mathbf{A}} B.
\end{equation}
For the case when \(\mathbf{A}\) is finite, Proposition~\ref{p2.5} implies that the meet \(\wedge \mathbf{A}\) does not exist if and only if \(B_1 \parallel_X B_2\) holds for some \(B_1\), \(B_2 \in \mathbf{A}\).
\end{example}

Recall that an ultrametric space \((X, d)\) is \emph{spherically complete} if every sequence \((B_n)_{n\in \mathbb{N}} \subseteq \overline{\mathbf{B}}_X\) with \(B_1 \supset B_2 \supset \ldots\) has a nonempty intersection (see, for example, Definition~20.1 \cite{Sch1985})

\begin{proposition}\label{p5.7}
Let \((X, d)\) be a nonempty ultrametric space. Then the following conditions are equivalent:
\begin{enumerate}
\item \label{p5.7:s1} \((X, d)\) is spherically complete.
\item \label{p5.7:s2} The intersection \(\bigcap_{B \in \mathbf{A}} B\) is nonempty for every nonempty chain \(\mathbf{A}\) in \((\overline{\BB}_X, {\preccurlyeq}_{X})\).
\item \label{p5.7:s3} The set \(\mathbf{A} \subseteq \overline{\BB}_X\) is a maximal chain of \((\overline{\BB}_X, {\preccurlyeq}_X)\) if and only if there is \(a \in X\) such that
\begin{equation}\label{p5.7:e1}
\mathbf{A} = \{B \in \overline{\BB}_X \colon B \ni a\}.
\end{equation}
\end{enumerate}
\end{proposition}

\begin{proof}
\(\ref{p5.7:s1} \Rightarrow \ref{p5.7:s2}\). Let \((X, d)\) be spherically complete and let \(\mathbf{A}\) be a nonempty chain of the poset \((\overline{\BB}_X, {\preccurlyeq}_{X})\). We must shown that \eqref{ex5.6:e1} holds. Write
\begin{equation}\label{p5.7:e3}
D^{*} = \inf \{\diam B \colon B \in \mathbf{A}\}.
\end{equation}
If there is \(B^{*} \in \mathbf{A}\) such that \(\diam B^{*} = D^{*}\), then Proposition~\ref{p2.7} implies
\[
\bigcap_{B \in \mathbf{A}} B = B^{*} \neq \varnothing.
\]
For the case when the inequality \(\diam B > D^{*}\) holds for every \(B \in \mathbf{A}\), we can find a sequence \((B_n)_{n\in \mathbb{N}} \subseteq \mathbf{A}\) such that \(\diam B_1 > \diam B_2 > \ldots\) and
\begin{equation}\label{p5.7:e4}
\lim_{n \to \infty} \diam B_n = D^{*}.
\end{equation}
Write \(\mathbf{A}_n = \{B \in \mathbf{A} \colon \diam B \geqslant \diam B_n\}\) for every \(n \in \mathbb{N}\). Using \eqref{p5.7:e3}, \eqref{p5.7:e4} and Proposition~\ref{p2.7} we obtain
\[
\bigcap_{B \in \mathbf{A}} B = \bigcap_{n \in \mathbb{N}} \bigcap_{B \in \mathbf{A}_n} B = \bigcap_{n \in \mathbb{N}} B_n.
\]
Since \((X, d)\) is spherically complete, \(\bigcap_{n \in \mathbb{N}} B_n \neq \varnothing\) holds. Thus, \eqref{ex5.6:e1} is true for anyway \(\mathbf{A}\).

\(\ref{p5.7:s2} \Rightarrow \ref{p5.7:s3}\). Let \ref{p5.7:s2} hold. Let us consider an arbitrary maximal chain \(\mathbf{A} \subseteq \overline{\BB}_X\). Then the intersection \(\bigcap_{B \in \mathbf{A}} B\) is nonempty and, consequently, there is \(a \in X\) such that \(a\) belongs to \(B\) for every \(B \in \mathbf{A}\). Write
\[
\mathbf{A}^* = \{B \in \overline{\BB}_X \colon B \ni a\}.
\]
Proposition~\ref{p2.4} implies that \(\mathbf{A}^*\) is a chain in \((\overline{\BB}_X, {\preccurlyeq}_X)\). Since \(\mathbf{A} \subseteq \mathbf{A}^*\) evidently holds and \(\mathbf{A}\) is a maximal chain in \((\overline{\BB}_X, {\preccurlyeq}_X)\), we obtain the equality \(\mathbf{A} = \mathbf{A}^*\), that implies \eqref{p5.7:e1}.

Now let \(\mathbf{A}\) be an arbitrary subset of \(\overline{\BB}_X\) satisfying \eqref{p5.7:e1} for some \(a \in X\). Then, as was noted above, \(\mathbf{A}\) is a chain in \((\overline{\BB}_X, {\preccurlyeq}_X)\). It is easy to see that \(\mathbf{A}\) is a maximal chain in \((\overline{\BB}_X, {\preccurlyeq}_X)\). Indeed, if \(B^{*} \in \overline{\BB}_X\) and, for every \(B \in \mathbf{A}\), the ball \(B^{*}\) is comparable with \(B\),
\begin{equation}\label{p5.7:e2}
(B \subseteq B^{*}) \quad \text{or} \quad (B^{*} \subseteq B),
\end{equation}
then, using \eqref{p5.7:e2} with \(B = \{a\}\) (it should be noted here that \(\{a\}\) belongs to \(\mathbf{A}\) by \eqref{p5.7:e1}), we obtain \(\{a\} \subseteq B^{*}\). That implies \(a \in B^{*}\) and, consequently, \(B^{*} \in \mathbf{A}\). Thus, \(\mathbf{A}\) is a maximal chain as required.

\(\ref{p5.7:s3} \Rightarrow \ref{p5.7:s2}\). Let \ref{p5.7:s3} hold and let \(\mathbf{C}\) be a nonempty chain in \((\overline{\BB}_X, {\preccurlyeq}_X)\). Then, by Proposition~\ref{l11.18}, there is a maximal (in \((\overline{\BB}_X, {\preccurlyeq}_X)\)) chain \(\mathbf{A}\) such that \(\mathbf{C} \subseteq \mathbf{A}\). By condition~\ref{p5.7:s3}, we can find \(a \in X\) for which \eqref{p5.7:e1} holds. Using \eqref{p5.7:e1} and the inclusion \(\mathbf{C} \subseteq \mathbf{A}\) we have
\[
\bigcap_{B \in \mathbf{C}} B \supseteq \bigcap_{B \in \mathbf{A}} B = \{a\} \neq \varnothing.
\]
Condition \ref{p5.7:s2} follows.

To complete the proof it suffices to note that \(\ref{p5.7:s2} \Rightarrow \ref{p5.7:s1}\) is trivially valid.
\end{proof}

\begin{remark}\label{r5.8}
For the ultrametric spaces, the equivalence of the spherical completeness to the fulfillment of condition~\ref{p5.7:s2} of Proposition~\ref{p5.7} is known, but usually this fact is formulated without using the language of the theory of ordered sets (see, for example, Theorem~1 \cite{KakAmBP1995}).
\end{remark}

Using Proposition~\ref{p5.7} and equalities~\eqref{ex5.6:e2} we also obtain the following.

\begin{corollary}\label{c5.9}
An ultrametric space \((X, d)\) is spherically complete if and only if the met \(\wedge \mathbf{A}\) exists for every nonempty chain \(\mathbf{A}\) of the poset \((\overline{\BB}_X, {\preccurlyeq}_X)\)
\end{corollary}

\begin{remark}\label{r5.9}
The spherical complete ultrametric spaces were first introduced by Ingleton \cite{IngPCPS1952} in order to obtain an analog of the Hanh---Banach theorem for non-Archimedean valued fields. This notion has numerous applications in studies of fixed point results for ultrametric spaces \cite{Hitzler2002, PRAMSUH1993, KSTaiA2012, BMNZJM2005}. It was shown by Bayod and Mart\'{\i}nez-Maurica \cite{BMPAMS1987} that an ultrametric space is spherically complete if and only if this space is ultrametrically injective. Recall that an ultrametric space \((Y, \rho)\) is ultrametrically injective if for each \(F \colon A \to X\), where \(A \subseteq Y\) and \(X\) is a space with an ultrametric \(d\), the condition
\[
d(F(x), F(y)) \leqslant \rho(x, y), \quad \forall x, y \in A
\]
implies the existence of an extension \(\widehat{F} \colon Y \to X\) of the mapping \(F\) such that
\[
d(\widehat{F}(x), \widehat{F}(y)) \leqslant \rho(x, y), \quad \forall x, y \in Y.
\]
Thus, an ultrametric space is ultrametrically injective if every contractive mapping from this space to arbitrary ultrametric space has a contractive extension. Some interesting results related to spherical completeness of ultrametric spaces can also be found in \cite{Sch1985} and \cite{KakAmBP1995}.
\end{remark}

\subsection{Distance sets of ultrametric spaces}

The following general fact was proved in Proposition~2 of paper~\cite{DLPS2008TaiA}.

\begin{theorem}\label{t5.6}
Let \(A\) be a subset of \(\RR^{+}\). Then \(A\) is the distance set of a nonempty ultrametric space if and only if \(0\) belongs to \(A\).
\end{theorem}

Below we refine this theorem for totally bounded ultrametric spaces and for separable ones.

\begin{lemma}\label{l6.3}
Let \((X, d)\) be a totally bounded infinite ultrametric space. Then \(D(X)\) is countably infinite, \(|D(X)| = \aleph_0\).
\end{lemma}

\begin{proof}
Every totally bounded metric space is separable. Consequently, \(X\) contains an at most countable dense subset \(A\). Thus, we have the inequalities
\begin{equation}\label{l6.3:e2}
|D(A, d|_{A \times A})| \leqslant |A \times A| \leqslant \aleph_0.
\end{equation}
Corollary~\ref{c2.15} and \eqref{l6.3:e2} imply that
\[
|D(X)| = |D(A, d|_{A \times A})| \leqslant \aleph_0.
\]

Suppose \(|D(X)|\) is a finite number, \(|D(X)| < \aleph_0\). Since \(X\) is infinite, the inequality \(|D(X)| \geqslant 2\) holds. The double inequality \(2 \leqslant |D(X)| < \aleph_0\) implies that there is \(d_1 \in D(X)\) such that
\begin{equation}\label{l6.3:e1}
0 < d_1 \leqslant d^*
\end{equation}
for every \(d^* \in D(X) \setminus \{0\}\). Let us consider the family
\[
\mathcal{F} = \left\{B_{d_1/2}(x) \colon x \in X\right\}.
\]
Then \(\mathcal{F}\) is a cover of \(X\),
\[
X \subseteq \bigcup_{B \in \mathcal{F}} B.
\]
Using~\eqref{l6.3:e1} we obtain \(B_{d_1/2}(x) \cap B_{d_1/2}(y) = \varnothing\) for all distinct \(x\), \(y \in X\). Hence, the cover \(\mathcal{F}\) does not contain any finite subcover. This contradicts Definition~\ref{d2.10}. The equality \(|D(X)| = \aleph_0\) follows.
\end{proof}

\begin{lemma}\label{l6.5}
Let \((X, d)\) be a compact ultrametric space with the distance set \(D(X)\) and let \(p > 0\) belong to \(D(X)\). Then \(p\) is an isolated point of the metric space \((D(X), \rho)\) with the standard metric \(\rho(x, y) = |x-y|\).
\end{lemma}

\begin{proof}
Suppose for every \(\varepsilon > 0\) there is \(s \in D(X)\) satisfying
\begin{equation}\label{l6.5:e1}
p < s < p+\varepsilon.
\end{equation}
Then there is \((s_n)_{n \in \mathbb{N}} \subseteq D(X)\) such that
\begin{equation}\label{l6.5:e7}
\lim_{n \to \infty} s_n = p \quad \text{and} \quad p < s_n
\end{equation}
for every \(n \in \mathbb{N}\). We can find two sequences \((x_n)_{n \in \mathbb{N}} \subseteq X\) and \((y_n)_{n \in \mathbb{N}} \subseteq X\) such that
\begin{equation}\label{l6.5:e2}
d(x_n, y_n) = s_n
\end{equation}
for every \(n \in \mathbb{N}\) and
\begin{equation}\label{l6.5:e3}
\lim_{n \to \infty} d(x_n, y_n) = p.
\end{equation}
By Bolzano---Weierstrass property (Proposition~\ref{p2.9}), there are points \(x\), \(y \in X\) and subsequences \((x_{n_k})_{k \in \mathbb{N}}\) of \((x_n)_{n \in \mathbb{N}}\) and \((y_{n_k})_{k \in \mathbb{N}}\) of \((y_n)_{n \in \mathbb{N}}\) such that
\begin{equation}\label{l6.5:e4}
\lim_{k \to \infty} d(x, x_{n_k}) = \lim_{k \to \infty} d(y, y_{n_k}) = 0.
\end{equation}
Since \(d \colon X \times X \to \RR^{+}\) is continuous, from~\eqref{l6.5:e7}, \eqref{l6.5:e2} and \eqref{l6.5:e4} it follows that \(d(x, y) = p\).

Let us consider the cycle \(C_k\) with
\[
V(C_k) = \{x, x_{n_k}, y_{n_k}, y\} \quad \text{ and } \quad E(C_k) = \bigl\{\{x, x_{n_k}\}, \{x_{n_k}, y_{n_k}\}, \{y_{n_k}, y\}, \{y, x\}\bigr\}.
\]
Equalities \(d(x, y) = p\), \eqref{l6.5:e4} and Proposition~\ref{p2.3} give us
\[
p = d(x_{n_k}, y_{n_k})
\]
for all sufficiently large \(k \in \mathbb{N}\), that contradicts \eqref{l6.5:e7}.

Thus, there is \(\varepsilon_1 > 0\) such that
\begin{equation}\label{l6.5:e5}
(p, p+\varepsilon_1) \cap D(X) = \varnothing.
\end{equation}

Arguing similarly we can find \(\varepsilon_2 > 0\) satisfying
\begin{equation}\label{l6.5:e6}
(p - \varepsilon_2, p) \cap D(X) = \varnothing.
\end{equation}
From~\eqref{l6.5:e5} and \eqref{l6.5:e6} it follows that \(p\) is an isolated point of \(D(X)\).
\end{proof}

\begin{definition}\label{d6.3}
A poset has the order type \(\mathbf{1} + \bm{\omega}^{*}\) if this poset is order isomorphic to the subposet
\[
\{0\} \cup \left\{\frac{1}{n} \colon n \in \mathbb{N}\right\}
\]
of the poset \((\RR, {\leqslant})\), where \({\leqslant}\) is the standard order on \(\RR\).
\end{definition}

\begin{remark}\label{r6.2}
The symbol \(\bm{\omega}\) is the usual designation of the order type of \((\mathbb{N}, {\leqslant})\) and \(\mathbf{1}\) denotes the order type of one-point posets. The symbol \(\bm{\omega}^{*}\) is used for the backwards order type of \(\bm{\omega}\). Moreover, by \(\mathbf{1} + \bm{\omega}^{*}\) we denote the sum of the order types \(\mathbf{1}\) and \(\bm{\omega}^{*}\) (see, for example, Section~1.4 in \cite{Ros1982} for details).
\end{remark}

It was noted in Theorem~\ref{t5.6} that a set \(A \subseteq \RR^{+}\) is the distance set of a nonempty ultrametric space if and only if \(0 \in A\). For totally bounded ultrametric spaces this result admits the following refinement.

\begin{theorem}\label{t5.10}
Let \(A\) be a subset of \(\RR^{+}\). Then \(A\) is the distance set of an infinite totally bounded ultrametric space if and only if the following conditions simultaneously hold:
\begin{enumerate}
\item \label{t5.10:s1} \(0 \in A\).
\item \label{t5.10:s2} The poset \((A, {\leqslant}_A)\) with \({\leqslant}_A = A^2 \cap {\leqslant}\) has the order type \(\mathbf{1} + \bm{\omega}^{*}\).
\item \label{t5.10:s3} The point \(0\) is an accumulation point of the set \(A\) with respect to the standard Euclidean metric on \(\RR\).
\end{enumerate}
\end{theorem}

\begin{proof}
Let \(X\) be an infinite totally bounded ultrametric space with \(A = D(X)\). Then from \(X \neq \varnothing\) it follows that \(0 \in A\). Let us prove the validity of \ref{t5.10:s2}.

The completion \(\widetilde{X}\) of \(X\) is compact (see Proposition~\ref{p2.13}) and ultrametric (see Lemma~\ref{l2.20}). By Corollary~\ref{p4.4}, the equality \(D(X) = D(\widetilde{X})\) holds. Hence, without loss generality, we assume that \(X\) is compact.

By Lemma~\ref{l6.3}, we have \(|D(X)| = \aleph_0\). Consequently, it suffices to show that the cardinality of
\[
D_{d^*} = D(X) \cap [d^{*}, \diam X],
\]
where \([d^{*}, \diam X] = \{s \in \RR^{+} \colon d^{*} \leqslant s \leqslant \diam X\}\), is finite.

By Lemma~\ref{l6.4}, the distance set \(D(X)\) is a compact subset of \(\RR\). The classical Heine---Borel theorem (from Real Analysis) claims that a subset \(S\) of \(\RR\) is compact if and only if \(S\) is closed and bounded. Thus, \(D_{d^*}\) is compact. By Lemma~\ref{l6.5}, every \(p \in D_{d^*}\) is isolated. Hence, for every \(p \in D_{d^*}\), the one-point set \(\{p\}\) is an open ball in the metric subspace \(D_{d^*}\) of \(\RR\) (see Proposition~\ref{p2.2}). It is clear that the open cover
\[
\bigl\{\{p\} \colon p \in D_{d^*}\bigr\}
\]
contains a finite subcover if and only if the cardinal number \(|D_{d^*}|\) is finite. Now the inequality \(|D_{d^*}| < \infty\) follows from Borel---Lebesgue property (Definition~\ref{d2.3}). Condition~\ref{t5.10:s2} follows.

Let us prove condition \ref{t5.10:s3}. Using Proposition~\ref{p4.4} and Lemma~\ref{l6.4} we obtain that \(A\) is a compact subset of \(\RR^{+}\). If \ref{t5.10:s3} does not hold, then \(0\) is an isolated point of \(A\). Consequently, by Lemma~\ref{l6.5}, all points of \(A\) are isolated. Since \(A\) is compact, it implies the finiteness of \(A\), contrary to \ref{t5.10:s2}. Condition~\ref{t5.10:s3} follows.

Conversely, suppose that \(A\) satisfies conditions \ref{t5.10:s1}--\ref{t5.10:s3}. As in Example~\ref{ex2.6}, we write \(A = X\) and define the ultrametric \(d \colon X \times X \to \RR^{+}\) by~\eqref{ex2.6:e1}. Then the equality~\(A = D(X)\) follows from the definition of \(d\). Using \ref{t5.10:s1}--\ref{t5.10:s3}, it is easy to prove that \((X, d)\) is compact, ultrametric and infinite.
\end{proof}

\begin{remark}\label{r6.11}
It can be shown that characteristic properties of the distance sets of infinite totally bounded ultrametric spaces presented by conditions \ref{t5.10:s1}--\ref{t5.10:s3} of Theorem~\ref{t5.10} are independent of one another.
\end{remark}

\begin{example}\label{ex6.8}
Let \((x_{n})_{n \in \mathbb{N}} \subseteq \RR^{+}\) be a strictly decreasing sequence such that
\[
\lim_{n \to \infty} x_n = 0.
\]
Let us define a set \(A \subseteq \RR^{+}\) by the rule
\[
(x \in A) \Leftrightarrow (x=0 \text{ or } \exists n \in \mathbb{N} \colon x = x_n).
\]
Write \(X = A\) and let \(d \colon X \times X \to \RR^{+}\) be defined as in Example~\ref{ex2.6}. Then \((X, d)\) is a compact, infinite ultrametric space satisfying the equality \(D(X, d) = A\).
\end{example}

Example~\ref{ex6.8} gives us a ``constructive'' description of distance sets of all infinite totally bounded ultrametric spaces. For compact ultrametric spaces, the following reformulation of Theorem~\ref{t5.10} can be found in \cite{Sch1985} Proposition~19.2.

\begin{corollary}\label{c6.9}
The following statements are equivalent for every \(A \subseteq \RR^{+}\):
\begin{enumerate}
\item \label{c6.9:s1} There is an infinite compact ultrametric space \(X\) such that \(A = D(X)\).
\item \label{c6.9:s3} There is an infinite totally bounded ultrametric space \(X\) such that \(A = D(X)\).
\item \label{c6.9:s2} There is a strictly decreasing sequence \((x_n)_{n \in \mathbb{N}} \subseteq \RR^{+}\) such that
\[
\lim_{n \to \infty} x_n = 0
\]
holds and the equivalence
\[
(x \in A) \Leftrightarrow (x = 0 \text{ or } \exists n \in \mathbb{N} \colon x_n = x)
\]
is valid for every \(x \in \RR^{+}\).
\end{enumerate}
\end{corollary}

We conclude the second part of the section by following modification of Theorem~\ref{t5.6} for separable ultrametric spaces.

\begin{proposition}\label{p5.16}
Let \(A\) be a subset of \(\RR^{+}\). Then \(A\) is the distance set of a nonempty separable ultrametric space if and only if \(0 \in A\) and \(|A| \leqslant \aleph_0\) hold.
\end{proposition}

For the proof it suffices to note that \(|D(X)| \leqslant \aleph_0\) holds for every separable ultrametric space \(X\) (see inequality~\eqref{l6.3:e2} in the proof of Lemma~\ref{l6.3}) and use the Delhomm\'{e}---Laflamme---Pouzet---Sauer construction which was described in Example~\ref{ex2.6}.

\begin{example}\label{ex5.20}
Let \(p \geqslant 2\) be a prime number. The ultrametric space \((\mathbb{Q}_p, d_p)\) is a completion of the countable ultrametric space \((\mathbb{Q}, d_p)\). Hence, \((\mathbb{Q}_p, d_p)\) is separable. The distance sets \(D(\mathbb{Q}, d_p)\) and \(D(\mathbb{Q}_p, d_p)\) are both equal to the set
\[
\{p^{n} \colon n \in \mathbb{Z}\} \cup \{0\}
\]
of all integer powers of \(p\) together with the added zero (see, for example, \cite[p.~58]{Gou1993}).
\end{example}

\begin{remark}\label{r6.6}
Theorem~\ref{t5.10}, Corollary~\ref{c6.9} and Proposition~\ref{p5.16} can be found in paper \cite{DS2022TRoUCaS} that contains also some results describing the distance sets of ultrametric spaces with the same topology. It was proved by A.~J.~Lemin and V.~Lemin~\cite{LL1996} that, for every infinite, ultrametric space \(X\), the cardinality of \(D(X)\) is no greater than the weight of \(X\),
\begin{equation}\label{r6.6:e2}
|D(X)| \leqslant w(X).
\end{equation}
In particular, \eqref{r6.6:e2} implies \(|D(X)| \leqslant \aleph_0\) for all separable ultrametric spaces \(X\). For the case of finite ultrametric space \(X\), we evidently have the equality \(w(X) = |X|\) and, consequently, inequality \eqref{r6.6:e2} turns into the Gomory---Hu inequality
\[
|D(X)| \leqslant |X|,
\]
which was deduced R.~Gomory and T.~Hu \cite{GH1961S} by studying the flows in networks. The necessary and sufficient conditions under which the group of self-isometries of ultrametric space \(X\) that can be represented by some labeled star graph \(S\) coincides with the group of self-isomorphisms of \(S\) are also determined by properties of \(D(X)\) \cite{DR2025USGbLSG}.
\end{remark}

\subsection{Ultrametric preserving functions and weak similarities}

A function \(f \colon \RR^{+} \to \RR^{+}\) is called \emph{ultrametric preserving} if
\[
X \times X \xrightarrow{d} \RR^{+} \xrightarrow{f} \RR^{+}
\]
is an ultrametric on \(X\) for every ultrametric space \((X, d)\).

The following elegant result was obtained by Pongsriiam and Termwuttipong in paper~\cite{PTAbAppAn2014}.

\begin{theorem}\label{t5.21}
The following conditions are equivalent for every \(f \colon \RR^{+} \to \RR^{+}\):
\begin{enumerate}
\item \label{t5.21:s1} \(f\) is ultrametric preserving.
\item \label{t5.21:s2} \(f\) is increasing and the equality \(f(x) = 0\) holds if and only if \(x = 0\).
\end{enumerate}
\end{theorem}

Some facts related to ultrametric preserving functions can be found in \cite{Dov2020MS, SKP2020MS, VD2019a, PTAbAppAn2014, BDS2021UPFaWSoUS, Dov2024SUPF, BD2024OMoMPF, Dov2024UPFAME}.

\begin{example}\label{ex5.24}
Let \(f \colon \RR^{+} \to \RR^{+}\) be defined as
\begin{equation}\label{ex5.24:e1}
f(t) = \frac{d^{*} t}{1+t}
\end{equation}
with \(d^{*} \in (0, \infty)\). Then \(f\) is increasing and \(f(t) = 0\) holds if and only if \(t=0\). Hence, \(f\) is ultrametric preserving by Theorem~\ref{t5.21}. Moreover, \(f \circ d\) is a metric for every metric space \((X, d)\) (Example~2 in \cite{Dobos1998}). Thus, \(f\) is also \emph{metric preserving}. Using the function \(f\) it is easy to show that, for every unbounded metric space \((Y, \rho)\), the metric space \((Y, \delta)\) with \(\delta = f \circ \rho\) is bounded and has the same topology as \((Y, \rho)\).
\end{example}

In Theorem~\ref{t2.25} below we will use the transformation \(d \mapsto f \circ d\) with \(f \colon \RR^{+} \to \RR^{+}\) defined by equality~\eqref{ex5.24:e1} to characterize the bounded ultrametric spaces with empty diametrical graphs.

The above formulated Pongsriiam---Termwuttipong theorem can be partially generalized as follows.

\begin{lemma}\label{l5.22}
Let \((X, d)\) be a nonempty ultrametric space with the distance set \(D(X)\) and let \(f \colon D(X) \to \RR^{+}\) be an isotone mapping from the poset \((D(X), {\leqslant})\) to the poset \((\RR^{+}, {\leqslant})\). Then the mapping
\[
X \times X \ni \<x, y> \mapsto f(d(x, y)) \in \RR^{+}
\]
is an ultrametric on \(X\) if and only if the equation \(f(t) = 0\) has the unique solution \(t=0\).
\end{lemma}

\begin{proof}
It directly follows form Theorem~5.6 \cite{VD2019a}.
\end{proof}

%The next theorem shows that the non-degenerate, ultrametric spaces with empty diametrical graphs can be obtained from unbounded ultrametric spaces by standard ``ultrametric-preserving'' transformation.

\begin{theorem}\label{t2.25}
Let \((X, d)\) be an unbounded ultrametric space, let \(d^{*} \in (0, \infty)\) and \(\rho \colon X \times X \to \RR^{+}\) be defined as
\begin{equation}\label{t2.25:e1}
\rho(x, y) = \frac{d^{*} \cdot d(x, y)}{1 + d(x, y)}.
\end{equation}
Then \((X, \rho)\) is a bounded ultrametric space with empty diametrical graph \(G_{X, \rho}\).

Conversely, let \((X, \rho)\) be a bounded ultrametric space with \(|X| \geqslant 2\) and empty \(G_{X, \rho}\). Write \(d^{*} = \diam (X, \rho)\). Then there is an unbounded ultrametric space \((X, d)\) such that~\eqref{t2.25:e1} holds for all \(x\), \(y \in X\).
\end{theorem}

\begin{proof}
It was noted in Example~\ref{ex5.24} that the function \(f\) defined by~\eqref{ex5.24:e1} is ultrametric preserving. Hence, the mapping \(\rho \colon X \times X \to \RR^{+}\), defined by~\eqref{t2.25:e1}, is an ultrametric. Moreover, since \(f\) is strictly increasing and satisfies the equality
\begin{equation*}%\label{t2.25:e4}
\lim_{t \to \infty} f(t) = d^{*},
\end{equation*}
we have
\[
\rho(x, y) < \lim_{t \to \infty} f(t) = d^{*} = \diam (X, \rho)
\]
for all \(x\), \(y \in X\). Thus, the diametrical graph \(G_{X, \rho}\) is empty.

Conversely, let \((X, \rho)\) be a bonded ultrametric space with \(|X| \geqslant 2\) and empty diametrical graph \(G_{X, \rho}\). Write \(d^{*} = \diam (X, \rho)\). The inequality \(|X| \geqslant 2\) and boundedness of \((X, \rho)\) imply \(d^{*} \in (0, \infty)\). The function \(g \colon [0, d^{*}) \to \RR^{+}\),
\begin{equation}\label{t2.25:e7}
g(s) = \frac{s}{d^{*} - s},
\end{equation}
is strictly increasing and satisfies the equalities
\begin{equation}\label{t2.25:e5}
g(0) = 0 \quad \text{and} \quad \lim_{\substack{s \to d^{*} \\ s \in [0, d^{*})}} g(s) = +\infty.
\end{equation}
Since \(d^{*}\) equals \(\diam (X, \rho)\), there are sequences \((x_n)_{n \in \mathbb{N}} \subseteq X\) and \((y_n)_{n \in \mathbb{N}} \subseteq X\) such that
\begin{equation}\label{t2.25:e8}
\lim_{n \to \infty} \rho(x_n, y_n) = d^{*}.
\end{equation}
In addition, by Theorem~\ref{t2.24}, we have \(\rho(x, y) < d^{*}\) for all \(x\), \(y \in X\). Consequently, the inclusion \(D(X, \rho) \subseteq [0, d^{*})\) holds. Now Lemma~\ref{l5.22} implies that the mapping \(d \colon X \times X \to \RR^{+}\) satisfying the equality
\[
d(x, y) = g(\rho(x, y))
\]
for all \(x\), \(y \in X\) is an ultrametric on~\(X\). From the second equality in \eqref{t2.25:e5} and equality~\eqref{t2.25:e8} it follows that \((X, d)\) is unbounded. A direct calculation shows
\begin{equation}\label{t2.25:e6}
f(g(s)) = s \quad \text{and} \quad g(f(t)) = t
\end{equation}
for all \(s \in [0, d^{*})\) and \(t \in [0, +\infty)\), where \(f\) is defined by~\eqref{ex5.24:e1}. Now equality \eqref{t2.25:e1} follows from \eqref{t2.25:e6}.
\end{proof}

\begin{remark}\label{r2.33}
The condition \(|X| \geqslant 2\) cannot be dropped in the second part of Theorem~\ref{t2.25}. Indeed, if \(|X| = 1\), then, for every metric \(\rho\), the metric space \((X, \rho)\) is bounded and ultrametric with empty diametrical graph \(G_{X, \rho}\) and there are no ultrametrics \(d \colon X \times X \to \RR^{+}\) for which \(\diam (X, d) = +\infty\) holds.
\end{remark}

In the case of strictly increasing ultrametric preserving \(f \colon \RR^{+} \to  \RR^{+}\), the transition from \((X, d)\) to \((X, f \circ d)\) can also be described using the concept of weak similarity.

\begin{definition}\label{d2.34}
Let \((X, d)\) and \((Y, \rho)\) be nonempty metric spaces. A mapping \(\Phi \colon X \to Y\) is a \emph{weak similarity} of \((X, d)\) and \((Y, \rho)\) if \(\Phi\) is bijective and there is a strictly increasing function \(f \colon D(Y) \to D(X)\) such that the following diagram
\begin{equation}\label{d2.34:e2}
\ctdiagram{
\ctv 0,50:{X \times X}
\ctv 100,50:{Y \times Y}
\ctv 0,0:{D(X)}
\ctv 100,0:{D(Y)}
\ctet 0,50,100,50:{\Phi \otimes \Phi}
\ctel 0,50,0,0:{d}
\cter 100,50,100,0:{\rho}
\ctet 100,0,0,0:{f}
}
\end{equation}
is commutative, i.e., the equality
\begin{equation}\label{d2.34:e1}
d(x, y) = f\left(\rho\bigl(\Phi(x), \Phi(y)\bigr)\right)
\end{equation}
holds for all \(x\), \(y \in X\).
\end{definition}

\begin{remark}\label{r5.32}
In Diagram~\eqref{d2.34:e2} and other diagrams illustrating the application of weak similarities to metric spaces we write \(\Phi \otimes \Phi\) for the mapping
\[
X \times X \ni \<x, y> \mapsto \<\Phi(x), \Phi(y)> \in Y \times Y.
\]
Moreover, for every metric \(d\) on \(X\) we use the symbol \(d \colon X \times X \to D(X)\) for the surjection induced by restricting codomain \(\RR^{+}\) of \(d\) to its range \(D(X)\).
\end{remark}

If \(\Phi \colon X \to Y\) is a weak similarity, \(f \colon D(Y) \to D(X)\) is strictly increasing, and \eqref{d2.34:e1} is fulfilled for all \(x\), \(y \in X\), then we say that \((X, d)\) and \((Y, \rho)\) are \emph{weakly similar}, and \(f\) is the \emph{scaling function} of \(\Phi\).

\begin{remark}\label{r5.29}
It follows directly from Definition~\ref{d2.34} that the scaling function is uniquely determined by given weak similarity. Moreover, Corollary~\ref{c5.4} implies that every scaling function is an order isomorphism of corresponding distance sets.
\end{remark}

\begin{example}\label{ex5.29}
Let \((X, d)\) and \((Y, \rho)\) be nonempty metric spaces. A mapping \(\Phi \colon X \to Y\) is a \emph{similarity}, if \(\Phi\) is bijective and there is a strictly positive number \(r = r(\Phi)\), the \emph{ratio} of \(\Phi\), such that
\[
\rho\bigl(\Phi(x), \Phi(y)\bigr) = rd(x, y)
\]
for all \(x\), \(y \in X\) (see, for example, \cite[p.~45]{Edgar1992}). It is clear that every isometry is a similarity with the ratio \(r = 1\) and every similarity is a weak similarity. Furthermore, a weak similarity \(\Phi \colon X \to Y\) with the scaling function \(f \colon D(Y) \to D(X)\) is an isometry if and only if \(f(t) = t\) holds for every \(t \in D(Y)\).
\end{example}

\begin{example}\label{ex5.32}
Let \(\Phi \colon X \to Y\) be a bijection and let \(d \colon X \times X \to \RR^{+}\) and \(\rho \colon Y \times Y \to \RR^{+}\) be some metrics. The mapping \(\Phi\) is a weak similarity of \((X, d)\) and \((Y, \rho)\) if and only if the equivalence
\begin{equation}\label{ex5.32:e1}
(d(x, y) \leqslant d(w, z)) \Leftrightarrow (\rho(\Phi(x), \Phi(y)) \leqslant \rho(\Phi(w), \Phi(z)))
\end{equation}
is valid for all \(x\), \(y\), \(z\), \(w \in X\).
\end{example}

\begin{remark}\label{r5.36}
Equivalence~\eqref{ex5.32:e1} evidently implies the validity of
\begin{equation}\label{r5.36:e1}
(d(x, y) = d(w, z)) \Leftrightarrow (\rho(\Phi(x), \Phi(y)) = \rho(\Phi(w), \Phi(z))).
\end{equation}
The bijections \(\Phi \colon X \to Y\) satisfying~\eqref{r5.36:e1} for all \(x\), \(y\), \(z\), \(w \in X\) is said to be combinatorial similarities. Some questions connected with the weak similarities and combinatorial similarities were studied in \cite{DovBBMSSS2020, DLAMH2020, Dov2019IEJA}. The weak similarities of finite ultrametric and semimetric spaces were also considered by E.~Petrov in \cite{Pet2018pNUAA}.
\end{remark}

\begin{lemma}\label{l5.29}
Let \((X, d)\), \((Y, \rho)\) and \((Z, \delta)\) be nonempty metric spaces and mappings \(\Phi \colon X \to Y\) and \(\Psi \colon Y \to Z\) be weak similarities with the scaling functions \(f \colon D(Y) \to D(X)\) and \(g \colon D(Z) \to D(Y)\), respectively. Then the mapping
\[
X \xrightarrow{\Phi} Y \xrightarrow{\Psi} Z
\]
is a weak similarity of \((X, d)\) and \((Z, \delta)\), and the function
\[
D(Z) \xrightarrow{g} D(Y) \xrightarrow{f} D(X)
\]
is the scaling function of this weak similarity. Moreover, the inverse mapping \(\Phi^{-1} \colon Y \to X\) of \(\Phi\) is also a weak similarity, the scaling function \(f \colon D(Y) \to D(X)\) of \(\Phi\) is an order isomorphism of the posets \((D(Y), {\leqslant})\) and \((D(X), {\leqslant})\) such that the inverse isomorphism \(f^{-1} \colon D(X) \to D(Y)\) is the scaling function of the weak similarity \(\Phi^{-1}\).
\end{lemma}

\begin{proof}
The first part of the lemma follows from the commutativity of the diagram
\[
\ctdiagram{
\ctv 0,50:{X \times X}
\ctv 100,50:{Y \times Y}
\ctv 200,50:{Z \times Z}
\ctv 0,0:{D(X)}
\ctv 100,0:{D(Y)}
\ctv 200,0:{D(Z)}
\ctet 0,50,100,50:{\Phi \otimes \Phi}
\ctet 100,50,200,50:{\Psi \otimes \Psi}
\ctet 100,0,0,0:{f}
\ctet 200,0,100,0:{g}
\ctel 0,50,0,0:{d}
\ctel 100,50,100,0:{\rho}
\ctel 200,50,200,0:{\delta}
}.
\]

To prove the second part of the lemma, it suffices to use Corollary~\ref{c5.4} and to note that that the equality
\[
d = f \circ \rho \circ (\Phi \otimes \Phi)
\]
holds if and only if we have \(f^{-1} \circ d \circ (\Phi^{-1} \otimes \Phi^{-1}) = \rho\).
\end{proof}

Lemma~\ref{l5.29} implies the following.

\begin{corollary}\label{c5.36}
Let \((X, d)\) and \((Y, \rho)\) be nonempty weakly similar ultrametric spaces. Then the diametrical graph \(G_{X, d}\) is empty if and only if the diametrical graph \(G_{Y, \rho}\) is empty.
\end{corollary}

\begin{proof}
By Lemma~\ref{l5.29}, the posets \((D(X), {\leqslant})\) and \((D(Y), {\leqslant})\) are order isomorphic. Hence, \((D(X), {\leqslant})\) has largest element if and only if \((D(Y), {\leqslant})\) has the largest element. To complete the proof it suffices to remember that the largest element of the distance set of metric space, if such an element exists, coincides with the diameter of the space (see Example~\ref{ex5.7}).
\end{proof}

\begin{remark}\label{r5.31}
For more detailed proof of Lemma~\ref{l5.29} in more general case of semimetric spaces see Proposition~1.2 in \cite{DP2013AMH}.
\end{remark}

\begin{proposition}\label{p5.32}
Let \(f \colon \RR^{+} \to \RR^{+}\) be an ultrametric preserving function. Then the following condition are equivalent:
\begin{enumerate}
\item \label{p5.32:s1} The function \(f\) is strictly increasing.
\item \label{p5.32:s2} For every nonempty ultrametric space \((X, d)\), the spaces \((X, f \circ d)\) and \((X, d)\) are weakly similar.
\end{enumerate}
\end{proposition}

\begin{proof}
\(\ref{p5.32:s1} \Rightarrow \ref{p5.32:s2}\). Let \((X, d)\) be a nonempty ultrametric space with the distance set \(D(X, d)\). Then \((X, f \circ d)\) is an ultrametric space with the distance set
\[
D(X, f \circ d) = f(D(X, d)).
\]
If \(f\) is strictly increasing, the the identical function \(\operatorname{Id} \colon X \to X\) is a weak similarity of \((X, f \circ d)\) and \((X, d)\), and the restriction \(f|_{D(X, d)}\) is the corresponding scaling function,
\[
\ctdiagram{
\ctv 0,50:{X \times X}
\ctv 100,50:{X \times X}
\ctv 0,0:{D(X, f \circ d)}
\ctv 100,0:{D(X, d)}
\ctet 0,50,100,50:{\operatorname{Id} \otimes \operatorname{Id}}
\ctel 0,50,0,0:{f \circ d}
\cter 100,50,100,0:{d}
\ctet 100,0,0,0:{f|_{D(X, d)}}
}.
\]

\(\ref{p5.32:s2} \Rightarrow \ref{p5.32:s1}\). Suppose we have \ref{p5.32:s2} but \ref{p5.32:s1} does not hold. Then there are points \(a\), \(b \in \RR^{+}\) such that \(0 < a < b < \infty\) and \(f(a) = f(b)\). Let \(X = \{x_1, x_2, x_3\}\) be a tree-point set and let an ultrametric \(d \colon X \times X \to \RR^{+}\) satisfy the equalities
\[
d(x_1, x_2) = a, \quad d(x_1, x_3) = d(x_2, x_3) = b.
\]
Then \(f \circ d\) is an ultrametric on \(X\) such that
\[
f(d(x_1, x_2)) = f(d(x_1, x_3)) = f(d(x_2, x_3)).
\]
Thus, we have
\begin{equation}\label{p5.32:e1}
|D(X, d)| = 3 \neq 2 = |D(X, f \circ d)|.
\end{equation}
It was noted in Remark~\ref{r5.29} that every scaling function is bijective. Hence, if \((X, d)\) and \((X, f \circ d)\) are weakly similar, then the equality
\[
|D(X, d)| = |D(X, f \circ d)|
\]
holds, contrary to \eqref{p5.32:e1}.
\end{proof}

Using the concept of weak similarity we can give a more compact variant of Theorem~\ref{t2.25}.

\begin{theorem}\label{t2.35}
Let \((X, d)\) be an ultrametric space with \(|X| \geqslant 2\). Then the following statements are equivalent:
\begin{enumerate}
\item\label{t2.35:s1} \((X, d)\) is weakly similar to an unbounded ultrametric space.
\item\label{t2.35:s2} The diametrical graph \(G_{X, d}\) is empty.
\end{enumerate}
\end{theorem}

\begin{proof}
\(\ref{t2.35:s1} \Rightarrow \ref{t2.35:s2}\). Let \((X, d)\) be a weakly similar to an unbounded ultrametric space \((Y, \rho)\). The diametrical graph \(G_{Y, \rho}\) is empty. Hence, \(G_{X, d}\) is also empty by Corollary~\ref{c5.36}.

\(\ref{t2.35:s2} \Rightarrow \ref{t2.35:s1}\). Suppose that the diametrical graph \(G_{X, d}\) is empty. If \((X, d)\) is unbounded, then \(\ref{t2.35:s2}\) is valid because \((X, d)\) is weakly similar to itself. If \((X, d)\) is bounded, then, by Theorem~\ref{t2.25}, there is an unbounded ultrametric space \((Y, \rho)\) such that
\[
d(x, y) = \diam X \frac{\rho(x, y)}{1 + \rho(x, y)}
\]
for all \(x\), \(y \in X\). The function \(f \colon \RR^{+} \to \RR^{+}\),
\[
f(t) = \diam X \frac{t}{1+t}
\]
is strictly increasing and ultrametric preserving. Hence, \((X, d)\) and \((Y, \rho)\) are weakly similar by Proposition~\ref{p5.32}.
\end{proof}

Recall that a topological space \((X, \tau)\) is \emph{connected} if it is not the union of two disjoint nonempty open subsets of \(X\). A subset \(S\) of \(X\) is said to be connected if \(S\) is connected as subspace of \((X, \tau)\). A \emph{component} (= connected component) of \((X, \tau)\) is, by definition, a connected set \(A \subseteq X\) for which the implication
\[
(S \supseteq A) \Rightarrow (S = A)
\]
is valid for every connected \(S \subseteq X\). The set of all components of nonempty \((X, \tau)\) forms a partition of \(X\). Thus, for every \(x \in X\), there exists a unique component \(C(x)\) of \((X, \tau)\) such that \(x \in C(x)\).

\begin{lemma}\label{l5.34}
A nonempty subset \(S\) of the real line \(\RR\) is connected if and only if \(S\) is an interval.
\end{lemma}

For the proof see Proposition~4 in \cite[p.~336]{Bourbaki1995}.

\begin{remark}\label{r5.35}
By bounded interval in \(\RR\) we mean any subset of \(\RR\) which can be given in the one of the following forms:
\begin{align*}
[a, b] &= \{x \in \RR \colon a \leqslant x \leqslant b\}, &  (a, b] &= \{x \in \RR \colon a < x \leqslant b\}, \\
[a, b) &= \{x \in \RR \colon a \leqslant x < b\}, & (a, b) &= \{x \in \RR \colon a < x < b\},
\end{align*}
where \(a\), \(b \in \RR\) with \(a \leqslant b\) for \([a, b]\) and \(a < b\) for \((a, b]\), \([a, b)\) and \((a, b)\). Similarly, by definition, any unbounded interval in \(\RR\) has one of the forms:
\begin{align*}
(-\infty, \infty) &= \RR, &  [a, \infty) &= \{x \in \RR \colon x \geqslant a\}, \\
(a, \infty) &= \{x \in \RR \colon x > a\}, & (-\infty, a] &= \{x \in \RR \colon x \leqslant a\}
\end{align*}
or \((-\infty, a) = \{x \in \RR \colon x < a\}\) with arbitrary \(a \in \RR\).
\end{remark}

\begin{remark}\label{r5.42}
In all topological spaces every single-point set and the empty set are connected (\cite[p.~108]{Bourbaki1995}). Moreover, if \((X, \tau)\) is a connected topological space, then the set \(X\) is the unique component of \((X, \tau)\). In particular, the empty set \(\varnothing\) is a component of \((X, \tau)\) if and only if \(X = \varnothing\).
\end{remark}

\begin{lemma}\label{l5.36}
Let \(A\) be a nonempty subset of \(\RR\). Then all components of \(A\) are intervals in \(\RR\).
\end{lemma}

\begin{proof}
Let \(p\) be a point of \(A\). Then we denote by \(C(p)\) the connected component of \(A\) containing the point \(p\).

Suppose contrary that there is a point \(p \in A\) for which the component \(C(p)\) is not an interval in \(\RR\). Then, by Lemma~\ref{l5.34}, the set \(C(p)\) is not a connected subset of \(\RR\). A subset \(E\) of \(\RR\) is connected if and only if \([x, y] \subseteq E\) holds whenever \(x\), \(y \in E\) and \(x < y\) (see, for example, Theorem~2.47 \cite{Rudin1976}). Consequently, there are \(x\), \(y \in C(p)\) and \(z \in \RR \setminus C(p)\) such that \(x < z < y\). Then the sets \((-\infty, z) \cap C(p)\) and \(C(p) \cap (z, \infty)\) are disjoint nonempty open subsets of the topological space \(C(p)\) and we have the equality
\[
C(p) = ((-\infty, z) \cap C(p)) \cup (C(p) \cap (z, \infty)).
\]
Hence, \(C(p)\) is not a connected subset of \(A\), that contradicts the definition of connected components.
\end{proof}

The useful generalization of the concept of ultrametric is the concept of pseudoultrametric.

\begin{definition}\label{d5.43}
Let \(X\) be a set and let \(d \colon X \times X \to \RR^{+}\) be a symmetric function such that \(d(x, x) = 0\) holds for every \(x \in X\). The function \(d\) is a \emph{pseudoultrametric} on \(X\) if it satisfies the strong triangle inequality.
\end{definition}

If \(d\) is a pseudoultrametric on \(X\), then we will say that \((X, d)\) is a \emph{pseudoultrametric} \emph{space}.

Every ultrametric space is a pseudoultrametric space but not conversely. In contrast to ultrametric spaces, pseudoultrametric spaces can contain some distinct points with zero distance between them.

\begin{definition}\label{d5.44}
A function \(f \colon \RR^{+} \to \RR^{+}\) is \emph{pseudoultrametric preserving} if \(f \circ d\) is a pseudoultrametric for every pseudoultrametric space \((X, d)\).
\end{definition}

The following result is a continuation of the Pongsriiam---Termwuttipong theorem to the case of pseudoultrametric preserving functions.

\begin{proposition}\label{p5.45}
The following conditions are equivalent for every \(f \colon \RR^{+} \to \RR^{+}\):
\begin{enumerate}
\item \label{p5.45:s1} The function \(f\) is increasing and \(f(0) = 0\) holds.
\item \label{p5.45:s2} The function \(f\) is pseudoultrametric preserving.
\end{enumerate}
\end{proposition}

For the proof see Proposition~2.4 \cite{Dov2020MS}.

\begin{remark}\label{r5.47}
Theorem~\ref{t5.21} and Proposition~\ref{p5.45} imply that every ultrametric preserving function is pseudoultrametric preserving. Moreover, a pseudoultrametric preserving \(f \colon \RR^{+} \to \RR^{+}\) is ultrametric preserving if and only if \(f^{-1}(0) = \{0\}\) holds.
\end{remark}

Using pseudoultrametric preserving functions we can now obtain a new description of bounded ultrametric spaces with empty diametrical graphs.

In what follows we will use the symbol \(\RR^{+} \setminus D(X)\) to denote the relative complement of the distance set \(D(X)\) of an ultrametric space \(X\) with respect to \(\RR^{+}\).

\begin{theorem}\label{t5.46}
Let \((X, d)\) be an ultrametric space with \(|X| \geqslant 2\) and let \(D(X)\) be the distance set of \((X, d)\). Then the following conditions are equivalent.
\begin{enumerate}
\item \label{t5.46:s1} The space \((X, d)\) is bounded and the diametrical graph \(G_{X, d}\) is empty.
\item \label{t5.46:s2} There is \(a \in (0, \infty)\) such that the set \([a, \infty)\) is a component of \(\RR^{+} \setminus D(X)\).
\item \label{t5.46:s3} There is an ultrametric space \((Z, \delta)\) such that \((Z, \delta)\) and \((X, d)\) are weakly similar, but if \(\Phi \colon Z \to X\) is a weak similarity and \(\psi \colon D(X) \to D(Z)\) is the scaling function of \(\Phi\), then
\begin{equation}\label{t5.46:e1}
g|_{D(X)} \neq \psi
\end{equation}
holds for every pseudoultrametric preserving function \(g \colon \RR^{+} \to \RR^{+}\).
\end{enumerate}
\end{theorem}

\begin{proof}
\(\ref{t5.46:s1} \Leftrightarrow \ref{t5.46:s2}\). Let \(\ref{t5.46:s1}\) hold. The inequality \(|X| \geqslant 2\) implies the strict inequality \(\diam X > 0\). Since \((X, d)\) is bounded and the diametrical graph \(G_{X, d}\) is empty, the double inequality
\[
0 < \diam X < \infty
\]
and the inclusion
\[
[\diam X, \infty) \subseteq \RR^{+} \setminus D(X)
\]
take place. It follows from the equality
\[
\diam X = \sup D(X)
\]
that \([\diam X, \infty)\) is the maximal interval in \(\RR^{+} \setminus D(X)\) which contains the point \(\diam X\). Hence, \([\diam X, \infty)\) is a component of \(\RR^{+} \setminus D(X)\) by Lemma~\ref{l5.36}. Thus, \(\ref{t5.46:s2}\) follows.

Let \(a\) be a point of \((0, \infty)\) and let \([a, \infty)\) be a component of \(\RR^{+} \setminus D(X)\). By Lemma~\ref{l5.36}, \([a, b)\) is the maximal interval in \(\RR^{+} \setminus D(X)\) contained the point \(a\). Hence, the intersection \((a-\varepsilon, a) \cap D(X)\) is nonempty for every \(\varepsilon \in (0, a)\). It implies the equalities
\[
a = \sup D(X) = \diam X.
\]

\(\ref{t5.46:s1} \Leftrightarrow \ref{t5.46:s3}\). Let \(\ref{t5.46:s1}\) hold. By Theorem~\ref{t2.35}, there is an unbounded ultrametric space \((Z, \delta)\) such that \((X, d)\) and \((Z, \delta)\) are weakly similar. Let \(\Phi \colon Z \to X\) be a weak similarity and let \(\psi \colon D(X) \to D(Z)\) be the scaling function of \(\Phi\). Since \((Z, \delta)\) is unbounded and \(\psi\) is surjective, the set \(\psi(D(X))\) is unbounded. Now if \(g \colon \RR^{+} \to \RR^{+}\) is an arbitrary pseudoultrametric preserving function, then we have the inequality
\[
g(t) \leqslant g(\diam X)
\]
for every \(t \in D(X)\), because \(g\) is increasing. Thus,
\[
g(D(X)) \subseteq [0, g(\diam X)]
\]
holds. The last statement and unboundedness of \(\psi(D(X))\) imply \eqref{t5.46:e1}.

Let condition \(\ref{t5.46:s3}\) hold. If \(\ref{t5.46:s1}\) does not hold, then either \((X, d)\) is unbounded or \((X, d)\) is bounded and has the nonempty diametrical graph \(G_{X, d}\).

Consider first the case when \((X, d)\) is unbounded. By condition~\(\ref{t5.46:s3}\), there is an ultrametric space \((Z, \delta)\) such that \((X, d)\) and \((Z, \delta)\) are weakly similar, and if \(\Phi \colon Z \to X\) is a weak similarity with a scaling function \(\psi\), then \eqref{t5.46:e1} holds for every pseudoultrametric preserving \(g\). Let us define a function \(f \colon \RR^{+} \to \RR^{+}\) as
\begin{equation}\label{t5.46:e2}
f(t) = \sup \{\psi(s) \colon s \in [0, t] \cap D(X)\}, \quad t \in \RR^{+}.
\end{equation}
Then \(f\) is increasing and satisfies the equality \(f(0) = 0\). By Proposition~\ref{p5.45}, the function \(f\) is pseudoultrametric preserving. Moreover, from~\eqref{t5.46:e2} it follows the equality
\begin{equation}\label{t5.46:e3}
f|_{D(X)} = \psi.
\end{equation}
The last equality contradicts~\eqref{t5.46:e1} for \(g = f\).

Reasoning in completely similar way, we can prove the validity of equality~\eqref{t5.46:e3} in the case when \((X, d)\) is bounded and has nonempty \(G_{X, d}\), which completes the proof of the theorem.
\end{proof}

\begin{theorem}\label{t5.51}
Let \((X, d)\) be a nonempty ultrametric space and let \(D(X) = D(X, d)\) be the distance set of \((X, d)\). Then the following conditions are equivalent:
\begin{enumerate}
\item \label{t5.51:s1} There is \(a \in (0, \infty)\) such that \((0, a]\) or \([a, \infty)\) is a component of \(\RR^{+} \setminus D(X)\).
\item \label{t5.51:s2} There is an ultrametric space \((Z, \delta)\) such that \((Z, \delta)\) and \((X, d)\) are weakly similar, but if \(\Phi \colon Z \to X\) is a weak similarity and \(\psi \colon D(X) \to D(Z)\) is the scaling function of \(\Phi\), then
\begin{equation}\label{t5.51:e1}
g|_{D(X)} \neq \psi
\end{equation}
holds for every ultrametric preserving function \(g \colon \RR^{+} \to \RR^{+}\).
\end{enumerate}
\end{theorem}

\begin{proof}
The theorem is evidently valid for the case \(|X| = 1\). So in what follows we will suppose that \(|X| \geqslant 2\).

\(\ref{t5.51:s1} \Rightarrow \ref{t5.51:s2}\). Let condition~\ref{t5.51:s1} hold. If there is \(a > 0\) such that \([a, \infty)\) is a component of \(\RR^{+} \setminus D(X)\), then \ref{t5.51:s2} takes place by Theorem~\ref{t5.46}.

If \((0, a]\) is a component of \(\RR^{+} \setminus D(X)\) for some \(a > 0\), then we define a function \(\psi\) on \(D(X) = D(X, d)\) as
\begin{equation}\label{t5.51:e2}
\psi(t) = \begin{cases}
0 & \text{if } t = 0,\\
t-a & \text{if } t \in D(X) \cap [a, \infty).
\end{cases}
\end{equation}
By Lemma~\ref{l5.22}, the mapping \(\rho \colon X \times X \to \RR^{+}\) with
\begin{equation}\label{t5.51:e3}
\rho(x, y) = \psi (d(x, y))
\end{equation}
is an ultrametric on \(X\). From \eqref{t5.51:e2} and \eqref{t5.51:e3} it follows that the diagram
\[
\ctdiagram{
\ctv 0,50:{X \times X}
\ctv 100,50:{X \times X}
\ctv 0,0:{D(X, \rho)}
\ctv 100,0:{D(X, d)}
\ctet 0,50,100,50:{\operatorname{Id} \otimes \operatorname{Id}}
\ctel 0,50,0,0:{\rho}
\cter 100,50,100,0:{d}
\ctet 100,0,0,0:{\psi}
}
\]
is commutative and \(\psi \colon D(X, d) \to D(X, \rho)\) is strictly increasing and bijective. Hence, \(\operatorname{Id} \colon X \to X\) is a weak similarity of the ultrametric spaces \((X, \rho)\), \((X, d)\) and \(\psi\) is the scaling function of this weak similarity. Suppose that there is an ultrametric preserving function \(g \colon \RR^{+} \to \RR^{+}\) such that
\begin{equation}\label{t5.51:e4}
g|_{D(X, d)} = \psi.
\end{equation}
Since \((0, a]\) is a component of \(\RR^{+} \setminus D(X)\), there is a strictly decreasing sequence \((t_n)_{n \in \mathbb{N}} \subseteq D(X, d)\) for which
\begin{equation}\label{t5.51:e5}
\lim_{n \to \infty} t_n = a
\end{equation}
holds. Now using \eqref{t5.51:e2}, \eqref{t5.51:e4} and \eqref{t5.51:e5} we obtain the contradiction
\[
0 < g(a) \leqslant \lim_{n \to \infty} g(t_n) = \lim_{n \to \infty} (t_n - a) = 0.
\]
So, equality~\eqref{t5.51:e4} is not possible for any ultrametric preserving function \(g\).

\(\ref{t5.51:s2} \Rightarrow \ref{t5.51:s1}\). Let \ref{t5.51:s2} be valid. Suppose condition \ref{t5.51:s1} does not hold. Then the point \(0\) is an accumulation point of \(D(X)\) and either \(\diam X \in D(X)\) or \((X, d)\) is unbounded. By condition~\(\ref{t5.51:s2}\), there is an ultrametric space \((Z, \delta)\) such that \((X, d)\) and \((Z, \delta)\) are weakly similar, and if \(\Phi \colon Z \to X\) is a weak similarity with a scaling function \(\psi\), then \eqref{t5.51:e1} holds for every ultrametric preserving \(g\).

Let us define \(g \colon \RR^{+} \to \RR^{+}\) by equality~\eqref{t5.46:e2} with \(f = g\). Then, as in the proof of Theorem~\ref{t5.46}, we obtain that \(g\) is pseudoultrametric preserving and the equality
\begin{equation}\label{t5.51:e6}
g|_{D(X, d)} = \psi
\end{equation}
holds. To complete the proof it suffices to note that \(g\) is ultrametric preserving. Indeed, the inequality \(g(t) > 0\) holds for every \(t \in (0, \infty)\) because \(g\) is increasing, \(\psi(t) > 0\) takes place for every \(t \in D(X) \setminus \{0\}\) and the point \(0\) is an accumulation point of \(D(X)\). Hence, \(g\) is ultrametric preserving by Remark~\ref{r5.47}.
\end{proof}

\begin{example}\label{ex5.52}
Let \(X\) be a subset of an interval \((a, b)\), \(0 < a < b < \infty\), and let \(|X| \geqslant 2\) hold. As in Example~\ref{ex2.6}, we define an ultrametric \(d \colon X \times X \to \RR^{+}\) as
\[
d(x, y) = \begin{cases}
0 & \text{if } x = y,\\
\max\{x, y\} & \text{if } x \neq y.
\end{cases}
\]
Then \((X, d)\) is spherically complete and has a nonempty diametrical graph if and only if, for every ultrametric space \((Z, \delta)\), which is weakly similar to \((X, d)\), and every weak similarity \(\Phi \colon Z \to X\), there is an ultrametric preserving function \(g \colon \RR^{+} \to \RR^{+}\) such that \(g|_{D(X)} = \psi\), where \(\psi\) is the scaling function of \(\Phi\).
\end{example}

The following theorem is a reformulation of Theorem~2.14 from \cite{BDS2021UPFaWSoUS}.

\begin{theorem}\label{t5.34}
Let \((X, d)\) be a nonempty ultrametric space and let \(D(X) = D(X, d)\) be the distance set of \((X, d)\). Then the following conditions are equivalent:
\begin{enumerate}
\item \label{t5.34:s1} Every component of \(\RR^{+} \setminus D(X, d)\) is either open in \(\RR\) or it is a single-point set.
\item \label{t5.34:s2} If \(\Phi \colon Z \to X\) is a weak similarity of an ultrametric space \((Z, \delta)\) and the ultrametric space \((X, d)\), then there is a strictly increasing ultrametric preserving function \(g \colon \RR^{+} \to \RR^{+}\) such that
\begin{equation}\label{t5.34:e0}
g|_{D(X, d)} = \psi,
\end{equation}
where \(\psi\) is the scaling function of the weak similarity \(\Phi\).
\end{enumerate}
\end{theorem}

\begin{proof}
This theorem is obviously true for \(|X| = 1\). Furthermore, if the equality \(D(X) = \RR^{+}\) holds, then we have \(\RR^{+} \setminus D(X) = \varnothing\). Hence, in this case condition~\ref{t5.34:s1} is valid by Remark~\ref{r5.42}. Moreover, for the case \(D(X) = \RR^{+}\), the domain of the scaling function \(\psi\) of any weak similarity \(\Phi \colon Z \to X\) coincides with \(\RR^{+}\). By Definition~\ref{d2.34}, the function \(\psi\) is strictly increasing and satisfies \(\psi(0) = 0\). Hence, \(\psi\) is strictly increasing and ultrametric preserving by Theorem~\ref{t5.21}. So, in what follows, without loss of generality, we can assume that \(|X| \geqslant 2\) and \(\RR^{+} \setminus D(X) \neq \varnothing\).

\(\ref{t5.34:s1} \Rightarrow \ref{t5.34:s2}\). Suppose that condition~\ref{t5.34:s1} holds. Let \((Z, \delta)\) be a nonempty ultrametric space and let \(\Phi \colon Z \to X\) be a weak similarity with a scaling function \(\psi \colon D(X) \to D(Z)\),
\[
\ctdiagram{
\ctv 0,50:{Z \times Z}
\ctv 100,50:{X \times X}
\ctv 0,0:{D(Z, \delta)}
\ctv 100,0:{D(X, d)}
\ctet 0,50,100,50:{\Phi \otimes \Phi}
\ctel 0,50,0,0:{\delta}
\cter 100,50,100,0:{d}
\ctet 100,0,0,0:{\psi}
}.
\]
We want to continue the scaling function \(\psi \colon D(X) \to D(Z)\) to a strictly increasing ultrametric preserving function \(g \colon \RR^{+} \to \RR^{+}\).

Let \(t \in \RR^{+}\).

Consider first the case when \(t \in D(X)\) or \(\{t\}\) is a component of \(\RR^{+} \setminus D(X)\). Then we have
\[
D(X) \cap [0, t] \neq \varnothing \neq D(X) \cap [t, \infty).
\]
Consequently, there is a point \(t^{*} \in \RR^{+}\) such that
\begin{equation}\label{t5.34:e2}
\sup\{\psi(s) \colon s \in D(X) \cap [0, t]\} \leqslant t^{*} \leqslant \inf\{\psi(s) \colon s \in D(X) \cap [t, \infty)\}.
\end{equation}
Let us define \(g(t)\) as
\begin{equation}\label{t5.34:e5}
g(t) = t^{*}.
\end{equation}

If there is an interval \((a, b) \ni t\) such that \((a, b)\) is a component of \(\RR^{+} \setminus D(X)\), then the points \(a\) and \(b\) belong to \(D(X)\) and we can define \(g(t)\) as
\begin{equation}\label{t5.34:e3}
g(t) = \psi(a) + \frac{\psi(b) - \psi(a)}{b - a} (t - a).
\end{equation}

Similarly, if \(t \in (a, \infty)\), where \((a, \infty)\) is a component of \(\RR^{+} \setminus D(X)\), then \(a \in D(X)\). From \(|X| \geqslant 2\) and Lemma~\ref{l5.29} it follows that the inequalities \(a > 0\) and \(\psi(a)  > 0\) hold. In this case we write
\begin{equation}\label{t5.34:e4}
g(t) = \frac{\psi(a)}{a} t.
\end{equation}

Since every point of \(\RR^{+} \setminus D(X)\) belongs to the component of this point, condition~\ref{t5.34:s1} implies that \(g(t)\) is defined now for all \(t \in \RR^{+} \setminus D(X)\). It should be noted here that the components of two distinct points of any topological space either coincide or are disjoint, so that this definition is correct.

If \(t \in D(X)\), then we have
\[
\psi(t) = \sup\{\psi(s) \colon s \in D(X) \cap [0, t]\}
\]
and
\[
\psi(t) = \inf\{\psi(s) \colon s \in D(X) \cap [t, \infty)\}
\]
because \(\psi\) is increasing. Consequently, the equality \(\psi(t) = g(t)\) follows from~\eqref{t5.34:e2} and \eqref{t5.34:e5} for every \(t \in D(X)\), that implies~\eqref{t5.34:e0} and the equality \(g(0) = 0\).

Let us prove the validity of the implication
\begin{equation}\label{t5.34:e9}
(t_1 < t_2) \Rightarrow (g(t_1) < g(t_2))
\end{equation}
for all \(t_1\), \(t_2 \in \RR^{+}\).

First of all, it should be noted that \eqref{t5.34:e0} implies \eqref{t5.34:e9} for all \(t_1\), \(t_2 \in D(X)\) since the scaling function \(\psi\) is strictly increasing. Let \(t_1 \in \RR^{+} \setminus D(X)\), \(t_2 \in D(X)\) and \(t_1 < t_2\). If \(\{t_1\}\) is a component of \(\RR^{+} \setminus D(X)\), then there is a point \(t_3 \in D(X) \cap [t_1, \infty)\) such that \(\psi(t_3) < \psi(t_2)\). From \eqref{t5.34:e2} and \eqref{t5.34:e5} it follows
\[
g(t_1) \leqslant g(t_3) = \psi(t_3),
\]
which together with \(\psi(t_3) < \psi(t_2) = g(t_2)\) implies \(g(t_1) < g(t_2)\). The case when \(t_1 \in D(X)\), \(t_2 \in \RR^{+} \setminus D(X)\) and \(t_1 < t_2\) can be considered similarly.

Let \(t_1\) and \(t_2\) be points of \(\RR^{+} \setminus D(X)\) such that \(t_1 < t_2\). If the components of these points are the same, then \eqref{t5.34:e9} follows from \eqref{t5.34:e3} and \eqref{t5.34:e4}. Suppose now that the points \(t_1\) and \(t_2\) belong to the different components of \(\RR^{+} \setminus D(X)\). Then there is a point \(t_3 \in D(X)\) such that
\[
t_1 < t_3 < t_2.
\]
It was shown above that in this case we have
\[
g(t_1) < g(t_3) < g(t_2).
\]
Thus, \eqref{t5.34:e9} is valid for all \(t_1\), \(t_2 \in \RR^{+}\).

For every strictly increasing function \(f \colon \RR^{+} \to \RR^{+}\) satisfying \(f(0) = 0\), we have \(f^{-1}(0) = \{0\}\). Thus, the function \(g\) defined above is strictly increasing and ultrametric preserving by Pongsriiam---Termwuttipong theorem.

\(\ref{t5.34:s2} \Rightarrow \ref{t5.34:s1}\). Suppose condition \ref{t5.34:s1} is false. We must show that \ref{t5.34:s2} is also false.

The classification of intervals of \(\RR\) given in Remark~\ref{r5.35} and Lemma~\ref{l5.36} imply that at least one from the following statements are valid:
\begin{enumerate}
\item [\((s_1)\)] There is \(a \in (0, \infty)\) such that \((0, a]\) or \([a, \infty)\) is a component of \(\RR^{+} \setminus D(X)\).
\item [\((s_2)\)] There are some points \(a\), \(b\) such that \(0 < a < b < \infty\) and \([a, b)\) or \((a, b]\) is a component of \(\RR^{+} \setminus D(X)\).
\end{enumerate}

If \((s_1)\) holds, then \ref{t5.34:s2} is false by Theorem~\ref{t5.51}.

Let us consider the case, when \((s_2)\) holds and \([a, b)\) is a component of the set \(\RR^{+} \setminus D(X)\). Let us consider the function \(\psi\) defined as
\begin{equation}\label{t5.34:e6}
\psi(t) = \begin{cases}
t &\text{if } t \in D(X) \cap [0, a),\\
t - (b-a) &\text{if } t \in D(X) \cap [b, \infty).
\end{cases}
\end{equation}
From \([a, b) \cap D(X) = \varnothing\) it follows that \(\psi\) is defined for all \(t \in D(X)\). It is clear that \(\psi\) is strictly increasing on \(D(X)\) and \(\psi(t) = 0\) holds if and only if \(t = 0\). By Lemma~\ref{l5.22}, the mapping
\[
X \times X \ni \<x, y> \mapsto \psi(d(x, y))
\]
is an ultrametric on \(X\).

Write \(\delta = \psi \circ d\). Then the ultrametric spaces \((X, d)\) and \((X, \delta)\) are weakly similar, the identical mapping \(\operatorname{Id} \colon X \to X\) is a weak similarity of \((X, \delta)\) and \((X, d)\), and the function \(\psi\),
\[
D(X, d) \ni t \mapsto \psi(t) \in D(X, \delta),
\]
is the scaling function of this weak similarity.

We claim that there is no a strictly increasing ultrametric preserving function \(g \colon \RR^{+} \to \RR^{+}\) for which the equality
\begin{equation}\label{t5.34:e10}
g|_{D(X, d)} = \psi
\end{equation}
holds. To prove it, we may consider a point \(b_1 \in [a, b)\) such that
\begin{equation}\label{t5.34:e7}
a < b_1 < b.
\end{equation}
If \(g \colon \RR^{+} \to \RR^{+}\) is an arbitrary strictly increasing ultrametric preserving function, then \eqref{t5.34:e7} implies
\begin{equation}\label{t5.34:e8}
g(a) < g(b_1) < g(b).
\end{equation}
Since \([a, b)\) is a component of \(\RR^{+} \setminus D(X, d)\), the membership \(b \in D(X, d)\) and the equality \(a > 0\) hold. Hence, there is a strictly increasing sequence \((a_n)_{n \in \mathbb{N}} \subseteq D(X, d) \cap [0, a)\) such that
\[
\lim_{n \to \infty} a_n = a.
\]
The last equality, \eqref{t5.34:e10} and \eqref{t5.34:e6} imply that
\[
\lim_{n \to \infty} g(a_n) = \lim_{n \to \infty} \psi(a_n) = a.
\]
Moreover, \eqref{t5.34:e10}, \eqref{t5.34:e6} and \(b \in D(X, d)\) imply that \(\psi(b) = g(b) = b - (b-a) = a\). Now using \eqref{t5.34:e8} we have the contradiction
\[
a = \lim_{n \to \infty} g(a_n) \leqslant g(a) < g(b_1) < g(b) = a.
\]

If \((s_2)\) holds for a component \((a, b]\) of \(\RR^{+} \setminus D(X, d)\), then using the function
\[
\varphi(t) = \begin{cases}
t & \text{if } t \in D(X) \cap [0, a],\\
t - (b-a) & \text{if } t \in D(X) \cap (b, \infty)
\end{cases}
\]
instead of the function \(\psi\) defined by \eqref{t5.34:e6} and arguing as above we can see that the equality
\[
g|_{D(X, d)} = \varphi
\]
is impossible for any strictly increasing ultrametric preserving function \(g\).
\end{proof}

\begin{corollary}\label{c5.38}
Let \(X\) and \(Z\) be nonempty weakly similar ultrametric spaces and let the distance set \(D(X)\) be a closed subset of \(\RR^{+}\). Then, for every weak similarity \(\Phi \colon Z \to X\), the scaling function of \(\Phi\) admits a continuation to a strictly increasing ultrametric preserving function.
\end{corollary}

\begin{proof}
It follows from Theorem~\ref{t5.34} because the set \(\RR^{+} \setminus D(X)\) is an open subset of \(\RR\) and every component of any nonempty open subset of \(\RR\) is an open interval.
\end{proof}

\begin{corollary}\label{c5.35}
Let \(X\) be a nonempty totally bounded ultrametric space. Then, for every weak similarity \(\Phi \colon Z \to X\) of an ultrametric space \(Z\) and the ultrametric space \(X\), there is a strictly increasing ultrametric preserving function \(g \colon \RR^{+} \to \RR^{+}\) such that
\[
g|_{D(X)} = \psi,
\]
where \(\psi \colon D(X) \to D(Z)\) is the scaling function of \(\Phi\).
\end{corollary}

\begin{proof}
It follows from Corollary~\ref{c5.38} and Corollary~\ref{c6.9}.
\end{proof}

\begin{example}\label{ex5.30}
Let \(X\) be a subset of \([0, 1]\) such that
\[
X = \left\{\frac{1}{n} \colon n \in \mathbb{N}\right\}
\]
and let \(d \colon X \times X \to \RR^{+}\) and \(\delta \colon X \times X \to \RR^{+}\) be ultrametrics on \(X\) defined for all \(x\), \(y \in X\) as
\[
d(x, y) = \begin{cases}
0 & \text{if } x = y,\\
\max\{x^2, y^2\} & \text{if } x \neq y
\end{cases}
\]
and, respectively,
\begin{equation}\label{ex5.30:e1}
\delta(x, y) = \begin{cases}
0 & \text{if } x = y,\\
1 + \max\{x, y\} & \text{if } x \neq y.
\end{cases}
\end{equation}
Then \((X, d)\) and \((X, \delta)\) are weakly similar, the identical mapping \(\operatorname{Id} \colon X \to X\) is a (unique) weak similarity of \((X, d)\) and \((X, \delta)\), and the equality \(d(x, y) = f(\delta(x, y))\) holds for all \(x\), \(y \in X\) with the scaling function \(f \colon D(X, \delta) \to D(X, d)\) such that
\begin{equation}\label{ex5.30:e2}
D(X, d) = \{0\} \cup \left\{\frac{1}{n^2} \colon n \in \mathbb{N}\right\}, \quad D(X, \delta) = \{0\} \cup \left\{1 + \frac{1}{n} \colon n \in \mathbb{N}\right\}
\end{equation}
and
\begin{equation}\label{ex5.30:e3}
f(t) = \begin{cases}
0 & \text{if } t = 0,\\
(t-1)^2 & \text{if } t \in D(X, \delta) \setminus \{0\}.
\end{cases}
\end{equation}
It is interesting to note that the ultrametric space \((X, \delta)\) is not totally bounded because \(|X| = \aleph_0\) and \(\delta(x, y) \geqslant 1\) whenever \(x \neq y\). In addition, there is no ultrametric preserving function \(\varphi \colon \RR^{+} \to \RR^{+}\) such that
\begin{equation}\label{ex5.30:e4}
\varphi|_{D(X, \delta)} = f.
\end{equation}
Indeed, if \(\varphi \colon \RR^{+} \to \RR^{+}\) is increasing and \eqref{ex5.30:e4} holds, then using \eqref{ex5.30:e2} and \eqref{ex5.30:e3} we obtain the equality \(\varphi(t) = 0\) for all \(t \in [0, 1]\). Thus, \(\varphi\) is not ultrametric preserving by Pongsriiam---Termwuttipong theorem.

By Lemma~\ref{l5.29}, there is the inverse function \(\psi \colon D(X, d) \to D(X, \delta)\) for the scaling function \(f\). It is easy to show that
\begin{equation}\label{ex5.30:e5}
\psi(s) = \begin{cases}
0 & \text{if } s = 0,\\
1 + \sqrt{s} & \text{if } s \in D(X, d) \setminus \{0\}.
\end{cases}
\end{equation}
In the proof of Theorem~\ref{t5.34} it was used the piecewise linear expansion of \(\psi\) to a strictly increasing ultrametric preserving function \(g\) depicted in Figure~\ref{fig14} below.
\end{example}

\begin{figure}[ht]
\begin{center}
\begin{tikzpicture}
\def\dx{4cm}
\def\dy{1.7cm}
\def\de{3pt}
\draw [->] (-1, 0) -- (1.5*\dx, 0);
\draw [->] (0, -1) -- (0, 3.1*\dy);
\draw (\dx, -\de) node [label = below:{\(1\)}] {} -- (\dx, \de);
\draw (-\de, 2*\dy) node [label = left:{\(2\)}] {} -- (\de, 2*\dy);
\draw (\dx, 2*\dy) -- (1.5*\dx, 3*\dy);

\draw (\dx/4, -\de) node [label = below:{\(\frac{1}{4}\)}] {} -- (\dx/4, \de);
\draw (\dx/4, 3*\dy/2) -- (\dx, 2*\dy);
\draw (-\de, 3*\dy/2) node [label = left:{\(3/2\)}] {} -- (\de, 3*\dy/2);

\draw (\dx/9, -\de) node [label = below:{\(\frac{1}{9}\)}] {} -- (\dx/9, \de);
\draw (\dx/9, 4*\dy/3) -- (\dx/4, 3*\dy/2);
\draw (-\de, 4*\dy/3) -- (\de, 4*\dy/3);
\draw (-\de, 3.9*\dy/3) node [label = left:{\(4/3\)}] {};

\draw [dashed] (0, \dy) -- (\dx/9, 4*\dy/3);
\draw (-\de, \dy) node [label = left:{\(1\)}] {} -- (\de, \dy);
\end{tikzpicture}
\end{center}
\caption{The graph of strictly increasing ultrametric preserving function \(g \colon \RR^{+} \to \RR^{+}\) corresponding to the scaling function \(\psi\) defined by rule \eqref{ex5.30:e5}.}
\label{fig14}
\end{figure}
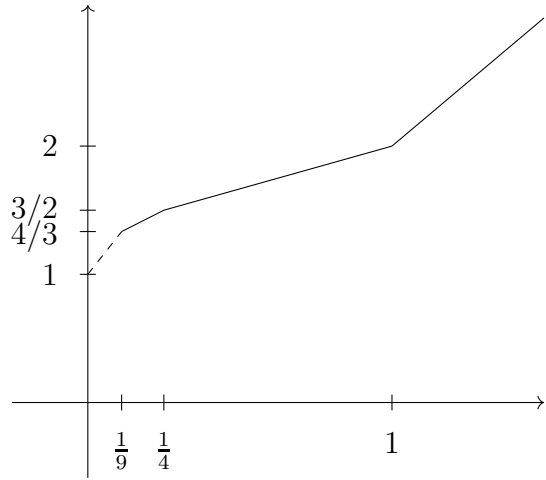

Theorems~\ref{t5.46}--\ref{t5.34} can be considered as reformulations of some results guaranteeing, in the language of ultrametric spaces, the existence of isotonic extension of given isotonic mappings. For example, Theorem~\ref{t5.46} is a simple corollary of the following.

\begin{theorem}\label{t5.57}
Let \(D\) be a subset of \(\RR^{+}\) such that \(0 \in D\). Then the following conditions are equivalent:
\begin{enumerate}
\item \label{t5.57:s1} For every strictly increasing function \(\psi \colon D \to \RR^{+}\) with \(\psi(0) = 0\) there is an increasing function \(g \colon \RR^{+} \to \RR^{+}\) such that \(g|_{D} = \psi\).
\item \label{t5.57:s2} The poset \((D, {\leqslant})\) has a largest element or \(D\) is a confinal subset of \((\RR^{+}, {\leqslant})\).
\end{enumerate}
\end{theorem}

The above theorem admits also future generalizations for partially ordered sets and, in particular, for different types of semilattices and lattices \cite{GLPFM1970, SorMSN1952, DovJMSUS2020, BGAU2012, BNSPRSESA2003, FofMN1969}. We note only that lattices are a special kind of posets which play a fundamental role in the theory of ultrametric spaces \cite{Lem2003AU}.

The next our goal is to find a refinement of the Pongsriiam---Termwuttipong theorem for totally bounded ultrametric spaces.

\begin{definition}\label{d5.16}
Let \((X, d)\) and \((X, \rho)\) be metric spaces. The metrics \(d\) and \(\rho\) are (\emph{topologically}) \emph{equivalent} if these metrics generate one and the same topology,
\begin{equation*}
(A \text{ is open in } (X, d)) \Leftrightarrow (A \text{ is open in } (X, \rho))
\end{equation*}
for every \(A \subseteq X\).
\end{definition}

The following result can be considered as a special case of Theorem~29 from \cite{Dov2024SUPF}.

\begin{theorem}\label{t5.15}
The following conditions are equivalent for every \(\psi \colon \RR^{+} \to \RR^{+}\):
\begin{enumerate}
\item \label{t5.15:s1} \(\psi\) is ultrametric preserving and continuous at \(0\).
\item \label{t5.15:s2} For every nonempty totally bounded ultrametric space \((X, d)\) the mapping \(\rho = \psi \circ d\) is an ultrametric on \(X\) and \((X, \rho)\) is totally bounded.
\item \label{t5.15:s3} For every ultrametric space \((X, d)\), the mapping \(\rho = \psi \circ d\) is an ultrametric on \(X\) such that \(\rho\) and \(d\) are topologically equivalent.
\end{enumerate}
\end{theorem}

\begin{proof}
\(\ref{t5.15:s1} \Rightarrow \ref{t5.15:s2}\). Let \ref{t5.15:s1} hold and let \((X, d)\) be a nonempty totally bounded ultrametric space. Then the mapping \(\rho = \psi \circ d\) is an ultrametric on \(X\). By Propositions~\ref{p2.11} and \ref{p4.5}, \((X, \rho)\) is totally bounded if and only if every infinite sequence \((x_n)_{n \in \mathbb{N}} \subseteq X\) contains a subsequence \((x_{n_k})_{k \in \mathbb{N}}\) such that
\begin{equation}\label{t5.15:e1}
\lim_{k \to \infty} \rho(x_{n_k}, x_{n_{k+1}}) = 0.
\end{equation}

Let \((x_n)_{n \in \mathbb{N}}\) be an arbitrary sequence of points of \(X\). Using Propositions~\ref{p2.11} and \ref{p4.5} again, we can find a subsequence \((x_{n_k})_{k \in \mathbb{N}}\) of \((x_n)_{n \in \mathbb{N}}\) such that
\begin{equation}\label{t5.15:e2}
\lim_{k \to \infty} d(x_{n_k}, x_{n_{k+1}}) = 0.
\end{equation}
Limit relation \eqref{t5.15:e1} can be written as
\begin{equation}\label{t5.15:e10}
\lim_{k \to \infty} \psi(d(x_{n_k}, x_{n_{k+1}})) = 0.
\end{equation}
Since \(\psi(0) = 0\) holds and \(\psi\) is continuous at \(0\), equality~\eqref{t5.15:e10} follows from \eqref{t5.15:e2}.

\(\ref{t5.15:s2} \Rightarrow \ref{t5.15:s1}\). Let \ref{t5.15:s2} hold. It is easy to prove directly (or can be obtained as a particular case of Corollary~5.7 \cite{VD2019a}) that \(\psi\) is ultrametric preserving if and only if
\[
X \times X \xrightarrow{d} \RR^{+} \xrightarrow{\psi} \RR^{+}
\]
is an ultrametric for every three-point ultrametric space \((X, d)\). Since every three-point ultrametric space is totally bounded, condition~\ref{t5.15:s2} implies that \(\psi\) is ultrametric preserving.

To complete the proof it suffices to show that
\[
\lim_{t \to 0} \psi(t) = 0
\]
holds.

Suppose contrary that there is \((t_n)_{n \in \mathbb{N}} \subseteq \RR^{+}\) such that
\begin{equation}\label{t5.15:e3}
\lim_{n \to \infty} t_n = 0 \quad \text{and} \quad \limsup_{n \to \infty} \psi(t_n) > 0.
\end{equation}
Then we can find a subsequence \((t_{n_k})_{k \in \mathbb{N}}\) of \((t_n)_{n \in \mathbb{N}}\) and a point \(\varepsilon^{*} \in (0, \infty)\) such that \(t_{n_{k_1}} \neq t_{n_{k_2}}\) whenever \(k_1 \neq k_2\) and, in addition,
\begin{equation}\label{t5.15:e4}
\psi(t_{n_k}) > \varepsilon^{*}
\end{equation}
holds for all \(k \in \mathbb{N}\). Let \(A \subseteq \RR^{+}\) and \(d \colon A \times A \to \RR^{+}\) be defined as
\begin{equation}\label{t5.15:e5}
(t \in A) \Leftrightarrow (t = 0 \text{ or } \exists k \in \mathbb{N} \colon t = t_{n_k}),
\end{equation}
and
\begin{equation}\label{t5.15:e6}
d(s, t) = \begin{cases}
0 & \text{if } s = t,\\
\max\{s ,t\} & \text{if } s \neq t.
\end{cases}
\end{equation}
Then \((A, d)\) is an infinite compact ultrametric space (see Example~\ref{ex6.8}). Let us consider the ultrametric space \((A, \rho)\) with \(\rho = \psi \circ d\). Using~\eqref{t5.15:e4} we obtain
\[
\rho(s, t) > \varepsilon^{*} > 0
\]
whenever \(s\) and \(t\) are different points of \(A\). The last statement and the equality \(|A| = \infty\) imply that the ultrametric space \((A, \rho)\) is not totally bonded, contrary to~\ref{t5.15:s2}.

\(\ref{t5.15:s1} \Rightarrow \ref{t5.15:s3}\). Let \ref{t5.15:s1} hold, let \((X, d)\) be an ultrametric space, and let \(\rho = \psi \circ d\). Then \(\rho\) and \(d\) are topologically equivalent if and only if
\[
\left(\lim_{n \to \infty} d(x_n, a) = 0\right) \Leftrightarrow \left(\lim_{n \to \infty} \rho(x_n, a) = 0\right)
\]
is valid for every \((x_n)_{n \in \mathbb{N}} \subseteq X\) and every \(a \in X\). Since \(\psi\) is continuous at \(0\), the limit relation \(\lim_{n \to \infty} d(x_n, a) = 0\) implies the equality
\begin{equation}\label{t5.15:e7}
\lim_{n \to \infty} \rho(x_n, a) = 0.
\end{equation}
Conversely, suppose that there are \((x_n)_{n \in \mathbb{N}} \subseteq X\) and \(a \in X\) satisfying \eqref{t5.15:e7} such that
\[
\limsup_{n \to \infty} d(x_n, a) > \varepsilon
\]
for some \(\varepsilon > 0\). Then there is an infinite subsequence \((x_{n_k})_{k \in \mathbb{N}}\) of \((x_n)_{n \in \mathbb{N}}\) for which we have
\begin{equation}\label{t5.15:e8}
d(x_{n_k}, a) > \frac{1}{2} \varepsilon
\end{equation}
for every \(k \in \mathbb{N}\) and
\begin{equation}\label{t5.15:e9}
\lim_{k \to \infty} \psi (d(x_{n_k}, a)) = 0.
\end{equation}
By Theorem~\ref{t5.21}, the mapping \(\psi\) is increasing. Hence, from~\eqref{t5.15:e8} and \eqref{t5.15:e9} it follows that
\[
0 \leqslant \psi\left(\frac{1}{2} \varepsilon\right) \leqslant \lim_{k \to \infty} \psi (d(x_{n_k}, a)) = 0.
\]
Thus, \(\psi\left(\frac{1}{2} \varepsilon\right) = 0\) holds, contrary to Theorem~\ref{t5.21}.

\(\ref{t5.15:s3} \Rightarrow \ref{t5.15:s1}\). Let \(d\) and \(\psi \circ d\) be topologically equivalent ultrametrics for every ultrametric space \((X, d)\). Then, in particular, \(\psi\) is ultrametric preserving. If \(\psi\) is discontinuous at \(0\), then we can find an infinite compact ultrametric space \((A, d)\) (see \eqref{t5.15:e5} and \eqref{t5.15:e6}) such that every point of the ultrametric space \((A, \psi \circ d)\) is isolated. Hence, \(d\) and \(\psi \circ d\) is not topologically equivalent, contrary to condition~\ref{t5.15:s3}.
\end{proof}

\begin{example}\label{ex5.31}
A valuation of the field \(\mathbb{Q}\) is a mapping \(|\cdot| \colon \mathbb{Q} \to \RR^{+}\) satisfying the following conditions:
\begin{enumerate}
\item \(|x| = 0\) holds if and only if \(x = 0\);
\item \(|x + y| \leqslant |x| + |y|\);
\item \(|xy| = |x| \cdot |y|\)
\end{enumerate}
for all \(x\), \(y \in \mathbb{Q}\). A valuation \(|\cdot|\) is trivial if we have \(|x| = 1\) whenever \(x \neq 0\). Every valuation \(|\cdot| \colon \mathbb{Q} \to \RR^{+}\) generates a metric \(d \colon \mathbb{Q} \times \mathbb{Q} \to \RR^{+}\) by the rule
\[
d (x, y) = |x - y|.
\]
If metrics \(d \colon \mathbb{Q} \times \mathbb{Q} \to \RR^{+}\) and \(\rho \colon \mathbb{Q} \times \mathbb{Q} \to \RR^{+}\) are topologically equivalent and generated by some valuations on \(\mathbb{Q}\), then there is \(\alpha \in (0, \infty)\) such that
\[
\rho(x, y) = d^{\alpha}(x, y)
\]
for all \(x\), \(y \in \mathbb{Q}\) (for the proof see, for example, Theorem~9.2 in \cite{Sch1985}).

The famous Ostrovski Theorem claims that every metric \(d \colon \mathbb{Q} \times \mathbb{Q} \to \RR^{+}\) generated by nontrivial valuation \(|\cdot| \colon \mathbb{Q} \to \RR^{+}\) is topologically equivalent either to standard Euclidean metric on \(\mathbb{Q}\) or to some \(p\)-adic metric \(d_p \colon \mathbb{Q} \times \mathbb{Q} \to \RR^{+}\).
\end{example}

\begin{remark}\label{r5.17}
Theorem~\ref{t5.15} can be considered as a modification of corresponding results for general metric spaces. In this connection see, in particular, Theorem~3.3 in \cite{BDMS1981}, Theorem~2.2 in \cite{HMCM1991} and Theorem~2.8 in \cite{DM2013}.
\end{remark}

Applying suitable ultrametric preserving functions for constructing diametrical graphs, we can obtain the following refinement of Theorem~\ref{t5.18}.

\begin{theorem}\label{t5.9}
Let \((X, d)\) be a metric space with \(|X| \geqslant 2\). Then \((X, d)\) is totally bounded and ultrametric if and only if \(G_{X, d}^{r}\) is complete \(k\)-partite with an integer \(k = k(r)\) for every \(r \in (0, \diam X]\).
\end{theorem}

\begin{proof}
Suppose that \((X, d)\) is totally bounded and ultrametric. Let \(r \in (0, \diam X]\) and let \(\psi_r \colon \RR^{+} \to \RR^{+}\) be defined as
\begin{equation*}
\psi_r(t) = \min\{r, t\}, \quad t \in \RR^{+}.
\end{equation*}
Then using Theorems \ref{t5.21} and \ref{t5.15} we obtain  that \((X, \rho_r)\) is a totally bounded ultrametric space for
\begin{equation}\label{t5.9:e5}
\rho_r = \psi_r \circ d.
\end{equation}
By Proposition~\ref{p2.36}, the diametrical graph \(G_{X, \rho_r}\) is complete \(k\)-partite for an integer \(k = k(r)\). As in the proof of Theorem~\ref{t5.18}, we obtain the equality
\begin{equation}\label{t5.9:e9}
G_{X, \rho_r} = G_{X, d}^{r}.
\end{equation}
Hence, \(G_{X, d}^{r}\) is also \(k\)-partite with the same \(k\).

Suppose now that, for every \(r \in (0, \diam(X, d)]\), \(G_{X, d}^{r}\) is complete \(k\)-partite with an integer \(k = k(r)\). Using Theorem~\ref{t5.18} we obtain that \((X, d)\) is ultrametric.

Let \(r \in (0, \diam(X, d)]\) be given. By Pongsriiam---Termwuttipong theorem, the space \((X, \rho_r)\) is also ultrametric. Now equality~\eqref{t5.9:e9} and the second part of Theorem~\ref{t2.24} imply that there are points \(x_1\), \(\ldots\), \(x_{k(r)} \in X\) such that
\begin{equation}\label{t5.9:e6}
X \subseteq \bigcup_{i=1}^{k(r)} B_{r^{*}}^{\rho}(x_i)
\end{equation}
where
\begin{equation}\label{t5.9:e7}
r^{*} = \diam (X, \rho_r) \quad \text{and} \quad B_{r^{*}}^{\rho}(x_i) = \{x \in X \colon \rho_r(x, x_i) < r^{*}\}.
\end{equation}
From \eqref{t5.9:e5} and the first equality in \eqref{t5.9:e7} it follows that
\begin{equation}\label{t5.9:e8}
B_{r^{*}}^{\rho}(x_i) \subseteq B_{r}(x_i) = \{x \in X \colon d(x, x_i) < r\}
\end{equation}
for every \(i \in \{1, \ldots, k(r)\}\). Since \(r\) is an arbitrary point of \((0, \diam (X, d)]\), Definition~\ref{d2.10} and formulas \eqref{t5.9:e6}, \eqref{t5.9:e8} imply the total boundedness of \((X, d)\).
\end{proof}

Let us now look at another interesting application of ultrametric preserving functions. Recall that a metric space \((X, d)\) is metrically discrete if there is \(r \in (0, \infty)\) such that \(d(x, y) > r\) holds for all distinct \(x\), \(y \in X\).

\begin{proposition}\label{p6.14}
Let an ultrametric space \((X, d)\) be not metrically discrete. Then there is an ultrametric \(\rho \colon X \times X \to \RR^{+}\) such that \(d\) and \(\rho\) are topologically equivalent, the (ordered) distance set
\[
D(X, \rho) = \{\rho(x, y) \colon x, y \in X\}
\]
has the order type \(\mathbf{1} + \bm{\omega}^{*}\) and \(0\) is an unique accumulation point of \(D(X, \rho)\).
\end{proposition}

\begin{proof}
Let \((\varepsilon_{n})_{n \in \mathbb{N}} \subseteq \RR^{+}\) be a strictly decreasing sequence with
\[
\lim_{n \to \infty} \varepsilon_{n} = 0
\]
and let \(f \colon \RR^{+} \to \RR^{+}\) be defined as
\begin{equation}\label{p6.14:e1}
f(t) = \begin{cases}
0 & \text{if } t = 0,\\
\varepsilon_{n+1} & \text{if } \varepsilon_{n+1} \leqslant t < \varepsilon_{n},\\
\varepsilon_{1} & \text{if } t \geqslant \varepsilon_{1}.\\
\end{cases}
\end{equation}
It is clear that \(f\) is increasing and \(f(t) = 0\) holds if and only if \(t = 0\). By Theorem~\ref{t5.21}, the function \(\rho = f \circ d\) is an ultrametric on~\(X\). Moreover, since \(0\) is an accumulation point of \(D(X, d)\) (because \((X, d)\) is not metrically discrete) and \((\varepsilon_{n})_{n \in \mathbb{N}} \subseteq \RR^{+}\) is strictly decreasing, formula~\eqref{p6.14:e1} implies that \(D(X, \rho)\) has the order type \(\mathbf{1} + \bm{\omega}^{*}\).

If \((x_{n})_{n \in \mathbb{N}} \subseteq X\) and \(a \in X\), then using \eqref{p6.14:e1} we also obtain
\[
\left(\lim_{n \to \infty} d(x_n, a) = 0\right) \Leftrightarrow \left(\lim_{n \to \infty} \rho(x_n, a) = 0\right).
\]
Thus, \(d\) and \(\rho\) are topologically equivalent and \(0\) is an unique accumulation point of \(D(X, \rho)\).
\end{proof}

%\begin{theorem}\label{t5.16}
%Let \((X, d)\) be an ultrametric space. Then there are an ultrametric \(\rho \colon X \times X \to \RR^{+}\) and an ultrametric space \((Y, \delta)\) such that \(d\) and \(\rho\) are topologically equivalent, the equality
%\begin{equation}\label{t5.16:e1}
%D(X, \rho) = D(Y, \delta)
%\end{equation}
%holds and \((Y, \delta)\) is a compact ultrametric space.
%\end{theorem}
%
%\begin{proof}
%The theorem being proof is obviously true if \(|X| < \infty\) and follows from Proposition~\ref{p6.14} and Theorem~\ref{t5.10} if \((X, d)\) is not metrically discrete. Suppose \((X, d)\) is metrically discrete. Let \(r_0\) be a strictly positive real numbers such that
%\[
%d(x, y) \geqslant r_0
%\]
%holds for all distinct \(x\), \(y \in X\). Write, for all \(x\), \(y \in X\),
%\[
%\rho(x, y) = \begin{cases}
%r_0 & \text{if } x \neq y,\\
%0 & \text{if } x = y.
%\end{cases}
%\]
%Let \(Y\) be a two point subset of \(X\) and let \(\delta\) be the restriction of \(\rho\) on \(Y \times Y\), \(\delta = \rho|_{Y \times Y}\). Then \((X, \rho)\) and \((Y, \delta)\) are topologically equivalent and \eqref{t5.16:e1} holds.
%\end{proof}

\begin{remark}\label{r8.18}
Proposition~\ref{p6.14} is closely connected with the so-called rationalization problem for ultrametric spaces (\(=\) approximation of ultrametric space \((X, d)\) by ultrametric space \((X, \rho)\) such that \(D(X, \rho) \subseteq \mathbb{Q}\)) which was considered by Alex J.~Lemin and Vladimir Lemin \cite{LL1997TP}. Indeed, if, for every \(n \in \mathbb{N}\), we write \(\varepsilon_{n} = \left(\frac{1}{2}\right)^{n}\) in formula~\eqref{p6.14:e1}, then all values of ultrametric \(\rho = f \circ d\) are binary rational numbers.
\end{remark}

Theorem~3 from \cite{LL1997TP} and Proposition~\ref{p6.14}, which was proved above, can be partially reinforce as follows.

\begin{proposition}\label{p8.19}
Let \((X, d)\) be a bounded, non-discrete, ultrametric space and let \(\varepsilon > 0\), \(K > 1\) be given. Then there is an ultrametric \(\rho \colon X \times X \to \RR^{+}\) such that:
\begin{enumerate}
\item\label{p8.19:s1} The identity mapping \(\operatorname{id}\colon (X, d) \to (X, \rho)\) is non-expanding and \(K\)-Lipschitz, i.e., we have
\[
\rho(x, y) \leqslant d(x, y) \leqslant K \rho(x, y)
\]
for all \(x\), \(y \in X\).
\item\label{p8.19:s2} The inequality
\[
|d(x, y) - \rho(x, y)| \leqslant \varepsilon
\]
holds for all \(x\), \(y \in X\).
\item\label{p8.19:s3} All values of \(\rho\) are binary rational numbers.
\item\label{p8.19:s4} The distance set \(D(X, \rho)\) has the order type \(\mathbf{1} + \bm{\omega}^{*}\) and \(0\) is an unique accumulation point of \(D(X, \rho)\).
\end{enumerate}
\end{proposition}

\begin{remark}\label{r5.67}
One of the central concepts of the present section, the concept of weak similarity, was recently introduced in~\cite{DP2013AMH}, but this is closely related to the ordinal scaling, multidimensional scaling and ranking that appear naturally in many applied researches \cite{Agarwal2007, Borg2005, Jamieson2011, Kleindessner2014, KruP1964, QYJ2004, Rosales2006, SheP1962, SheJoMP1966, Wauthier2013}. The presence of these relationships and, thus, the existence of potential applications, makes the study of weak similarities more important and promising.
\end{remark}

\section{The first look on ultrametric balleans}
\label{sec6}

In this section we prove the inclusion \(\BB_{X} \subseteq \overline{\BB}_{X}\) for all totally bounded ultrametric spaces \((X, d)\) (see Corollary~\ref{c2.41}) and find the conditions under which \(\BB_{X} = \overline{\BB}_{X}\) holds (see Corollary~\ref{c6.7}). Moreover, Proposition~\ref{p3.10} describes some interrelations between the ballean \(\BB_{X}\) of ultrametric space \((X, d)\) and the ballean \(\BB_{Y}\) of dense \(Y \subseteq X\), that implies the order isomorphism of balleans of ultrametric spaces which have isometric completions, Theorem~\ref{t7.12}.

Recall that if \(A\) is a nonempty, bounded subset in an ultrametric space \((X, d)\), then there is the smallest ball \(B^{*} \in \overline{\BB}_{X}\) containing \(A\) (see Proposition~\ref{p2.12}).

\begin{proposition}\label{p2.40}
Let \((X, d)\) be a totally bounded ultrametric space, let \(B \in \mathbf{B}_X\) and let \(B^{*} \in \overline{\mathbf{B}}_X\) be the smallest ball containing \(B\). Then the equality
\begin{equation}\label{p2.40:e1}
B = B^{*}
\end{equation}
holds.
\end{proposition}

\begin{proof}
By Definition~\ref{d2.9}, we have the inclusion \(B^{*} \supseteq B\). Thus, \eqref{p2.40:e1} holds if
\begin{equation}\label{p2.40:e2}
B^{*} \subseteq B.
\end{equation}
Inclusion \eqref{p2.40:e2} is trivially valid if \(|B| = 1\). Let \(B = B_r(c)\) and let \(|B| \geqslant 2\) hold. By Corollary~\ref{c2.39}, the diametrical graph \(G_{B, d}\) is complete multipartite. Hence, \(G_{B, d}\) is nonempty and there is \(b \in B\) such that
\[
\diam B = d(b, c).
\]
Write \(r^{*} = \diam B\). Proposition~\ref{p2.12} implies the equality \(B^{*} = \overline{B}_{r^{*}}(c)\). Since \(b \in B\), we have the strict inequality \(r^{*} < r\). Inclusion~\eqref{p2.40:e2} follows.
\end{proof}

\begin{corollary}\label{c2.41}
Let \((X, d)\) be a totally bounded ultrametric space. Then the inclusion
\begin{equation}\label{c2.41:e1}
\mathbf{B}_X \subseteq \overline{\mathbf{B}}_X
\end{equation}
holds.
\end{corollary}

\begin{proof}
Let \(B \in \BB_{X}\). By Proposition~\ref{p2.12}, there is the smallest ball \(B^{*} \in \overline{\BB}_X\) containing \(B\). Proposition~\ref{p2.40} implies the equality \(B = B^{*}\). Hence, we have \(B \in \overline{\BB}_X\). Inclusion~\eqref{c2.41:e1} follows.
\end{proof}

The inclusion \(\BB_{X} \subseteq \overline{\BB}_{X}\) can be false if ultrametric space \((X, d)\) is not totally bounded.

\begin{example}\label{ex7.3}
As in Example~\ref{ex2.6}, let \(X = [0, 1]\) and let \(d\) be the ultrametric defined by formula \eqref{ex2.6:e1}. Write
\[
B_1(0) = \{x \in X \colon d(x, 0) < 1\}.
\]
Then \(B_1(0) = [0, 1)\) and, by definition of \(\BB_{X}\), \(B_1(0) \in \BB_{X}\). For every closed ball \(\overline{B}_r(a) \supseteq B_1(0)\) we have \(r \geqslant 1\), that implies \(1 \in \overline{B}_r(a)\). Thus, \(B_1(0) \notin \overline{\BB}_{X}\) holds.
\end{example}

\begin{lemma}\label{c6.6}
Let \(A\) be a subset of a totally bounded ultrametric space \((X, d)\) and let \(\diam A > 0\). Then the equivalence
\[
(A \in \mathbf{B}_X) \Leftrightarrow (A \in \overline{\mathbf{B}}_X)
\]
is valid.
\end{lemma}

\begin{proof}
The validity
\[
(A \in \mathbf{B}_X) \Rightarrow (A \in \overline{\mathbf{B}}_X)
\]
follows from Corollary~\ref{c2.41}. Suppose we have \(A \in \overline{\mathbf{B}}_X\). Let \(c^{*} \in A\). Write \(r^{*} = \diam A\). Then we have \(r^{*} \in D(X)\) and \(r^{*} > 0\). By Lemma~\ref{l2.6}, the condition \(A \in \overline{\mathbf{B}}_X\) implies
\begin{equation}\label{c6.6:e1}
A = \overline{B}_{r^*}(c^{*}) = \{x \in X \colon d(x, c^{*}) \leqslant r^{*}\}.
\end{equation}
Using Theorem~\ref{t5.10}, we can find \(\varepsilon > 0\) such that the annulus
\[
\{x \in X \colon r^{*} < d(x, c^{*}) < r^{*} + \varepsilon\}
\]
is empty. The last statement and \eqref{c6.6:e1} imply
\[
A = \{x \in X \colon d(x, c^{*}) < r^{*} + \varepsilon\} = B_{r^{*}+\varepsilon}(c^{*}).
\]
Thus, the membership relation \(A \in {\mathbf{B}}_X\) is valid.
\end{proof}

\begin{corollary}\label{c6.8}
Let \((X, d)\) be a nonempty totally bounded ultrametric space, \(B_1 \in \BB_{X}\) and \(B_1 \neq X\). Then there is \(B_2 \in \BB_{X}\) such that \(B_1\) is a part of the diametrical graph \(G_{B_2, d|_{B_2 \times B_2}}\).
\end{corollary}

\begin{proof}
Let \(b_1\) be a point of \(B_1\) and let
\begin{equation}\label{c6.8:e4}
D_{b_1} = \{d(x, b_1) \colon x \in X\}.
\end{equation}
It is clear that \(D_{b_1}\) is a subset of the distance set \(D(X)\). By Proposition~\ref{p2.4}, there is \(r_1 > 0\) such that \(B_1 = B_{r_1}(b_1)\). Condition \(B_1 \neq X\) implies \(\diam X \geqslant r_1\). Using equality~\eqref{e2.3} with \(A = X\) and \(a = b_1\) we obtain \(\sup D_{b_1} = \diam X\). By Proposition~\ref{p2.36}, there are some points \(a_1\), \(c_1 \in X\) for which
\begin{equation}\label{c6.8:e3}
d(a_1, c_1) = \diam X
\end{equation}
holds. Now from the strong triangle inequality
\[
\max\{d(a_1, b_1), d(b_1, c_1)\} \geqslant d(a_1, c_1)
\]
and \eqref{c6.8:e3} it follows that
\[
\diam X \in D_{b_1}.
\]
Consequently, the set \([r_1, \diam X] \cap D_{b_1}\) is nonempty. Theorem~\ref{t5.10} and the inequality \(r_1 > 0\) imply that the subposet \([r_1, \diam X] \cap D_{b_1}\) of \((\RR, {\leqslant})\) contains a smallest element. Let us consider the closed ball
\begin{equation}\label{c6.8:e1}
B_2 = \overline{B}_{r_2}(b_1)
\end{equation}
where \(r_2\) is the smallest element of \([r_1, \diam X] \cap D_{b_1}\). Since \(r_1 > 0\) and \(r_2 \geqslant r_1\), we have \(r_2 > 0\). From \(r_2 \in D_{b_1}\) and equalities \eqref{c6.8:e4}, \eqref{c6.8:e1} it follows that \(\diam B_2 = r_2  > 0\). Hence, \(B_2 \in \BB_{X}\) by Lemma~\ref{c6.6}. From the definition of \(r_2\) it follows that
\begin{equation}\label{c6.8:e2}
B_1 = B_{r_1}(b_1) = B_{r_2}(b_1).
\end{equation}
By Theorem~\ref{t2.24} (with \(X = B_1\)), equalities \eqref{c6.8:e1} and \eqref{c6.8:e2} imply that \(B_1\) is a part of \(G_{B_2, d|_{B_2 \times B_2}}\).
\end{proof}

\begin{remark}\label{r6.8}
It was shown in the proof of Corollary~\ref{c6.8} that the inequality
\begin{equation}\label{r6.8:e1}
\diam B_2 > \diam B_1
\end{equation}
holds but, in general, we can have \(\diam B_1 = \diam B_2\) for \(B_1\), \(B_2 \in \BB_{X}\), where \(B_1\) is a part of \(G_{B_2, d|_{B_2 \times B_2}}\).
\end{remark}

\begin{example}\label{ex6.9}
Let \((X, d)\) be the separable ultrametric space from Example~\ref{ex2.6}. Then the diametrical graph \(G_{X, d}\) is complete bipartite (more specifically \(G_{X, d}\) is a star having a countable many leaves) with the parts
\[
X_1 = [0, 1) \cap \mathbb{Q} \quad \text{and} \quad X_2 = \{1\}.
\]
Hence, the equality \(\diam X = \diam X_1\) holds.
\end{example}

Using Lemma~\ref{c6.6} we also can simply describe the totally bounded ultrametric spaces \((X, d)\) for which the inclusion \(\BB_{X} \subseteq \overline{\BB}_{X}\) turns into equality.

\begin{corollary}\label{c6.7}
Let \((X, d)\) be a totally bounded ultrametric space. Then the equality
\begin{equation}\label{c6.7:e1}
{\mathbf{B}}_X = \overline{\mathbf{B}}_X
\end{equation}
holds if and only if all points of \(X\) are isolated.
\end{corollary}

\begin{proof}
If all points of \(X\) are isolated, then we evidently have
\[
\{x\} \in \mathbf{B}_X \quad \text{and} \quad \{x\} \in \overline{\mathbf{B}}_X
\]
for every \(x \in X\). This statement and Lemma~\ref{c6.6} yield equality~\eqref{c6.7:e1}.

Conversely, suppose \(X\) contains an accumulation point \(x^{*}\). Then the closed ball \(\{x^{*}\} = \overline{B}_0(x^{*})\) belongs to \(\overline{\mathbf{B}}_X\) but \(\{x^{*}\} \notin \mathbf{B}_X\). Thus, \(\mathbf{B}_X \neq \overline{\mathbf{B}}_X\) holds.
\end{proof}

\begin{definition}\label{d6.9}
A topological space \((X, \tau)\) is said to be \emph{discrete} if every subset of \(X\) is open. We say that a metric space \((Y, \rho)\) is \emph{discrete} if the topology generated by metric \(\rho\) is discrete.
\end{definition}

It is easy to see that a metric space is discrete if and only if every point of this space is isolated. Thus, by Corollary~\ref{c6.7}, the equality \(\mathbf{B}_X = \overline{\mathbf{B}}_X\) holds for totally bounded ultrametric space \(X\) if and only if this is a discrete space.

\begin{example}\label{ex6.10}
Let \(X = [0,1] \subseteq \RR\) and let \(d \colon X \times X \to \RR^{+}\) be defined as in~\eqref{ex2.6:e1}. Then the ultrametric space \((X, d)\) is bounded and complete but not totally bounded. Using Theorem~\ref{t5.10} we can simply characterize all compact subspaces of \((X, d)\). Every finite subset of \(X\) is compact. An infinite \(Y \subseteq X\) is compact if and only if \(0 \in Y\), the equality \(\inf(Y \setminus \{0\}) = 0\) holds, and the order type of \((Y, {\leqslant}_Y)\) is \(\mathbf{1} + \bm{\omega}^{*}\), where \({\leqslant}_Y = {\leqslant} \cap Y^{2}\). Consequently, every infinite and compact \(Y \subseteq X\) contains the discrete dense set \(Y \setminus \{0\}\).
\end{example}

In the rest of the section, for arbitrary ultrametric space \(X\) and dense \(Y \subseteq X\), we introduce into consideration a ``natural'' bijection \(\BB_{X} \to \BB_{Y}\) and describe some properties of this bijection.

\begin{lemma}\label{l3.9}
Let \((X, d)\) be a metric space, \(Y\) be a dense subset of \(X\), and \(Z \subseteq X\) be nonempty, bonded and open in \(X\). Then the equality
\begin{equation}\label{l3.9:e1}
\diam Z = \diam (Z \cap Y)
\end{equation}
holds.
\end{lemma}

\begin{proof}
It is clear that
\begin{equation}\label{l3.9:e2}
\diam Z \geqslant \diam (Z \cap Y).
\end{equation}
Let \(\varepsilon > 0\). Then there exist \(a\), \(b \in Z\) and \(a^{*}\), \(b^{*} \in Z \cap Y\) such that
\[
\diam Z \leqslant d(a, b) + \varepsilon
\]
and
\[
d(a, a^{*}) \leqslant \varepsilon, \quad d(b, b^{*}) \leqslant \varepsilon.
\]
Using the triangle inequality we obtain
\[
|d(a, b) - d(a^{*}, b^{*})| \leqslant d(a, a^{*}) + d(b, b^{*}) \leqslant 2\varepsilon.
\]
Hence, the inequality
\[
|\diam Z - d(a^{*}, b^{*})| \leqslant 3\varepsilon
\]
holds. Consequently,
\begin{equation}\label{l3.9:e3}
\diam Z \leqslant d(a^{*}, b^{*}) + 3\varepsilon \leqslant \diam (Z \cap Y) + 3\varepsilon.
\end{equation}
Since \(\varepsilon\) is an arbitrary positive number, \eqref{l3.9:e3} implies
\[
\diam Z \leqslant \diam (Z \cap Y).
\]
Now \eqref{l3.9:e1} follows from the last inequality and \eqref{l3.9:e2}.
\end{proof}

The ballean \(\BB_{X}\) of metric space \((X, d)\) can be considered as a poset with the partial order \(\preccurlyeq_{X}\) generated by set inclusion
\begin{equation}\label{e7.18}
(B_1 \preccurlyeq_{X} B_2) \Leftrightarrow (B_1 \subseteq B_2), \quad B_1, B_2 \in \BB_{X}.
\end{equation}

\begin{proposition}\label{p3.10}
Let \((X, d)\) be a nonempty, ultrametric space and let \(Y\) be a dense subset of \(X\). Then the following statements hold:
\begin{enumerate}
\item \label{p3.10:s1} For every \(B \in \BB_{X}\), we have \(B \cap Y \in \BB_Y\).
\item \label{p3.10:s2} The mapping \(\Phi \colon \BB_{X} \to \BB_Y\) defined by
\begin{equation}\label{p3.10:e1}
\Phi(B) = B \cap Y
\end{equation}
is an order isomorphism of \((\BB_{X}, {\preccurlyeq}_X)\) and \((\BB_{Y}, {\preccurlyeq}_Y)\).
\item \label{p3.10:s3} The equality
\begin{equation}\label{p3.10:e2}
\diam \Phi(B) = \diam B
\end{equation}
holds for every \(B \in \BB_{X}\).
\end{enumerate}
\end{proposition}

\begin{proof}
\ref{p3.10:s1} Let \(B \in \BB_{X}\) and let \(r\) be  the radius of \(B\). Since \(B\) is a nonempty and open subset of \(X\), the intersection \(B \cap Y\) is nonempty. Let \(y^{*} \in B \cap Y\). By Proposition~\ref{p2.4}, we have
\begin{equation}\label{p3.10:e3}
B = \{x \in X \colon d(x, y^*) < r\}.
\end{equation}
The inclusion \(Y \subseteq X\) and equality~\eqref{p3.10:e3} imply that
\[
B \cap Y = \{y \in Y \colon d(y, y^*) < r\}.
\]
Thus, \(B \cap Y\) is an open ball in \((Y, d|_{Y \times Y})\) and \(\Phi\), defined by~\eqref{p3.10:e1}, is a mapping from \(\BB_{X}\) to \(\BB_Y\).

\ref{p3.10:s2} By definition, for every \(B^{*} \in \BB_Y\) there are \(c^{*} \in Y\) and \(r^{*} \in (0, \infty)\) such that
\[
B^{*} = \{y \in Y \colon d(y, c^{*}) < r^{*}\}.
\]
Consequently, \(B^{*} = Y \cap \{x \in X \colon d(x, c^{*}) < r^{*}\}\) holds. Hence, \(\Phi\) is surjective.

Now let \(B_1\) and \(B_2\) be different elements of \(\BB_{X}\). Then we have
\[
B_1 \setminus B_2 \neq \varnothing \quad \text{or} \quad B_2 \setminus B_1 \neq \varnothing,
\]
where \(B_i \setminus B_j\) is the relative complement of \(B_j\) in \(B_i\). Without loss of generality we assume
\begin{equation}\label{p3.10:e4}
B_1 \setminus B_2 \neq \varnothing.
\end{equation}
Proposition~\ref{p2.2} implies that \(B_1\) is open and \(B_2\) is closed (in \(X\)). Consequently, \(B_1 \setminus B_2\) is an open subset of \(X\). Since \(Y\) is dense, the existence of a point \(y_0 \in (B_1 \setminus B_2) \cap Y\) follows from~\eqref{p3.10:e4}. Consequently, we have
\[
y_0 \in B_1 \cap Y = \Phi(B_1) \quad \text{and} \quad y_0 \notin B_2 \cap Y = \Phi(B_2),
\]
i.e., \(\Phi(B_1) \neq \Phi(B_2)\). Hence, \(\Phi \colon \BB_{X} \to \BB_Y\) is injective. Thus, \(\Phi\) is a bijective mapping.

It is clear that the mapping \(\Phi\) is isotone. Since \(\Phi\) is isotone and bijective, this mapping is also strictly isotone and surjective. Moreover, it is clear that
\begin{equation}\label{p3.10:e5}
(B_1 \cap B_2 = \varnothing) \Rightarrow (\Phi(B_1) \cap \Phi(B_2) = \varnothing)
\end{equation}
is valid for all \(B_1\), \(B_2 \in \BB_{X}\). By Proposition~\ref{p2.5}, implication \eqref{p3.10:e5} can be rewriting in the form
\[
(B_1 \parallel_X B_2) \Rightarrow (\Phi(B_1) \parallel_Y \Phi(B_2)).
\]
Now using Lemma~\ref{l6.2} we see that \(\Phi\) is an order isomorphism.

\ref{p3.10:s3} Equality~\eqref{p3.10:e2} is a special case of equality~\eqref{l3.9:e1} because every \(B \in \BB_{X}\) is an open subset of \(X\).
\end{proof}

If \(Z\) and \(W\) are isometric ultrametric spaces with an isometry \(\Psi \colon Z \to W\), then the mapping
\[
\BB_{Z} \ni B \mapsto \Psi(B) \in \BB_{W}
\]
is an order isomorphism of \((\BB_{Z}, {\preccurlyeq_{Z}})\) and \((\BB_{W}, {\preccurlyeq_{W}})\). Consequently, Proposition~\ref{p3.10} implies the following.

\begin{theorem}\label{t7.12}
Let \((X, d)\) and \((Y, \rho)\) be nonempty, ultrametric spaces with completions \((\widetilde{X}, \widetilde{d})\) and \((\widetilde{Y}, \widetilde{\rho})\), respectively. If \((\widetilde{X}, \widetilde{d})\) and \((\widetilde{Y}, \widetilde{\rho})\) are isometric, then \((\BB_{X}, {\preccurlyeq_{X}})\) and \((\BB_{Y}, {\preccurlyeq_{Y}})\) are order isomorphic.
\end{theorem}

It will be shown in the proof of Proposition~\ref{p7.10} from the next section of the paper that the mapping \(\Phi \colon \BB_{X} \to \BB_{Y}\), defined by \eqref{p3.10:e1}, is an isometry of \(\BB_{X}\) and \(\BB_{Y}\) if \((X, d)\) is a totally bounded ultrametric space and if we consider \(\BB_{X}\) and \(\BB_{Y}\) as metric spaces with corresponding Hausdorff distances.

\section{Metric properties of ultrametric balleans}
\label{sec8}

In this section, we consider the sets of balls of completely bounded, ultrametric spaces \((X, d)\) together with the Hausdorff distance \(d_H\) on them. It is shown for every compact, ultrametric space \((X, d)\) that the space \((\overline{\BB}_X, d_H)\) is also a compact, ultrametric space, and that the ballean \(\BB_{X}\) of \((X, d)\) is dense (in \(\overline{\BB}_X\)) and coincides with the set of all isolated points of \((\overline{\BB}_X, d_H)\), Theorem~\ref{t7.8}. This result implies that every totally bounded ultrametric space admits an isometric embedding in a compact, ultrametric space with the dense set of isolated points, Theorem~\ref{p6.11}. Moreover, using the topological universality of Cantor set for the class of separable, ultrametric spaces, we show that every closed ball from the field of \(p\)-adic numbers \(\mathbb{Q}_p\) with \(p=2\) is also topologically universal for this class, Theorem~\ref{t8.15}.

Let \(A\) and \(B\) be two nonempty, bounded subsets of a metric space \((X, d)\). The \emph{Hausdorff distance} \(d_H(A, B)\) between \(A\) and \(B\) is defined by
\begin{equation}\label{e6.11}
d_H(A, B) := \max\left\{\sup_{a \in A} \inf_{b \in B} d(a, b), \sup_{b \in B} \inf_{a \in A} d(a, b)\right\}.
\end{equation}
This definition and some properties of the Hausdorff distance can be found in~\cite{BBI2001}. We note only that \eqref{e6.11} implies
\begin{equation}\label{e6.12}
d_H(\{a\}, \{b\}) = d(a, b)
\end{equation}
for all \(a\), \(b \in X\).

\begin{remark}\label{r8.21}
The Hausdorff distance was introduced by Felix Hausdorff in~1914 \cite{Hau1914}. Later, Edwards \cite{Edw1975} and, independently, Gromov~\cite{Gro1981PMIHS} expanded Hausdorff construction to the class of all compact metric spaces. A modification of Gromov---Hausdorff distance and Hausdorff distance on ultrametrics were considered by Derong Qiu in \cite{Qiu2009pNUAA} and \cite{Qiu2014pNUAA}, respectively.
\end{remark}

It is well-known that the Hausdorff distance \(d_H\) is a metric on the space of all bounded, closed, nonempty subsets of arbitrary metric space \((X, d)\) (see, for example, \cite[Proposition~7.3.3]{BBI2001}).

Using formula~\eqref{e6.12}, we easily get the following.

\begin{lemma}\label{l6.11}
For every nonempty metric space \((X, d)\) the mapping
\[
F \colon X \to \overline{\mathbf{B}}_X, \quad F(x) = \{x\},
\]
is an isometric embedding of \((X, d)\) in \((\overline{\mathbf{B}}_X, d_H)\).
\end{lemma}

It follows from Proposition~\ref{p2.12} and the next lemma that the Hausdorff distance between any two distinct balls \(B_1\) and \(B_2\) of a compact, ultrametric space coincides with the diameter of the smallest ball containing \(B_1 \cup B_2\).

\begin{lemma}\label{l6.12}
Let \((X, d)\) be a nonempty, ultrametric space. Then the equality
\begin{equation}\label{l6.12:e1}
d_H(B_1, B_2) = \diam(B_1 \cup B_2)
\end{equation}
holds for all distinct balls \(B_1\), \(B_2 \in \overline{\mathbf{B}}_X\).
\end{lemma}

\begin{proof}
Let \(B_1\) and \(B_2\) be distinct balls in \((X, d)\). If \(B_1 \cap B_2 = \varnothing\) holds, then equality \eqref{l6.12:e1} follows from \eqref{p2.5:e1} and \eqref{e6.11}. Suppose \(B_1 \cap B_2 \neq \varnothing\). Then, by Proposition~\ref{p2.5}, we have either \(B_1 \subset B_2\) or \(B_2 \subset B_1\). Without loss of generality we assume
\begin{equation}\label{l6.12:e2}
B_1 \subset B_2.
\end{equation}
If \(|B_1| = 1\), then there is \(p_1 \in X\) such that
\begin{equation}\label{l6.12:e3}
B_1  = \{p_1\}.
\end{equation}
Inclusion~\eqref{l6.12:e2} implies that \(p_1 \in B_2\). Using \eqref{e2.3}, \eqref{e6.11} and \eqref{l6.12:e3} we obtain
\begin{align*}
d_H(B_1, B_2) & = \max\left\{\inf_{p_2 \in B_2} d(p_1, p_2), \sup_{p_2 \in B_2} d(p_1, p_2)\right\} \\
& = \sup\{d(p_1, p_2) \colon p_2 \in B_2\} = \diam B_2.
\end{align*}
From \eqref{l6.12:e2} it follows that \(\diam B_2 = \diam (B_1 \cup B_2)\). Thus, \eqref{l6.12:e1} holds.

To complete the proof we must consider the case when
\begin{equation}\label{l6.12:e4}
B_1 \subset B_2
\end{equation}
and \(\diam B_1 > 0\). If \eqref{l6.12:e4} is true, then, in accordance with Lemma~2.1 \cite{Qiu2014pNUAA},
\begin{equation}\label{l6.12:e5}
d_H(B_1, B_2) = \dist (B_1, B_2)
\end{equation}
holds. Now \eqref{l6.12:e1} follows from \eqref{l6.12:e5} and \eqref{r2.10:e2} (see Remark \ref{r2.10}).
\end{proof}

The following lemma can be easily derived from Lemma~2.4 of \cite{Qiu2014pNUAA}.

\begin{lemma}\label{l6.13}
Let \((X, d)\) be a nonempty ultrametric space and let \(\overline{\BB}_X\) be the set of all closed balls of \((X, d)\). Then the metric space \((\overline{\BB}_X, d_H)\) is ultrametric.
\end{lemma}

The following lemma is a reformulation of the classical Blaschke theorem.
\begin{lemma}\label{l6.14}
Let \((X, d)\) be a compact, nonempty metric space and let \(\mathfrak{M}_X\) denote the set of all nonempty, closed subsets of \(X\). Then the metric space \((\mathfrak{M}_X, d_H)\) is also compact.
\end{lemma}

For the proof see, for example, \cite[Theorem~7.3.8]{BBI2001}.

\begin{lemma}\label{l6.15}
Let \((X, d)\) be a compact, nonempty ultrametric space. Then the metric space \((\overline{\BB}_X, d_H)\) is compact.
\end{lemma}

\begin{proof}
It is clear that \(\overline{\BB}_X \subseteq \mathfrak{M}_X\). By Lemma~\ref{l6.14}, the metric space \((\mathfrak{M}_X, d_H)\) is compact. Consequently, \(\overline{\BB}_X\) is totally bounded by Corollary~\ref{c2.17}. The Bolzano---Weierstrass property implies that \((\overline{\BB}_X, d_H)\) is compact if every sequence \((B_n)_{n \in \mathbb{N}} \subseteq \overline{\BB}_X\) contains a convergent subsequence. From Proposition~\ref{p2.11} it follows that every \((B_n)_{n \in \mathbb{N}} \subseteq \overline{\BB}_X\) contains a Cauchy subsequence \((B_{n_k})_{k \in \mathbb{N}}\). Using Lemma~\ref{l6.12} we see that if \((B_{n_k})_{k \in \mathbb{N}}\) is a Cauchy sequence, then either there are \(B \in \overline{\BB}_X\) and \(k_0 \in \mathbb{N}\) such that
\begin{equation}\label{l6.15:e1}
B_{n_k} = B
\end{equation}
for all \(k \geqslant k_0\), or the limit relation
\begin{equation}\label{l6.15:e2}
\lim_{n \to \infty} \diam B_{n_k} = 0
\end{equation}
holds. If we have \eqref{l6.15:e1} for all sufficiently large \(k\), then \((B_{n_k})_{k \in \mathbb{N}}\) is convergent and
\[
\lim_{k \to \infty} B_{n_k} = B.
\]
Suppose \eqref{l6.15:e2} holds. Since \((\mathfrak{M}_X, d_H)\) is compact and \(\overline{\BB}_X \subseteq \mathfrak{M}_X\), the Cauchy sequence \((B_{n_k})_{k \in \mathbb{N}}\) is convergent to a set \(A \in \mathfrak{M}_X\),
\begin{equation}\label{l6.15:e3}
\lim_{k \to \infty} d_H(A, B_{n_k}) = 0.
\end{equation}
From \eqref{l6.15:e3} and \eqref{e6.11} it follows that
\begin{equation}\label{l6.15:e4}
\lim_{k \to \infty} \sup_{x \in A} \inf_{y \in B_{n_k}} d(x, y)= 0.
\end{equation}
Using limit relations \eqref{l6.15:e2} and \eqref{l6.15:e4} it is easy to prove that there is \(a \in X\) for which \(\{a\} = A\) holds. To complete the proof it suffices to note that \(\{a\} \in \overline{\BB}_X\).
\end{proof}

\begin{lemma}\label{l7.6}
Let \((X, d)\) be a compact, nonempty ultrametric space. Then the ballean \(\BB_{X}\) of \((X, d)\) coincides with the set of all isolated points of the ultrametric space \((\overline{\BB}_X, d_H)\).
\end{lemma}

\begin{proof}
By Corollary~\ref{c2.41}, we have the inclusion
\begin{equation}\label{l7.6:e1}
\BB_{X} \subseteq \overline{\BB}_X.
\end{equation}
We first show that every \(B^* \in \BB_{X}\) is an isolated point of \((\overline{\BB}_X, d_H)\). If \(B^* \in \BB_{X}\) and \(\diam B^* > 0\), then, by Lemma~\ref{l6.12}, we obtain
\[
d_H(B^*, B) = \diam (B^* \cup B) \geqslant \diam B^* > 0
\]
for every \(B \in \overline{\BB}_X\). Thus, \(B^*\) is an isolated point in \((\overline{\BB}_X, d_H)\).

Let us consider now the case when \(B^* \in \BB_{X}\) and \(\diam B^* = 0\). In this case there is \(c^{*} \in X\) such that \(B^* = \{c^{*}\}\). Since \(B^{*}\) is a clopen subset of \((X, d)\), the point \(c^{*}\) is an isolated point of \((X, d)\) (by Proposition~\ref{p2.2}). Consequently, there is \(\varepsilon^{*} > 0\) such that
\begin{equation}\label{l7.6:e2}
d(c^*, x) > \varepsilon^{*} > 0
\end{equation}
for every \(x \in X \setminus \{c^{*}\}\). Double inequality~\eqref{l7.6:e2} and definition \eqref{e6.11} imply
\[
d_H(B^*, B) > \varepsilon^{*} > 0
\]
whenever \(B \in \BB_{X}\) and \(B \neq B^{*}\).

To complete the proof it suffices to show that every ball
\begin{equation}\label{l7.6:e3}
B^* \in \overline{\BB}_X \setminus \BB_{X}
\end{equation}
is an accumulation point of \((\overline{\BB}_X, d_H)\).

Suppose \eqref{l7.6:e3} holds. Then, by Lemma~\ref{c6.6}, we obtain
\[
B^{*} = \{c^{*}\}
\]
for a point \(c^{*} \in X\) and, in addition, the point \(c^{*}\) is an accumulation point of \(X\), by Proposition~\ref{p2.2}. Hence, for every \(\varepsilon > 0\), the ball \(B_{\varepsilon} (c^{*}) \in \overline{\BB}_X\) satisfies the inequality
\[
d_H(B^{*}, B_{\varepsilon} (c^{*})) \leqslant \varepsilon
\]
and the condition \(B^{*} \neq B_{\varepsilon} (c^{*})\). Thus, \(B^{*}\) is an accumulation point in \((\overline{\BB}_X, d_H)\).
\end{proof}

Analyzing the proof of Lemma~\ref{l7.6} we also obtain the following.

\begin{lemma}\label{l7.7}
Let \((X, d)\) be a compact, nonempty ultrametric space. Then the set of all isolated points of \((\overline{\BB}_X, d_H)\) is a dense subset of \(\overline{\BB}_X\).
\end{lemma}

Putting together all lemmas proved above, we obtain the following result.

\begin{theorem}\label{t7.8}
Let \((X, d)\) be a compact, nonempty ultrametric space and let \((\overline{\BB}_X, d_H)\) be the space of all closed balls of \((X, d)\) with the Hausdorff metric \(d_H\). Then \((\overline{\BB}_X, d_H)\) is also a compact, ultrametric space for which the ballean \(\BB_{X} \subseteq \overline{\BB}_X\) coincides with the set of all isolated points of \((\overline{\BB}_X, d_H)\) and this set is dense in \(\overline{\BB}_X\).
\end{theorem}

\begin{remark}\label{r8.10}
Using Proposition~\ref{p2.5} (see, in particular, formulas \eqref{p2.5:e1} and \eqref{r2.10:e2}) one can easily proves that, for ultrametric space \((X, d)\), a sequence \((\overline{B}_i)_{i \in \mathbb{N}} \subseteq \overline{\BB}_X\) is convergent in \((\overline{\BB}_X, d_H)\) if and only if \((\overline{B}_i)_{i \in \mathbb{N}}\) is convergent in the so-called Wijsman topology, which was firstly studied by Lechicki and Levi~\cite{LL1987BUMI}. It is interesting to note that, for every metric space \((Y, \rho)\), a sequence of subsets of \(Y\) is convergent with respect to the Vietoris topology generated by \(\rho\) if and only if this sequence is Wijsman convergent for all metrics which are topologically equivalent to \(\rho\) \cite{BNLL1992AMPAIS}. See also \cite{BDD2017BBMSSS} for characterization of Wijsman convergence via Wijsman statistical convergence by moduli.
\end{remark}

Let \((X, d)\) and \((Y, \rho)\) be metric spaces. An injective mapping \(\Phi \colon X \to Y\) is called an \emph{isometric embedding} of \((X, d)\) in \((Y, \rho)\) if the equality
\[
d(x, y) = \rho(\Phi(x), \Phi(y))
\]
holds for all \(x\), \(y \in X\). In this case we write \((X, d) \hookrightarrow (Y, \rho)\) and say that \((X, d)\) \emph{admits an isometric embedding} in \((Y, \rho)\).

\begin{theorem}\label{p6.11}
Let \((X, d)\) be a totally bounded ultrametric space. Then there is a compact ultrametric space \((Y, \rho)\) such that the set of isolated points of \((Y, \rho)\) is a dense subset of \(Y\) and \((X, d)\) admits an isometric embedding \((X, d) \hookrightarrow (Y, \rho)\).
\end{theorem}

\begin{proof}
Let \((\widetilde{X}, \widetilde{d})\) be the completion of \((X, d)\). Then \((\widetilde{X}, \widetilde{d})\) is compact (by Proposition~\ref{p2.13}) and ultrametric (by Lemma~\ref{l2.20}). It follows directly from the definition of \((\widetilde{X}, \widetilde{d})\) that there is an isometric embedding
\[
(X, d) \xrightarrow{\Phi} (\widetilde{X}, \widetilde{d}).
\]
By Lemma~\ref{l6.11}, the mapping
\[
F \colon \widetilde{X} \to \overline{\BB}_X, \quad F(x) = \{x\}
\]
is also an isometric embedding. Consequently,
\[
X \xrightarrow{\Phi} \widetilde{X} \xrightarrow{F} \overline{\BB}_X
\]
is an isometric embedding of \((X, d)\) in \((\overline{\BB}_X, d_H)\). The metric space \((\overline{\BB}_X, d_H)\) is compact and ultrametric with the dense set of isolated points (by Theorem~\ref{t7.8}).
\end{proof}

Theorem~\ref{p6.11} has some extensions for subspaces \((X, d)\) of finite dimensional Euclidean space \(E^{n}\) and corresponding subspaces \((Y, \rho)\) of \(E^{n+1}\). In this case the isometric embedding \((X, d) \hookrightarrow (Y, \rho)\) admits a clear geometric interpretation (see Figure~\ref{fig1}).

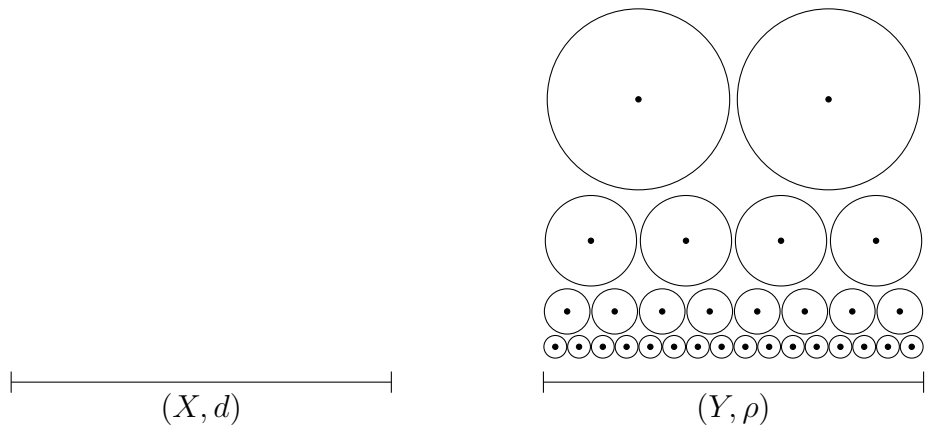
\begin{figure}[htb]
\begin{center}
\begin{tikzpicture}[x=1cm, y=1cm,
level 1/.style={level distance=1.5cm,sibling distance=25mm},
level 2/.style={level distance=1.5cm,sibling distance=12mm},
level 3/.style={level distance=1.5cm,sibling distance=6mm},
solid node/.style={circle,draw,inner sep=1.5,fill=black},
hollow node/.style={circle,draw,inner sep=1.5}]
\def\vRule(#1,#2){\draw (#1,#2-4pt)--(#1,#2+4pt)}

\vRule(0,0); \vRule(5,0);
\draw (0,0) -- node [below] {\((X, d)\)} (5cm,0);

\vRule(7,0); \vRule(12,0);
\draw (7,0) -- node [below] {\((Y, \rho)\)} (12,0);
\def\xx{2.5}; \def\dx{2.4cm};

\foreach \x in {1,3}{
	\draw (7+\x*\xx/2, 3*\xx/2) circle (\dx/2);
	\draw [fill, black] (7+\x*\xx/2, 3*\xx/2) circle (1pt);
}
\foreach \x in {1,3,5,7}{
	\draw (7+\x*\xx/4, 3*\xx/4) circle (\dx/4);
	\draw [fill, black] (7+\x*\xx/4, 3*\xx/4) circle (1pt);
}
\foreach \x in {1,3,...,15}{
	\draw (7+\x*\xx/8, 3*\xx/8) circle (\dx/8);
	\draw [fill, black] (7+\x*\xx/8, 3*\xx/8) circle (1pt);
}
\foreach \x in {1,3,...,31}{
	\draw (7+\x*\xx/16, 3*\xx/16) circle (\dx/16);
	\draw [fill, black] (7+\x*\xx/16, 3*\xx/16) circle (1pt);
}
\end{tikzpicture}
\end{center}
\caption{\((X, d)\) is an interval \(I\) in \(E^{1} = \RR\) and \((Y, \rho)\) is a subspace of the Euclidean plane \(E^{2}\) such that \(Y \supset I\), all points \(y \in Y \setminus I\) are isolated and the distances between them decrease as they approach to~\(I\).}
\label{fig1}
\end{figure}

\begin{proposition}\label{p7.10}
Let \((X, d)\) be a totally bounded ultrametric space and let \((\widetilde{X}, \widetilde{d})\) be the completion of \((X, d)\). Then the metric spaces \((\BB_{\widetilde{X}}, \widetilde{d}_H)\) and \((\BB_{X}, d_H)\) are isometric.
\end{proposition}

\begin{proof}
It is clear that, if \((X_1, d_1)\) and \((X_2, d_2)\) are isometric ultrametric spaces, then \((\BB_{X_1}, d_{1, H})\) and \((\BB_{X_2}, d_{2, H})\) are also isometric. Consequently, without loss of generality, we suppose that \(X\) is a dense subset of \(\widetilde{X}\) and \(d = \widetilde{d}|_{X \times X}\).

By Proposition~\ref{p3.10}, the mapping
\[
\Phi \colon \BB_{\widetilde{X}} \to \BB_{X}, \quad \Phi(B) = B \cap X,
\]
is bijective.

Let \(B_1\) and \(B_2\) be different elements of \(\BB_{\widetilde{X}}\). Using Lemma~\ref{l3.9} and
\[
(B_1 \cup B_2) \cap X = (B_1 \cap X) \cup (B_2 \cap X)
\]
we obtain
\begin{equation*}
\diam (B_1 \cup B_2) = \diam ((B_1 \cup B_2) \cap X) = \diam (\Phi(B_1) \cup \Phi(B_2)).
\end{equation*}
By Lemma~\ref{l6.12} it implies
\[
\widetilde{d}_H (B_1, B_2) = d_H (\Phi(B_1), \Phi(B_2)).
\]
Hence, \(\Phi\) is an isometry as required.
\end{proof}

\begin{corollary}\label{c7.11}
Let \((X, d)\) and \((Y, \rho)\) be totally bounded ultrametric spaces. If the completion \((\widetilde{X}, \widetilde{d})\) and \((\widetilde{Y}, \widetilde{\rho})\) are isometric, then the ultrametric spaces \((\BB_{X}, d_H)\) and \((\BB_Y, \rho_H)\) are isometric.
\end{corollary}

Let \(\mathcal{F}\) be a family of metric spaces. A metric space \((Y, \rho)\) is \emph{metrically} (\emph{topologically}) \emph{universal} for \(\mathcal{F}\) if every \((X, d) \in \mathcal{F}\) admits an isometric (topological) embedding \((X, d) \hookrightarrow (Y, \rho)\).

Theorem~\ref{p6.11} gives rise the following interesting question.

\begin{problem}\label{pr6.1}
Let \(\mathcal{F}\) be a family of totally bounded ultrametric spaces. Under what condition there exists a totally bounded ultrametric space which is metrically universal for \(\mathcal{F}\)?
\end{problem}

The following proposition and Theorem~\ref{t6.13} below show that the sets of isolated points of totally bounded ultrametric spaces and, consequently, of open balls in such spaces are small from set-theoretic point of view.

\begin{proposition}\label{p6.12}
Let \((X, d)\) be a totally bounded metric space and let \(I = I(X)\) be the set of all isolated points of \((X, d)\). Then \(I\) is at most countable.
\end{proposition}

\begin{proof}
Let \((\varepsilon_n)_{n \in \mathbb{N}} \subseteq \RR^{+}\) be a strictly increasing sequence with
\begin{equation}\label{p6.12:e1}
\lim_{n \to \infty} \varepsilon_n = 0.
\end{equation}
If we define a subset \(I_n\) of \(I\) as
\[
I_n = \{x \in I \colon B_{\varepsilon_n}(x) \cap X = \{x\}\},
\]
then limit relation~\eqref{p6.12:e1} implies the equality
\[
I = \bigcup_{n \in \mathbb{N}} I_n.
\]
Consequently, it suffices to show that every \(I_n\) is finite, \(|I_n| < \infty\). To prove the last inequality, we note that \(I_n\) is totally bounded (by Corollary~\ref{c2.17}) and use Definition~\ref{d2.10}.
\end{proof}

\begin{theorem}\label{t6.13}
Let \((X, d)\) be a totally bounded ultrametric space with the distance set \(D = D(X)\). Then the ballean \(\mathbf{B}_X\) is at most countable and, for every strictly positive \(r \in D\), the set
\begin{equation*}%\label{t6.13:e1}
\mathbf{B}_X(r) = \{B \in \mathbf{B}_X \colon \diam B = r\}
\end{equation*}
is finite and nonempty, \(1 \leqslant |\mathbf{B}_X(r)| < \infty\).
\end{theorem}

\begin{proof}
Let \((\widetilde{X}, \widetilde{d})\) be the completion of \((X, d)\). Then \((\BB_{\widetilde{X}}, \widetilde{d}_H)\) and \((\BB_{X}, d_H)\) are isometric by Proposition~\ref{p7.10}. The set \(\BB_{\widetilde{X}}\) coincides with the set of all isolated points of the compact ultrametric space \((\BB_{\widetilde{X}}, \widetilde{d}_H)\) by Theorem~\ref{t7.8}. Consequently, \(\BB_{\widetilde{X}}\) is at most countable by Proposition~\ref{p6.12}. Hence, \(\BB_{X}\) is also at most countable.

Let \(r \in D\) be strictly positive. If \(r = \diam X\), then the set \(\mathbf{B}_X(r)\) contains the unique ball \(B = X\). Suppose \(0 < r < \diam X\) holds. By Theorem~\ref{t5.10}, we can find \(r^{*} \in \RR^{+}\) such that \(r < r^{*}\) and
\[
(r, r^{*}) \cap D = \varnothing.
\]
Let us consider the family
\[
\mathcal{F}_r = \{B_{r^{*}}(x) \colon x \in X\}.
\]
It follows from Proposition~\ref{p2.5} that there is \(X^{*} \subseteq X\) such that, for all distinct \(y^{*}\), \(z^{*} \in X^{*}\),
\begin{equation}\label{t6.13:e2}
B_{r^*}(y^{*}) \cap B_{r^*}(z^{*}) = \varnothing
\end{equation}
and
\[
\bigcup_{x \in X} B_{r^*}(x) = \bigcup_{x^{*} \in X^{*}} B_{r^*}(x^{*}).
\]

Let \(\mathcal{F}_r^{*} = \{B_{r^*}(x^{*}) \colon x^{*} \in X^{*}\}\). Then the inclusion
\begin{equation}\label{t6.13:e3}
\mathbf{B}_X(r) \subseteq \mathcal{F}_r^{*}
\end{equation}
holds. Using Definition~\ref{d2.3} with \(A = X\) and \(\mathcal{F} = \mathcal{F}_r^{*}\) we see that~\eqref{t6.13:e2} implies \(|\mathcal{F}_r^{*}| < \infty\). The inequality \(|\mathbf{B}_X(r)| < \infty\) follows from \(|\mathcal{F}_r^{*}| < \infty\) and \eqref{t6.13:e3}.
\end{proof}

In the rest of the section we will discuss some questions related to topologically universal, compact, ultrametric spaces and metrically universal ones.

\begin{remark}\label{r8.15}
The universal spaces for the class $\mathfrak{S}$ of all separable metric spaces are studied in some details. The investigations in this direction were initiated by P.~S.~Urysohn for topologically \(\mathfrak{S}\)-universal spaces in~\cite{Ury1925MA} and in~\cite{Ury1925CRASP, Ury1927BSM} for metrically \(\mathfrak{S}\)-universal ones. The space \(C([0,1])\) of all continuous, real-valued functions with the sup-norm is separable and metrically \(\mathfrak{S}\)-universal by famous Banach theorem~\cite{Ban1932}. Metrically universal, separable, metric spaces are constructed S. Iliadis~\cite{Ili2005} for classes of separable metric spaces having some restrictions on dimension of these spaces. (See also \cite{Ili2012CSP, Ili2013TA, Pol1986TAMS, Nag1959FM, Noe1930MA, Smi1962AMSTS2, Wen1973FM, IN2013TA}.) The class \(\mathfrak{SU}\) of all ultrametric, separable spaces is an important subclass of \(\mathfrak{S}\). It was proved by A.~F.~Timan and I.~A.~Vestfrid \cite{TV1983FAA} that the space $l_2$ of real sequences $(x_n)_{n\in\mathbb N}$ with the norm $\Bigl(\sum_{n\in\mathbb N}x_{n}^{2}\Bigr)^{1/2}$ is metrically \(\mathfrak{SU}\)-universal and it is clear that \(l_2\) is not metrically \(\mathfrak{S}\)-universal. The first example of metrically \(\mathfrak{SU}\)-universal ultrametric space was constructed by I.~A.~Vestfrid \cite{Ves1994UMJ}. As was proved by A.~Lemin and V.~Lemin in \cite{LL2000TaiA}, if an ultrametric space \((Y, \rho)\) is metrically universal for all two-points spaces, then the weight of \((Y, \rho)\) is not less than the continuum $\mathfrak{c}$. Consequently, no $\mathfrak{SU}$-universal, ultrametric space can be separable. In this connection it should be pointed out that, under some set-theoretic assumptions, for every cardinal $\tau> \mathfrak{c}$ there is an ultrametric space \((X, d)\) of the weight \(\tau\) such that every ultrametric space with the weight \(\tau\) can be isometrically embedded into \((X, d)\). The last statement was proved by J.~Vaughan~\cite{Vau1999TP}. Some results related to ``minimal'', metrically universal, ultrametric spaces can be found in~\cite{BDKP2017AASFM}.
\end{remark}

The following theorem implies that every nonempty open set in the field \(\mathbb{Q}_2\) is topologically universal for the class \(\mathfrak{SU}\) of all separable, ultrametric spaces.

\begin{theorem}\label{t8.15}
Let \((\mathbb{Q}_2, d_2)\) be the ultrametric space of \(2\)-adic numbers, \(d_2 (x, y) = |x-y|_2\). Then every \(B \in \BB_{\mathbb{Q}_2}\) is topologically universal for the class \(\mathfrak{SU}\) of all separable, ultrametric spaces.
\end{theorem}

\begin{proof}
Theorem~6.2.16 \cite{Eng1989} implies that the Cantor set \(D^{\aleph_0}\) is a topologically universal space for \(\mathfrak{SU}\). It was shown in Example~\ref{ex3.11} that \(D^{\aleph_0}\) is homeomorphic the closed unit ball \(\overline{B}_1(0)\) of ultrametric space \((\mathbb{Q}_2, d_2)\). From Theorem~51 \cite{Sch1985} it follows  that every bounded subset of \((\mathbb{Q}_p, d_p)\) is totally bounded. Hence, by Lemma~\ref{c6.6}, we obtain \(\overline{B}_1(0) \in \BB_{\mathbb{Q}_2}\). For every nonzero \(a \in \mathbb{Q}_2\) and every \(b \in \mathbb{Q}_2\) the function \(F \colon \mathbb{Q}_2 \to \mathbb{Q}_2\) defined as
\begin{equation}\label{t8.15:e1}
F(x) = ax + b
\end{equation}
is simultaneously a linear bijection on the field \(\mathbb{Q}_2\) and a self-homeomorphism of the ultrametric space \((\mathbb{Q}_2, d_2)\). The distance set \(D(\mathbb{Q}_2, d_2)\) of the ultrametric space \((\mathbb{Q}_2, d_2)\) is
\[
\{2^{j} \colon j \in \mathbb{Z}\} \cup \{0\}.
\]
Hence, the inclusion
\begin{equation}\label{t8.15:e2}
D(\mathbb{Q}_2, d_2) \subseteq \mathbb{Q}_2
\end{equation}
holds. Now let \(B \in \BB_{\mathbb{Q}_2}\) be given and let \(b \in B\). Write \(a = \diam B\) and consider \(F\) defined by \eqref{t8.15:e1}. From \(b \in B \subseteq \mathbb{Q}_2\) and \eqref{t8.15:e2} it follows that \(F\) is correctly defined and \(F(\overline{B}_1(0)) = B\) holds. The restriction \(F|_{\overline{B}_1(0)}\) is a homeomorphism of \(\overline{B}_1(0)\) and \(B\). Thus, \(B\) is topologically universal for \(\mathfrak{SU}\) as required.
\end{proof}

\begin{remark}\label{r7.19}
Isometrical embeddings of ultrametric spaces in non-Archimedean valued field were constructed by W. H. Schikhof \cite{SchIM1984}.
\end{remark}

It can be proved that a metric space \((X, d)\) is separable if and only if there is a metric \(\rho\) such that \(d\) and \(\rho\) are topologically equivalent and \((X, \rho)\) is totally bounded. (It is a reformulation of Theorem~4.3.5 \cite{Eng1989}.) Theorem~\ref{t8.15} gives us the following refinement of this result for ultrametric spaces.

\begin{corollary}\label{c8.16}
Let \((X, d)\) be an ultrametric space. Then \((X, d)\) is separable if and only if there is an ultrametric \(\rho \colon X \times X \to \RR^{+}\) such that \((X, \rho)\) is totally bounded and \(\rho\) is topologically equivalent to \(d\).
\end{corollary}

Corollary~\ref{c8.16} and Theorem~\ref{t5.10} imply that, for every infinite separable ultrametric space \((X, d)\), there is an ultrametric \(\rho \colon X \times X \to \RR^{+}\) such that \(d\) and \(\rho\) are topologically equivalent, \(0\) is an unique accumulation point of \(D(X, \rho)\), and the (ordered) distance set
\[
D = \{\rho(x, y) \colon x, y \in X\}
\]
has the order type \(\mathbf{1} + \bm{\omega}^{*}\).

\section{Free representing trees}\label{sec9}

The main goal of the section is to introduce into consideration a (free) representing tree \(T_X\) for every nonempty totally bounded ultrametric space \((X, d)\) and prove some properties of \(T_X\). Using concept of diametrical graph, we inductively construct \(T_X\) for every nonempty totally bounded ultrametric space \((X, d)\) (see Definition~\ref{d9.1} below). In Proposition~\ref{p8.4} it is proved that \(T_X\) is a locally finite tree for which the equality \(V(T_X) = \BB_{X}\) holds. The structure of representing trees \(T_X\) corresponding to ultrametric compacts which are homeomorphic to the Cantor set is described in Proposition~\ref{p9.10}.

\begin{definition}\label{d9.1}
Let \((X, d)\) be a nonempty totally bounded ultrametric space $(X, d)$. Then we associate with \((X, d)\) a graph \(T_X\) by the following rule:

\begin{enumerate}
\renewcommand{\theenumi}{\ensuremath{(i_{\arabic{enumi}})}}
\item\label{d9.1:s1} \(V(T_X)\) is a subset of \(2^{X} = \{A \colon A \subseteq X\}\).
\item\label{d9.1:s2} The set \(X\) is a vertex of \(T_X\).
\item\label{d9.1:s3} A set \(A \subset X\) belongs to \(V(T_X)\) if and only if there exists a finite sequence \(X_1\), \(\ldots\), \(X_n\) of subsets of \(X\) such that \(X_1 = X\), \(X_n = A\) and, for every \(i \in \{1, \ldots, n-1\}\), \(X_{i+1}\) is a part of \(G_{X_i, d|_{X_i \times X_i}}\).
\item\label{d9.1:s4} The vertices \(A\), \(B \in V(T_X)\) are adjacent if and only if \(A\) is a part of \(G_{B, d|_{B \times B}}\) or \(B\) is a part of \(G_{A, d|_{A \times A}}\).
\end{enumerate}
\end{definition}

\begin{remark}\label{r8.2}
Let \(A\) be a subset of totally bounded ultrametric space \((X, d)\). Using Example~\ref{ex2.24}, Theorem~\ref{t2.24} and Corollary~\ref{c2.39} it is easy to prove that the diametrical graph \(G_{A, d|_{A \times A}}\) is complete multipartite if and only if \(|A| \geqslant 2\) holds. Now in the correspondence with condition~\ref{d9.1:s3} of Definition~\ref{d9.1}, we see that if \(B \in V(T_X)\) and \(B \neq X\), then there is \(A \in V(T_X)\) such that \(B\) is a part of \(G_{A, d|_{A \times A}}\). Consequently, every \(B \in V(T_X)\) is a nonempty subset of \(X\). Moreover, condition~\ref{d9.1:s4} implies that the inequality
\[
\max \{|A|, |B|\} \geqslant 2
\]
holds whenever \(\{A, B\} \in E(T_X)\).
\end{remark}

To study the properties of \(T_X\), it is convenient to rank the set \(V(T_X)\) by labeling
\[
V(T_X) \ni A \mapsto \lev(A) \in \mathbb{N} \cup \{0\}
\]
defined as \(\lev(X) = 0\) and if \(A \in V(T_X)\) and \(A \neq X\), then
\begin{equation}\label{e8.1}
\lev(A) = n-1,
\end{equation}
where \(n\) is the smallest positive integer satisfying condition \ref{d9.1:s3} of Definition~\ref{d9.1}. We will say that \(A\) is a vertex on the \emph{level} \(n-1\) (or \(A\) has the level \(n-1\)) if~\eqref{e8.1} holds.

\begin{example}\label{ex8.3}
Let \(X = \{x_1, x_2, x_3, x_4\}\) and let \(d \colon X \times X \to \RR^{+}\) be defined as
\[
d(x, y) = \begin{cases}
0 & \text{if } x = y,\\
2 & \text{if } x_1 \in \{x, y\} \text{ and } x \neq y,\\
1 & \text{othetwise}.
\end{cases}
\]
Then \((X, d)\) is a finite ultrametric space. For the graph \(T_X\) (see Figure~\ref{fig2}) we have
\begin{gather*}
V(T_X) = \bigl\{\{x_1, x_2, x_3, x_4\}, \{x_2, x_3, x_4\}, \{x_1\}, \{x_2\}, \{x_3\}, \{x_4\}\bigr\}, \\
\lev(\{x_1, x_2, x_3, x_4\}) = 0, \quad \lev(\{x_1\}) = \lev(\{x_2, x_3, x_4\}) = 1, \\
\lev(\{x_2\}) = \lev(\{x_3\}) = \lev(\{x_4\}) = 2,\\
\begin{aligned}
E(T_X) & = \Bigl\{\bigl\{\{x_1, x_2, x_3, x_4\}, \{x_2, x_3, x_4\}\bigr\}, \bigl\{\{x_1, x_2, x_3, x_4\}, \{x_1\}\bigr\}, \\
& \bigl\{\{x_2, x_3, x_4\}, \{x_2\}\bigr\}, \bigl\{\{x_2, x_3, x_4\}, \{x_3\}\bigr\}, \bigl\{\{x_2, x_3, x_4\}, \{x_4\}\bigr\}\Bigr\}.
\end{aligned}
\end{gather*}
\end{example}

\begin{figure}[ht]
\begin{center}
\begin{tikzpicture}[
level 1/.style={level distance=1cm,sibling distance=15mm},
level 2/.style={level distance=1cm,sibling distance=10mm},
level 3/.style={level distance=1cm,sibling distance=10mm},
solid node/.style={circle,draw,inner sep=1.5,fill=black},
hollow node/.style={circle,draw,inner sep=1.5}]

\node [solid node] at (0,0) {}
child{node [solid node]{}}
child{node [solid node]{}
	child{node [solid node]{}}
	child{node [solid node]{}}
	child{node [solid node]{}}
};
\end{tikzpicture}
\end{center}
\caption{The graph \(T_X\) corresponding to \((X, d)\) is a tree.}
\label{fig2}
\end{figure}
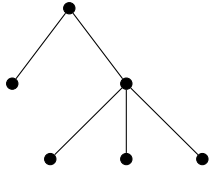

\begin{lemma}\label{l8.2}
Let \((X, d)\) be a nonempty totally bounded ultrametric space. Then the equality
\[
V(T_X) = \BB_{X}
\]
holds.
\end{lemma}

\begin{proof}
This lemma is trivial if \(|X| = 1\). Suppose that \(|X| \geqslant 2\). At first we prove the inclusion \(V(T_X) \subseteq \BB_{X}\).

It is clear that \(X \in \BB_{X}\). Let \(G_{X, d}\) be the diametrical graph of \((X ,d)\). Then, by Theorem~\ref{t2.24} and Proposition~\ref{p2.36}, there is an integer \(k \geqslant 2\) such that \(G_{X, d}\) is complete \(k\)-partite with some parts \(X_1\), \(\ldots\), \(X_k \in \BB_{X}\). Since \(X\) is the unique subset of \(X\) on the zero level, a set \(A \subseteq X\) belongs to \(V(T_X)\) and has the level \(1\), \(\lev(A) = 1\), if and only if \(A = X_i\) for some \(i \in \{1, \ldots, k\}\).

Let us define a subset \(\mathcal{F}_1\) of \(\{X_1, \ldots, X_n\}\) as
\[
\mathcal{F}_1 := \{X_i \colon |X_i| \geqslant 2 \text{ and } 1 \leqslant i \leqslant k\}.
\]
If \(\mathcal{F}_1 = \varnothing\), then \(V(T_X) = \{X\} \cup \{X_1, \ldots, X_n\}\) and the inclusion \(V(T_X) \subseteq \BB_{X}\) follows from Theorem~\ref{t2.24}. Let \(\mathcal{F}_1 \neq \varnothing\) and \(X^{*} \in \mathcal{F}_1\). Then \((X^{*}, d|_{X^{*} \times X^{*}})\) is a totally bounded ultrametric space by Proposition~\ref{p2.11}. Consequently, by Proposition~\ref{p2.36}, the diametrical graph \(G_{X^{*}, d|_{X^{*} \times X^{*}}}\) is complete multipartite and, moreover, by Theorem~\ref{t2.24}, every part of \(G_{X^{*}, d|_{X^{*} \times X^{*}}}\) is an open ball in the space \((X^{*}, d|_{X^{*} \times X^{*}})\). By Corollary~\ref{c2.4}, we have the inclusion
\[
\BB_{X^{*}} \subseteq \BB_{X}.
\]
Consequently, if \(A \in V(T_X)\) and \(\lev(A) = 2\), then the membership
\begin{equation}\label{l8.2:e1}
A \in \BB_{X}
\end{equation}
is valid. Repeating the above reasons for \(X^{*} \in V(T_X)\) with \(\diam X^{*} > 0\) and \(\lev(X^{*}) = n\), \(n = 2\), \(3\), \(\ldots\) we obtain \eqref{l8.2:e1} for all \(A \in V(T_X)\).

Now the desirable inclusion
\begin{equation}\label{l8.2:e2}
V(T_X) \subseteq \BB_{X}
\end{equation}
follows from Definition~\ref{d9.1}.

Let us prove the converse inclusion
\begin{equation}\label{l8.2:e3}
\BB_{X} \subseteq V(T_X).
\end{equation}
Suppose \(\BB_{X} \setminus V(T_X) \neq \varnothing\) and let \(B_1\) belong to \(\BB_{X} \setminus V(T_X)\). By Corollary~\ref{c6.8}, there is \(B_2 \in \BB_{X}\) such that \(B_1\) is a part of \(G_{B_2, d|_{B_2 \times B_2}}\) and \(\diam B_2 > 0\) (see Remark~\ref{r6.8}). If \(B_2 \in V(T_X)\), then we also have \(B_1 \in V(T_X)\) by Definition~\ref{d9.1}. Thus, \(B_2 \in \BB_{X} \setminus V(T_X)\) holds. Consequently, there is \(B_3 \in \BB_{X}\) such that \(B_2\) is a part of \(G_{B_3, d|_{B_3 \times B_3}}\). Since \(B_1 \subseteq B_2 \subseteq B_3\), the double inequality
\begin{equation}\label{l8.2:e4}
\diam B_1 \leqslant \diam B_2 \leqslant \diam B_3
\end{equation}
holds. By Corollary~\ref{c2.41}, we have \(\BB_{X} \subseteq \overline{\BB}_{X}\). Using \eqref{l8.2:e4}, Lemma~\ref{l2.6} and the relation \(B_1 \neq B_2 \neq B_3\), we obtain
\[
\diam B_1 < \diam B_2 < \diam B_3.
\]
Repeating this construction, we can find a sequence
\[
(B_i)_{i \in \mathbb{N}} \subseteq \BB_{X}
\]
such that \(\diam B_i < \diam B_{i+1}\) holds for every \(i \in \mathbb{N}\). Consequently, the distance set \(D\) contains a strictly increasing, infinite sequence, contrary to Theorem~\ref{t5.10}.
\end{proof}

\begin{lemma}\label{l8.3}
Let \((X, d)\) be a nonempty totally bounded ultrametric space with the distance set \(D = D(X)\) and let \(B_1\), \(B_2 \in \BB_{X}\) be different. Then the following statements are equivalent:
\begin{enumerate}
\item \label{l8.3:s1} \(\{B_1, B_2\} \in E(T_X)\) is valid.
\item \label{l8.3:s2} There are \(r \in D \setminus \{0\}\) and \(c \in B_1 \cap B_2\) such that
\[
\diam B_1 = r \quad \text{and} \quad B_2 = \{x \in X \colon d(x, c) < r\}
\]
or
\[
\diam B_2 = r \quad \text{and} \quad B_1 = \{x \in X \colon d(x, c) < r\}.
\]
\item \label{l8.3:s3} For every \(B \in \BB_{X}\) we have either
\[
(B_1 \subset B_2) \text{ and } ((B_1 \subseteq B \subseteq B_2) \Rightarrow (B_1 = B \text{ or } B_2 = B))
\]
or
\[
(B_2 \subset B_1) \text{ and } ((B_2 \subseteq B \subseteq B_1) \Rightarrow (B_1 = B \text{ or } B_2 = B)).
\]
\end{enumerate}
\end{lemma}

\begin{proof}
Let us prove the validity of \(\ref{l8.3:s1} \Leftrightarrow \ref{l8.3:s2}\). First of all we note that \(V(T_X) = \BB_{X}\) holds by Lemma~\ref{l8.2}. Using condition~\ref{d9.1:s4} from the definition of \(T_X\) we see that the equivalence \(\ref{l8.3:s1} \Leftrightarrow \ref{l8.3:s2}\) is valid by Theorem~\ref{t2.24} if \(B_1 = X\) or \(B_2 = X\). If \(B_1\) and \(B_2\) are arbitrary open balls in \((X, d)\), then to prove the validity of \(\ref{l8.3:s1} \Leftrightarrow \ref{l8.3:s2}\) we can use Theorem~\ref{t2.24} together with Corollary~\ref{c2.4}.

\(\ref{l8.3:s2} \Rightarrow \ref{l8.3:s3}\). Suppose \ref{l8.3:s2} holds.

Let \(B_1 \subset B_2\), \(B \in \BB_{X}\) and
\begin{equation}\label{l8.3:e1}
B_1 \subseteq B \subseteq B_2.
\end{equation}
We must show that \(B_1 = B\) or \(B_2 = B\). Suppose \(B_1 \neq B\). Using~\ref{l8.3:s2} we obtain
\[
B_1 = \{x \in X \colon d(x, c) < r\} \text{ and } \diam B_2 = r.
\]
Consequently, from \(B_1 \neq B\) and \(B_1 \subseteq B\) it follows that \(d(c, p_1) \geqslant r\). Hence
\begin{equation}\label{l8.3:e2}
\diam B \geqslant r = \diam B_2.
\end{equation}
Now using Proposition~\ref{p2.7} (recall that \(\BB_{X} \subseteq \overline{\BB}_X\) holds by Corollary~\ref{c2.41}), the inclusion \(B \subseteq B_2\) and \eqref{l8.3:e2} we obtain the equality \(B_2 = B\).

The case \(B_2 \subset B_1\) is similar. Thus, the implication \(\ref{l8.3:s2} \Rightarrow \ref{l8.3:s3}\) is true.

\(\ref{l8.3:s3} \Rightarrow \ref{l8.3:s1}\). Let \ref{l8.3:s3} hold and let \(B_2 \subset B_1\). It suffices to show that \(B_2\) is a part of the complete multipartite graph \(G_{B_1, d|_{B_1 \times B_1}}\). Let \(\{X_1, \ldots, X_n\}\) be the set of all parts of this graph. If there are different \(n_1\), \(n_2 \in \{1, \ldots, n\}\) such that
\[
B_2 \cap X_{n_1} \neq \varnothing \neq B_2 \cap X_{n_2},
\]
then we have \(\diam B_2 = \diam B_1\) and, consequently, \(B_1 = B_2\) by Proposition~\ref{p2.7}, contrary to \(B_2 \subset B_1\). Hence, there is \(n^{*} \in \{1, \ldots, n\}\) such that \(B_2 \subseteq X_{n^{*}}\). If \(B_2 \neq X_{n^{*}}\), then using \ref{l8.3:s3} with \(B = X_{n^{*}}\) we obtain the equality \(X_{n^{*}} = B_1\), that contradicts the definition of \(G_{B_1, d|_{B_1 \times B_1}}\). Thus, \(B_2 = X_{n^{*}}\) holds, as required.
\end{proof}

\begin{lemma}\label{l8.5}
Let \((X, d)\) be a nonempty totally bounded ultrametric space and let \(B_1\), \(B_2\), \(B_3 \in \BB_{X}\). If \(B_1\) is a part of \(G_{B_i, d|_{B_i \times B_i}}\) for \(i = 2\), \(3\), then \(B_2 = B_3\) holds.
\end{lemma}

\begin{proof}
Let \(B_1\) be a part of \(G_{B_2, d|_{B_2 \times B_2}}\) and, simultaneously, a part of \(G_{B_3, d|_{B_3 \times B_3}}\). Since \(B_1 \subseteq B_2\), and \(B_1 \subseteq B_3\), and \(B_1 \neq \varnothing\), we have \(B_2 \subseteq B_3\) or \(B_3 \subseteq B_2\) by Proposition~\ref{p2.5}. Without loss of generality we assume \(B_2 \subseteq B_3\). Then we have \(\{B_1, B_3\} \in E(T_X)\) and \(B_1 \subseteq B_2 \subseteq B_3\). Now \(B_2 = B_3\) follows from Lemma~\ref{l8.3}.
\end{proof}

Let \(v\) be a vertex of a graph \(G\) and let \(k\) be a cardinal number. The vertex \(v \in V(G)\) has the \emph{degree} \(k\) if
\[
k = \card \{u \in V(G) \colon \{u, v\} \in E(G)\}.
\]
In this case we write \(\delta_{G}(v) = k\) or simply \(\delta(v) = k\).

The next lemma follows directly from the definition of \(T_X\).

\begin{lemma}\label{l8.7}
Let \((X, d)\) be a nonempty totally bounded ultrametric space. Then the following equality
\begin{equation}\label{l8.7:e1}
\delta(B) =
\begin{cases}
0 & \text{if } B = X \text{ and } \diam X = 0,\\
k & \text{if } B = X \text{ and } \diam X > 0,\\
1 + k & \text{if } \diam B > 0 \text{ and } B \neq X,\\
1 & \text{if } \diam B = 0 \text{ and } B \neq X,\\
\end{cases}
\end{equation}
holds for every \(B \in V(T_X)\), where \(k\) is the number of parts of complete multipartite graph \(G_{B, d|_{B \times B}}\).
\end{lemma}

Recall that a graph \(C\) is a \emph{subgraph} of a graph \(G\), \(C \subseteq G\), if
\[
V(C) \subseteq V(G) \quad \text{and} \quad E(C) \subseteq E(G).
\]
A \emph{path} is a finite, nonempty graph $P$ whose vertices can be numbered so that
\[
V(P) = \{x_0,x_1,...,x_k\},\ k \geqslant 1, \quad \text{and} \quad E(P) = \{\{x_0,x_1\},...,\{x_{k-1},x_k\}\}.
\]
In this case we say that \(P\) is a path joining \(x_0\) and \(x_k\).

A graph \(G\) is \emph{connected} if for every two distinct \(u\), \(v \in V(G)\) there is a path \(P \subseteq G\) joining \(u\) and \(v\).

\begin{definition}\label{d9.8}
A connected graph without cycles is called a \emph{tree}. A tree \(T\) is \emph{locally finite} if the inequality \(\delta(v) < \infty\) holds for every \(v \in V(T)\).
\end{definition}

We will say that a vertex \(v\) of a tree \(T\) is a \emph{leaf} of \(T\) if the inequality \(\delta(v) \leqslant 1\) holds.

\begin{proposition}\label{p8.4}
Let \((X, d)\) be a nonempty totally bounded ultrametric space. Then \(T_X\) is a locally finite tree such that \(V(T_X) = \BB_{X}\) and the equation
\begin{equation}\label{p8.4:e0}
\delta(v) = 2
\end{equation}
has at most one solution \(v \in V(T_X)\).
\end{proposition}

\begin{proof}
First of all, \(V(T_X) = \BB_{X}\) holds by Lemma~\ref{l8.2}. We claim that \(T_X\) is a tree. Suppose contrary that there is a cycle \(C \subseteq T_X\). Let
\[
V(C) = \{B_1, \ldots, B_n\} \text{ and } E(C) = \{\{B_1, B_2\}, \ldots, \{B_{n-1}, B_{n}\}, \{B_{n}, B_{1}\}\}.
\]
After renumbering, if necessary, we can assume
\[
\diam B_1 \geqslant \diam B
\]
for every \(B \in V(C)\). If we have \(\diam B_1 = \diam B_2\), then the equality
\begin{equation}\label{p8.4:e1}
B_1 = B_2
\end{equation}
holds. Indeed, we have \(V(T_X) = \BB_{X}\) and, by Corollary~\ref{c2.41}, the inclusion \(\BB_{X} \subseteq \overline{\BB}_{X}\) holds. Consequently, the conditions
\[
B_1, \ B_2 \in \overline{\BB}_{X}, \quad \diam B_1 = \diam B_2 \quad \text{and} \quad B_1 \cap B_2 \neq \varnothing
\]
are satisfied. Now using Proposition~\ref{p2.7} we obtain \eqref{p8.4:e1}. Thus, \(\diam B_1 > \diam B_2\) holds. Similarly, we obtain the inequality
\begin{equation}\label{p8.4:e2}
\diam B_1 > \diam B_n.
\end{equation}
Inequality \(\diam B_1 > \diam B_2\) and \(\{B_1, B_2\} \in E(T_X)\) imply that \(B_2\) is a part of \(G_{B_1, d|_{B_1 \times B_1}}\). From \(\{B_2, B_3\} \in E(T_X)\) and the definition of \(E(T_X)\) it follows that
\begin{equation}\label{p8.4:e3}
B_2 \text{ is a part of } G_{B_3, d|_{B_3 \times B_3}}
\end{equation}
or
\begin{equation}\label{p8.4:e4}
B_3 \text{ is a part of } G_{B_2, d|_{B_2 \times B_2}}.
\end{equation}
If \eqref{p8.4:e3} holds, then, by Lemma~\ref{l8.5}, we have \(B_3 = B_1\), that contradicts the definition of graphs. Similarly, we can show that \(B_{i+1}\) is a part of \(G_{B_{i}, d|_{B_{i} \times B_{i}}}\) for \(i = 3\), \(\ldots\), \(n\) and, in addition,
\begin{equation}\label{p8.4:e5}
B_1 \text{ is a part of } G_{B_{n}, d|_{B_{n} \times B_{n}}}.
\end{equation}
From~\eqref{p8.4:e5} it follows that \(\diam B_n > \diam B_1\), contrary to \eqref{p8.4:e2}. Thus, \(T_X\) does not contain any cycle.

Now we claim that \(T_X\) is a connected graph. Indeed, let \(B\) belong to \(V(T_X)\) and \(\lev(B) = n\), \(n \in \mathbb{N}\). Then, by condition~\ref{d9.1:s3} of Definition~\ref{d9.1}, there is \(B_1 \in V(T_X)\) such that \(\lev(B_1) = n - 1\) and \(B\) is a part of \(G_{B_1, d|_{B_1 \times B_1}}\). Hence, \(\{B, B_1\} \in E(T_X)\) by condition~\ref{d9.1:s4}. If \(n - 1 \neq 0\), then completely similarly we can find \(B_2 \in V(T_X)\), \(\lev(B_2) = n - 2\) and \(\{B_1, B_2\} \in E(T_X)\) and so on. Then the graph \(P \subseteq T_X\) with
\[
V(P) = \{B, B_1, \ldots, B_n\} \quad \text{and} \quad E(P) = \{\{B, B_1\}, \{B_1, B_2\}, \ldots, \{B_{n-1}, B_{n}\}\}
\]
is a path joining \(B\) and \(B_n\). It is clear that \(P \subseteq T_X\). Moreover, we have the equalities
\[
\lev(B) = n, \lev(B_1) = n-1, \lev(B_j) = n-j, \ldots, \lev(B_n) = n - n = 0.
\]
Consequently, by Definition~\ref{d9.1}, the equality \(B_n = X\) holds. Thus, for every \(B \in V(T_X)\), there is a path joining \(B\) and \(X\) in \(T_X\). Hence, \(T_X\) is connected and acyclic. Thus, it is a tree.

Using Definition~\ref{d9.1}, Proposition~\ref{p2.3} and Lemma~\ref{l8.5} we see that \(T_X\) is locally finite.

To complete the proof it suffices to note that if \eqref{p8.4:e0} holds for \(v \in V(T_X)\), then \(v = X\) by Lemma~\ref{l8.7}.
\end{proof}

In what follows we will say that \(T_X\) is a \emph{free representing tree} of \((X, d)\).

As an application of Proposition~\ref{p8.4}, we consider a characterization of the Cantor set \(D^{\aleph_0}\) (see Example~\ref{ex3.11}) on the language of free representing trees. The following proposition is, in fact, a partial reformulation of the classical Brouwer's result \cite{Bro1910PAA}.

\begin{proposition}\label{p9.10}
Let \((X, d)\) be a compact ultrametric space. Then the following statements are equivalent:
\begin{enumerate}
\item \label{p9.10:s1} \((X, d)\) is homeomorphic to the Cantor set \(D^{\aleph_0}\).
\item \label{p9.10:s2} The representing tree \(T_X\) contains no leaves.
\end{enumerate}
\end{proposition}

\begin{proof}
\(\ref{p9.10:s1} \Rightarrow \ref{p9.10:s2}\). Let \((X, d)\) and \(D^{\aleph_0}\) be homeomorphic. By definition, a vertex \(v \in V(T_X)\) is a leaf of \(T_X\) if and only if the inequality \(\delta(v) \leqslant 1\) holds. Since \((X, d)\) is homeomorphic to \(D^{\aleph_0}\), we have the inequality \(\diam X > 0\). The last inequality, the equality \(\BB_{X} = V(T_X)\), and Lemma~\ref{l8.7} imply that \(\delta(v) \geqslant 1\) holds for every \(v \in V(T_X)\). Consequently, if \(v_0\) is a leaf of \(T_X\), then the equality
\begin{equation}\label{p9.10:e1}
\delta(v_0) = 1
\end{equation}
holds. Using Lemma~\ref{l8.7} again we see \eqref{p9.10:e1} implies the existence of an open ball \(B_0\) with \(\diam B_0 = 0\), i.e., we have \(B_0 = \{p_0\}\) for a point \(p_0 \in X\). Hence, \(p_0\) is an isolated point of \(X\). Since \((X, d)\) and \(D^{\aleph_0}\) are homeomorphic, the existence of isolated \(p_0 \in X\) implies the existence of an isolated point in \(D^{\aleph_0}\), contrary to the definition of \(D^{\aleph_0}\).

\(\ref{p9.10:s2} \Rightarrow \ref{p9.10:s1}\). Suppose that \(T_X\) contains no leaves. By Brouwer theorem, every zero-dimensional compact metric space without isolated points is homeomorphic to the Cantor set \cite{Bro1910PAA}, \cite[p.~370, Exercise~6.2.A]{Eng1989}. A nonempty metric space is called zero-dimensional if \(Y\) has a base consisting of clopen sets. Since the ballean of a metric space is a base for this space, every nonempty ultrametric space is zero-dimensional by Proposition~\ref{p2.2}. Consequently, it suffices to show that \((X, d)\) has no isolated points. To prove the last statement we note that if \(p_0\) is an isolated point of \((X, d)\), then \(B_0 = \{p_0\} \in \BB_{X}\) holds and, consequently, \(\delta(B_0) \leqslant 1\) holds by Lemma~\ref{l8.7}. The last inequality and definition of the leaves imply that \(B_0\) is a leaf of \(T_X\), contrary to supposition.
\end{proof}

\section{Rooted trees and related orderings}
\label{sec10}

The present section is organized as follows. The first part of section deals with partial order generated by choice of a root on trees. The orderings of such type are characterized in Theorem~\ref{t9.20}. In the second part, we consider the rooted representing trees of totally bounded ultrametric spaces and describe them in Proposition~\ref{p10.17}.

\subsection{Generating order by choosing root of the tree}

Before introducing into consideration the concept of isomorphism of rooted trees, it is useful to recall the definition of isomorphism for free trees.

\begin{definition}\label{d8.7}
Let $T_1$ and $T_2$ be trees. A bijection $f\colon V(T_1)\to V(T_2)$ is an \emph{isomorphism} of $T_1$ and $T_2$ if
\begin{equation}\label{e2.2}
(\{u,v\} \in E(T_1)) \Leftrightarrow (\{f(u),f(v)\} \in E(T_2))
\end{equation}
is valid for all $u$, $v \in V(T_1)$. Two trees are \emph{isomorphic} if there exists an isomorphism of these trees.
\end{definition}

A rooted tree is a pair \((T, r)\), where \(T\) is a tree and \(r\) is a distinguished vertex of \(T\). In this case we say that \(r\) is the \emph{root} of \((T, r)\) and \(T\) is a \emph{free} tree corresponding to \((T, r)\). In what follows, we will write \(T = T(r)\) instead of \((T, r)\).

\begin{definition}\label{d8.10}
Let $T_1 = T_1(r_1)$ and $T_2 = T_2(r_2)$ be rooted trees. A bijection $f\colon V(T_1) \to V(T_2)$ is an \emph{isomorphism} of \(T_1(r_1)\) and \(T_2(r_2)\) if it is an isomorphism of free trees \(T_1\) and \(T_1\) and the equality $f(r_1) = r_2$ holds. Two rooted trees are \emph{isomorphic} if there exists an isomorphism of these trees.
\end{definition}

\begin{remark}\label{r10.3}
In the case of finite trees, Definitions~\ref{d8.7} and \ref{d8.10} introduced above as well as Definition~\ref{d9.27} from the next section of the paper can be considered as some specifications of Definition~1 from~\cite{HHH2006}.
\end{remark}

The following statement is well-known for finite trees (see, for example, Proposition~4.1 \cite{BM2008}).

\begin{lemma}\label{l9.14}
In each tree, every two different vertices are connected by exactly one path.
\end{lemma}

\begin{proof}
If \(T\) is an infinite tree and \(v_1\), \(v_2\) are two different vertices of \(T\) connected by some paths \(P_1 \subseteq T\) and \(P_2 \subseteq T\), then the graph \(G\),
\[
V(G) = V(P_1) \cup V(P_2) \quad \text{and} \quad E(G) = E(P_1) \cup E(P_2),
\]
is a finite, connected subgraph of \(T\). Since \(T\) does not have cycles, \(G\) is also acyclic. Hence, \(G\) is a finite tree and \(P_1\), \(P_2\) are paths connected \(v_1\) and \(v_2\) in \(G\). Thus, \(P_1 = P_2\) holds.
\end{proof}

\begin{corollary}\label{c9.15}
Let \(T\) be a tree and let \(P_1\) and \(P_2\) be some paths in \(T\). Then the equality \(P_1 = P_2\) holds if and only if \(V(P_1) = V(P_2)\).
\end{corollary}

Let introduce a partial order generated by choice of a root in a given tree.

Let \(T(r)\) be a rooted tree and let \(\mathcal{P} \colon V(T(r)) \to 2^{V(T(r))}\) be a mapping such that
\begin{equation}\label{e9.17}
\mathcal{P}(v) = \begin{cases}
V(P_{v}) & \text{if } v \neq r, \\
\{r\} & \text{if } v = r,
\end{cases}
\end{equation}
where \(P_{v}\) is the unique path connected \(v\) and \(r\). Then we define a binary relation \(\preccurlyeq_{T(r)}\) on \(V(T(r))\) by the rule
\begin{equation}\label{e9.18}
(v_1 \preccurlyeq_{T(r)} v_2) \Leftrightarrow (\mathcal{P}(v_2) \subseteq \mathcal{P}(v_1)).
\end{equation}

\begin{proposition}\label{p9.16}
The binary relation \(\preccurlyeq_{T(r)}\), defined by \eqref{e9.17}, \eqref{e9.18}, is a partial order on \(V(T(r))\) for every rooted tree \(T(r)\).
\end{proposition}

\begin{proof}
Let \(T = T(r)\) be a rooted tree. It is clear from \eqref{e9.17} and \eqref{e9.18} that \(\preccurlyeq_{T(r)}\) is reflexive and transitive. To justify the antisymmetric property suppose that we have
\begin{equation}\label{p9.16:e1}
v_1 \preccurlyeq_{T(r)} v_2 \quad \text{and} \quad v_2 \preccurlyeq_{T(r)} v_1.
\end{equation}
for some \(v_1\), \(v_2 \in V(T)\) and prove the equality \(v_1 = v_2\).

From \eqref{e9.17} and \eqref{p9.16:e1} it follows that \(\mathcal{P}(v_1) = \mathcal{P}(v_2)\). If \(|\mathcal{P}(v_1)| = |\mathcal{P}(v_2)| = 1\), then \(v_1 = v_2 = r\) holds by \eqref{e9.17}. If \(\mathcal{P}(v_1)\) and \(\mathcal{P}(v_2)\) contain more than one point, then, by \eqref{e9.17}, we have \(v_1 \neq r \neq v_2\) and, consequently, there is a path \(P_{v_i} \subseteq T\) connected \(r\) with \(v_i\), \(i = 1\), \(2\). By Corollary~\ref{c9.15}, it follows that \(P_{v_1} = P_{v_2}\). Now the equality \(v_1 = v_2\) follows from the definition of paths.
\end{proof}

We want to describe the posets \((V(T(r)), \preccurlyeq_{T(r)})\) up to (order) isomorphism. For this description the concepts of covering relation and transitive closure of binary relation shall be used.

Let us define the covering relation on a given partially ordered set.

\begin{definition}[{\cite[p.~6]{Sch2016}}]\label{d8.11}
Let \((P, \preccurlyeq_P)\) be a poset and let \(x\), \(y \in P\). Then \(x\) is called a \emph{lower cover} of \(y\) (and \(y\) is called an \emph{upper cover} of \(x\)) if \(x \prec_P y\) and
\begin{equation}\label{d8.11:e1}
(x \preccurlyeq_P z \preccurlyeq_P y) \Rightarrow (z = x \text{ or } z = y)
\end{equation}
is valid for every \(z \in P\). If \(x\) is a lower cover of \(y\), we write \(x \sqsubset_P y\) (\(y \sqsupset_P x\)). Moreover, we use the symbol \(x \sqsubseteq_P y\) (\(y \sqsupseteq_P x\)) if \(x \sqsubset_P y\) or \(x = y\).
\end{definition}

Let \(X\) be a set and let \(R_1\) and \(R_2\) be binary relations on \(X\), i.e., \(R_1\) and \(R_2\) are some subsets of the Cartesian square \(X^{2} = \{\<x, y> \colon x, y \in X\}\). Recall that a \emph{composition of binary relations} \(R_1\) and \(R_2\) is a binary relation \(R_1 \circ R_2 \subseteq X^{2}\) for which \(\<x, y> \in R_1 \circ R_2\) holds if and only if there is \(z \in X\) such that \(\<x, z> \in R_1\) and \(\<z, y> \in R_2\).

Using the composition of binary relations we can define the transitive closure of a relation as follows.

Let \(\gamma\) be a binary relation on \(X\). We will write \(\gamma^{1} = \gamma\) and \(\gamma^{n+1} = \gamma^{n} \circ \gamma\) for every integer \(n \geqslant 1\). The \emph{transitive closure} \(\gamma^{t}\) of \(\gamma\) is the relation
\begin{equation}\label{e9.20}
\gamma^{t} = \bigcup_{n=1}^{\infty} \gamma^{n}.
\end{equation}
For every \(\beta \subseteq X^{2}\), the transitive closure \(\beta^{t}\) is a transitive binary relation on \(X\) and the inclusion \(\beta \subseteq \beta^{t}\) holds. Moreover, if \(\tau \subseteq X^{2}\) is an arbitrary transitive binary relation for which \(\beta \subseteq \tau\), then we also have \(\beta^{t} \subseteq \tau\), i.e., \(\beta^{t}\) is the smallest transitive binary relation containing~\(\beta\).

It is easy to see that, for every poset \((P, {\preccurlyeq}_{P})\), the relation \({\sqsubseteq}_{P}^{t}\) is a partial order on \(P\) satisfying the inclusion
\[
{\sqsubseteq}_{P}^{t} \subseteq {\preccurlyeq}_{P},
\]
where \(\sqsubseteq_{P}^{t}\) is the transitive closure of \(\sqsubseteq_{P}\).

For the finite posets, it is well-known that the upper covering relation uniquely determines the partial order: If \((P, {\preccurlyeq_{P}})\) is a finite poset, then the equality
\begin{equation}\label{e9.22}
{\preccurlyeq_P} = {\sqsubseteq_{P}^{t}}
\end{equation}
holds.

It should be noted here that \eqref{e9.22} can be fails in general.

\begin{example}\label{ex9.18}
Let \((\RR, \leqslant_{\RR})\) be the totally ordered set of all real numbers with the standard order \({\leqslant_{\RR}}\). Then the transitive closure \(\sqsubseteq_{\RR}^{t}\) coincides with the diagonal relation on \(\RR\),
\[
{\sqsubseteq_{\RR}^{t}} = \{\<x, x> \colon x \in \RR\}.
\]
\end{example}

\begin{theorem}\label{t9.20}
Let \((S, \preccurlyeq_{S})\) be a poset. Then following conditions \ref{t9.20:s1} and \ref{t9.20:s2} are equivalent.
\begin{enumerate}
\item \label{t9.20:s1} There is a rooted tree \(T(r)\) such that \((V(T(r)), \preccurlyeq_{T(r)})\) and \((S, \preccurlyeq_{S})\) are order isomorphic.
\item \label{t9.20:s2} The poset \((S, \preccurlyeq_{S})\) has the following properties:
\begin{enumerate}
\item\label{t9.20:s2:1} \(S\) contains the largest element \(l\).
\item\label{t9.20:s2:2} If \(p \in S\) is not the largest element of \(S\), then \(p\) admits a unique upper cover \(q\), i.e., there is the unique \(q \in S\) such that \(p \sqsubset_S q\).
\item\label{t9.20:s2:3} The equality \({\preccurlyeq}_S = {\sqsubseteq_{S}^{t}}\) holds.
\end{enumerate}
\end{enumerate}
\end{theorem}

\begin{proof}
\(\ref{t9.20:s1} \Rightarrow \ref{t9.20:s2}\). To justify this implication it suffices to show that, for every rooted tree \(T(r)\), the poset \((V(T(r)), \preccurlyeq_{T(r)})\) has properties \ref{t9.20:s2:1}, \ref{t9.20:s2:2} and \ref{t9.20:s2:3} with \(S = V(T(r))\) and \({\preccurlyeq}_S = {\preccurlyeq}_{T(r)}\).

\ref{t9.20:s2:1}. Let \(T = T(r)\) be a rooted tree and let \({\preccurlyeq_{T(r)}}\) be defined by~\eqref{e9.18}. Then, for every \(v \in V(T) \setminus \{r\}\), it is clear that the inclusion \(\{r\} \subseteq V(P_{v})\) holds, where \(P_{v} \subseteq T\) is a path joining \(v\) with the root \(r\). Hence, \(r\) is the largest element of \((V(T(r)), \preccurlyeq_{T(r)})\).

\ref{t9.20:s2:2}. Let \(v_0 \in V(T) \setminus \{r\}\). Then there is a unique path \(P_{v_0} \subseteq T\) such that
\begin{equation}\label{t9.20:e1}
V(P_{v_0}) = \{v_0, v_1, \ldots, v_k\}, \quad k \geqslant 1, \quad v_k = r
\end{equation}
and
\begin{equation}\label{t9.20:e2}
E(P_{v_0}) = \bigl\{\{v_0, v_1\}, \ldots, \{v_{k-1}, v_k\}\bigr\}.
\end{equation}
It is clear from \eqref{t9.20:e1}, \eqref{t9.20:e2} that
\begin{equation}\label{t9.20:e3}
V(P_{v_1}) = \{v_1, \ldots, v_k\},
\end{equation}
where \(P_{v_1} \subseteq T\) is the path joining \(v_1\) with \(r\). Equalities \eqref{t9.20:e1}, \eqref{t9.20:e3} and the definition of \(\preccurlyeq_{T(r)}\) imply \(v_0 \prec_{T(r)} v_1\).

To prove the validity of
\begin{equation}\label{t9.20:e4}
v_0 \sqsubset_{T(r)} v_1
\end{equation}
it suffices to show that if \(u \in V(T)\) and
\begin{equation}\label{t9.20:e5}
v_0 \preccurlyeq_{T(r)} u \preccurlyeq_{T(r)} v_1
\end{equation}
holds, then we have
\begin{equation}\label{t9.20:e6}
u = v_0 \quad \text{or} \quad  u = v_1
\end{equation}
(see Definition~\ref{d8.11}). By \eqref{e9.17}, double inequality \eqref{t9.20:e5} can be written as
\[
\{v_1, \ldots, v_k\} \subseteq V(P_u) \subseteq \{v_0, v_1, \ldots, v_k\},
\]
that implies
\begin{equation}\label{t9.20:e7}
V(P_u) = \{v_1, \ldots, v_k\} \quad \text{or} \quad V(P_u) = \{v_0, v_1, \ldots, v_k\},
\end{equation}
where \(V(P_u)\) is the path joining \(u\) with \(r\) in \(T\). Now \eqref{t9.20:e6} follows from \eqref{t9.20:e1}, \eqref{t9.20:e3}, \eqref{t9.20:e7}.

\ref{t9.20:s2:3}. Let us prove the equality
\begin{equation}\label{t9.20:e8}
{\preccurlyeq_{T(r)}} = {\sqsubseteq_{T(r)}^{t}}.
\end{equation}
First of all we recall that, for every binary relation \(\beta\), the transitive closure \(\beta^{t}\) is the smallest transitive binary relation containing \(\beta\). Hence, we have \({\sqsubseteq_{T(r)}^{t}} \subseteq {\preccurlyeq_{T(r)}}\), because \({\preccurlyeq_{T(r)}}\) is transitive and \({\sqsubseteq_{T(r)}} \subseteq {\preccurlyeq_{T(r)}}\) holds. To prove the converse inclusion
\begin{equation}\label{t9.20:e9}
{\preccurlyeq_{T(r)}} \subseteq {\sqsubseteq_{T(r)}^{t}},
\end{equation}
we consider an arbitrary \(\<v_0, v> \in V(T) \times V(T)\) satisfying \(v_0 \preccurlyeq_{T(r)} v\). The inequality \(v_0 \preccurlyeq_{T(r)} v\) means that \(V(P_{v_0}) \supseteq V(P_{v})\), where \(P_{v_0}\) and \(P_{v}\) are the paths joining \(v_0\) with \(r\) and, respectively, \(v\) with \(r\). Let \(P_{v_0}\) satisfy \eqref{t9.20:e1}, \eqref{t9.20:e2} and let \(v_0 \neq v\). The membership relation \(v \in V(P_{v})\) implies that there is \(n \in \{1, \ldots, k\}\) for which \(v_n = v\). It was shown above that
\begin{equation}\label{t9.20:e10}
v_0 \sqsubset_{T(r)} v_1.
\end{equation}
Similarly we obtain
\begin{equation}\label{t9.20:e11}
v_1 \sqsubset_{T(r)} v_2, \quad v_2 \sqsubset_{T(r)} v_3, \quad \ldots, \quad v_{n-1} \sqsubset_{T(r)} v_{n}.
\end{equation}
From~\eqref{t9.20:e10} and \eqref{t9.20:e11} it follows that
\[
\<v_0, v> = \<v_0, v_n> \in {\sqsubset}_{T(r)}^{n} \subseteq {\sqsubseteq}_{T(r)}^t.
\]
Inclusion \eqref{t9.20:e9} holds.

\(\ref{t9.20:s2} \Rightarrow \ref{t9.20:s1}\). Let a partially ordered set \((S, {\preccurlyeq_{S}})\) have properties \ref{t9.20:s2:1}, \ref{t9.20:s2:2} and \ref{t9.20:s2:3}. Let us define a rooted graph \(T = T(r)\) such that \(V(T) = S\), \(r = l\), where \(l\) is the largest element of \(S\), and
\begin{equation}\label{t9.20:e13}
(\{u, v\} \in E(T)) \Leftrightarrow (u \sqsubset_{S} v \text{ or } v \sqsubset_{S} u)
\end{equation}
for all \(u\), \(v \in S\). We claim that \(T(r)\) is a rooted tree for which the equality
\begin{equation}\label{t9.20:e12}
{\preccurlyeq_{T(r)}} = {\preccurlyeq_{S}}
\end{equation}
holds, where \({\preccurlyeq_{T(r)}}\) is the partial order on \(V(T)\) defined by \eqref{e9.18}. First of all, we note that if \eqref{t9.20:e12} holds, then the identical mapping \(\operatorname{id} \colon S \to S\), \(\operatorname{id}(s) = s\) for all \(s \in S\), is an order isomorphism of \((V(T(r)), {\preccurlyeq_{T(r)}})\) and \((S, {\preccurlyeq_{S}})\).

Let \(v\) be an element of \(S\) such that \(v \neq l\). By Definition~\ref{d9.19}, we have the inequality \(v \preccurlyeq_{S} l\). The last inequality, equality~\eqref{e9.20} and \({\preccurlyeq_{S}} = {\sqsubseteq_S^t}\) imply the existence of finite sequence \(s_0\), \(\ldots\), \(s_n\), \(n \geqslant 1\), such that
\begin{equation}\label{t9.20:e14}
v = s_0 \sqsubset_S s_1 \sqsubset_S \ldots \sqsubset_S s_n = l.
\end{equation}
Using \eqref{t9.20:e13} and \eqref{t9.20:e14} we see that \((s_0, \ldots, s_n)\) is a path in \(T\) connected \(v = s_0\) with \(l = s_n\). Thus, \(T\) is a connected graph. Suppose \(T\) contains a cycle \(C\). Let \(V(C) = \{v_1, \ldots, v_n\}\), \(n \geqslant 3\), and
\[
(\{v_i, v_j\}\in E(C)) \Leftrightarrow (|i-j|=1 \quad \text{or} \quad |i-j|=n-1).
\]
Then \((V(C), \preccurlyeq_{V(C)})\) is a finite poset, where
\begin{equation}\label{t9.20:e17}
{\preccurlyeq_{V(C)}} = {\preccurlyeq}_S \cap (V(C) \times V(C)),
\end{equation}
Consequently, this poset contains a maximal element \(m\). Without loss of generality, we assume \(v_1 = m\). The membership \(\{v_1, v_2\} \in E(C)\), the inclusion \(E(C) \subseteq E(T)\) and \eqref{t9.20:e13} imply \(v_1 \sqsubset_S v_2\) or \(v_2 \sqsubset_S v_1\). If \(v_1 \sqsubset_S v_2\) holds, then \(v_1 \preccurlyeq_S v_2\) also holds and, consequently, \(v_1 \preccurlyeq_{V(C)} v_2\) is valid, where \({\preccurlyeq_{V(C)}}\) is defined by \eqref{t9.20:e17}. Since \(v_1\) is a maximal element of \(V(C)\), the inequality \(v_1 \preccurlyeq_{V(C)} v_2\) implies \(v_1 = v_2\), contrary to the definition of (simple) graph. Hence, we have
\begin{equation}\label{t9.20:e15}
v_2 \sqsubset_S v_1.
\end{equation}
As above, we can see that \(\{v_2, v_3\} \in E(C) \subseteq E(T)\) implies
\begin{equation}\label{t9.20:e16}
v_2 \sqsubset_S v_3 \quad \text{or} \quad v_3 \sqsubset_S v_2.
\end{equation}
If the first relation in \eqref{t9.20:e16} is valid, then from \eqref{t9.20:e15} it follows that \(v_1\) and \(v_3\) are two different upper covers of \(v_2\), contrary to \ref{t9.20:s2:2}. Consequently, \(v_3 \sqsubset_S v_2\) holds. Now, using induction, we obtain the contradiction
\[
v_1 \sqsupset_S v_2 \sqsupset_S \ldots \sqsupset_S v_{n-1} \sqsupset_S v_{n} \sqsupset_S v_{1}.
\]
Hence, \(T\) contains no cycles. Thus, \(T = T(r)\) is a rooted tree.

To complete the proof we must show that \eqref{t9.20:e12} holds. By condition \ref{t9.20:s2:3}, we have \({\preccurlyeq_S} = {\sqsubseteq_S^t}\). Moreover, as was shown in the first part of the present proof, we also have \({\preccurlyeq_{T(r)}} = {\sqsubseteq_{T(r)}^t}\) (see \eqref{t9.20:e8}). Consequently, \eqref{t9.20:e12} holds if and only if
\begin{equation}\label{t9.20:e18}
{\sqsubset_S} \subseteq {\preccurlyeq_{T(r)}}
\end{equation}
and
\begin{equation}\label{t9.20:e19}
{\sqsubset_{T(r)}} \subseteq {\preccurlyeq_{S}}.
\end{equation}
Let us prove \eqref{t9.20:e18}. Assume \(s_0\), \(s_1 \in S\) and \(s_0 \sqsubset_S s_1\) holds. The inequality \(s_0 \preccurlyeq_{T(r)} s_1\) evidently holds if \(s_1 = l\). Suppose we have \(s_1 \prec_S l\). Then \(s_1 \prec_S l\) and \(s_0 \sqsubset_S s_1\) imply \(s_0 \prec_S l\). By \ref{t9.20:s2:3}, there are \(s_1^{*}\), \(\ldots\), \(s_n^{*}\) such that
\begin{equation}\label{t9.20:e20}
s_0 \sqsubset_S s_1^{*} \sqsubset_S \ldots \sqsubset_S s_n^{*} = l.
\end{equation}
By \ref{t9.20:s2:2}, we have the equality \(s_1 = s_1^{*}\). It follows from \eqref{t9.20:e20} and \eqref{t9.20:e13} that \(P_{s_0} = (s_0, s_1^{*}, \ldots, s_n^{*})\) is the path joining \(s_0\) with the root in \(T\). Similarly, \(P_{s_1^*} = (s_1^*, \ldots, s_n^{*})\) is the path joining \(s_1^*\) with \(r\) in \(T\). The inclusion
\[
\{s_0, s_1^*, \ldots, s_n^{*}\} \supseteq \{s_1^*, \ldots, s_n^{*}\},
\]
the equalities \(s_1 = s_1^{*}\), \(s_n^{*} = l = r\), \eqref{t9.20:e17} and \eqref{t9.20:e18} imply \(s_0 \preccurlyeq_{T(r)} s_1\). Inclusion \eqref{t9.20:e18} follows.

Suppose now we have
\begin{equation}\label{t9.20:e21}
v_0 \sqsubset_{T(r)} v_1
\end{equation}
for \(v_0\), \(v_1 \in V(T)\). Let \(P_{v_0} = (v_1^{*}, v_2^{*}, \ldots, v_n^{*})\) be the path in \(T\) joining \(v_1^{*} = v_0\) with \(v_{n}^{*} = r\). Then we have \(v_1^{*} \sqsubset_{T(r)} v_2^{*}\) (see \eqref{t9.20:e4}). Since the upper cover of \(v_0\) (with respect the order \(\preccurlyeq_{T(r)}\)) is unique, \(v_1^{*} = v_0\) implies \(v_2^{*} = v_1\). Consequently, \(v_0\) and \(v_1\) are adjacent in \(T\). From \(\{v_0, v_1\} \in E(T)\) it follows that
\[
v_0 \sqsubset_S v_1 \text{ or } v_1 \sqsubset_S v_0
\]
(see \eqref{t9.20:e13}). Using \eqref{t9.20:e18} we see that \(v_1 \sqsubset_S v_0\) implies \(v_1 \preccurlyeq_{T(r)} v_0\). The last inequality contradicts \eqref{t9.20:e21}. Hence, \(v_0 \sqsubset_S v_1\) holds. Thus, the implication
\[
(v_0 \sqsubset_{T(r)} v_1) \Rightarrow (v_0 \sqsubset_S v_1)
\]
is valid. Inclusion \eqref{t9.20:e19} follows.
\end{proof}

The following example shows that characteristic properties of \(\preccurlyeq_{T(r)}\) presented in condition~\ref{t9.20:s2} of Theorem~\ref{t9.20} are independent of one another.

\begin{example}\label{ex9.21}
\((e_1)\) Let \((S, {\preccurlyeq_S})\) be a subposet of \((\RR, \leqslant_{\RR})\) defined such that
\[
(s \in S) \Leftrightarrow \left(\exists n \in \mathbb{N}\colon s = \frac{1}{n} \text{ or } s = 1 + \frac{1}{n}\right).
\]
Then \((S, {\preccurlyeq_S})\) has properties \ref{t9.20:s2:2} and \ref{t9.20:s2:3} but the poset \((S, {\sqsubseteq}_S^t)\) is not totally ordered. Hence, we have \({\preccurlyeq}_S \neq {\sqsubseteq}_S^t\). Thus, \(\ref{t9.20:s2:1} \mathbin{\&} \ref{t9.20:s2:2} \mathbin{\&} \neg\ref{t9.20:s2:3}\) is valid for \((S, {\preccurlyeq_S})\).

\((e_2)\) The subposet \((S, {\preccurlyeq_S})\) of \((\RR, \leqslant_{\RR})\) with \(S = \mathbb{Z}\) does not have the largest element but satisfies \ref{t9.20:s2:2} and \ref{t9.20:s2:3}.

\((e_3)\) A diagram of a finite poset \((S, {\preccurlyeq_S})\) satisfying \(\ref{t9.20:s2:1} \mathbin{\&} \neg\ref{t9.20:s2:2} \mathbin{\&} \ref{t9.20:s2:3}\) is given by Figure~\ref{fig3} below.
\end{example}

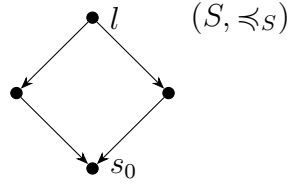
\begin{figure}[ht]
\begin{center}
\begin{tikzpicture}[
level 1/.style={level distance=1cm,sibling distance=25mm},
level 2/.style={level distance=1cm,sibling distance=12mm},
level 3/.style={level distance=1cm,sibling distance=12mm},
solid node/.style={circle,draw,inner sep=1.5,fill=black},
hollow node/.style={circle,draw,inner sep=1.5}]

\node at (1,4) [label={right:{\((S, {\preccurlyeq_S})\)}}] {};
\node at (0,4) [solid node, label={right:{\(l\)}}] (A) {};
\node at (1,3)  [solid node] (B1) {};
\node at (-1,3) [solid node] (B2) {};
\node at (0,2) [solid node, label={right:{\(s_0\)}}] (C) {};
\path (A) edge[-Stealth] (B1);
\path (A) edge[-Stealth] (B2);
\path (B1) edge[-Stealth] (C);
\path (B2) edge[-Stealth] (C);
\end{tikzpicture}
\end{center}
\caption{The upper cover of \(s_0\) is not unique in \((S, {\preccurlyeq_S})\).}
\label{fig3}
\end{figure}

\begin{lemma}\label{l10.10}
Let \(T = T(r)\) be a rooted tree. Let us denote by \(\sqsubset_{T(r)}\) the upper covering relation corresponding to the partial order \(\preccurlyeq_{T(r)}\). Then the equivalence
\begin{equation}\label{l10.10:e1}
(v \sqsubset_{T(r)} u \text{ or } u \sqsubset_{T(r)} v) \Leftrightarrow (\{u, v\} \in E(T))
\end{equation}
is valid for all \(u\), \(v \in V(T)\).
\end{lemma}

\begin{proof}
Equivalence \eqref{l10.10:e1} is evidently valid if \(|V(T)| = 1\). Suppose \(|V(T)| \geqslant 2\). Let \(u\) and \(v\) be vertices of \(T\) such that \(v \sqsubset_{T(r)} u\). The root \(r\) is the largest element of the poset \((V(T), {\preccurlyeq_{T(r)}})\). Hence, \(v \sqsubset_{T(r)} u\) implies \(v \neq r\). By Lemma~\ref{l9.14}, there is the unique path \(P_{v} = (v_1, \ldots, v_n) \subseteq T\) with \(v_1 = v\) and \(v_n = r\). Let \(w\) be an arbitrary vertex of \(T\). By definition of \({\preccurlyeq_{T(r)}}\), the inequality \(v \preccurlyeq_{T(r)} w\) holds if and only if there is \(i \in \{1, \ldots, n\}\) for which we have \(w = v_i\). Consequently, there is \(i_0 \in \{1, \ldots, n\}\) such that \(v_{i_0} = u\). It is clear that \(i_0 \geqslant 2\). If the equality \(i_0 = 2\) holds, then
\[
\{v, u\} = \{v_1, v_{i_0}\} \in E(P_v).
\]
Suppose that \(i_0 \geqslant 3\). Then \(v_2 \neq u\) and we get
\[
v \preccurlyeq_{T(r)} v_2 \preccurlyeq_{T(r)} v_{i_0} = u.
\]
It means that \(v \sqsubset_{T(r)} v_2\) and, consequently, \(u\) is not the upper cover of \(v\), contrary to \(v \sqsubset_{T(r)} u\). The implication
\[
(v \sqsubset_{T(r)} u) \Rightarrow (\{u, v\} \in E(T))
\]
follows. Since \(T\) is a simple graph, we have
\[
(\{u, v\} \in E(T)) \Leftrightarrow (\{v, u\} \in E(T)).
\]
Thus, the implication
\[
(v \sqsubset_{T(r)} u \text{ or } u \sqsubset_{T(r)} v) \Rightarrow (\{u, v\} \in E(T))
\]
is true for all \(u\), \(v \in V(T)\).

Suppose now that \(\{u, v\} \in E(T)\). It is clear that \(u \neq r\) or \(v \neq r\). Without loss of generality, we assume \(v \neq r\). Let, as above, \(P_{v} = (v_1, \ldots, v_n) \subseteq T\) be a path in \(T\) with \(v_1 = v\) and \(v_n = r\). Suppose that \(u \in V(P_v)\). Since \(T\) has no cycles, conditions \(u \in V(P_v)\) and \(\{u, v\} \in E(P_v)\) imply that \(u = v_2\). As in the first part of the proof, we obtain \(v \sqsubset_{T(r)} v_2\). Thus,
\[
v \sqsubset_{T(r)} u
\]
holds. If \(u \notin V(P_v)\) holds, then using \(\{u, v\} \in E(T)\) we can add the edge \(\{u, v\}\) to the path \(P_v\) and obtain a path \(P_{u} = (u, v_1, \ldots, v_n)\) which connects \(u\) and \(v_n = r\) in \(T\). It is clear that \(v \in P_{u}\) and, reasoning as before, we get \(u \sqsubset_{T(r)} v\). It implies
\[
(\{u, v\} \in E(T)) \Rightarrow (v \sqsubset_{T(r)} u \text{ or } u \sqsubset_{T(r)} v).
\]
Equivalence~\eqref{l10.10:e1} follows.
\end{proof}

\begin{remark}\label{r10.11}
Let \(T = T(r)\) be a rooted tree. In the theory of trees it is often said that a vertex \(u \in V(T)\) is a \emph{direct successor} of a vertex \(v \in V(T)\) (= \(v\) is a \emph{direct predecessor} of \(u\))  if \(u \sqsubset_{T(r)} v\) holds. By Theorem~\ref{t9.20}, if \(u \in V(T)\) is not the root of \(T\), then \(u\) admits the unique direct predecessor. Moreover, Lemma~\ref{l10.10} claims that two vertices of \(T\) are adjacent if and only if one of them is a direct successor of another.
\end{remark}

\begin{proposition}\label{p9.23}
Let \(T_1 = T_1(r_1)\) and \(T_2 = T_2(r_2)\) be rooted trees and let \(f \colon V(T_1) \to V(T_2)\) be a bijective mapping. Then the following statements are equivalent:
\begin{enumerate}
\item \label{p9.23:s1} \(f\) is an isomorphism of rooted trees \(T_1(r_1)\) and \(T_2(r_2)\).
\item \label{p9.23:s2} \(f\) is an order isomorphism of posets \((V(T_1), \preccurlyeq_{T_1(r_1)})\) and \((V(T_2), \preccurlyeq_{T_2(r_2)})\).
\end{enumerate}
\end{proposition}

\begin{proof}
\(\ref{p9.23:s1} \Rightarrow \ref{p9.23:s2}\). Let \(f\) be an isomorphism of \(T_1(r_1)\) and \(T_2(r_2)\) and let \(u\), \(v \in V(T_1)\) satisfy the inequality \(u \preccurlyeq_{T_1(r_1)} v\). We must show that
\begin{equation}\label{p9.23:e1}
f(u) \preccurlyeq_{T_2(r_2)} f(v)
\end{equation}
holds. By Definition~\ref{d8.10}, we have the equality
\begin{equation}\label{p9.23:e2}
f(r_1) = r_2.
\end{equation}
Since \(r_2\) is the largest element of \((V(T_2), \preccurlyeq_{T_2(r_2)})\), equality \eqref{p9.23:e2} implies \eqref{p9.23:e1} if \(v = r_1\). Thus, it suffices to prove \eqref{p9.23:e1} if \(v \neq r_1\). From~\eqref{e9.17}, \eqref{e9.18} and \(u \preccurlyeq_{T_1(r_1)} v\) it follows that \(u \neq r_1\) holds if \(v \neq r_1\). Let
\[
P_u = (u, u_1, \ldots, u_n), \ u_n = r_1 \text{ and } P_v = (v, v_1, \ldots, v_k), \ v_k = r_1,
\]
be the paths (in \(T_1\)) joining \(u\) with \(r_1\) and, respectively, \(v\) with \(r_1\). Then
\[
(f(u), f(u_1), \ldots, f(u_n)), \ f(u_n) = r_2 \text{ and } (f(v), f(v_1), \ldots, f(v_k)), \ f(v_k) = r_2,
\]
are the paths (in \(T_2\)) joining \(f(u)\) with \(r_2\) and, respectively, \(f(v)\) with \(r_2\). From \(u \preccurlyeq_{T_1(r_1)} v\) and \eqref{e9.18} we obtain the inclusion \(V(P_{v}) \subseteq V(P_{u})\), that implies
\begin{equation}\label{p9.23:e3}
f(V(P_{v})) \subseteq f(V(P_{u})).
\end{equation}
Now inequality \eqref{p9.23:e1} follows from \eqref{p9.23:e3} and the definition \(\preccurlyeq_{T(r_2)}\). Thus, the implication
\begin{equation}\label{p9.23:e4}
(u \preccurlyeq_{T_1(r_1)} v) \Rightarrow (f(u) \preccurlyeq_{T_2(r_2)} f(v))
\end{equation}
is valid for all \(u\), \(v \in V(T_1)\). Since the bijection \(f^{-1} \colon V(T_2) \to V(T_1)\) is an isomorphism of the rooted trees \(T_2(r_2)\) and \(T_1(r_1)\), the converse implication for implication \eqref{p9.23:e4} is also valid. Thus, \(f \colon V(T_1) \to V(T_2)\) is an order isomorphism of \((V(T_1), \preccurlyeq_{T_1(r_1)})\) and \((V(T_2), \preccurlyeq_{T_2(r_2)})\).

\(\ref{p9.23:s2} \Rightarrow \ref{p9.23:s1}\). Let \(f \colon V(T_1) \to V(T_2)\) be an order isomorphism of \((V(T_1), \preccurlyeq_{T_1(r_1)})\) and \((V(T_2), \preccurlyeq_{T_2(r_2)})\). To prove that \(f\) is an isomorphism of rooted trees \(T_1(r_1)\) and \(T_2(r_2)\), we must show that
\begin{equation}\label{p9.23:e5}
f(r_1) = r_2
\end{equation}
holds and
\begin{equation}\label{p9.23:e6}
(\{u, v\} \in E(T_1)) \Leftrightarrow (\{f(u), f(v)\} \in E(T_2))
\end{equation}
is valid for all \(u\), \(v \in V(T_1)\). Since \(r_i\) is the largest element of \((V(T_i), \preccurlyeq_{T_i(r_i)})\) and \(f\) is an order isomorphism, equality \eqref{p9.23:e5} follows from Definitions \ref{d2.8} and \ref{d9.19}. By Lemma~\ref{l10.10}, we have
\[
(\{u, v\} \in E(T_1)) \Leftrightarrow (v \sqsubset_{T_1(r_1)} u \text{ or } u \sqsubset_{T_1(r_1)} v)
\]
and, since \(f\) is an order isomorphism, the equivalence
\[
(v \sqsubset_{T_1(r_1)} u \text{ or } u \sqsubset_{T_1(r_1)} v) \Leftrightarrow (f(v) \sqsubset_{T_2(r_2)} f(u) \text{ or } f(u) \sqsubset_{T_2(r_2)} f(v))
\]
is valid. Now using Lemma~\ref{l10.10} again, we obtain
\[
(f(v) \sqsubset_{T_2(r_2)} f(u) \text{ or } f(u) \sqsubset_{T_2(r_2)} f(v)) \Leftrightarrow (\{f(u), f(v)\} \in E(T_2)).
\]
Equivalence \eqref{p9.23:e6} follows.
\end{proof}

\subsection{Rooted representing trees of totally bounded ultrametric spaces}

Let us consider now in more details the case when our tree is the representing tree of a totally bounded ultrametric space.

\begin{definition}\label{d10.12}
Let \((X, d)\) be a nonempty totally bounded ultrametric space. The representing tree \(T_X\) together with the root \(r_X = X\) is said to be the \emph{rooted representing tree} of \((X, d)\) and denoted as \(T_X = T_X(r_X)\).
\end{definition}

Let \((X, d)\) be a nonempty totally bounded ultrametric space and let \(\BB_{X}\) be the ballean of \((X, d)\). As in Theorem~\ref{t7.12}, we can consider \(\BB_{X}\) as a partially ordered set with order \(\preccurlyeq_X\) defined by~\eqref{e7.18}.

\begin{lemma}\label{l10.14}
Let \((X, d)\) be a nonempty totally bounded ultrametric space. Then the posets \((\BB_{X}, \preccurlyeq_{X})\) and \((V(T_X), \preccurlyeq_{T_X(r_X)})\) are the same, i.e., the equalities
\[
\BB_{X} = V(T_X) \quad \text{and} \quad {\preccurlyeq_{X}} = {\preccurlyeq_{T_X(r_X)}}
\]
hold.
\end{lemma}

\begin{proof}
The equality \(\BB_{X} = V(T_X)\) was proved in Lemma~\ref{l8.2}. To prove the equality \({\preccurlyeq_{X}} = {\preccurlyeq_{T_X(r_X)}}\) it suffices to show that the inclusions
\begin{equation}\label{l10.14:e2}
{\preccurlyeq_{X}} \subseteq {\preccurlyeq_{T_X(r_X)}}
\end{equation}
and
\begin{equation}\label{l10.14:e3}
{\preccurlyeq_{T_X(r_X)}} \subseteq {\preccurlyeq_{X}}
\end{equation}
are valid.

Let us prove the validity of~\eqref{l10.14:e2}. Let \(A_1\) and \(A_2\) belong to \(\BB_{X}\) and \(A_1 \preccurlyeq_{X} A_2\) hold. We want to show that
\begin{equation}\label{l10.14:e4}
A_1 \preccurlyeq_{T_X(r_X)} A_2.
\end{equation}
Inequality~\eqref{l10.14:e4} is evidently holds if \(A_1 = X\), or \(A_1 = A_2\), or \(A_2 = X\) (see \eqref{e9.17} and \eqref{e9.18}).

Let \(A_1 \neq X \neq A_2\) and \(A_1 \neq A_2\). Then, by condition~\ref{d9.1:s3} of Definition~\ref{d9.1}, there are subsets \(X_1^2\), \(\ldots\), \(X_n^2\) of \(X\) such that \(X_1^2 = X\), \(X_n^2 = A_2\) and \(X_{i+1}^2\) is a part of \(G_{X_i^2, d|_{X_i^2 \times X_i^2}}\) for every \(i \in \{1, \ldots, n-1\}\). Similarly, there are subsets \(X_1^1\), \(\ldots\), \(X_m^1\) of \(X\) such that \(X_1^1 = X\), \(X_m^1 = A_1\) and \(X_{j+1}^1\) is a part of \(G_{X_j^1, d|_{X_j^1 \times X_j^1}}\) for every \(j \in \{1, \ldots, m-1\}\). Using condition~\ref{d9.1:s4} of Definition~\ref{d9.1} we see that \((X_1^1, \ldots, X_m^1)\) is a path in \(T_X\) joining \(X\) and \(A_1\) and \((X_1^2, \ldots, X_n^2)\) is a path in \(T_X\) joining \(X\) and \(A_2\). By \eqref{e9.17}, \eqref{e9.18}, to prove inequality \eqref{l10.14:e4} we must show that, for every \(i \in \{1, \ldots, n\}\) there is \(j \in \{1, \ldots, m\}\) such that \(X_i^2 = X_j^1\). We claim that every \(i \in \{1, \ldots, n\}\) belongs also to \(\{1, \ldots, m\}\) (i.e., \(n \leqslant m\) holds) and the equality
\begin{equation}\label{l10.14:e5}
X_i^1 = X_i^2
\end{equation}
holds. The last equality is trivially valid for \(j = 1\) and \(i = 1\).

The condition \(A_1 \neq X \neq A_2\) implies that \(m \geqslant 2\) and \(n \geqslant 2\). Consequently, there exist the sets \(X_2^2\) and \(X_2^1\). If \(X_2^2 \neq X_2^1\) holds, then \(X_2^2\) and \(X_2^1\) are disjoint as different parts of complete multipartite graph \(G_{X, d}\),
\begin{equation}\label{l10.14:e6}
X_2^2 \cap X_2^1 = \varnothing.
\end{equation}
Since \(A_2 \subseteq X_2^{2}\) and \(A_1 \subseteq X_2^{1}\), equality \eqref{l10.14:e6} implies
\begin{equation}\label{l10.14:e7}
A_2 \cap A_1 = \varnothing.
\end{equation}
The sets \(A_1\) and \(A_2\) are nonempty (see Remark~\ref{r8.2}). Hence, \eqref{l10.14:e6} contradicts \(A_1 \preccurlyeq_X A_2\). Thus, \(X_2^1 = X_2^2\) holds. Similarly, we see that \(X_i^1 = X_i^2\) for every \(i \leqslant \min\{m, n\}\). If \(n = \min\{m, n\}\) holds, then the last statement and \eqref{e9.17}, \eqref{e9.18} imply \eqref{l10.14:e4}.

To complete the proof of inequality \eqref{l10.14:e4} it suffices to note that if \(m < n\) holds, then \(A_2\) is a subset of a part of \(G_{A_1, d|_{A_1 \times A_1}}\). Consequently, we have the inclusion \(A_2 \subseteq A_1\). Moreover, the inequality \(A_1 \preccurlyeq_X A_2\) means the inclusion \(A_1 \subseteq A_2\) (see \eqref{e7.18}). Thus, the equality \(A_1 = A_2\) holds, contrary to \(A_1 \neq A_2\).

To prove inclusion \eqref{l10.14:e3} suppose that \(B_1\) and \(B_2\) belong to \(V(T_X)\) and
\begin{equation}\label{l10.14:e8}
B_1 \preccurlyeq_{T_X(r_X)} B_2.
\end{equation}
It suffices to prove the inclusion \(B_1 \subseteq B_2\). Without loss of generality we assume \(B_1 \neq B_2\).

Let \(P_i\) be a path connecting \(r_X = X\) with \(B_i\) in \(T_X\), \(i = 1\), \(2\). Using Lemma~\ref{l9.14} (the uniqueness of path), relations \eqref{e9.17}, \eqref{e9.18} (the definition of the order \(\preccurlyeq_{T(r)}\)) and Definition~\ref{d9.1} we see that there is a sequence \(X_1\), \(\ldots\), \(X_n\) of subsets of \(X\) such that \(X = X_1\), \(B_1 = X_n\) and \(B_2 = X_{i_0}\) for some \(i_0 \in \{1, \ldots, n-1\}\), and \(X_{i+1}\) is a part of \(G_{X_i, d_{X_i \times X_i}}\) for every \(i \in \{1, \ldots, n-1\}\). Consequently, \(B_1\) is a subset of a part of the complete multipartite graph \(G_{B_2, d|_{B_2 \times B_2}}\). It implies the inclusion \(B_1 \subseteq B_2\), that is equivalent to \eqref{l10.14:e8}.
\end{proof}

\begin{theorem}\label{t8.7}
Let \((X, d)\) and \((Y, \rho)\) be nonempty totally bounded and ultrametric. Then the posets \((\BB_{X}, \preccurlyeq_X)\) and \((\BB_{Y}, \preccurlyeq_Y)\) are order isomorphic if and only if the rooted trees \(T_X = T_X(r_X)\) and \(T_Y = T_Y(r_Y)\) are isomorphic. Moreover, a bijection \(\Phi \colon \BB_{X} \to \BB_{Y}\) is an isomorphism of \(T_X(r_X)\) and \(T_Y(r_Y)\) if and only if this bijection is an order isomorphism of \((\BB_{X}, \preccurlyeq_X)\) and \((\BB_{Y}, \preccurlyeq_Y)\).
\end{theorem}

\begin{proof}
It follows from Lemma~\ref{l10.14} and Proposition~\ref{p9.23}.
\end{proof}

Let \(T = T(r)\) be a rooted tree and let $v$ be a vertex of \(T\). Denote by \(\delta^+(v) = \delta_{T(r)}^+(v)\) the \emph{out-degree} of $v$,
\begin{equation}\label{e10.43}
\delta^+(v) = \begin{cases}
\delta(v) & \text{if } v = r,\\
\delta(v)-1 & \text{if } v \neq r.
\end{cases}
\end{equation}
The root $r$ is a \emph{leaf} of \(T\) if and only if $\delta^+(r) \leqslant 1$. Moreover, for a vertex $v$ different from the root $r$, the equality $\delta^+(v) = 0$ holds if and only if $v$ is a leaf of \(T\).

\begin{remark}\label{r10.16}
Thus, the root \(r\) of \(T = T(r)\) is a leaf of \(T\) if and only if \(T\) is empty, \(E(T) = \varnothing\).
\end{remark}

\begin{proposition}\label{p10.17}
Let \((X, d)\) be a nonempty totally bounded ultrametric space with the rooted representing tree \(T_X(r_X)\). Then \(T_X(r_X)\) is locally finite and
\begin{equation}\label{p10.17:e1}
\delta_{T_X(r_X)}^{+}(B) \neq 1
\end{equation}
holds for every \(B \in V(T_X(r_X))\).
\end{proposition}

\begin{proof}
By Proposition~\ref{p8.4}, the free representing tree \(T_X\) is locally finite. Thus, it suffices to show that \eqref{p10.17:e1} holds for every \(B \in V(T_X(r_X))\). Lemma~\eqref{l8.2} implies the equality \(V(T_X(r_X)) = \BB_{X}\). Using \eqref{e10.43} and Lemma~\ref{l8.7} we obtain the equality
\begin{equation}\label{p10.17:e2}
\delta_{T_X(r_X)}^{+}(B) =
\begin{cases}
k & \text{if } \diam B > 0,\\
0 & \text{if } \diam B = 0\\
\end{cases}
\end{equation}
for every \(B \in \BB_{X}\), where \(k\) is the number of parts of complete multipartite graph \(G_{B, d|_{B \times B}}\). By Definition~\ref{d2.6}, the inequality \(k \geqslant 2\) holds for every complete \(k\)-partite graph. Hence, \eqref{p10.17:e2} implies \eqref{p10.17:e1}.
\end{proof}

It will be shown in Theorem~\ref{t10.16} of the paper that the local finiteness and the validity of \eqref{p10.17:e1} for all vertices of \(T_X(r_X)\) completely characterized the rooted representing trees of totally bounded ultrametric spaces.

\section{Labelings on trees}
\label{sec11}

In this section we consider the trees whose vertices are labeled by nonnegative real numbers and, in particular, introduce the labeling on vertex sets of representing trees of totally bounded ultrametric spaces.

The first part of the section deals with the ultrametric spaces \((V(T), d_l)\) generated by given labelings \(l\) on the trees \(T\). The structure of labeled trees \(T = T(l)\), for which the ultrametric space \((V(T), d_l)\) are totally bounded, is described in Theorem~\ref{t11.8}. Theorem~\ref{t11.16} gives us a characterization of labeled trees having totally bounded discrete \((V(T), d_l)\).

For general totally bounded ultrametric spaces, the concept of labeled representing trees is introduced in Definition~\ref{d9.27} of the second part of the section. It is proved in Theorem~\ref{t11.12} that the labeled representing trees of totally bounded ultrametric spaces are isomorphic if and only if the completions of these spaces are isometric.

The ultrametric space \(\mathbf{MC}_{T_X}\) of all maximal chains of balls of a totally bounded ultrametric space \((X, d)\) is investigated in the third part of the section. It is proved in Proposition~\ref{p11.25} that \(\mathbf{MC}_{T_X}\) is isometric to the completion of \((X, d)\).

Labeled representing trees of \(p\)-adic balls are considered in the final part of the section. It is shown that such trees coincide up to isomorphism with irregular Bethe Lattices of degree \(p+1\) labeled by numbers \(p^{k}\), \(p^{k-1}\), \(p^{k-2}\), \(\ldots\) with \(k \in \mathbb{Z}\).

\subsection{General labeled trees}

Let us introduce the concept of labeled trees.

\begin{definition}\label{d2.4}
A labeled tree is a pair \((T, l)\), where \(T\) is a tree and \(l\) is a mapping defined on the set \(V(T)\).

As in the case of the rooted trees, we will say that \(T\) is a \emph{free tree} corresponding to \((T, l)\) and write \(T = T(l)\) instead of \((T, l)\). Moreover, in what follows, we will consider only the nonnegative real-valued labelings \(l\colon V(T)\to \RR^{+}\).
\end{definition}

For the case of labeled trees Definitions~\ref{d8.7} and \ref{d8.10} must be modified as follows.

\begin{definition}\label{d2.5}
Let \(T_i=T_i(l_i)\) be a labeled tree, \(i = 1\), \(2\). A mapping $f\colon V(T_1) \to V(T_2)$ is an isomorphism of \(T_1(l_1)\) and \(T_2(l_2)\) if it is an isomorphism of the free trees $T_1$ and $T_2$ and the equality
\begin{equation}\label{d2.5:e1}
l_2(f(v))=l_1(v)
\end{equation}
holds for every $v \in V(T_1)$.
\end{definition}

Following~\cite{DovBBMSSS2020}, for arbitrary labeled tree \(T = T(l)\), we define a mapping \(d_l \colon V(T) \times V(T) \to \RR^{+}\) as
\begin{equation}\label{e11.3}
d_l(u, v) = \begin{cases}
0 & \text{if } u = v,\\
\max\limits_{v^{*} \in V(P)} l(v^{*}) & \text{if } u \neq v,
\end{cases}
\end{equation}
where \(P\) is a path joining \(u\) and \(v\) in \(T\).

\begin{remark}\label{r11.10}
The correctness of the definition of \(d_l\) follows from Lemma~\ref{l9.14}.
\end{remark}

To formulate the first theorem of this section we recall the concept of \emph{pseudoultrametric space}. Let \(X\) be a set. A symmetric mapping \(d \colon X \times X \to \RR^{+}\) is a \emph{pseudoultrametric} on \(X\) if
\[
d(x, x) = 0 \quad \text{and} \quad d(x, y) \leqslant \max\{d(x, z), d(z, y)\}
\]
hold for all \(x\), \(y\), \(z \in X\). Every ultrametric is a pseudoultrametric, but a pseudoultrametric \(d\) on a set \(X\) is an ultrametric if and only if \(d(x, y) > 0\) holds for all distinct \(x\), \(y \in X\).

\begin{theorem}\label{t11.9}
Let \(T = T(l)\) be a labeled tree. Then \((V(T), d_l)\) is a pseudoultrametric space. The function \(d_l\) is an ultrametric on \(V(T)\) if and only if the inequality
\begin{equation}\label{t11.9:e1}
\max\{l(u_1), l(v_1)\} > 0
\end{equation}
holds for every \(\{u_1, v_1\} \in E(T)\).
\end{theorem}

A proof of Theorem~\ref{t11.9} can be obtained by simple modification of the proof of Proposition~3.2 from \cite{Dov2020TaAoG}.

\begin{proposition}\label{l11.12}
Let \(T_1 = T_1(l_1)\) and \(T_2 = T_2(l_2)\) be labeled trees and let \(f \colon V(T_1) \to V(T_2)\) be an isomorphism of these trees. Then the equality
\[
d_{l_1}(u, v) = d_{l_2}(f(u), f(v))
\]
holds for all \(u\), \(v \in V(T_1)\).
\end{proposition}

\begin{proof}
It directly follows from Definition~\ref{d2.5} and formula~\eqref{e11.3}, because if \((v_1, \ldots, v_n)\) is a path joining some distinct \(v = v_1\) and \(u = v_n\) in \(T_1\), then \(f(u) \neq f(v)\) holds and \((f(v_1), \ldots, f(v_n))\) is a path joining \(f(v)\) and \(f(u)\) in \(T_2\) and we have the equality
\[
\{l_1(v_1), \ldots, l_1(v_n)\} = \{l_2(f(v_1)), \ldots, l_2(f(v_n))\}. \qedhere
\]
\end{proof}

\begin{corollary}\label{c11.6}
Let \(T_1 = T_1(l_1)\) and \(T_2 = T_2(l_2)\) be isomorphic labeled trees. Then ultrametric (pseudoultrametric) spaces \((V(T_1), d_{l_1})\) and \((V(T_2), d_{l_2})\) are isometric.
\end{corollary}

The converse statement is not valid in general.

\begin{example}\label{ex11.13}
Let \(V = \{v_0, v_1, v_2, v_3, v_4\}\) be a five-point set, and let \(S = S(l_S)\) and \(P = P(l_P)\) be a labeled star with the center \(v_0\) and, respectively, a labeled path such that \(V(S) = V(P) = V\), \(l_S(v_0) = l_P(v_0) = 1\) and
\[
l_S(v_i) + 1 = l_P(v_i) = 1
\]
for \(i = 1\), \(\ldots\), \(4\) (see Figure~\ref{fig11}). Then, for all distinct \(u\), \(v \in V\), we have
\[
d_{l_P}(u, v) = d_{l_S}(u, v) = 1.
\]
Thus, the ultrametric spaces \((V, d_{l_P})\) and \((V, d_{l_S})\) coincide, but \(P(l_P)\) and \(S(l_S)\) are not isomorphic as labeled trees or even as free trees.

\begin{figure}[ht]
\begin{center}
\begin{tikzpicture}[
level 1/.style={level distance=\levdist,sibling distance=24mm},
level 2/.style={level distance=\levdist,sibling distance=12mm},
level 3/.style={level distance=\levdist,sibling distance=6mm},
solid node/.style={circle,draw,inner sep=1.5,fill=black},
micro node/.style={circle,draw,inner sep=0.5,fill=black}]

\def\dx{1cm}
\node at (0, \dx) [label= above:\(S(l_S)\)] {};
\node [solid node] at (0, 0) [label= below:\(1\)] {};
\node [solid node] at (\dx, \dx) [label= right:\(0\)] {};
\node [solid node] at (\dx, -\dx) [label= right:\(0\)] {};
\node [solid node] at (-\dx, \dx) [label= left:\(0\)] {};
\node [solid node] at (-\dx, -\dx) [label= left:\(0\)] {};
\draw (0, 0) -- (\dx, \dx);
\draw (0, 0) -- (\dx, -\dx);
\draw (0, 0) -- (-\dx, \dx);
\draw (0, 0) -- (-\dx, -\dx);

\node at (5*\dx, \dx) [label= above:\(P(l_P)\)] {};
\node [solid node] at (3*\dx, 0) [label= below:\(1\)] {};
\node [solid node] at (4*\dx, 0) [label= below:\(1\)] {};
\node [solid node] at (5*\dx, 0) [label= below:\(1\)] {};
\node [solid node] at (6*\dx, 0) [label= below:\(1\)] {};
\node [solid node] at (7*\dx, 0) [label= below:\(1\)] {};
\draw (3*\dx, 0) -- (7*\dx, 0);
\end{tikzpicture}
\end{center}
\caption{The star \(S\) and the path \(P\) are not isomorphic as free trees, but the labelings \(l_S\) and \(l_P\) generate the same ultrametric on \(V\).}
\label{fig11}
\end{figure}
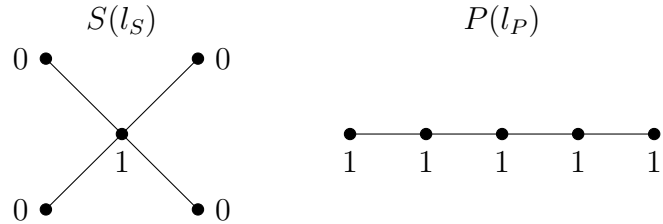
\end{example}

\begin{remark}\label{r11.8}
In the next section of the paper, after introduction of the concept of monotone labeling, it will be shown that for every discrete ultrametric space of the form \((V(T), d_l)\) there is a unique, up to isomorphism, rooted tree \(T^{*}(r^{*})\) with monotone labeling \(l^{*} \colon V(T^{*}) \to \RR^{+}\) such that the ultrametric spaces \((V(T), d_l)\) and \((V(T^{*}), d_{l^{*}})\) are isometric.
\end{remark}

Recall that a metric space \((X, d)\) is \emph{discrete} if every point of \(X\) is isolated.

\begin{theorem}\label{t11.16}
Let \(T = T(l)\) be a labeled tree such that the inequality
\[
\max\{l(u), l(v)\} > 0
\]
holds for every \(\{u, v\} \in E(T)\). Then the following statements are equivalent:
\begin{enumerate}
\item \label{t11.16:s1} The ultrametric space \((V(T), d_{l})\) is discrete and totally bounded.
\item \label{t11.16:s2} The tree \(T\) is locally finite and the set
\begin{equation*}%\label{t11.16:e1}
V_{\varepsilon} = \bigl\{v \in V(T) \colon l(v) \geqslant \varepsilon\bigr\}
\end{equation*}
is finite for every \(\varepsilon > 0\).
\end{enumerate}
\end{theorem}

\begin{proof}
\(\ref{t11.16:s1} \Rightarrow \ref{t11.16:s2}\). Let \ref{t11.16:s1} hold. Suppose that the set \(V_{\varepsilon}\) is infinite for some \(\varepsilon > 0\). Using the definition of \(d_{l}\) (see~\eqref{e11.3}) it is easy to prove the inequality
\begin{equation}\label{t11.16:e2}
d_{l}(v_1, v_2) \geqslant \varepsilon
\end{equation}
for all distinct \(v_1\), \(v_2 \in V_{\varepsilon}\). Hence, the subspace \((V_{\varepsilon}, d_{l}|_{V_{\varepsilon} \times V_{\varepsilon}})\) of totally bounded ultrametric space \((V(T), d_{l})\) is not totally bounded, contrary to Corollary~\ref{c2.17}. Thus, \(V_{\varepsilon}\) is finite for every \(\varepsilon > 0\).

Assume now that \(T\) contains a vertex \(v^*\) of infinite degree, \(\delta_{T}(v^*) = \infty\). Let us consider first the case when \(l(v^{*}) > 0\) holds.

Write
\begin{equation}\label{t11.16:e3}
U_{v^{*}} = \bigl\{u \in V(T) \colon \{u, v^*\} \in E(T)\bigr\}.
\end{equation}
The equality \(\delta_{T}(v^*) = \infty\) implies that \(U_{v^{*}}\) has an infinite cardinality, \(|U_{v^{*}}| = \infty\). For all distinct \(u_1\), \(u_2 \in U_{v^{*}}\) the unique path \(P\) joining \(u_1\) and \(u_2\) in \(T\) has the form \((u_1, v^{*}, u_2)\). Hence,
\[
d_l(u_1, u_2) \geqslant l(v^{*}) > 0
\]
holds by \eqref{e11.3}. It implies that the ultrametric space \((U_{v^{*}}, d_l|_{U_{v^{*}} \times U_{v^{*}}})\) is not totally bounded, contrary to \ref{t11.16:s1}.

Thus, if \(\delta_{T}(v^*) = \infty\) holds, then we have the equality
\begin{equation}\label{t11.16:e4}
l(v^{*}) = 0.
\end{equation}
It is clear that the inclusion
\begin{equation}\label{t11.16:e6}
\{u \in U_{v^{*}} \colon l(u) > \varepsilon\} \subseteq V_{\varepsilon}
\end{equation}
is true for every \(\varepsilon > 0\). Since \(V_{\varepsilon}\) is finite for \(\varepsilon > 0\), inclusion~\eqref{t11.16:e6} and the equality \(|U_{v^{*}}| = \infty\) guarantee the existence of a sequence \((u_n)_{n \in \mathbb{N}} \subseteq U_{v^{*}}\) such that
\[
\lim_{n \to \infty} l(u_n) = 0.
\]
From the last limit relation, \eqref{e11.3}, \eqref{t11.16:e3} and \eqref{t11.16:e4}, we obtain
\[
d_l(v^{*}, u_n) = l(u_n)
\]
for every \(n \in \mathbb{N}\). Consequently, \(\lim_{n \to \infty} d_l(v^{*}, u_n) = 0\) holds. Thus, \(v^{*}\) is an accumulation point of the ultrametric space \((V(T), d_{l})\), contrary to statement~\ref{t11.16:s1}.

\(\ref{t11.16:s2} \Rightarrow \ref{t11.16:s1}\). Let \ref{t11.16:s2} hold. Statement \ref{t11.16:s1} is trivially valid if \(T\) is a finite tree. Let \(|V(T)| = \infty\) hold. At first we prove that the ultrametric space \((V(T), d_l)\) is discrete.

Suppose that \(v_0\) is a vertex of \(T\) such that \(l(v_0) > 0\). Then \eqref{e11.3} implies
\[
d_l(u, v_0) \geqslant l(v_0)
\]
for every \(u \in V(T) \setminus \{v_0\}\). Hence, \(v_0\) is an isolated point of \((V(T), d_{l})\).

Let us consider now the case when \(v_0 \in V(T)\) has the zero label, \(l(v_0) = 0\). As in \eqref{t11.16:e3}, write
\[
U_{v_0} = \{u \in V(T) \colon \{u, v_0\} \in E(T)\}.
\]
Since the inequality \(\max\{l(u), l(v_0)\} > 0\) holds for every \(\{u, v_0\} \in E(T)\), from \(0 < \delta_{T}(v_0) < \infty\) it follows the inequality \(\min\{l(u) \colon u \in U_{v_0}\} > 0\). Let us consider now an arbitrary vertex \(u_0 \in V(T) \setminus \{v_0\}\). If \(P\) is a path joining \(v_0\) and \(u_0\) in \(T\), then there is a vertex \(u^{*}\) of \(T\) such that
\begin{equation}\label{t11.16:e5}
u^{*} \in V(P) \cap U_{v_0}.
\end{equation}
Now using \eqref{e11.3} and \eqref{t11.16:e5} we obtain
\[
d_l(v_0, u_0) = \max_{u \in V(P)} l(u) \geqslant l(v_0, u^{*}) \geqslant \min\{l(u) \colon u \in U_{v_0}\} > 0.
\]
Hence, \(v_0\) is an isolated point of \((V(T), d_{l})\). Thus, the ultrametric space \((V(T), d_{l})\) is discrete.

To complete the proof, it suffices to show that \((V(T), d_{l})\) is totally bounded.

By Proposition~\ref{p2.11}, the ultrametric space \((V(T), d_{l})\) is totally bounded if every sequence of vertices of \(T\) contains a Cauchy subsequence. Let us consider a sequence \((u_j^0)_{j \in \mathbb{N}}\) of pairwise distinct points of \(V(T)\). Let \((\varepsilon_i)_{i \in \mathbb{N}}\) be a decreasing sequence of strictly positive real numbers such that \(\lim_{i \to \infty} \varepsilon_i = 0\). Statement~\ref{t11.16:s2} implies that the set \(V_{\varepsilon_{1}}\) is finite. Write \(G^{1}\) for the subgraph of \(T\) induced by \(V(T) \setminus V_{\varepsilon_{1}}\), i.e., \(V(G^{1}) = V(T) \setminus V_{\varepsilon_{1}}\) and
\[
\bigl(\{u, v\} \in E(G^{1})\bigr) \Leftrightarrow \bigl(u, v \in V(G^{1}) \text{ and } \{u, v\} \in E(T)\bigr).
\]
Every connected component of \(G^{1}\) is a tree. Since \(V_{\varepsilon_{1}}\) is a finite set and \(T\) is a locally finite, infinite tree, the number of these trees are finite. Hence, there is an infinite tree \(T^{1} \subseteq T\), \(|V(T^{1})| = \infty\), with \(l(v^{1}) < \varepsilon_{1}\) for all \(v^{1} \in V(T^{1})\) and such that \((u_{j}^{1})_{j \in \mathbb{N}} \subseteq V(T^{1})\) holds for an infinite subsequence \((u_{j}^{1})_{j \in \mathbb{N}}\) of the sequence \((u_{j}^{0})_{j \in \mathbb{N}} \subseteq V(T)\). Write \(u^{1} = u_1^1\).

Let us consider the subgraph \(G^{2}\) of \(T^{1}\) induced by \(V(T^{1}) \setminus V_{\varepsilon_{2}}\). As above, the inequality \(|V_{\varepsilon_{2}}| < \infty\) implies the existence of an infinite tree \(T^{2} \subseteq T^{1}\) and a subsequence \((u_{j}^{2})_{j \in \mathbb{N}}\) of the sequence \((u_{j}^{1})_{j \in \mathbb{N}} \subseteq T^{1}\) for which \((u_{j}^{2})_{j \in \mathbb{N}} \subseteq V(T^{2})\) holds. Let us write \(u^{2} = u_1^2\).

By induction, for every \(i \geqslant 2\), we find an infinite connected component \(T^{i+1}\) of the subgraph \(G^{i+1}\) of \(T^{i}\) induced by \(V(T^{i}) \setminus V_{\varepsilon_{i+1}}\) and a subsequence \((u_{j}^{i+1})_{j \in \mathbb{N}}\) of the sequence \((u_{j}^{i})_{j \in \mathbb{N}}\) such that
\begin{equation}\label{t11.16:e8}
(u_{j}^{i+1})_{j \in \mathbb{N}} \subseteq V(T^{i+1}).
\end{equation}
Write \(u^{i+1}\) for the first element \(u_{1}^{i+1}\) of this sequence.

Let us consider now the sequence \((u^{i})_{i \in \mathbb{N}}\). It is clear that \((u^{i})_{i \in \mathbb{N}}\) is a subsequence of the original sequence \((u_j^0)_{j \in \mathbb{N}}\). From~\eqref{t11.16:e8} it follows that
\[
l(u^{i+1}) < \varepsilon_{i+1}
\]
holds for every \(i \in \mathbb{N}\). The last inequality and the limit relation \(\lim_{i \to \infty} \varepsilon_i = 0\) imply
\begin{equation}\label{t11.16:e9}
\lim_{i \to \infty} l(u^{i}) = 0.
\end{equation}
Moreover, since for every \(i \geqslant 2\) the points \(u^{i}\) and \(u^{i+1}\) are vertices of the tree \(T^{i}\) and the inequality \(l(v) \leqslant \varepsilon_{i}\) holds for every \(v \in V(T^{i})\), formula \eqref{e11.3} implies the inequality
\[
d_{l}(u^{i}, u^{i+1}) \leqslant l(u^{i-1})
\]
for every \(i \geqslant 2\). Now using Proposition~\ref{p4.5} and limit relation~\eqref{t11.16:e9} we obtain that \((u^{i})_{i \in \mathbb{N}}\) is a Cauchy sequence.

The same reasons show that \((u_j^0)_{j \in \mathbb{N}}\) contains a Cauchy subsequence whenever \((u_j^0)_{j \in \mathbb{N}}\) contains an infinite subsequence of pairwise distinct members.

To complete the proof, we note only that if all subsequences of pairwise distinct members of a sequence are finite, then there is an infinite constant subsequence of that sequence.
\end{proof}

The following result gives us the necessary and sufficient conditions under which the ultrametric space \((V(T), d_l)\) is totally bounded.

\begin{theorem}\label{t11.8}
Let \(T = T(l)\) be a labeled tree such that the inequality
\[
\max\{l(u), l(v)\} > 0
\]
holds for every \(\{u, v\} \in E(T)\). Then the following statements are equivalent:
\begin{enumerate}
\item \label{t11.8:s1} The ultrametric space \((V(T), d_{l})\) is totally bounded.
\item \label{t11.8:s2} The set \(V_{\varepsilon} = \bigl\{v \in V(T) \colon l(v) \geqslant \varepsilon\bigr\}\) is finite for every \(\varepsilon > 0\) and the inequality \(\delta_{T}(v) < \infty\) holds whenever \(l(v) > 0\).
\end{enumerate}
\end{theorem}

This theorem can be proved similarly to Theorem~\ref{t11.16} and we omit it here.

The following example shows that \((V(T), d_{l})\) can be totally bounded even if the tree \(T\) contains infinitely many vertices of infinite degree.

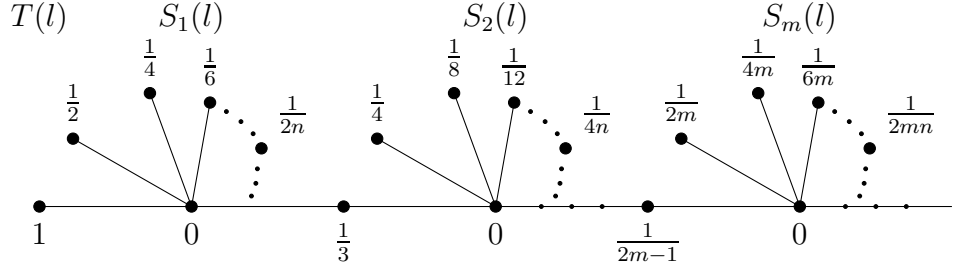
\begin{figure}[ht]
\begin{center}
\begin{tikzpicture}[
level 1/.style={level distance=\levdist,sibling distance=24mm},
level 2/.style={level distance=\levdist,sibling distance=12mm},
level 3/.style={level distance=\levdist,sibling distance=6mm},
solid node/.style={circle,draw,inner sep=1.5,fill=black},
micro node/.style={circle,draw,inner sep=0.5,fill=black}]

\def\dx{2cm}
\draw (0, 0) -- (6*\dx, 0);
\node [solid node] at (0, 0) [label= below:\(1\)] {};
\node [solid node] at (\dx, 0) [label= below:\(0\)] {};
\node at (0, \dx) [label= above:\(T(l)\)] {};
\draw (\dx, 0) -- +(150:0.9*\dx) node [solid node, label= above:\(\frac{1}{2}\)] {};
\draw (\dx, 0) -- +(110:0.8*\dx) node [solid node, label= above:\(\frac{1}{4}\)] {};
\draw (\dx, 0) -- +(80:0.7*\dx) node [solid node, label= above:\(\frac{1}{6}\)] {};
\draw (\dx, 0) +(70:0.67*\dx) node [micro node] {};
\draw (\dx, 0) +(60:0.65*\dx) node [micro node] {};
\draw (\dx, 0) +(50:0.63*\dx) node [micro node] {};
\draw (\dx, 0) +(40:0.6*\dx) node [solid node, label= above right:\(\frac{1}{2n}\)] {};
\draw (\dx, 0) +(30:0.5*\dx) node [micro node] {};
\draw (\dx, 0) +(20:0.45*\dx) node [micro node] {};
\draw (\dx, 0) +(10:0.4*\dx) node [micro node] {};
\node at (\dx, \dx) [label= above:\(S_1(l)\)] {};

\node [solid node] at (2*\dx, 0) [label= below:\(\frac{1}{3}\)] {};

\node [solid node] at (3*\dx, 0) [label= below:\(0\)] {};
\draw (3*\dx, 0) -- +(150:0.9*\dx) node [solid node, label= above:\(\frac{1}{4}\)] {};
\draw (3*\dx, 0) -- +(110:0.8*\dx) node [solid node, label= above:\(\frac{1}{8}\)] {};
\draw (3*\dx, 0) -- +(80:0.7*\dx) node [solid node, label= above:\(\frac{1}{12}\)] {};
\draw (3*\dx, 0) +(70:0.67*\dx) node [micro node] {};
\draw (3*\dx, 0) +(60:0.65*\dx) node [micro node] {};
\draw (3*\dx, 0) +(50:0.63*\dx) node [micro node] {};
\draw (3*\dx, 0) +(40:0.6*\dx) node [solid node, label= above right:\(\frac{1}{4n}\)] {};
\draw (3*\dx, 0) +(30:0.5*\dx) node [micro node] {};
\draw (3*\dx, 0) +(20:0.45*\dx) node [micro node] {};
\draw (3*\dx, 0) +(10:0.4*\dx) node [micro node] {};
\node at (3*\dx, \dx) [label= above:\(S_2(l)\)] {};

\node [micro node] at (3.3*\dx, 0) {};
\node [micro node] at (3.5*\dx, 0) {};
\node [micro node] at (3.7*\dx, 0) {};

\node [solid node] at (4*\dx, 0) [label= below:\(\frac{1}{2m-1}\)] {};

\node [solid node] at (5*\dx, 0) [label= below:\(0\)] {};
\draw (5*\dx, 0) -- +(150:0.9*\dx) node [solid node, label= above:\(\frac{1}{2m}\)] {};
\draw (5*\dx, 0) -- +(110:0.8*\dx) node [solid node, label= above:\(\frac{1}{4m}\)] {};
\draw (5*\dx, 0) -- +(80:0.7*\dx) node [solid node, label= above:\(\frac{1}{6m}\)] {};
\draw (5*\dx, 0) +(70:0.67*\dx) node [micro node] {};
\draw (5*\dx, 0) +(60:0.65*\dx) node [micro node] {};
\draw (5*\dx, 0) +(50:0.63*\dx) node [micro node] {};
\draw (5*\dx, 0) +(40:0.6*\dx) node [solid node, label= above right:\(\frac{1}{2mn}\)] {};
\draw (5*\dx, 0) +(30:0.5*\dx) node [micro node] {};
\draw (5*\dx, 0) +(20:0.45*\dx) node [micro node] {};
\draw (5*\dx, 0) +(10:0.4*\dx) node [micro node] {};
\node at (5*\dx, \dx) [label= above:\(S_m(l)\)] {};

\node [micro node] at (5.3*\dx, 0) {};
\node [micro node] at (5.5*\dx, 0) {};
\node [micro node] at (5.7*\dx, 0) {};
\end{tikzpicture}
\end{center}
\caption{The tree \(T = T(l)\) is not locally finite but the ultrametric space \((V(T), d_l)\) is totally bounded.}
\label{fig10}
\end{figure}

\begin{example}\label{ex11.10}
Let \(T_1(l_1)\) be the labeled tree such that \(V(T_1) = \mathbb{N}\) and
\[
\bigl(\{m, n\} \in E(T_1)\bigr) \Leftrightarrow \bigl(|m-n|=1\bigr)
\]
for all \(m\), \(n \in \mathbb{N}\) and, for each \(m \in V(T_1)\),
\[
l_1(m) = \begin{cases}
\frac{1}{m} & \text{if \(m\) is odd},\\
0 & \text{if \(m\) is even}.
\end{cases}
\]
For every \(m \in \mathbb{N}\) we define a labeled star \(S_m(l_m)\) with a center \(c_m\) such that \(|V(S_m)|\) is countably infinite, \(l_m(c_m) = 0\),
\[
V(S_{m_1}) \cap V(S_{m_2}) = V(T) \cap V(S_{m}) = \varnothing
\]
for all distinct \(m_1\), \(m_1 \in \mathbb{N}\), and the restriction \(l_m|_{V(S_m) \setminus \{c_m\}}\) of the labeling \(l_m\) on the set \(V(S_m) \setminus \{c_m\}\) is a bijection on the set \(\{mn \colon n \in \mathbb{N}\}\). The labeled tree \(T = T(l)\) is obtained from \(T_1(l_1)\) by gluing the center \(c_m\) of labeled star \(S_m(l_m)\) to vertex \(2m \in V(T_1)\) (see Figure~\ref{fig10}). The ultrametric space \((V(T), d_l)\) is totally bounded by Theorem~\ref{t11.8}.
\end{example}

The following simple example shows that the class of totally bounded ultrametric spaces which are representable in the form \((V(T), d_l)\) is a proper subclass of all totally bounded ultrametric spaces.

\begin{example}\label{ex11.11}
Let \(V = \{v_1, v_2, v_3, v_4\}\) be a four-point set and let an ultrametric \(d \colon V \times V \to \RR^{+}\) satisfy the equalities
\begin{gather}\label{ex11.11:e1}
d(v_1, v_2) = d(v_3, v_4) = 1\\
\intertext{and}
\label{ex11.11:e2}
d(v_1, v_3) = d(v_1, v_4) = d(v_2, v_3) = d(v_2, v_4) = 2.
\end{gather}
Let \(T = T(l)\) be a labeled tree such that \(V(T) = V\). We claim that the ultrametric spaces \((V, d)\) and \((V(T), d_l)\) are not isometric. Indeed, from~\eqref{e11.3} and \eqref{ex11.11:e1} it follows that
\[
\max_{1 \leqslant i \leqslant 4} l(v_i) \leqslant 1.
\]
The last inequality and \eqref{e11.3} imply that \(d_l(v, u) \leqslant 1\) holds for all \(u\), \(v \in V\), contrary to \eqref{ex11.11:e2}.
\end{example}

It seems to be interesting to get a purely metric characterization of totally bounded ultrametric spaces generated by labelings on the trees. For discrete totally bounded ultrametric spaces, such characterization, apparently, can be given in the language of diametrical graphs.

\begin{remark}
Theorems~\ref{t11.16} and \ref{t11.8} were proved in paper~\cite{DK2022LTGCCaDUS}.
\end{remark}

\begin{conjecture}\label{con11.1}
Let \((X, d)\) be a discrete nonempty totally bounded ultrametric space. Then the following statements are equivalent:
\begin{enumerate}
\item\label{con11.1:s1} There is a labeled tree \(T = T(l)\) such that \((V(T), d_l)\) and \((X, d)\) are isometric.
\item\label{con11.1:s2} At least one of the parts of the diametrical graph \(G_{B, d|_{B \times B}}\) is a single-point set for every ball \(B \in \BB_{X}\) with \(\diam B > 0\).
\end{enumerate}
\end{conjecture}

Statement~\ref{con11.1:s2} of the above conjecture can also be formulated as follows.

\begin{enumerate}
\item[ ] For every \(B \in \BB_{X}\), there are \(c \in X\) and \(r \geqslant 0\) such that
\[
B = \{x \in X \colon d(x, c) = r\} \cup \{c\}
\]
i.e., the ball \(B\) is the sphere \(S(c, r) = \{x \in X \colon d(x, c) = r\}\) with the added center \(c\).
\end{enumerate}

\begin{remark}
Conjecture~\ref{con11.1} was formulated in~\cite{DK2022LTGCCaDUS}. E.~Petrov proved in~\cite{Pet2022HPoFUS} the validity of Conjecture~\ref{con11.1} for finite ultrametric spaces using some other terms. In~\cite{DK2024DLPSSAG} it is shown that the equivalence \((i) \Leftrightarrow (ii)\) is valid for all locally finite ultrametric spaces \((X, d)\).
\end{remark}

\subsection{Isomorphism of labeled representing trees}

Let us define the labeling for representing trees of totally bounded ultrametric spaces.

\begin{definition}\label{d9.27}
Let \((X, d)\) be a nonempty totally bounded ultrametric space. The representing tree \(T_X\) together with the labeling
\[
l_X \colon \BB_{X} \to \RR^{+}, \quad l_X(B) = \diam B,
\]
is said to be the \emph{labeled representing tree} of \((X, d)\) and denoted as \(T_X = T_X(l_X)\).
\end{definition}

\begin{remark}\label{r11.4}
Using Lemma~\ref{l2.6} and Corollary~\ref{c2.41}, we see that, for every nonempty totally bounded ultrametric space \((X, d)\), the ball \(B = X\) is an unique element of \(\BB_{X}\) satisfying the equality
\[
l_X(B) = \diam X.
\]
Hence, for every totally bounded ultrametric space \((X, d)\), the labeled representing tree \(T_X(l_X)\) has a ``natural'' root \(r_X = X\). Consequently, for any two nonempty totally bounded ultrametric spaces \((Y, \rho)\) and \((Z, \delta)\), each isomorphism of the labeled representing trees \(T_Y(l_Y)\) and \(T_Z(l_Z)\) is also an isomorphism of the rooted representing trees \(T_Y(r_Y)\) and \(T_Z(r_Z)\).
\end{remark}

\begin{remark}\label{r11.5}
It follows directly from the definition of leaves of trees and Lemma~\ref{l8.7} that \(v \in V(T_X)\) is a leaf of \(T_X\) of and only if \(l_X(v) = 0\) holds.
\end{remark}

Let us consider the pseudoultrametric spaces \((V(T_X), d_{l_X})\) generated by the labeled representing trees \(T_X(l_X)\) of totally bounded ultrametric spaces \((X, d)\). To begin with, it should be noted that \(d_{l_X}\) is an ultrametric on \(V(T_X)\) for every nonempty totally bounded ultrametric space \((X, d)\). Indeed, if \(\{B_1, B_2\} \in E(T_X)\) holds, then using Lemma~\ref{l8.3} and the definition of \(l_X\) we obtain
\[
\max\{l_X(B_1), l_X(B_2)\} > 0.
\]
Hence, \(d_{l_X}\) is an ultrametric by Theorem~\ref{t11.9}.

\begin{proposition}\label{p11.11}
Let \((X, d)\) be a nonempty totally bounded ultrametric space with the labeled representing tree \(T_X = T_X(l_X)\). Then the equality
\begin{equation}\label{p11.11:e1}
d_{l_X}(B_1, B_2) = d_H(B_1, B_2)
\end{equation}
holds for all \(B_1\), \(B_2 \in \BB_{X}\), where \(d_H\) is the restriction of Hausdorff distance on the ballean \(\BB_{X}\).
\end{proposition}

\begin{proof}
By Corollary~\ref{c2.41}, we have the inclusion \(\BB_{X} \subseteq \overline{\BB}_{X}\) and, by Lemma~\ref{l6.12}, the equality
\[
d_H(B_1, B_2) = \diam (B_1 \cup B_2)
\]
holds for all distinct \(B_1\), \(B_2 \in \overline{\BB}_{X}\). Since \eqref{p11.11:e1} is evidently valid whenever \(B_1 = B_2\), it suffices to to show that
\begin{equation}\label{p11.11:e2}
d_{l_X}(B_1, B_2) = \diam (B_1 \cup B_2)
\end{equation}
holds for all distinct \(B_1\), \(B_2 \in \BB_{X}\).

Suppose first that \(B_1\) and \(B_2\) are disjoint open ball, \(B_1 \cap B_2 = \varnothing\). Write \(B^{*}\) for the smallest closed ball containing \(B_1 \cup B_2\). From Proposition~\ref{p2.12} it follows that
\begin{equation}\label{p11.11:e3}
\diam B^{*} = \diam (B_1 \cup B_2).
\end{equation}
Since every \(B \in \BB_{X}\) is nonempty and \(B_1 \cap B_2 = \varnothing\) holds, we have \(\diam (B_1 \cup B_2) > 0\). The last inequality, equality~\eqref{p11.11:e3} and Lemma~\ref{c6.6} imply that \(B^{*} \in \BB_{X}\). Using the (proper) inclusions \(B_1 \subset B^{*}\), \(B_2 \subset B^{*}\), Theorem~\ref{t9.20} and Lemma~\ref{l10.14} we can find \(B_1^{1}\), \(\ldots\), \(B_1^{n} \in \BB_{X}\) and \(B_2^{1}\), \(\ldots\), \(B_2^{m} \in \BB_{X}\) such that \(B_1 = B_1^{1}\), \(B_2 = B_2^{1}\),
\begin{equation}\label{p11.11:e4}
B_1^{n} = B^{*} = B_2^{m}
\end{equation}
and
\begin{equation}\label{p11.11:e5}
B_1^{i} \sqsubset_{T_X(r_X)} B_1^{i+1}, \quad B_2^{j} \sqsubset_{T_X(r_X)} B_2^{j+1}
\end{equation}
whenever \(1 \leqslant i \leqslant n-1\) and \(1 \leqslant j \leqslant m-1\). From the uniqueness of the smallest ball \(B^{*}\) and Lemma~\ref{l10.10} it follows that
\[
P_{B_1, B_2}^{*} = (B_1^{1}, \ldots, B_1^{n-1}, B^{*}, B_2^{m-1}, \ldots, B_2^{1})
\]
is a path in \(T_X(r_X)\) joining the balls \(B_1\) and \(B_2\). Relations \eqref{p11.11:e4} and \eqref{p11.11:e5} imply \(B^{*} \supseteq B_1^{i}\) and \(B^{*} \supseteq B_2^{j}\) for every \(i \in \{1, \ldots, n\}\) and \(j \in \{1, \ldots, m\}\). Hence, the equalities
\begin{equation}\label{p11.11:e6}
\diam B^{*} = \max \{\diam B \colon B \in V(P_{B_1, B_2}^{*})\} = d_{l_X}(B_1, B_2)
\end{equation}
hold.

If \(B_1 \cap B_2 \neq \varnothing\), then, by Proposition~\ref{p2.5}, we have either \(B_1 \subset B_2\) or \(B_2 \subset B_1\). Without loss of generality, we may assume that \(B_1 \subset B_2\). Then, arguing as above, we can find \(B_1^{1}\), \(\ldots\), \(B_1^{n} \in \BB_{X}\) such that \(B_1 = B_1^{1}\) and \(B_2 = B_1^{n}\) and
\[
B_1^{i} \sqsubset_{T_X(r_X)} B_1^{i+1}
\]
holds whenever \(1 \leqslant i \leqslant n-1\). As in \eqref{p11.11:e6} we obtain
\[
\diam (B_1 \cup B_2) = \diam B_2 = \max \{\diam B \colon B \in V(P_{B_1, B_2})\} = d_{l_X}(B_1, B_2),
\]
where \(P_{B_1, B_2} = (B_1^{1}, \ldots, B_1^{n})\) is the path joining \(B_1\) and \(B_2\) in \(T_X\).
\end{proof}

Proposition~\ref{p11.11} and Theorem~\ref{t7.8} give us.

\begin{corollary}\label{c11.12}
Let \((X, d)\) be a nonempty totally bounded ultrametric space with labeled representing tree \(T_X = T_X(l_X)\). Then the ultrametric space \((V(T_X), d_{l_X})\) is isometric to the subspace of all isolated points of the ultrametric space \((\overline{\BB}_X, d_H)\).
\end{corollary}

Proposition~\ref{p11.11} also implies the following.

\begin{lemma}\label{l11.15}
Let \((X, d)\) be a nonempty totally bounded ultrametric space with labeled representing tree \(T_X = T_X(l_X)\), let \(B_1\), \(B_2\) be disjoint open balls. Then the equality
\begin{equation}\label{l11.15:e1}
d(x_1, x_2) = d_{l_X}(B_1, B_2)
\end{equation}
holds for every \(x_1 \in B_1\) and every \(x_2 \in B_2\).
\end{lemma}

\begin{proof}
Let \(x_1 \in B_1\) and \(x_2 \in B_2\). Then, by Proposition~\ref{p2.5} and Proposition~\ref{p11.11}, we have
\begin{equation*}
d(x_1, x_2) = \diam (B_1 \cup B_2) = d_{l_X}(B_1, B_2). \qedhere
\end{equation*}
\end{proof}

In what follows we write \(T_1(l_1) \simeq T_2(l_2)\) if the labeled trees \(T_1(l_1)\) and \(T_2(l_2)\) are isomorphic.

\begin{lemma}\label{l11.10}
Let \((X, d)\) and \((Y, \rho)\) be nonempty totally bounded ultrametric spaces. If \((X, d)\) and \((Y, \rho)\) are isometric, then \(T_X(l_X) \simeq T_Y(l_Y)\) holds.
\end{lemma}

\begin{proof}
Let \((X, d)\) and \((Y, \rho)\) be isometric with an isometry \(\Phi \colon X \to Y\). Then the equivalence
\[
(A \in \BB_{X}) \Leftrightarrow (\{\Phi(x) \colon x \in A\} \in \BB_{Y})
\]
is valid for every \(A \subseteq X\). Consequently, the mapping \(\Phi^* \colon \BB_{X} \to \BB_{Y}\) with
\begin{equation}\label{l11.10:e2}
\Phi^{*}(B) = \{\Phi(x) \colon x \in B\}, \quad B \in \BB_{X},
\end{equation}
is bijective. Proposition~\ref{p8.4} implies the equalities \(V(T_X) = \BB_{X}\) and \(V(T_Y) = \BB_{Y}\). Hence, \(\Phi^*\) is a bijective mapping from \(V(T_X)\) to \(V(T_Y)\). We claim that \(\Phi^*\) is an isomorphism of the labeled rooted trees \(T_X(r_X, l_X)\) and \(T_Y(r_Y, l_Y)\).

Since \(\Phi\) is bijective, we have
\[
\Phi^*(r_X) = \Phi^*(X) = Y = r_Y.
\]
Furthermore, since \(\Phi\) is an isometry, Definition~\ref{d9.27} and \eqref{l11.10:e2} imply
\[
l_Y(\Phi(B)) = l_X(B)
\]
for every \(B \in \BB_{X} = V(T_X)\). Now using Definition~\ref{d2.5} we see that \(\Phi^*\) is an isomorphism of the labeled rooted trees \(T_X(r_X, l_X)\) and \(T_Y(r_Y, l_Y)\) if and only if it is an isomorphism of free representing trees \(T_X\) and \(T_Y\) that follows from Lemma~\ref{l8.3} and \eqref{l11.10:e2}, because \(\Phi\) is bijective.
\end{proof}

\begin{proposition}\label{p11.10}
Let \((X, d)\) and \((Y, \rho)\) be nonempty compact ultrametric spaces. Then \((X, d)\) and \((Y, \rho)\) are isometric if and only if
\begin{equation}\label{p11.10:e1}
T_X(l_X) \simeq T_Y(l_Y).
\end{equation}
\end{proposition}

\begin{proof}
Let \((X, d)\) and \((Y, \rho)\) be isometric. Then Lemma~\ref{l11.10} implies \eqref{p11.10:e1}.

Conversely, suppose \eqref{p11.10:e1} holds and let \(\Psi^{*} \colon V(T_X) \to V(T_Y)\) be an isomorphism of the labeled trees \(T_X(l_X)\) and \(T_Y(l_Y)\). We must show that \((X, d)\) and \((Y, \rho)\) are isometric.

Let \(p\) be a point of \(X\). Write
\[
\BB_{X}^{(p)} = \{B \in \BB_{X} \colon p \in B\}.
\]
Then \(\bigcap_{B \in \BB_{X}^{(p)}} B\) is a single-point set and \(p\) is the unique point of this set,
\[
\{p\} = \bigcap_{B \in \BB_{X}^{(p)}} B.
\]
We claim that there is an unique \(q \in Y\) for which
\begin{equation}\label{p11.10:e3}
\{q\} = \bigcap_{B \in \BB_{X}^{(p)}} \Psi^{*}(B)
\end{equation}
holds and that the mapping
\begin{equation}\label{p11.10:e4}
\Psi \colon X \to Y, \quad \Psi(p) = q, \quad p \in X,
\end{equation}
is an isometry of \((X, d)\) and \((Y, \rho)\).

Let us prove the existence and uniqueness of \(q \in Y\) satisfying~\eqref{p11.10:e3}.

Suppose first that \(p\) is an isolated point of \(X\). Then \(\{p\}\) is an open ball in \(X\), \(\{p\} \in \BB_{X}\), and
\[
\diam \{p\} = l_X(\{p\}) = 0.
\]
By Remark~\ref{r11.5}, the ball \(\{p\}\) is a leaf of \(T_X\). Since \(\Psi^{*}\) is an isomorphism of \(T_X\) and \(T_Y\), the ball \(\Psi^{*}(\{p\}) \in \BB_{Y}\) is a leaf of \(T_Y\). Now using Remark~\ref{r11.5} again, we see that there is an unique \(q \in Y\) such that \(\{q\} = \Psi^{*}(\{p\})\) holds. Since \(\Psi^{*}\) is an isomorphism of the labeled trees \(T_X(l_X)\) and \(T_Y(l_Y)\), it is also an isomorphism of the rooted trees \(T_X(r_X)\) and \(T_Y(r_Y)\) (see Remark~\ref{r11.4}). Hence, by Theorem~\ref{t8.7}, the mapping \(\Psi^{*}\) is an isomorphism of posets \((\BB_{X}, {\preccurlyeq}_X)\) and \((\BB_{Y}, {\preccurlyeq}_Y)\). Now \eqref{p11.10:e3} follows from the validity of the inclusion \(\{p\} \subseteq B\) for all \(B \in \BB_{X}^{(p)}\).

Let \(p\) be an accumulation point of \((X, d)\). By Proposition~\ref{p2.5}, for all distinct \(B_1\), \(B_2 \in \BB_{X}^{(p)}\) we have either \(B_1 \subset B_2\) or \(B_2 \subset B_1\). Hence, \(\BB_{X}^{(p)}\) is a totally ordered subposet of \((\BB_{X}, {\preccurlyeq}_X)\). Using Theorem~\ref{t8.7} we see that
\[
\{\Psi^{*}(B) \colon B \in \BB_{X}^{(p)}\}
\]
is a totally ordered subposet of \((\BB_{Y}, {\preccurlyeq}_Y)\). Write
\[
D_X^{(p)} = \{\diam B \colon B \in \BB_{X}^{(p)}\}.
\]
From Proposition~\ref{p2.3} and Corollary~\ref{c2.41} it follows that (see Definition~\ref{d2.8}) the mapping
\begin{equation}\label{p11.10:e5}
\BB_{X}^{(p)} \ni B \mapsto \diam B \in D_X^{p}
\end{equation}
is an isomorphism of totally ordered spaces \((\BB_{X}^{(p)}, {\preccurlyeq}_X^{(p)})\) and \((D_X^{(p)}, {\leqslant}^{(p)})\), where
\[
{\preccurlyeq}_X^{(p)} = (\BB_{X}^{(p)} \times \BB_{X}^{(p)}) \cap ({\preccurlyeq}_X) \quad \text{and} \quad {\leqslant}^{(p)} = (D_X^{(p)} \times D_X^{(p)}) \cap {\leqslant}.
\]
Hence, by Corollary~\ref{c6.9}, there is an enumeration \(B_1\), \(\ldots\), \(B_n\), \(\ldots\) of elements of \(\BB_{X}^{(p)}\) such that
\begin{equation}\label{p11.10:e6}
B_1 \supset B_2 \supset \ldots \supset B_n \supset B_{n+1} \supset \ldots
\end{equation}
and
\begin{equation}\label{p11.10:e7}
\lim_{n \to \infty} \diam B_n = 0.
\end{equation}
Since \(\Psi^{*}\) is a graph isomorphism of \(T_X(l_X)\), \(T_Y(l_Y)\) and, simultaneously, an ordered isomorphism of \((\BB_{X}, {\preccurlyeq}_X)\), \((\BB_{Y}, {\preccurlyeq}_Y)\), from \eqref{p11.10:e6} and \eqref{p11.10:e7} it follows that
\begin{equation}\label{p11.10:e8}
\Psi^{*}(B_1) \supset \Psi^{*}(B_2) \supset \ldots \supset \Psi^{*}(B_n) \supset \Psi^{*}(B_{n+1}) \supset \ldots
\end{equation}
and
\begin{equation}\label{p11.10:e9}
\lim_{n \to \infty} \diam \Psi^{*}(B_n) = 0.
\end{equation}
The open balls \(\Psi^{*}(B)\), \(B \in \BB_{X}\), are closed subsets of \((Y, \rho)\) by Proposition~\ref{p2.2}. The Cantor theorem claims that a metric space \((Z, \delta)\) is complete if and only if for every decreasing sequence \(Z_1 \supset Z_2 \supset \ldots\) of nonempty closed subsets of \(Z\) which satisfy
\[
\lim_{n \to \infty} \diam(Z_n) = 0,
\]
the intersection \(\bigcap_{i=1}^{\infty} Z_i\) is nonempty \cite[Theorem~4.3.9]{Eng1989}. Consequently, \eqref{p11.10:e8} and \eqref{p11.10:e9} imply \(\bigcap_{i=1}^{\infty} \Psi^{*}(B_i) \neq \varnothing\). Moreover, using \eqref{p11.10:e9} we see that
\[
\left|\bigcap_{i=1}^{\infty} \Psi^{*}(B_i)\right| = 1.
\]
Thus, the mapping \(\Psi \colon X \to Y\) is correctly defined by~\eqref{p11.10:e4} in each point \(p \in X\).

It is clear that \(\Psi \colon X \to Y\) is an isometry if \(|X| = 1\) holds. Suppose \(|X| \geqslant 2\). Let \(p_1\) and \(p_2\) be distinct points of \(X\) and let \(B_1\), \(B_2 \in \BB_{X}\) such that \(p_1 \in B_1\), \(p_2 \in B_2\) and
\begin{equation}\label{p11.10:e10}
B_1 \cap B_2 = \varnothing.
\end{equation}
Then the balls \(B_1^{*} = \Psi^{*}(B_1)\) and \(B_2^{*} = \Psi^{*}(B_2)\) are also disjoint. Indeed, suppose that \(B_1^{*} \cap B_2^{*} \neq \varnothing\) holds, then, by Proposition~\ref{p2.5}, we have
\begin{equation}\label{p11.10:e11}
B_1^{*} \subseteq B_2^{*} \quad \text{or} \quad B_2^{*} \subseteq B_1^{*}.
\end{equation}
The inverse mapping \({\Psi^{*}}^{-1} \colon \BB_Y \to \BB_{X}\) is an order isomorphism of \((\BB_{Y}, {\preccurlyeq}_Y)\) and \((\BB_{X}, {\preccurlyeq}_X)\). Hence, \eqref{p11.10:e11} implies
\[
B_1 \subseteq B_2 \quad \text{or} \quad B_2 \subseteq B_1,
\]
contrary to \eqref{p11.10:e10}.

By Lemma~\ref{l11.15}, we have the equalities
\begin{gather}\label{p11.10:e12}
d(p_1, p_2) = \max \{l_X(B) \colon B \in V(P_{B_1, B_2})\}\\
\intertext{and}
\label{p11.10:e13}
\rho(q_1, q_2) = \max \{l_Y(B^{*}) \colon B^{*} \in V(P_{B_1^{*}, B_2^{*}})\},
\end{gather}
where \(q_1 = \Psi(p_1)\), \(q_2 = \Psi(p_2)\) and \(P_{B_1, B_2}\) (\(P_{B_1^{*}, B_2^{*}}\)) is the path joining the vertices \(B_1\) and \(B_2\) (\(B_1^{*}\) and \(B_2^{*}\)) in \(T_X\) (in \(T_Y\)). It follows from Definition~\ref{d2.5} that
\begin{equation}\label{p11.10:e14}
l_X(B) = l_Y(\Psi^{*}(B))
\end{equation}
holds for every \(B \in \BB_{X}\) and
\begin{equation}\label{p11.10:e15}
V(P_{B_1^{*}, B_2^{*}}) = \{\Psi^{*}(B) \colon B \in V(P_{B_1, B_2})\}.
\end{equation}
Equalities~\eqref{p11.10:e12}--\eqref{p11.10:e15} imply the equality
\[
d(p_1, p_2) = \rho(\Psi(p_1), \Psi(p_2))
\]
for all distinct \(p_1\), \(p_2 \in X\). Thus, \(\Psi \colon X \to Y\) is an isometric embedding of \((X, d)\) in \((Y, \rho)\).

The existence of an isometric embedding of \((Y, \rho)\) in \((X, d)\) can be proved similarly. Hence, by Proposition~\ref{p3.9}, the mapping \(\Psi \colon X \to Y\) is an isometry as required.
\end{proof}

\begin{proposition}\label{p8.13}
Let \((X, d)\) be a nonempty totally bounded ultrametric space. If \(X_1\) and \(X_2\) are dense subsets of \(X\), then \(T_{X_1}(l_{X_1})\) and \(T_{X_2}(l_{X_2})\) are isomorphic with an isomorphism \(\Psi \colon \BB_{X_1} \to \BB_{X_2}\), where \(\Psi\) is a bijective mapping for which the diagram
\[
\ctdiagram{
\ctv -60,0:{\BB_{X_1}}
\ctv 60,0:{\BB_{X_2}}
\ctv 0,-60:{\BB_{X}}
\ctet -60,0,60,0:{\Psi}
\ctelg 0,-60,-60,0;-30:{\Phi_1}
\cterg 0,-60,60,0;-30:{\Phi_2}
}
\]
is commutative and \(\Phi_i \colon \BB_{X} \to \BB_{X_i}\) is defined by \(\Phi_i(B) = B \cap X_i\), \(i = 1\), \(2\).
\end{proposition}

\begin{proof}
Let \(X_1\) and \(X_2\) be dense subsets of \(X\). Using Proposition~\ref{p3.10} we see that \(\Psi \colon \BB_{X_1} \to \BB_{X_2}\) is an order isomorphism of the posets \((\BB_{X_1}, {\preccurlyeq}_{X_1})\) and \((\BB_{X_2}, {\preccurlyeq}_{X_2})\) and the equality
\begin{equation}\label{p8.13:e1}
\diam \Psi(B) = \diam B
\end{equation}
holds for every \(B \in \BB_{X_1}\). By Lemma~\ref{l10.14}, the equality
\[
(\BB_{X_i}, {\preccurlyeq}_{X_i}) = (V(T_{X_i}), {\preccurlyeq}_{T_{X_i}(r_{X_i})})
\]
holds for \(i = 1\), \(2\). Hence, \(\Psi\) is also an order isomorphism of the posets
\[
(V(T_{X_1}), {\preccurlyeq}_{T_{X_1}(r_{X_1})}) \quad \text{and} \quad (V(T_{X_2}), {\preccurlyeq}_{T_{X_2}(r_{X_2})}).
\]
Now by Proposition~\ref{p9.23}, the mapping \(\Psi\) is an isomorphism of the rooted representing trees \(T_{X_1}(r_{X_1})\) and \(T_{X_2}(r_{X_2})\), and, consequently, it is an isomorphism of free representing trees \(T_{X_1}\) and \(T_{X_2}\). Using Definition~\ref{d9.27} we can rewrite \eqref{p8.13:e1} as \(l_{X_2}(\Psi(B)) = l_{X_1}(B)\). Hence, by Definition~\ref{d2.5}, \(\Psi\) is an isomorphism of \(T_{X_1}(l_{X_1})\) and \(T_{X_2}(l_{X_2})\).
\end{proof}

\begin{corollary}\label{c9.13}
Let \((X_i, d_i)\) be a nonempty totally bounded ultrametric space with completion \((\widetilde{X}_i, \widetilde{d}_i)\), \(i = 1\), \(2\). Then \(T_{\widetilde{X}_1}(l_{\widetilde{X}_1})\) and \(T_{\widetilde{X}_2}(l_{\widetilde{X}_2})\) are isomorphic if and only if the trees \(T_{X_1}(l_{X_1})\) and \(T_{X_2}(l_{X_2})\) are isomorphic.
\end{corollary}

\begin{proof}
Let
\begin{equation}\label{c9.13:e1}
T_{X_1}(l_{X_1}) \simeq T_{X_2}(l_{X_2})
\end{equation}
be valid. The metric space \((\widetilde{X}_i, \widetilde{d}_i)\) contains a dense isometric to \((X_i, d_i)\) subspace \((\widetilde{X}_i^{*}, \widetilde{d}_i)\), \(i =1\), \(2\). By Proposition~\ref{p8.13}, we obtain
\begin{equation}\label{c9.13:e2}
T_{\widetilde{X}_i}(l_{\widetilde{X}_i}) \simeq T_{\widetilde{X}_i^{*}}(l_{\widetilde{X}_i^{*}}),
\end{equation}
\(i =1\), \(2\). Lemma~\ref{l11.10} implies the validity of
\begin{equation}\label{c9.13:e3}
T_{\widetilde{X}_i^{*}}(l_{\widetilde{X}_i^{*}}) \simeq T_{X_i}(l_{X_i}).
\end{equation}
Now
\begin{equation}\label{c9.13:e4}
T_{\widetilde{X}_1}(l_{\widetilde{X}_1}) \simeq T_{\widetilde{X}_2}(l_{\widetilde{X}_2})
\end{equation}
follows from \eqref{c9.13:e1}, \eqref{c9.13:e2} and \eqref{c9.13:e3}. Similarly, \eqref{c9.13:e4}, \eqref{c9.13:e3} and \eqref{c9.13:e2} give us \eqref{c9.13:e1}.
\end{proof}

Using Proposition~\ref{p11.10} and Corollary~\ref{c9.13} we obtain.

\begin{theorem}\label{t11.12}
Let \((X, d)\) and \((Y, \rho)\) be nonempty totally bounded ultrametric spaces. Then the completions \((\widetilde{X}, \widetilde{d})\) and \((\widetilde{Y}, \widetilde{\rho})\) are isometric if and only if
\begin{equation}\label{t11.12:e1}
T_X(l_X) \simeq T_Y(l_Y).
\end{equation}
\end{theorem}

\begin{proof}
By Proposition~\ref{p11.10}, the spaces \((\widetilde{X}, \widetilde{d})\) and \((\widetilde{Y}, \widetilde{\rho})\) are isometric if and only if
\begin{equation*}
T_{\widetilde{X}}(l_{\widetilde{X}}) \simeq T_{\widetilde{Y}}(l_{\widetilde{Y}}).
\end{equation*}
In addition, by Corollary~\ref{c9.13}, the last equivalence is valid if and only if
\begin{equation*}
T_X(l_X) \simeq T_Y(l_Y).
\end{equation*}
Thus, \eqref{t11.12:e1} holds if and only if \((\widetilde{X}, \widetilde{d})\) and \((\widetilde{Y}, \widetilde{\rho})\) are isometric.
\end{proof}

\subsection{Maximal chains of balls and completion of space}

We start by characterizing the maximal chains in the posets \((V(T_X), {\preccurlyeq}_{T_X(r_X)})\) generated by labeled representing tree \(T_X = T_X(l_X)\) (see Remark~\ref{r11.4} and Proposition~\ref{p9.16}). Recall that the presence of maximal chains is guaranteed by Proposition~\ref{l11.18} in any poset.

\begin{lemma}\label{l11.19}
Let \((X, d)\) be a nonempty compact ultrametric space with the rooted representing tree \(T_X = T_X(r_X)\). Then the following statements are equivalent for every \(C \subseteq V(T_X)\):
\begin{enumerate}
\item \label{l11.19:s1} \(C\) is a maximal chain in \((V(T_X), {\preccurlyeq}_{T_X(r_X)})\).
\item \label{l11.19:s2} There is a point \(c \in X\) such that
\begin{equation}\label{l11.19:e1}
C = \{B \in \BB_{X} \colon c \in B\}.
\end{equation}
\end{enumerate}
\end{lemma}

\begin{proof}
\(\ref{l11.19:s1} \Rightarrow \ref{l11.19:s2}\). Suppose \(\ref{l11.19:s1}\) holds. By Lemma~\ref{l10.14}, we have the equality
\[
(V(T_X), {\preccurlyeq}_{T_X(r_X)}) = (\BB_{X}, {\preccurlyeq}_X).
\]
In particular, from \ref{l11.19:s1} it follows that \(C\) is a maximal chain in \((\BB_{X}, {\preccurlyeq}_{X})\). We must prove that there is \(c \in X\) such that
\begin{equation}\label{l11.19:e2}
C = \{B \in \BB_{X} \colon B \ni c\}.
\end{equation}
First of all, we claim that \(\bigcap_{B \in C} B\) is nonempty,
\begin{equation}\label{l11.19:e3}
\bigcap_{B \in C} B \neq \varnothing.
\end{equation}

To prove that the set \(\bigcap_{B \in C} B\) is not empty, we will use the so-called \emph{finite intersection property}. Recall that a nonempty family \(\mathcal{F} = \{F_s\}_{s \in S}\) of sets has the finite intersection property if
\[
F_{s_1} \cap F_{s_2} \cap \ldots \cap F_{s_k} \neq \varnothing
\]
is valid for every finite \(\{s_1, s_2, \ldots, s_k\} \subseteq S\). It is known that a Hausdorff topological space \(Y\) is compact if and only if every family of closed subsets of \(Y\), which has this property, has nonempty intersection (\cite[Theorem~3.1.1]{Eng1989}). Using Proposition~\ref{p2.4} it is easy to see that
\[
\bigcap_{i=1}^{k} B_i \in \BB_{X}
\]
holds for every \(\{B_1, \ldots, B_k\} \subseteq C\) and \(k \in \mathbb{N}\). The elements of \(\BB_{X}\) are nonempty. Thus, \(C\) has the finite intersection property that implies \eqref{l11.19:e3}.

Let \(c\) be a point of \(\bigcap_{B \in C} B\). By Proposition~\ref{p2.5}, the set
\[
\{B \in \BB_{X} \colon B \ni c\}
\]
is a chain in \((\BB_{X}, {\preccurlyeq}_X)\). From \(c \in \bigcap_{B \in C} B\) follows the inclusion
\begin{equation}\label{l11.19:e4}
C \subseteq \{B \in \BB_{X} \colon B \ni c\}.
\end{equation}
Since \(C\) is a maximal chain in \((\BB_{X}, {\preccurlyeq}_X)\), inclusion~\eqref{l11.19:e4} yields~\eqref{l11.19:e2}.

\(\ref{l11.19:s2} \Rightarrow \ref{l11.19:s1}\). Let \(c\) be a point of \(X\) for which equality~\eqref{l11.19:e2} holds. Proposition~\ref{p2.5} and equality~\eqref{l11.19:e2} imply that \(C\) is a chain in \((\BB_{X}, {\preccurlyeq}_X)\).

Let \(C^{1} \in \BB_{X}\) be an arbitrary chain in \((\BB_{X}, {\preccurlyeq}_X)\) such that \(C^{1} \supseteq C\). To prove the maximality of \(C\) it suffices to show that
\begin{equation}\label{l11.19:e5}
c \in B^{1}
\end{equation}
holds for every \(B^{1} \in C^{1}\). Let \(B^{1}\) be an arbitrary element of \(C^{1}\). By Proposition~\ref{p2.2}, the ball \(B^{1}\) is a closed subset of \(X\). Consequently, \eqref{l11.19:e5} holds if \(c\) is an accumulation point of \(B^{1}\). Let \(\varepsilon\) be an arbitrary strictly positive real number. The ball \(B_{\varepsilon}(c) = \{x \in X \colon d(x, c) < \varepsilon\}\) belongs to \(C\). Since \(C^{1}\) is a chain and
\[
B^{1} \in C^{1} \supseteq C,
\]
we have
\[
B^{1} \preccurlyeq_X B_{\varepsilon}(c) \quad \text{or} \quad B_{\varepsilon}(c) \preccurlyeq_X B^{1}.
\]
Consequently, \(B^{1} \cap B_{\varepsilon}(c) \neq \varnothing\) holds for every \(\varepsilon > 0\). Thus, \(c\) is an accumulation point of \(B^{1}\).
\end{proof}

\begin{corollary}\label{c11.24}
Let \((X, d)\) be a nonempty totally bounded ultrametric space with the rooted representing tree \(T_X = T_X(r_X)\) and let \((\widetilde{X}, \widetilde{d})\) be a completion of \((X, d)\) such that \(X \subseteq \widetilde{X}\). Then the following statements are equivalent for every \(C \subseteq V(T_X)\):
\begin{enumerate}
\item \label{c11.24:s1} \(C\) is a maximal chain in \((V(T_X), {\preccurlyeq}_{T_X(r_X)})\).
\item \label{c11.24:s2} There is a point \(c \in \widetilde{X}\) such that \(C = \{\widetilde{B} \cap X \colon \widetilde{B} \in \BB_{\widetilde{X}} \text{ and } c \in \widetilde{B}\}\).
\end{enumerate}
\end{corollary}

\begin{proof}
By Lemma~\ref{l10.14}, we have \((V(T_X), {\preccurlyeq}_{T_X(r_X)}) = (\BB_{X}, {\preccurlyeq}_X)\). Proposition~\ref{p3.10} implies that the mapping
\[
\Phi \colon \BB_{\widetilde{X}} \to \BB_X, \quad \Phi(\widetilde{B}) = \widetilde{B} \cap X,
\]
is a correctly defined order isomorphism of \((\BB_{\widetilde{X}}, {\preccurlyeq}_{\widetilde{X}})\) and \((\BB_{X}, {\preccurlyeq}_X)\). In particular, a set \(C \subseteq V(T_X)\) is a maximal chain in \((V(T_X), {\preccurlyeq}_{T_X(r_X)})\) if and only if
\[
C = \{\Phi(\widetilde{B}) \colon \widetilde{B} \in \widetilde{C}\}
\]
holds for a maximal chain \(\widetilde{C}\) from \((\BB_{\widetilde{X}}, {\preccurlyeq}_{\widetilde{X}})\). Thus, the corollary under consideration is easily deduced from Lemma~\ref{l11.19}.
\end{proof}

Let \((X, d)\) be a nonempty totally bounded ultrametric space. Write \(\mathbf{MC}_{T_X}\) for the set of all maximal chains in \((V(T_X), {\preccurlyeq}_{T_X(r_X)})\). Let us define a function
\[
d_{\mathbf{MC}_{T_X}} \colon \mathbf{MC}_{T_X} \times \mathbf{MC}_{T_X} \to \RR^{+}
\]
by the rule
\begin{equation}\label{e11.34}
d_{\mathbf{MC}_{T_X}}(C_1, C_2) = \begin{cases}
0 & \text{if } C_1 = C_2,\\
l_X(\wedge (C_1 \cap C_2)) & \text{if } C_1 \neq C_2,
\end{cases}
\end{equation}
where \(\wedge (C_1 \cap C_2)\) is the meet of the set \(C_1 \cap C_2\) in the poset \((V(T_X), {\preccurlyeq}_{T_X(r_X)})\) and \(l_X \colon V(T_X) \to \RR^{+}\) is the labeling introduced by Definition~\ref{d9.27}.

\begin{remark}\label{r11.24}
For every pair \(C_1\), \(C_2\) of distinct maximal chains, the set \(C_1 \cap C_2\) is finite, that follows from Theorem~\ref{t9.20}, and nonempty, because \(X \in C_1 \cap C_2\). Hence, \(C_1 \cap C_2\) is a finite nonempty chain in \((V(T_X), {\preccurlyeq}_{T_X(r_X)})\) and, consequently, there is \(B^{*} \in C_1 \cap C_2\) such that \(B^{*} \preccurlyeq_{T_X(r_X)} B\) holds for every \(B \in C_1 \cap C_2\). It is clear that \(B^{*} = \wedge (C_1 \cap C_2)\). Thus, \(d_{\mathbf{MC}_{T_X}}\) is correctly defined.
\end{remark}

\begin{proposition}\label{p11.25}
Let \((X, d)\) be a nonempty totally bounded ultrametric space with a completion \((\widetilde{X}, \widetilde{d})\), \(\widetilde{X} \supseteq X\), let \(\mathbf{MC}_{T_X}\) be the set of all maximal chains in the poset \((V(T_X), {\preccurlyeq}_{T_X(r_X)})\) and let \(F \colon \mathbf{MC}_{T_X} \to \widetilde{X}\) be a mapping such that the equality \(F(C)  = c\) holds for every \(C \in \mathbf{MC}_{T_X}\) with \(c\) satisfying \eqref{l11.19:e1}. Then
\[
d_{\mathbf{MC}_{T_X}} \colon \mathbf{MC}_{T_X} \times \mathbf{MC}_{T_X} \to \RR^{+}
\]
is an ultrametric on \(\mathbf{MC}_{T_X}\) and \(F\) is an isometry of \((\mathbf{MC}_{T_X}, d_{\mathbf{MC}_{T_X}})\) and \((\widetilde{X}, \widetilde{d})\).
\end{proposition}

\begin{proof}
By Corollary~\ref{c11.24}, the mapping \(F\) is bijective. Hence, it suffices to prove the equality
\begin{equation}\label{p11.25:e1}
d_{\mathbf{MC}_{T_X}}(C_1, C_2) = \widetilde{d}(F(C_1), F(C_2))
\end{equation}
for all distinct \(C_1\), \(C_2 \in \mathbf{MC}_{T_X}\).

Let \(C_1\) and \(C_2\) be distinct maximal chains in \((V(T_X), {\preccurlyeq}_{T_X(r_X)})\) and let \(c_i\) be a point of \(\widetilde{X}\) such that \(c_i = F(C_i)\), \(i = 1\), \(2\). By Lemma~\ref{l10.14}, the equality \((\BB_{X}, {\preccurlyeq}_X) = (V(T_X), {\preccurlyeq}_{T_X(r_X)})\) holds. Hence, \(C_1\) and \(C_2\) are maximal chains in \((\BB_{X}, {\preccurlyeq}_X)\). From Proposition~\ref{p3.10} it follows that the mapping \(\Phi \colon \BB_{\widetilde{X}} \to \BB_{X}\) defined by
\[
\Phi (\widetilde{B}) = \widetilde{B} \cap X, \quad \widetilde{B} \in \BB_{\widetilde{X}},
\]
is an order isomorphism of \((\BB_{\widetilde{X}}, {\preccurlyeq}_{\widetilde{X}})\) and \((\BB_{X}, {\preccurlyeq}_X)\). Consequently, the sets \(\widetilde{C}_1 \subseteq \BB_{\widetilde{X}}\) and \(\widetilde{C}_2 \subseteq \BB_{\widetilde{X}}\), which satisfy
\[
(\widetilde{B} \in \widetilde{C}_i) \Leftrightarrow (\Phi(\widetilde{B}) \in \widetilde{C}_i), \quad i = 1, 2,
\]
for every \(\widetilde{B} \in \BB_{\widetilde{X}}\), are maximal chains in \((\BB_{\widetilde{X}}, {\preccurlyeq}_{\widetilde{X}})\). Now Proposition~\ref{p3.10} and Definition~\ref{d9.27} imply the equalities
\begin{equation}\label{p11.25:e2}
l_X(\wedge (C_1 \cap C_2)) = \diam (\wedge (C_1 \cap C_2)) = \diam (\wedge (\widetilde{C}_1 \cap \widetilde{C}_2)).
\end{equation}
Using~\eqref{e11.34} and \eqref{p11.25:e2} we see that \eqref{p11.25:e1} holds if and only if
\begin{equation}\label{p11.25:e3}
\widetilde{d}(c_1, c_2) = \diam (\wedge (\widetilde{C}_1 \cap \widetilde{C}_2)).
\end{equation}
By Corollary~\ref{c11.24}, we have
\[
\widetilde{C}_1 = \{\widetilde{B} \in \BB_{\widetilde{X}} \colon c_1 \in \widetilde{B}\} \quad \text{and} \quad \widetilde{C}_2 = \{\widetilde{B} \in \BB_{\widetilde{X}}\colon c_2 \in \widetilde{B}\}.
\]
Let \(B^{*} = B^{*}(\{c_1, c_2\})\) be the smallest ball containing the set \(\{c_1, c_2\}\). From Definition~\ref{d2.9} it follows that
\begin{equation}\label{p11.25:e4}
B^{*} \in \overline{\BB}_{\widetilde{X}}.
\end{equation}
Since \(F \colon \mathbf{MC}_{T_X} \to \widetilde{X}\) is bijective and \(C_1 \neq C_2\), the inequality \(\widetilde{d}(c_1, c_2) > 0\) holds. Lemma~\ref{c6.6}, the inequality \(\widetilde{d}(c_1, c_2) > 0\) and \eqref{p11.25:e4} imply that \(B^{*} \in \BB_{\widetilde{X}}\). Hence, we have \(B^{*} \in C_1 \cap C_2\) and, consequently, the inequality
\[
\wedge (\widetilde{C}_1 \cap \widetilde{C}_2) \preccurlyeq_{\widetilde{X}} \widetilde{B}^{*}
\]
is valid. The converse inequality follows from the inclusion \(\overline{\BB}_{\widetilde{X}} \supseteq \BB_{\widetilde{X}}\) (see Corollary~\ref{c2.41}) and Definition~\ref{d2.9}. Thus, we have the equality
\[
\wedge (\widetilde{C}_1 \cap \widetilde{C}_2) = \widetilde{B}^{*}(\{c_1, c_2\}).
\]
The last equality and Corollary~\ref{c2.18} imply \eqref{p11.25:e3}.
\end{proof}

Theorem~\ref{t11.12} and Proposition~\ref{p11.25} give us the following.

\begin{corollary}\label{c11.26}
The labeled representing trees \(T_X(l_X)\) and \(T_{\mathbf{MC}_{T_X}}(l_{\mathbf{MC}_{T_X}})\) are isomorphic for every nonempty totally bounded ultrametric space \((X, d)\).
\end{corollary}

\subsection{Labeled representing trees of \(p\)-adic balls and spheres}

In this part of the section we will consider several important examples of labeled representing trees. The given below results and pictures are partially motivated by papers \cite{HolAMM2001} and \cite{CuoAMM1991}.

Let \(p \geqslant 2\) be a prime number and let \(\mathbb{Q}_p\) be the corresponding field of all \(p\)-adic numbers. The open ball
\[
B_{r}(a) = \{x \in \mathbb{Q}_p \colon d_p(x, a) < r\}, \quad a \in \mathbb{Q}_p, \quad r > 0,
\]
is a compact subspace of the ultrametric space \((\mathbb{Q}_p, d_p)\), where \(d_p \colon \mathbb{Q}_p \times \mathbb{Q}_p \to \RR^{+}\) is the \(p\)-adic metric defined as an extension of the ultrametric
\[
\<x, y> \mapsto |x-y|_p, \quad x, y \in \mathbb{Q}
\]
from the field \(\mathbb{Q}\) of rational numbers to its completion \(\mathbb{Q}_p\) (we keep the symbol \(d_p\) for \(p\)-adic metric, when passing from \(\mathbb{Q}\) to \(\mathbb{Q}_p\), see~\eqref{e3.2} and \eqref{e3.3}). Consequently, we can consider the representing tree \(T_{B_{r}(a)}\) of the space \((B_{r}(a), d_p|_{B_{r}(a) \times B_{r}(a)})\) as a special example of representing trees introduced by Definition~\ref{d9.1} for all totally bounded ultrametric spaces.

From Lemma~\ref{l8.2} and Lemma~\ref{l8.3} it follows that
\[
V(T_{B_{r}(a)}) = \{B \in \mathbf{B}_{\mathbb{Q}_p} \colon B \subseteq {B}_{r}(a)\}
\]
and, moreover, that the membership \(\{B_1, B_2\} \in E(B_{r}(a))\) is valid if and only if we have either
\[
(B_1 \subset B_2) \text{ and } ((B_1 \subseteq B \subseteq B_2) \Rightarrow (B_1 = B \text{ or } B_2 = B))
\]
or
\[
(B_2 \subset B_1) \text{ and } ((B_2 \subseteq B \subseteq B_1) \Rightarrow (B_1 = B \text{ or } B_2 = B))
\]
for every \(B \in V(T_{B_{r}(a)})\). By Definition~\ref{d9.27}, the labeling
\[
l_{B_{r}(a)} \colon V(T_{B_{r}(a)}) \times V(T_{B_{r}(a)}) \to \RR^{+}
\]
satisfies
\[
l_{B_{r}(a)}(B) = \diam B = \sup\{d_p(x, y) \colon x, y \in B\}
\]
for every \(B \in V(T_{B_{r}(a)})\).

\begin{proposition}\label{p11.28}
Let \(p \geqslant 2\) be a prime number, \(a \in \mathbb{Q}_p\) and let \(r\) be a strictly positive real number. Then every \(B \in V(T_{B_r(a)})\) has exactly \(p\) direct successors \(B_1\), \(\ldots\), \(B_p\) which are labeled such that
\begin{equation}\label{p11.28:e0}
l_{{B}_{r}(a)}(B_1) = \ldots = l_{{B}_{r}(a)}(B_p) = p^{- 1}l_{{B}_{r}(a)}(B).
\end{equation}
\end{proposition}

\begin{proof}
Let us consider first the case when \(B \in V(T_{B_r(a)})\) coincides with \(B_r(a)\). The function
\[
\mathbb{Q}_p \ni x \mapsto x - a \in \mathbb{Q}_p
\]
is an isometry of \((\mathbb{Q}_p, d_p)\). (This statement is well-known in the theory of \(p\)-adic numbers or can be derived directly from formula~\eqref{e3.3}.) Consequently, the open ball
\[
B_r(0) = \{x - a \colon x \in B_r(a)\}
\]
is isometric to \(B_r(a)\). By Proposition~\ref{p11.10}, the isometricity of \(B_r(a)\) and \(B_r(0)\) implies
\[
T_{B_{r}(a)}(l_{B_{r}(a)}) \simeq T_{B_{r}(0)}(l_{B_{r}(0)}).
\]
Thus, without loss of generality, we can suppose \(a=0\).

Every \(x \in \mathbb{Q}_p\) can be uniquely presented as a convergent (in \((\mathbb{Q}_p, d_p)\)) series,
\begin{equation}\label{p11.28:e4}
x = \sum_{i \in \mathbb{Z}} d_i(x) p^i,
\end{equation}
where, for all \(i \in \mathbb{Z}\), we have \(d_{i}(x) \in \{0, 1, \ldots, p-1\}\) and \(d_i = 0\) holds whenever \(i \leqslant i(x)\) with some \(i(x) \in \mathbb{Z}\). The distance set \(D(\mathbb{Q}_p)\) of the ultrametric space \((\mathbb{Q}_p, d_p)\) has the form
\begin{equation}\label{p11.28:e1}
D(\mathbb{Q}_p) = \{p^{j} \colon j \in \mathbb{Z}\} \cup \{0\}.
\end{equation}
Since \((B_{r}(0), d_p|_{B_{r}(0) \times B_{r}(0)})\) is a compact metric space and \(|B_r(0)| \geqslant 2\) holds, by Proposition~\ref{p2.36}, there is \(x_1 \in B_{r}(0)\) such that
\begin{equation}\label{p11.28:e2}
\diam B_{r}(0) = d_p(x_1, 0) > 0.
\end{equation}
Using~\eqref{p11.28:e1} we can find \(j_0 = j_0(x)\) such that
\begin{equation}\label{p11.28:e3}
d_p(x_1, 0) = p^{j_0}.
\end{equation}
From \eqref{p11.28:e2}, \eqref{p11.28:e3} and Proposition~\ref{p2.36} it follows that
\[
B_r(0) = \overline{B}_{p^{j_0}}(0),
\]
where \(\overline{B}_{p^{j_0}}(0)\) is the closed ball with the center \(0\) and radius \(p^{j_0}\),
\begin{equation}\label{p11.28:e5}
\overline{B}_{p^{j_0}}(0) = \{x \in \mathbb{Q}_p \colon d_p(x, 0) \leqslant p^{j_0}\}.
\end{equation}

Let \(x \in \mathbb{Q}_p \setminus \{0\}\) be given by series \eqref{p11.28:e4} and let \(i_0 = i_0(x)\) be defined as
\[
i_0 = \inf\{i \in \mathbb{Z} \colon d_i(x) \neq 0\}.
\]
Then we have the equality
\begin{equation}\label{p11.28:e6}
d_p(x, 0) = p^{-i_0(x)}
\end{equation}
(see, for example, Theorem~2.1 from Chapter~II in \cite{Bachman1964}). Hence, \(x \in \mathbb{Q}_p\) is a point of \(\overline{B}_{p^{j_0}}(0)\) if and only if \(x\) has the expansion
\begin{equation}\label{p11.28:e7}
x = \sum_{i \geqslant -j_0} d_i(x) p^i.
\end{equation}
In particular, the points \(x_0 = 0 p^{-j_0}\), \(x_1 = 1 p^{-j_0}\), \(\ldots\), \(x_{p-1} = (p-1) p^{-j_0}\) belong to \(\overline{B}_{p^{j_0}}(0)\) and the equalities
\begin{equation}\label{p11.28:e8}
p^{j_0} = d_p(x_{k_1}, x_{k_2}) = \diam \overline{B}_{p^{j_0}}(0) = \diam B_r(0) = \diam B_r(a)
\end{equation}
hold for every pair of distinct \(k_1\), \(k_2 \in \{0, 1, \ldots, p-1\}\). Using~\eqref{p11.28:e6}--\eqref{p11.28:e8} we obtain
\[
\overline{B}_{p^{j_0}}(0) = \bigcup_{k=0}^{p-1} \overline{B}_{p^{j_0-1}}(x_k)
\]
and
\[
\overline{B}_{p^{j_0-1}}(x_{k_1}) \cap \overline{B}_{p^{j_0-1}}(x_{k_2}) = \varnothing
\]
for every pair of distinct \(k_1\), \(k_2 \in \{0, 1, \ldots, p-1\}\). From \eqref{p11.28:e1} it follows that the equalities
\begin{equation}\label{p11.28:e9}
\diam \overline{B}_{p^{j_0-1}}(x_{k}) = p^{j_0-1}, \quad \overline{B}_{p^{j_0-1}}(x_{k}) = B_r(x_k)
\end{equation}
hold for every \(r \in (p^{j_0-1}, p^{j_0}]\) and \(k \in \{0, 1, \ldots, p-1\}\). Hence, by Theorem~\ref{t2.24}, the balls \(\overline{B}_{p^{j_0-1}}(x_{0})\), \(\ldots\), \(\overline{B}_{p^{j_0-1}}(x_{p-1})\) are the parts of the complete \(p\)-partite diametrical graph \(G_{B_r(0), d_p|_{B_r(0) \times B_r(0)}}\).

It should be also noted that the balls \(\overline{B}_{p^{j_0-1}}(x_{0})\), \(\ldots\), \(\overline{B}_{p^{j_0-1}}(x_{p-1})\) are the direct successors of \(B_r(0)\) (see Lemma~\ref{l10.10} and Remark~\ref{r10.11}). Now, using \eqref{p11.28:e2}, \eqref{p11.28:e3}, \eqref{p11.28:e9} and Definition~\ref{d9.27}, we obtain \eqref{p11.28:e0} with \(B = B_r(0)\) and \(B_k = \overline{B}_{p^{j_0-1}}(x_{k})\), \(k = 0\), \(1\), \(\ldots\), \(p-1\).

Let us consider now an arbitrary \(B \in V(T_{B_r(a)})\). By Proposition~\ref{p8.4}, the equality
\[
V(T_{B_r(a)}) = \BB_{B_r(a)}
\]
holds. By Corollary~\ref{c2.4}, we have the inclusion
\[
\BB_{B^*} \subseteq \BB_{\mathbb{Q}_p}
\]
for every open ball \(B^* \in \BB_{\mathbb{Q}_p}\). Hence, every \(B \in V(T_{B_r(a)})\) is an open ball in \((\mathbb{Q}_p, d_p)\). Therefore, arguing as above, we can verify that every \(B \in V(T_{B_r(a)})\) has exactly \(p\) direct successors \(B_1\), \(\ldots\), \(B_p\) which satisfy \eqref{p11.28:e0}.
\end{proof}

\begin{example}\label{ex11.4}
If \(r \in (1,2]\), \(a = 0\) and \(p = 2\), then \(T_{B_r(a)} = T_{B_r(a)}(l_{{B}_{r}(a)})\) is the labeled tree depicted in Figure~\ref{fig5}.
\end{example}

\begin{figure}[ht]
\begin{center}
\begin{tikzpicture}[
level 1/.style={level distance=\levdist,sibling distance=24mm},
level 2/.style={level distance=\levdist,sibling distance=12mm},
level 3/.style={level distance=\levdist,sibling distance=6mm},
solid node/.style={circle,draw,inner sep=1.5,fill=black},
hollow node/.style={circle,draw,inner sep=1.5}]

\node [hollow node] at (0,0) [label={[label distance=1cm] right:\(T_{B_r(0)}\), \(p=2\), \(r \in (1,2]\)}] {\(1\)}
child{node [hollow node]{\(\frac{1}{2}\)}
	child{node [hollow node]{\(\frac{1}{4}\)}
		child{node [label=below:{\(\ldots\)}]{}}
		child{node [label=below:{\(\ldots\)}]{}}
	}
	child{node [hollow node]{\(\frac{1}{4}\)}
		child{node [label=below:{\(\ldots\)}]{}}
		child{node [label=below:{\(\ldots\)}]{}}
	}
}
child{node [hollow node]{\(\frac{1}{2}\)}
	child{node [hollow node]{\(\frac{1}{4}\)}
		child{node [label=below:{\(\ldots\)}]{}}
		child{node [label=below:{\(\ldots\)}]{}}
	}
	child{node [hollow node]{\(\frac{1}{4}\)}
		child{node [label=below:{\(\ldots\)}]{}}
		child{node [label=below:{\(\ldots\)}]{}}
	}
};
\end{tikzpicture}
\end{center}
\caption{}
\label{fig5}
\end{figure}

\begin{example}\label{ex11.5}
If \(r \in (1,3]\), \(a = 0\) and \(p = 3\), then \(T_{{B}_{r}(a)}\) is depicted in Figure~\ref{fig6}.
\end{example}

\begin{figure}[ht]
\begin{center}
\begin{tikzpicture}[
level 1/.style={level distance=\levdist,sibling distance=2.7cm},
level 2/.style={level distance=\levdist,sibling distance=0.9cm},
level 3/.style={level distance=\levdist,sibling distance=0.3cm},
solid node/.style={circle,draw,inner sep=1.5,fill=black},
hollow node/.style={circle,draw,inner sep=1.5}]

\node [hollow node] at (0,0)
[label={[label distance=1cm] right:\(T_{B_r(0)}\), \(p=3\), \(r \in (1,3]\)}] {\(1\)}
child{node [hollow node]{\(\frac{1}{3}\)}
	child{node [hollow node]{\(\frac{1}{9}\)}
		child{node [label=below:{\(.\)}]{}}
		child{node [label=below:{\(.\)}]{}}
		child{node [label=below:{\(.\)}]{}}
	}
	child{node [hollow node]{\(\frac{1}{9}\)}
		child{node [label=below:{\(.\)}]{}}
		child{node [label=below:{\(.\)}]{}}
		child{node [label=below:{\(.\)}]{}}
	}
	child{node [hollow node]{\(\frac{1}{9}\)}
		child{node [label=below:{\(.\)}]{}}
		child{node [label=below:{\(.\)}]{}}
		child{node [label=below:{\(.\)}]{}}
	}
}
child{node [hollow node]{\(\frac{1}{3}\)}
	child{node [hollow node]{\(\frac{1}{9}\)}
		child{node [label=below:{\(.\)}]{}}
		child{node [label=below:{\(.\)}]{}}
		child{node [label=below:{\(.\)}]{}}
	}
	child{node [hollow node]{\(\frac{1}{9}\)}
		child{node [label=below:{\(.\)}]{}}
		child{node [label=below:{\(.\)}]{}}
		child{node [label=below:{\(.\)}]{}}
	}
	child{node [hollow node]{\(\frac{1}{9}\)}
		child{node [label=below:{\(.\)}]{}}
		child{node [label=below:{\(.\)}]{}}
		child{node [label=below:{\(.\)}]{}}
	}
}
child{node [hollow node]{\(\frac{1}{3}\)}
	child{node [hollow node]{\(\frac{1}{9}\)}
		child{node [label=below:{\(.\)}]{}}
		child{node [label=below:{\(.\)}]{}}
		child{node [label=below:{\(.\)}]{}}
	}
	child{node [hollow node]{\(\frac{1}{9}\)}
		child{node [label=below:{\(.\)}]{}}
		child{node [label=below:{\(.\)}]{}}
		child{node [label=below:{\(.\)}]{}}
	}
	child{node [hollow node]{\(\frac{1}{9}\)}
		child{node [label=below:{\(.\)}]{}}
		child{node [label=below:{\(.\)}]{}}
		child{node [label=below:{\(.\)}]{}}
	}
};
\end{tikzpicture}
\end{center}
\caption{}
\label{fig6}
\end{figure}

Following M.~Ostill \cite{OstPASMaiA2012}, we will say that a rooted tree \(T\) is a \emph{Bethe Lattice} if there is \(q \geqslant 2\) such that
\begin{equation}\label{r11.30:e1}
\delta_{T}(v) = q
\end{equation}
for every \(v \in V(T)\). As in \cite{OstPASMaiA2012}, the term \emph{irregular Bethe Lattice of degree} \(q\) is used for the infinite rooted trees \(T = T(r)\) satisfying equality~\eqref{r11.30:e1} for every \(v \in V(T) \setminus \{r\}\) with the root \(r\) having degree \(q-1\). The concept of Bethe Lattice was firstly considered by Hans A. Bethe in~\cite{BetPRSLA1935} and it is an important mathematical tool for statistical mechanical models \cite{Sinai1982}.

Using the concept of Bethe Lattice, we can derive from Proposition~\ref{p11.28} the following.

\begin{corollary}\label{c11.31}
For any prime \(p \geqslant 2\), and any \(r > 0\), and each \(a \in \mathbb{Q}_p\), the rooted representing tree \(T_{B_r(a)}(r_{B_r(a)})\) is an irregular Bethe Lattice of degree \(p+1\).
\end{corollary}

The central result of the present part of Section~\ref{sec11} shows that arbitrary Bethe Lattice of degree \(n \geqslant 3\) is isomorphic to the rooted representing tree of open ball from a finite ultrametric product \(\mathbb{Q}_{p_1} \times \ldots \times \mathbb{Q}_{p_n}\) with some suitable prime numbers \(p_1\), \(\ldots\), \(p_n\).

The sphere
\[
S_{r}(a) = \{x \in \mathbb{Q}_p \colon |x - a|_p = r\},
\]
where \(r = p^{\gamma}\), \(\gamma \in \mathbb{Z}\), \(a \in \mathbb{Q}_p\), is also a compact ultrametric subspace of \((\mathbb{Q}_p, d_p)\).

If \(p \geqslant 3\), then the tree \(T_{S_{r}(a)}\) with \(r = p^{\gamma}\) can be obtained from the tree \(T_{B_{r_1}(a)}\) with \(r_1 \in (p^{\gamma}, p^{\gamma+1}]\) by deleting the direct successor of \(B_{r_1}(a)\) which contains the point \(a\).

\begin{example}\label{ex11.6}
Let \(p = 2\), \(a = 0\), \(r = 2^{0} = 1\) and \(r_1 \in (1/2, 1]\). Then the labeled representing tree \(T_{S_{r}(a)}\) is isomorphic to the labeled representing tree \(T_{{B}_{r_1}(a)}\) and, in addition, the sphere \(S_{r}(a)\) is isometric to the ball \({B}_{r_1}(a)\) (see Figure~\ref{fig7}).
\end{example}

\begin{figure}[ht]
\begin{center}
\begin{tikzpicture}[
level 1/.style={level distance=\levdist,sibling distance=4cm},
level 2/.style={level distance=\levdist,sibling distance=2cm},
level 3/.style={level distance=\levdist,sibling distance=1cm},
solid node/.style={circle,draw,inner sep=1.5,fill=black},
hollow node/.style={circle,draw,inner sep=1.5}]

\node [hollow node, label={[label distance=1cm] right:\(T_{S_{1}(0)}\), \(p=2\)}]
at (0,0) {\(\frac{1}{2}\)}
child{node [hollow node]{\(\frac{1}{4}\)}
	child{node [hollow node]{\(\frac{1}{8}\)}
		child{node [label=below:{\(\ldots\)}]{}}
		child{node [label=below:{\(\ldots\)}]{}}
	}
	child{node [hollow node]{\(\frac{1}{8}\)}
		child{node [label=below:{\(\ldots\)}]{}}
		child{node [label=below:{\(\ldots\)}]{}}
	}
}
child{node [hollow node]{\(\frac{1}{4}\)}
	child{node [hollow node]{\(\frac{1}{8}\)}
		child{node [label=below:{\(\ldots\)}]{}}
		child{node [label=below:{\(\ldots\)}]{}}
	}
	child{node [hollow node]{\(\frac{1}{8}\)}
		child{node [label=below:{\(\ldots\)}]{}}
		child{node [label=below:{\(\ldots\)}]{}}
	}
};
\end{tikzpicture}
\end{center}
\caption{}
\label{fig7}
\end{figure}

\begin{example}\label{ex11.7}
Let \(p = 3\). The representing tree of the sphere \(S_{1}(0) \subseteq (\mathbb{Q}_3, d_3)\) is depicted by Figure~\ref{fig8}.
\end{example}

\begin{figure}[ht]
\begin{center}
\begin{tikzpicture}[
level 1/.style={level distance=\levdist,sibling distance=5.4cm},
level 2/.style={level distance=\levdist,sibling distance=1.8cm},
level 3/.style={level distance=\levdist,sibling distance=6mm},
solid node/.style={circle,draw,inner sep=1.5,fill=black},
hollow node/.style={circle,draw,inner sep=1.5}]

\node [hollow node, label={[label distance=10pt] right:{\(T_{S_{1}(0)}\), \(p=3\)}}] at (0,0) {\(1\)}
child{node [hollow node]{\(\frac{1}{3}\)}
	child{node [hollow node]{\(\frac{1}{9}\)}
		child{node [label=below:{\(\ldots\)}]{}}
		child{node [label=below:{\(\ldots\)}]{}}
		child{node [label=below:{\(\ldots\)}]{}}
	}
	child{node [hollow node]{\(\frac{1}{9}\)}
		child{node [label=below:{\(\ldots\)}]{}}
		child{node [label=below:{\(\ldots\)}]{}}
		child{node [label=below:{\(\ldots\)}]{}}
	}
	child{node [hollow node]{\(\frac{1}{9}\)}
		child{node [label=below:{\(\ldots\)}]{}}
		child{node [label=below:{\(\ldots\)}]{}}
		child{node [label=below:{\(\ldots\)}]{}}
	}
}
child{node [hollow node]{\(\frac{1}{3}\)}
	child{node [hollow node]{\(\frac{1}{9}\)}
		child{node [label=below:{\(\ldots\)}]{}}
		child{node [label=below:{\(\ldots\)}]{}}
		child{node [label=below:{\(\ldots\)}]{}}
	}
	child{node [hollow node]{\(\frac{1}{9}\)}
		child{node [label=below:{\(\ldots\)}]{}}
		child{node [label=below:{\(\ldots\)}]{}}
		child{node [label=below:{\(\ldots\)}]{}}
	}
	child{node [hollow node]{\(\frac{1}{9}\)}
		child{node [label=below:{\(\ldots\)}]{}}
		child{node [label=below:{\(\ldots\)}]{}}
		child{node [label=below:{\(\ldots\)}]{}}
	}
};
\end{tikzpicture}
\end{center}
\caption{}
\label{fig8}
\end{figure}

In the space \(\mathbb{Q}_p^{n} = \mathbb{Q}_p \times \ldots \times \mathbb{Q}_p\) consisting points \(\bfx = (x_1, \ldots, x_n)\), \(x_k \in \mathbb{Q}_p\), \(k = 1\), \(\ldots\), \(n\), with metric
\[
d_p^n(\bfx, \bfy) = \max_{1 \leqslant k \leqslant n} |x_k - y_k|_p,
\]
for every ball
\[
B_{r}^{n}(\bfa) = \{\bfx \in \mathbb{Q}_p^{n} \colon d_p^n(\bfx, \bfa) < r\},
\]
every node of the representing tree \(T_{B_{r}^{n}(\bfa)}\) has exactly \(p^{n}\) direct successors with the same labels.

\begin{example}\label{ex8.7}
Let \(p = 2\), \(n = 2\), \(\mathbf{a} = (0, 0)\) and \(r \in (1,2)\). For this case the representing tree of \(B_{r}^{2}(\mathbf{a})\) is depicted by Figure~\ref{fig12}.
\end{example}

\begin{figure}[ht]
\begin{center}
\begin{tikzpicture}[
level 1/.style={level distance=1.5cm,sibling distance=3.6cm},
level 2/.style={level distance=1.5cm,sibling distance=0.9cm},
level 3/.style={level distance=1.5cm,sibling distance=0.2cm},
solid node/.style={circle,draw,inner sep=1.5,fill=black},
hollow node/.style={circle,draw,inner sep=1.5}]

\node [hollow node,
%label={[label distance=10pt] right:{\(B_{1}^{2}(\mathbf{0})\)}}
] at (0,0) {\(1\)}
child{node [hollow node]{\(\frac{1}{2}\)}
	child{node [hollow node]{\(\frac{1}{4}\)}
		child{node [label=below:{\(.\)}]{}}
		child{node [label=below:{\(.\)}]{}}
		child{node [label=below:{\(.\)}]{}}
		child{node [label=below:{\(.\)}]{}}
	}
	child{node [hollow node]{\(\frac{1}{4}\)}
		child{node [label=below:{\(.\)}]{}}
		child{node [label=below:{\(.\)}]{}}
		child{node [label=below:{\(.\)}]{}}
		child{node [label=below:{\(.\)}]{}}
	}
	child{node [hollow node]{\(\frac{1}{4}\)}
		child{node [label=below:{\(.\)}]{}}
		child{node [label=below:{\(.\)}]{}}
		child{node [label=below:{\(.\)}]{}}
		child{node [label=below:{\(.\)}]{}}
	}
	child{node [hollow node]{\(\frac{1}{4}\)}
		child{node [label=below:{\(.\)}]{}}
		child{node [label=below:{\(.\)}]{}}
		child{node [label=below:{\(.\)}]{}}
		child{node [label=below:{\(.\)}]{}}
	}
}
child{node [hollow node]{\(\frac{1}{2}\)}
	child{node [hollow node]{\(\frac{1}{4}\)}
		child{node [label=below:{\(.\)}]{}}
		child{node [label=below:{\(.\)}]{}}
		child{node [label=below:{\(.\)}]{}}
		child{node [label=below:{\(.\)}]{}}
	}
	child{node [hollow node]{\(\frac{1}{4}\)}
		child{node [label=below:{\(.\)}]{}}
		child{node [label=below:{\(.\)}]{}}
		child{node [label=below:{\(.\)}]{}}
		child{node [label=below:{\(.\)}]{}}
	}
	child{node [hollow node]{\(\frac{1}{4}\)}
		child{node [label=below:{\(.\)}]{}}
		child{node [label=below:{\(.\)}]{}}
		child{node [label=below:{\(.\)}]{}}
		child{node [label=below:{\(.\)}]{}}
	}
	child{node [hollow node]{\(\frac{1}{4}\)}
		child{node [label=below:{\(.\)}]{}}
		child{node [label=below:{\(.\)}]{}}
		child{node [label=below:{\(.\)}]{}}
		child{node [label=below:{\(.\)}]{}}
	}
}
child{node [hollow node]{\(\frac{1}{2}\)}
	child{node [hollow node]{\(\frac{1}{4}\)}
		child{node [label=below:{\(.\)}]{}}
		child{node [label=below:{\(.\)}]{}}
		child{node [label=below:{\(.\)}]{}}
		child{node [label=below:{\(.\)}]{}}
	}
	child{node [hollow node]{\(\frac{1}{4}\)}
		child{node [label=below:{\(.\)}]{}}
		child{node [label=below:{\(.\)}]{}}
		child{node [label=below:{\(.\)}]{}}
		child{node [label=below:{\(.\)}]{}}
	}
	child{node [hollow node]{\(\frac{1}{4}\)}
		child{node [label=below:{\(.\)}]{}}
		child{node [label=below:{\(.\)}]{}}
		child{node [label=below:{\(.\)}]{}}
		child{node [label=below:{\(.\)}]{}}
	}
	child{node [hollow node]{\(\frac{1}{4}\)}
		child{node [label=below:{\(.\)}]{}}
		child{node [label=below:{\(.\)}]{}}
		child{node [label=below:{\(.\)}]{}}
		child{node [label=below:{\(.\)}]{}}
	}
}
child{node [hollow node]{\(\frac{1}{2}\)}
	child{node [hollow node]{\(\frac{1}{4}\)}
		child{node [label=below:{\(.\)}]{}}
		child{node [label=below:{\(.\)}]{}}
		child{node [label=below:{\(.\)}]{}}
		child{node [label=below:{\(.\)}]{}}
	}
	child{node [hollow node]{\(\frac{1}{4}\)}
		child{node [label=below:{\(.\)}]{}}
		child{node [label=below:{\(.\)}]{}}
		child{node [label=below:{\(.\)}]{}}
		child{node [label=below:{\(.\)}]{}}
	}
	child{node [hollow node]{\(\frac{1}{4}\)}
		child{node [label=below:{\(.\)}]{}}
		child{node [label=below:{\(.\)}]{}}
		child{node [label=below:{\(.\)}]{}}
		child{node [label=below:{\(.\)}]{}}
	}
	child{node [hollow node]{\(\frac{1}{4}\)}
		child{node [label=below:{\(.\)}]{}}
		child{node [label=below:{\(.\)}]{}}
		child{node [label=below:{\(.\)}]{}}
		child{node [label=below:{\(.\)}]{}}
	}
};
\end{tikzpicture}
\end{center}
\caption{}
\label{fig12}
\end{figure}

Similarly we obtain the representing tree \(T_{S_{r}^{n}(\bfa)}\) with \(r = p^{\gamma}\), \(\gamma \in \mathbb{Z}\),
\[
S_{r}^{n}(\bfa) = \{\bfx \in \mathbb{Q}_p^{n} \colon d_p^n(\bfx, \bfa) = p^{\gamma}\}.
\]
To construct this tree it suffices to delete one direct successor of the root of \(T_{B_{r^*}^{n}(\bfa)}\) for \(r^* \in (p^{\gamma}, p^{\gamma+1})\) (see Figure~\ref{fig13}).

\begin{figure}[ht]
\begin{center}
\begin{tikzpicture}[
level 1/.style={level distance=1.5cm,sibling distance=4cm},
level 2/.style={level distance=1.5cm,sibling distance=1cm},
level 3/.style={level distance=1.5cm,sibling distance=0.2cm},
solid node/.style={circle,draw,inner sep=1.5,fill=black},
hollow node/.style={circle,draw,inner sep=1.5}]

\node [hollow node, label={[label distance=10pt] right:{\(S_{r}^{2}(\mathbf{0})\), \(p =2\), \(n = 2\), \(r = 1\)}}] at (0,0) {\(1\)}
child{node [hollow node]{\(\frac{1}{2}\)}
	child{node [hollow node]{\(\frac{1}{4}\)}
		child{node [label=below:{\(.\)}]{}}
		child{node [label=below:{\(.\)}]{}}
		child{node [label=below:{\(.\)}]{}}
		child{node [label=below:{\(.\)}]{}}
	}
	child{node [hollow node]{\(\frac{1}{4}\)}
		child{node [label=below:{\(.\)}]{}}
		child{node [label=below:{\(.\)}]{}}
		child{node [label=below:{\(.\)}]{}}
		child{node [label=below:{\(.\)}]{}}
	}
	child{node [hollow node]{\(\frac{1}{4}\)}
		child{node [label=below:{\(.\)}]{}}
		child{node [label=below:{\(.\)}]{}}
		child{node [label=below:{\(.\)}]{}}
		child{node [label=below:{\(.\)}]{}}
	}
	child{node [hollow node]{\(\frac{1}{4}\)}
		child{node [label=below:{\(.\)}]{}}
		child{node [label=below:{\(.\)}]{}}
		child{node [label=below:{\(.\)}]{}}
		child{node [label=below:{\(.\)}]{}}
	}
}
child{node [hollow node]{\(\frac{1}{2}\)}
	child{node [hollow node]{\(\frac{1}{4}\)}
		child{node [label=below:{\(.\)}]{}}
		child{node [label=below:{\(.\)}]{}}
		child{node [label=below:{\(.\)}]{}}
		child{node [label=below:{\(.\)}]{}}
	}
	child{node [hollow node]{\(\frac{1}{4}\)}
		child{node [label=below:{\(.\)}]{}}
		child{node [label=below:{\(.\)}]{}}
		child{node [label=below:{\(.\)}]{}}
		child{node [label=below:{\(.\)}]{}}
	}
	child{node [hollow node]{\(\frac{1}{4}\)}
		child{node [label=below:{\(.\)}]{}}
		child{node [label=below:{\(.\)}]{}}
		child{node [label=below:{\(.\)}]{}}
		child{node [label=below:{\(.\)}]{}}
	}
	child{node [hollow node]{\(\frac{1}{4}\)}
		child{node [label=below:{\(.\)}]{}}
		child{node [label=below:{\(.\)}]{}}
		child{node [label=below:{\(.\)}]{}}
		child{node [label=below:{\(.\)}]{}}
	}
}
child{node [hollow node]{\(\frac{1}{2}\)}
	child{node [hollow node]{\(\frac{1}{4}\)}
		child{node [label=below:{\(.\)}]{}}
		child{node [label=below:{\(.\)}]{}}
		child{node [label=below:{\(.\)}]{}}
		child{node [label=below:{\(.\)}]{}}
	}
	child{node [hollow node]{\(\frac{1}{4}\)}
		child{node [label=below:{\(.\)}]{}}
		child{node [label=below:{\(.\)}]{}}
		child{node [label=below:{\(.\)}]{}}
		child{node [label=below:{\(.\)}]{}}
	}
	child{node [hollow node]{\(\frac{1}{4}\)}
		child{node [label=below:{\(.\)}]{}}
		child{node [label=below:{\(.\)}]{}}
		child{node [label=below:{\(.\)}]{}}
		child{node [label=below:{\(.\)}]{}}
	}
	child{node [hollow node]{\(\frac{1}{4}\)}
		child{node [label=below:{\(.\)}]{}}
		child{node [label=below:{\(.\)}]{}}
		child{node [label=below:{\(.\)}]{}}
		child{node [label=below:{\(.\)}]{}}
	}
};
\end{tikzpicture}
\end{center}
\caption{}
\label{fig13}
\end{figure}

\section{From monotone labelings to ultrametric spaces and back}

The main result of this section, Theorem~\ref{t7.3}, gives a characterization of the labeled rooted trees, which are isomorphic to labeled representing trees of totally bounded ultrametric spaces. This result allows us to completely describe the corresponding free trees and rooted ones in Theorems~\ref{t9.10} and \ref{t10.16}, respectively. The characteristic properties of posets, which are order isomorphic to balleans of totally bounded ultrametric spaces, are presented in Theorem~\ref{t10.17}.

The next concept is a generalization of the notion of monotone rooted trees which was proposed by V.~Gurvich and M.~Vyalyi in~\cite{GV2012DAM}.

\begin{definition}\label{d11.27}
Let \(T = T(r)\) be a rooted tree and let \(l \colon V(T) \to \RR^{+}\) be a labeling satisfying the conditions:
\begin{enumerate}
\item \label{d11.27:s1} The inequality
\begin{equation}\label{d11.27:e1}
l(v) < l(u)
\end{equation}
holds whenever \(v\) is a direct successor of \(u\);
\item \label{d11.27:s2} The equality
\begin{equation}\label{d11.27:e2}
\inf\{l(v) \colon v \in V(C)\} = 0
\end{equation}
holds for every maximal chain \(C\) in \((V(T), {\preccurlyeq}_{T(r)})\).
\end{enumerate}
Then \(l\) is said to be a \emph{monotone labeling} on \(T(r)\).
\end{definition}

\begin{remark}\label{r11.28}
Condition~\ref{d11.27:s1} from the above definition means that every monotone labeling \(l \colon V(T) \to \RR^{+}\) is a strictly isotone mapping of the posets \((V(T), {\preccurlyeq}_{T(r)})\) and \((\RR^{+}, {\leqslant})\) (see formula~\eqref{d2.8:e1}).
\end{remark}

\begin{example}\label{ex11.29}
Let \((X, d)\) be a nonempty totally bonded ultrametric space and let \(T_X = T_X(r_X)\) be the rooted representing tree of \((X, d)\), \(r_X = X\), \(V(T_X) = \BB_{X}\). Then the labeling \(l_X \colon \BB_{X} \to \RR^{+}\),
\[
l_X(B) = \diam B, \quad B \in \BB_{X},
\]
is monotone on \(T_X(r_X)\).

Let \(B_1\), \(B_2 \in V(T_X)\) and let \(B_2\) is a direct successor of \(B_1\). Then
\[
B_2 \sqsubset_{T_X(r_X)} B_1
\]
holds (see Remark~\ref{r10.11}). In particular, we have the proper inclusion
\begin{equation}\label{ex11.29:e1}
B_2 \subset B_1,
\end{equation}
that implies
\[
\diam B_2 \leqslant l_X(B_1).
\]
If \(\diam B_1 = \diam B_2\) holds, then from Proposition~\ref{p2.7} and \(\BB_{X} \subseteq \overline{\BB}_{X}\) (see Corollary~\ref{c2.41}) it follows that \(B_1 = B_2\), contrary~\eqref{ex11.29:e1}. Thus, we have
\[
l_X(B_2) = \diam B_2 < \diam B_1 = l_X(B_1).
\]
Condition~\ref{d11.27:s1} of Definition~\ref{d11.27} is satisfied.

Let \(C\) be a maximal chain in \((V(T_X), {\preccurlyeq}_{T_X(r_X)})\) and let \((\widetilde{X}, \widetilde{d})\) be a completion of \((X, d)\) such that \(X \subseteq \widetilde{X}\). By Corollary~\ref{c11.24}, there is \(c \in \widetilde{X}\) satisfying the equality
\begin{equation}\label{ex11.29:e2}
C = \{\widetilde{B} \cap X \colon \widetilde{B} \in \BB_{\widetilde{X}} \text{ and } c \in \widetilde{B}\}.
\end{equation}
Hence, for every \(\varepsilon > 0\) there is \(B \in C\) such that \(B = \widetilde{B}_{\varepsilon}(c) \cap X\), where
\begin{equation}\label{ex11.29:e3}
\widetilde{B}_{\varepsilon}(c) = \{\widetilde{x} \in \widetilde{X} \colon \widetilde{d}(\widetilde{x}, c) < \varepsilon\}.
\end{equation}
Now condition~\ref{d11.27:s2} of Definition~\ref{d11.27} follows from \eqref{ex11.29:e2} and \eqref{ex11.29:e3}.
\end{example}

\begin{proposition}\label{p11.29}
Let \(T_i = T_i(r_i)\) be a rooted tree with a monotone labeling \(l_i\), \(i=1\), \(2\). Then every isomorphism \(f \colon V(T_1) \to V(T_2)\) of the labeled trees \(T_1(l_1)\) and \(T_2(l_2)\) is also an isomorphism of the rooted trees \(T_1(r_1)\) and \(T_2(r_2)\).
\end{proposition}

\begin{proof}
It follows directly from Definition~\ref{d11.27} that the inequality
\begin{equation}\label{p11.29:e1}
l_i(r_i) > l_i(v_i)
\end{equation}
holds for every \(v_i \in V(T_i) \setminus \{r_i\}\), \(i=1\), \(2\). Using Definition~\ref{d2.5} and \eqref{p11.29:e1} we see that if \(f \colon V(T_1) \to V(T_2)\) is an isomorphism of \(T_1(l_1)\) and \(T_2(l_2)\), then the equality
\begin{equation}\label{p11.29:e2}
f(r_1) = f(r_2)
\end{equation}
holds. Since every isomorphism of labeled trees are also an isomorphism of free trees, Definition~\ref{d8.10} and \eqref{p11.29:e2} imply that \(f\) is an isomorphism of the rooted trees \(T_1(r_1)\) and \(T_2(r_2)\).
\end{proof}

The next our goal is to show that analogies of Proposition~\ref{p11.25} and Corollary~\ref{c11.26} remain valid for all locally finite rooted trees \(T = T(r)\) which satisfy \(\delta^{+}(u) \neq 1\) for every \(u \in V(T)\) and have monotone labelings.

We start with a lemma giving a ``constructive'' description of maximal chains in \((V(T(r)), {\preccurlyeq}_{T(r)})\) for arbitrary rooted tree \(T = T(r)\). Below, as in Definition~\ref{d8.11}, the symbol \(x \sqsubset_{T(r)} y\) means that \(y\) is an upper cover of \(x\) in the poset \((V(T(r)), {\preccurlyeq}_{T(r)})\), where \({\preccurlyeq}_{T(r)}\) is defined by~\eqref{e9.18}.

\begin{lemma}\label{l11.27}
Let \(T = T(r)\) be a rooted tree, \(E(T) \neq \varnothing\), and let \(C\) be a subset of \(V(T)\). Then the following statements are equivalent:
\begin{enumerate}
\item \label{l11.27:s1} \(C\) is a maximal chain in \((V(T(r)), {\preccurlyeq}_{T(r)})\).
\item \label{l11.27:s2} Either the elements of \(C\) can be numbered in a finite sequence \(c_1\), \(\ldots\), \(c_n\) such that
\begin{equation}\label{l11.27:e1}
c_1 \sqsupset_{T(r)} \ldots \sqsupset_{T(r)} c_n,
\end{equation}
where \(c_1 = r\) and \(c_n\) is a leaf of \(T(r)\) or the elements of \(C\) can be numbered in an infinite sequence \((c_n)_{n \in \mathbb{N}} \subseteq V(T)\) such that
\begin{equation}\label{l11.27:e2}
c_1 \sqsupset_{T(r)} \ldots \sqsupset_{T(r)} c_n \sqsupset_{T(r)} c_{n+1} \sqsupset_{T(r)} \ldots
\end{equation}
and \(c_1 = r\).
\end{enumerate}
\end{lemma}

\begin{proof}
Consider first the case when \(C = \{c_1, \ldots, c_n\}\) is a finite maximal chain in \((V(T(r)), {\preccurlyeq}_{T(r)})\). Condition \(E(T) \neq \varnothing\) implies the inequality \(n \geqslant 2\). Since \(C\) is a chain, we can assume that
\[
c_1 \succcurlyeq_{T(r)} \ldots \succcurlyeq_{T(r)} c_n
\]
is satisfied after some renumbering, if necessary. It was shown in the proof of Theorem~\ref{t9.20} that \(r\) is the largest element of \((V(T(r)), {\preccurlyeq}_{T(r)})\) (see part \ref{t9.20:s2} of that proof). Hence, we have
\begin{equation}\label{l11.27:e3}
r \succcurlyeq_{T(r)} c_1 \succcurlyeq_{T(r)} c_2 \succcurlyeq_{T(r)} \ldots \succcurlyeq_{T(r)} c_n.
\end{equation}
Since \(C\) is maximal, \eqref{l11.27:e3} implies the equality \(r = c_1\). From maximality of \(C\) and \({\preccurlyeq}_{T(r)} = {\sqsubseteq}_{T(r)}^{t}\) (see \eqref{t9.20:e8}) it follows that
\[
c_i \sqsupset_{T(r)} c_{i+1}
\]
for every \(i \in \{1, \ldots, n-1\}\). If \(c_n\) is not a leaf of \(T\), then the inequality \begin{equation}\label{l11.27:e4}
\delta(c_n) \geqslant 2
\end{equation}
holds, where \(\delta(c_n)\) is the degree of \(c_n\) in \(T\). By Lemma~\ref{l10.10}, the equivalence
\begin{equation}\label{l11.27:e5}
(u \sqsupset_{T(r)} c_n \text{ or } c_n \sqsupset_{T(r)} u) \Leftrightarrow (\{c_n, u\} \in E(T))
\end{equation}
is valid for every \(u\in V(T)\). Moreover, using Theorem~\ref{t9.20} we obtain
\[
(u \sqsupset_{T(r)} c_n) \Rightarrow (u = c_{n-1})
\]
for every \(u \in V(T)\). Hence, from~\eqref{l11.27:e4} and \eqref{l11.27:e5} it follows that there is a vertex \(c_{n+1} \in V(T)\) such that \(c_n \sqsupset_{T(r)} c_{n+1}\). Thus, we have the chain \(\{c_1, \ldots, c_n, c_{n+1}\}\),
\[
c_1 \sqsupset_{T(r)} c_2 \sqsupset_{T(r)} \ldots \sqsupset_{T(r)} c_n \sqsupset_{T(r)} c_{n+1},
\]
contrary to the maximality of the chain \(C\).

Conversely, let \(C = \{c_1, \ldots, c_n\} \subseteq V(T)\), let \eqref{l11.27:e1} hold with \(c_1 = r\) and let \(c_n\) be a leaf of \(T\). Condition~\eqref{l11.27:e1} implies that \(C\) is a chain in \((V(T(r)), {\preccurlyeq}_{T(r)})\). We must show that \(C\) is maximal. By Lemma~\ref{l10.10}, \((c_1, \ldots, c_n)\) is a path joining \(c_n\) with the root \(r = c_1\) in the tree \(T(r)\). From the definition of the relation \({\preccurlyeq}_{T(r)}\) (see \eqref{e9.17}, \eqref{e9.18}) it follows that the inequality \(c_n \preccurlyeq_{T(r)} u\) holds if and only if \(u \in C\). Consequently, \(C\) is a maximal chain if \(c_n\) is a minimal element of \((V(T(r)), {\preccurlyeq}_{T(r)})\) (\(v \preccurlyeq_{T(r)} c_n\) implies \(v = c_n\) for every \(v \in V(T)\)). Since \(c_n\) is a leaf of \(T\), the inequality
\begin{equation}\label{l11.27:e6}
\delta(c_n) \leqslant 1
\end{equation}
holds. Suppose that there is \(v \in V(T)\) such that \(v \prec_{T(r)} c_n\). Then, by Theorem~\ref{t9.20}, we can find \(v^{*} \in V(T)\) such that
\begin{equation}\label{l11.27:e7}
v^{*} \sqsubset_{T(r)} c_n \sqsubset_{T(r)} c_{n-1} \sqsubset_{T(r)} \ldots.
\end{equation}
From Lemma~\ref{l10.10}, \eqref{l11.27:e7} and the definition of the degree of vertices of graph, we obtain the inequality \(\delta(c_n) \geqslant 2\), contrary to \eqref{l11.27:e6}. Hence, \(c_n\) is a minimal element of \((V(T(r)), {\preccurlyeq}_{T(r)})\). The maximality of the chain \(C\) follows.

Let us consider now the case when \(C\) is infinite.

In what follows, for every \(x \in V(T)\), we will denote by \(x_{V(T)}^{\Delta}\) and \(x_C^{\Delta}\) the set \(\{v \in V(T) \colon v \succcurlyeq_{T(r)} x\}\) and, respectively, the set \(\{c \in C \colon c \succcurlyeq_{T(r)} x\}\). It is clear that
\begin{equation}\label{l11.27:e8}
x_{V(T)}^{\Delta} \supseteq x_C^{\Delta}
\end{equation}
holds for every \(x \in V(T)\). By Theorem~\ref{t9.20}, the set \(x_{V(T)}^{\Delta}\) is a finite chain whenever \(x \in V(T)\). If \(x \in C\) and \(C\) is a maximal chain, then \eqref{l11.27:e8} yields the equality
\begin{equation}\label{l11.27:e9}
x_{V(T)}^{\Delta} = x_C^{\Delta}.
\end{equation}

Let \(C\) be a maximal chain in \((V(T(r)), {\preccurlyeq}_{T(r)})\). As in the case of finite \(C\), we obtain \(r \in C\). Write \(c_1 = r\). Since \(C\) is infinite, there is \(c^{(1)} \in C\) such that \(c_1 \succ_{T(r)} c^{(1)}\).

Using Theorem~\ref{t9.20} and equality~\eqref{l11.27:e9} with \(x = c^{(1)}\), we can find an integer \(n \geqslant 2\) and \(\{c_1, \ldots, c_n\} \subset C\) such that
\begin{equation}\label{l11.27:e10}
c_1 \sqsupset_{T(r)} \ldots \sqsupset_{T(r)} c_n
\end{equation}
and \(c_n = c^{(1)}\). Since \(C\) is infinite, from \eqref{l11.27:e10} follows the existence of an element \(c^{(2)} \in C\) for which \(c^{(1)} \succ_{T(r)} c^{(2)}\) holds. Using equality~\eqref{l11.27:e9} with \(x = c^{(2)}\) we find \(\{c_1, \ldots, c_{n}, \ldots, c_{n+m}\} \subset C\) such that \(c_n = c^{(1)}\), \(c_{n+m} = c^{(2)}\) and
\[
c_1 \sqsupset_{T(r)} \ldots \sqsupset_{T(r)} c_n \sqsupset_{T(r)} \ldots \sqsupset_{T(r)} c_{n+m}.
\]
Repeating this procedure we obtain a sequence \((c_n)_{n \in \mathbb{N}} \subseteq C\) such that \(c_i \sqsupset_{T(r)} c_{i+1}\) holds for every \(i \in \mathbb{N}\). We claim that the equality
\begin{equation}\label{l11.27:e11}
C = \{c_n \colon n \in \mathbb{N}\}
\end{equation}
holds. Since \((c_n)_{n \in \mathbb{N}}\) is a sequence of elements of \(C\), equality \eqref{l11.27:e11} holds if we have
\begin{equation}\label{l11.27:e13}
c \in \{c_n \colon n \in \mathbb{N}\}
\end{equation}
for every \(c \in C\).

Let \(c\) be an arbitrary element of \(C\). Suppose there is \(p \in \{c_n \colon n \in \mathbb{N}\}\) such that \(p \preccurlyeq_{T(r)} c\). Then we obtain
\[
c \in p_{V(T)}^{\Delta} \subseteq \{c_n \colon n \in \mathbb{N}\},
\]
i.e., \eqref{l11.27:e13} holds. Since \(C\) is a chain, we have \(c \preccurlyeq_{T(r)} c_n\) or \(c_n \preccurlyeq_{T(r)} c\) for all \(n \in \mathbb{N}\). Consequently, if \eqref{l11.27:e13} does not hold, then \(c \preccurlyeq_{T(r)} c_n\) is valid for every \(n \in \mathbb{N}\). It implies the inclusion
\begin{equation}\label{l11.27:e14}
\{c_n \colon n \in \mathbb{N}\} \subseteq c_{V(T)}^{\Delta}.
\end{equation}
The set \(c_{V(T)}^{\Delta}\) is finite by Theorem~\ref{t9.20}. Hence, from \eqref{l11.27:e14} and Theorem~\ref{t9.20} it follows that \(\{c_n \colon n \in \mathbb{N}\}\) is also finite, contrary to the construction.

A minor modification of the last part of the above proof shows that any chain \(C\) admitting a numbering such that \eqref{l11.27:e2} holds with \(c_1 = r\) is maximal in \((V(T(r)), {\preccurlyeq}_{T(r)})\). The proof is completed.
\end{proof}

\begin{corollary}\label{c11.30}
Let \(T = T(r)\) be a rooted tree with a monotone labeling \(l \colon V(T) \to \RR^{+}\) and let \(v \in V(T)\). Then \(l(v) = 0\) holds if and only if \(v\) is a leaf of \(T\).
\end{corollary}

\begin{proof}
It follows directly from Definition~\ref{d11.27} and the one-to-one correspondence between the leaves and finite maximal chains described in Lemma~\ref{l11.27}. We recall only that the root \(r\) is a leaf of \(T\) when \(T\) is an empty tree (see Remark~\ref{r10.16}).
\end{proof}

\begin{lemma}\label{l11.30}
Let \(T = T(r)\) be a rooted tree, let \(C_1\), \(C_2\) be different maximal chains in \((V(T), {\preccurlyeq}_{T(r)})\). Then we have
\begin{equation}\label{l11.30:e1}
c_1 \parallel_{T(r)} c_2
\end{equation}
for all \(c_1 \in C_1 \setminus (C_1 \cap C_2)\) and \(c_2 \in C_2 \setminus (C_1 \cap C_2)\).
\end{lemma}

\begin{proof}
Suppose that there exist \(c_1 \in C_1 \setminus (C_1 \cap C_2)\) and \(c_2 \in C_2 \setminus (C_1 \cap C_2)\) for which \eqref{l11.30:e1} does not hold. Then, without loss of generality, we can assume \(c_1 \prec_{T(r)} c_2\). From Theorem~\ref{t9.20} and the definition of the partial order \({\preccurlyeq}_{T(r)}\) it follows that \(c_2\) is an vertex of the path \(P_{c_1} = (v_1, \ldots, v_n)\) joining \(c_1\) with \(r\) in \(T(r)\) and the relations
\begin{equation}\label{l11.30:e2}
c_1 = v_1 \sqsubset_{T(r)} v_2 \sqsubset_{T(r)} \ldots \sqsubset_{T(r)} v_n = r
\end{equation}
hold. Lemma~\ref{l9.14} (the uniqueness of a path connected \(c_1\) and \(r\) in \(T(r)\)) and Lemma~\ref{l11.27} imply \(V(P_{c_1}) \subseteq C_1\). Consequently, \(c_2 \in C_1\) holds, that yields \(c_2 \in C_1 \cap C_2\), contrary to \(c_2 \in C_2 \setminus (C_1 \cap C_2)\).
\end{proof}

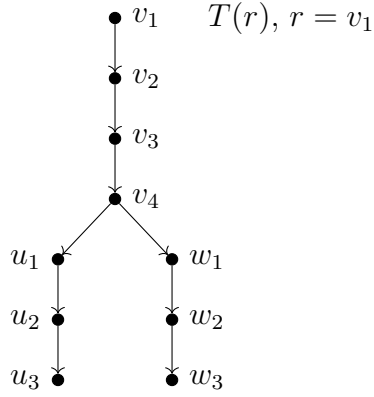
\begin{figure}[ht]
\begin{center}
\begin{tikzpicture}[->, sibling angle=30, level distance=0.8cm,
solid node/.style={circle,draw,inner sep=1.5,fill=black},
hollow node/.style={circle,draw,inner sep=1.5}]

\node [solid node, label={right:\(v_1\)}, label={[label distance=1cm] right:\(T(r)\), \(r=v_1\)}]
at (0,0) {}
child{node [solid node, label={right:\(v_2\)}] {}
	child{node [solid node, label={right:\(v_3\)}] {}
		child{node [solid node, label={right:\(v_4\)}] {}
			child{node [solid node, label={left:\(u_1\)}] {}
				child{node [solid node, label={left:\(u_2\)}] {}
					child{node [solid node, label={left:\(u_3\)}] {}}
				}
			}
			child{node [solid node, label={right:\(w_1\)}] {}
				child{node [solid node, label={right:\(w_2\)}] {}
					child{node [solid node, label={right:\(w_3\)}] {}}
				}
			}
		}
	}
};
\end{tikzpicture}
\end{center}
\caption{\(\mathbf{MC}_{T(r)}\) contains exactly two maximal chains \(C_1 = \{v_1, v_2, v_3, v_4\} \cup \{u_1, u_2, u_3\}\) and \(C_2 = \{v_1, v_2, v_3, v_4\} \cup \{w_1, w_2, w_3\}\). We have \(C_1 \cap C_2 = \{v_1, v_2, v_3, v_4\}\) and \(\wedge(C_1 \cap C_2) = v_4\).}
\label{fig9}
\end{figure}

Let \(T = T(r)\) be a rooted tree with a monotone labeling \(l \colon V(T) \to \RR^{+}\). Write \(\mathbf{MC}_{T(r)}\) for the set of all maximal chains in \((V(T), {\preccurlyeq}_{T(r)})\). Let us define a function
\[
d_{\mathbf{MC}_{T(r)}} \colon \mathbf{MC}_{T(r)} \times \mathbf{MC}_{T(r)} \to \RR^{+}
\]
by the rule
\begin{equation}\label{e11.53}
d_{\mathbf{MC}_{T(r)}}(C_1, C_2) = \begin{cases}
0 & \text{if } C_1 = C_2,\\
l(\wedge (C_1 \cap C_2)) & \text{if } C_1 \neq C_2,
\end{cases}
\end{equation}
where \(\wedge (C_1 \cap C_2)\) is the meet of the set \(C_1 \cap C_2\) in the poset \((V(T), {\preccurlyeq}_{T(r)})\).

The existence of the meet \(\wedge (C_1 \cap C_2)\) can be proved as in Remark~\ref{r11.24}. Thus, \(d_{\mathbf{MC}_{T(r)}}\) is correctly defined for every rooted tree \(T(r)\) with a monotone labeling \(l\). For an example of the set of all maximal chains corresponding to a rooted tree see Figure~\ref{fig9} above.

\begin{proposition}\label{p11.30}
For every rooted tree \(T = T(r)\) with a monotone labeling \(l \colon V(T) \to \RR^{+}\) the mapping
\[
d_{\mathbf{MC}_{T(r)}} \colon \mathbf{MC}_{T(r)} \times \mathbf{MC}_{T(r)} \to \RR^{+}
\]
is an ultrametric on \(\mathbf{MC}_{T(r)}\). In addition, if \(T\) is locally finite, then the ultrametric space \((\mathbf{MC}_{T(r)}, d_{\mathbf{MC}_{T(r)}})\) is compact.
\end{proposition}

\begin{proof}
Let \(T = T(r)\) be a rooted tree and let \(l \colon V(T) \to \RR^{+}\) be a monotone labeling.

First of all, we want to prove that \(d_{\mathbf{MC}_{T(r)}}\) is an ultrametric.

It is easy to prove that the inequality \(d_{\mathbf{MC}_{T(r)}}(C_1, C_2) > 0\) holds for all distinct \(C_1\), \(C_2 \in \mathbf{MC}_{T(r)}\). Indeed, if \(C_1 \neq C_2\), then, as in Remark~\ref{r11.24}, we obtain that there is \(v^{*} \in C_1 \cap C_2\) such that \(v^{*} = \wedge (C_1 \cap C_2)\). From \(C_1 \neq C_2\), Lemma~\ref{l11.27} and Corollary~\ref{c11.30} it follows that \(v^{*}\) is not a leaf of \(T\). Consequently, \(d_{\mathbf{MC}_{T(r)}}\) is an ultrametric if \(|\mathbf{MC}_{T(r)}| \leqslant 2\).

Let \(C_1\), \(C_2\), \(C_3\) be pairwise different maximal chains in \((V(T), {\preccurlyeq}_{T(r)})\). To prove the strong triangle inequality
\begin{equation}\label{p11.30:e1}
d_{\mathbf{MC}_{T(r)}}(C_1, C_2) \leqslant \max\{d_{\mathbf{MC}_{T(r)}}(C_1, C_3), d_{\mathbf{MC}_{T(r)}}(C_2, C_3)\},
\end{equation}
we note that the set \(\{\wedge (C_1 \cap C_2), \wedge (C_2 \cap C_3), \wedge (C_3 \cap C_1)\}\) is a chain in \((V(T), {\preccurlyeq}_{T(r)})\). Really, if \(i\), \(j\), \(k \in \{1, 2, 3\}\) and \(i \neq j \neq k\), then
\begin{equation}\label{p11.30:e2}
\wedge (C_i \cap C_j) \in C_j \quad \text{and} \quad \wedge (C_j \cap C_k) \in C_j
\end{equation}
hold. Since \(C_j\) is a chain in \((V(T), {\preccurlyeq}_{T(r)})\), from~\eqref{p11.30:e2} it follows that
\[
\wedge (C_i \cap C_j) \preccurlyeq_{T(r)} \wedge (C_j \cap C_k) \quad \text{or} \quad  \wedge (C_j \cap C_k) \preccurlyeq_{T(r)} \wedge (C_i \cap C_j).
\]

It is easy to see that inequality \eqref{p11.30:e1} is valid if
\begin{equation}\label{p11.30:e3}
\wedge (C_1 \cap C_3) \in C_1 \cap C_2 \quad \text{or} \quad \wedge (C_3 \cap C_2) \in C_1 \cap C_2.
\end{equation}
If \eqref{p11.30:e3} does not hold, then we obtain
\begin{equation}\label{p11.30:e4}
\wedge (C_1 \cap C_3) \in C_1 \setminus (C_1 \cap C_2) \quad \text{and} \quad \wedge (C_3 \cap C_2) \in C_2 \setminus (C_1 \cap C_2).
\end{equation}
By Lemma~\ref{l11.30}, from~\eqref{p11.30:e4} follows
\[
\wedge (C_1 \cap C_3) \parallel_{T(r)} \wedge (C_3 \cap C_2).
\]
Hence, \(\{\wedge (C_1 \cap C_2), \wedge (C_2 \cap C_3), \wedge (C_3 \cap C_1)\}\) is not a chain in \((V(T), {\preccurlyeq}_{T(r)})\), contrary to the statement which was obtained in the first part of the proof.

Thus, \((\mathbf{MC}_{T(r)}, d_{\mathbf{MC}_{T(r)}})\) is an ultrametric space.

Let \(T\) be locally finite. By Bolzano---Weierstrass property (see Proposition~\ref{p2.9}), the metric space \((\mathbf{MC}_{T(r)}, d_{\mathbf{MC}_{T(r)}})\) is compact if and only if, for every sequence \((C^{n})_{n \in \mathbb{N}} \subseteq \mathbf{MC}_{T(r)}\), there are \(C \in \mathbf{MC}_{T(r)}\) and a subsequence \((C^{n_j})_{j \in \mathbb{N}}\) of the sequence \((C^{n})_{n \in \mathbb{N}}\) such that
\begin{equation}\label{p11.30:e5}
\lim_{j \to \infty} d_{\mathbf{MC}_{T(r)}}(C^{n_j}, C) = 0.
\end{equation}
Let \((C^{n})_{n \in \mathbb{N}} \subseteq \mathbf{MC}_{T(r)}\). We will use an induction to construct the desirable \((C^{n_j})_{j \in \mathbb{N}}\) and \(C\). It suffices to consider the case when all elements of \((C^{n})_{n \in \mathbb{N}}\) are pairwise different,
\begin{equation}\label{p11.30:e6}
(C^{n} \neq C^{m}) \Leftrightarrow (n \neq m) \quad \forall n, m \in \mathbb{N}.
\end{equation}
Write \(\mathbf{A}_1 = \{C^{n} \colon n \in \mathbb{N}\}\), \(n_1 = 1\), and \(c_1 = r\). From \eqref{p11.30:e6} it follows \(|V(T)| = \infty\). Consequently, the set of all direct successors of \(r\) is nonempty. In addition, this set is finite, because \(T\) is locally finite. Using Lemma~\ref{l11.27} and \eqref{p11.30:e6}, we can find \(c_2 \in V(T)\) such that \(c_1 \sqsupset_{T(r)} c_2\) holds and the set
\[
\mathbf{A}_2 = \{C^{n} \in \mathbf{A}_1 \colon \{c_1, c_2\} \subseteq C^{n}, n \in \mathbb{N}\}
\]
is infinite. The set of all direct successors of \(c_2\) is also finite and nonempty. Hence, there is \(c_3 \in V(T)\) such that \(c_1 \sqsupset_{T(r)} c_2 \sqsupset_{T(r)} c_3\) and the set
\[
\mathbf{A}_3 = \{C^{n} \in \mathbf{A}_2 \colon \{c_1, c_2, c_3\} \subseteq C^{n}, n \in \mathbb{N}\}
\]
is infinite. Repeating this procedure, we can find a sequence \((c_j)_{j \in \mathbb{N}} \subseteq V(T)\) and a sequence \((\mathbf{A}_j)_{j \in \mathbb{N}} \subseteq 2^{V(T)}\) such that
\begin{equation}\label{p11.30:e7}
c_1 \sqsupset_{T(r)} \ldots \sqsupset_{T(r)} c_j \sqsupset_{T(r)} c_{j+1} \sqsupset_{T(r)} \ldots,
\end{equation}
\(c_1 = r\) and
\begin{equation}\label{p11.30:e8}
\{c_1, \ldots, c_j\} \subseteq \mathbf{A}_j \subseteq \{C^{n} \colon n \in \mathbb{N}\}
\end{equation}
for every \(j \in \mathbb{N}\). Write \(C = \{c_j \colon j \in \mathbb{N}\}\). By Lemma~\ref{l11.27}, \eqref{p11.30:e7} implies that \(C\) is a maximal chain in \((V(T), {\preccurlyeq}_{T(r)})\). Moreover, from \eqref{p11.30:e8} and \eqref{e11.53} it follows that
\begin{equation}\label{p11.30:e9}
d_{\mathbf{MC}_{T(r)}} (C, A) \leqslant l(c_j)
\end{equation}
holds for every \(j \in \mathbb{N}\) and every \(A \in \mathbf{A}_j\). The sequence \((A^{j})_{j \in \mathbb{N}} \subseteq 2^{V(T)}\), satisfying \(A^{j} \in \mathbf{A}_j\) for every \(j \in \mathbb{N}\), is a subsequence of the sequence \((C^{n})_{n \in \mathbb{N}} \subseteq \mathbf{MC}_{T(r)}\). Now condition~\ref{d11.27:s1} of Definition~\ref{d11.27} and inequality~\eqref{p11.30:e9} give us limit relation \eqref{p11.30:e5} with \(C^{n_j} = A^j\).
\end{proof}

The next lemma describes the structure of open balls in the space \((\mathbf{MC}_{T(r)}, d_{\mathbf{MC}_{T(r)}})\).

\begin{lemma}\label{l11.33}
Let \(T = T(r)\) be a locally finite rooted tree such that \(\delta_{T(r)}^{+}(u) \neq 1\) for every \(u \in V(T)\), let \(l \colon V(T) \to \RR^{+}\) be a monotone labeling, let \((\mathbf{MC}_{T(r)}, d_{\mathbf{MC}_{T(r)}})\) be the corresponding ultrametric space and let \(\BB_{\mathbf{MC}_{T(r)}}\) be the ballean of \((\mathbf{MC}_{T(r)}, d_{\mathbf{MC}_{T(r)}})\). Then the following conditions are equivalent for every \(B \subseteq \mathbf{MC}_{T(r)}\):
\begin{enumerate}
\item \label{l11.33:s1} The set \(B\) is an open ball in \((\mathbf{MC}_{T(r)}, d_{\mathbf{MC}_{T(r)}})\), \(B \in \BB_{\mathbf{MC}_{T(r)}}\).
\item \label{l11.33:s2} There is a vertex \(v \in V(T)\) such that
\begin{equation}\label{l11.33:e2}
B = \{C \in \mathbf{MC}_{T(r)} \colon v \in C\}.
\end{equation}
\end{enumerate}
\end{lemma}

\begin{proof}
\(\ref{l11.33:s1} \Rightarrow \ref{l11.33:s2}\). Let \(B \in \BB_{\mathbf{MC}_{T(r)}}\) and let \(C^{*}\) be a point of \(B\). Then the ball \(B\) has the form
\begin{equation}\label{l11.33:e3}
B = B_{r^{*}}(C^{*}) = \{C \in \mathbf{MC}_{T(r)} \colon d_{\mathbf{MC}_{T(r)}}(C, C^{*}) < r^{*}\},
\end{equation}
where \(r^{*} > 0\) is the radius of \(B\) (see Proposition~\ref{p2.4}). Write
\[
\diam B = \sup\{d_{\mathbf{MC}_{T(r)}}(C_1, C_2) \colon C_1, C_2 \in B\}.
\]
Suppose first that \(\diam B\) is a strictly positive real number. Since \(T\) is locally finite, the ultrametric space \((\mathbf{MC}_{T(r)}, d_{\mathbf{MC}_{T(r)}})\) is compact by Proposition~\ref{p11.30}. Hence, by Corollary~\ref{c2.39}, we can find \(C_{*} \in B\) such that
\begin{equation}\label{l11.33:e4}
\diam B = d_{\mathbf{MC}_{T(r)}}(C^{*}, C_{*}).
\end{equation}
From \(\diam B > 0\), the definition of \(d_{\mathbf{MC}_{T(r)}}\) (see \eqref{e11.53}) and \eqref{l11.33:e4} we obtain
\[
\diam B = l(\wedge (C^{*} \cap C_{*})).
\]
If \(|C^{*}| = \infty\), then, by Lemma~\ref{l11.27}, the elements of \(C^{*}\) can be numbered in an infinite sequence \((c_n^{*})_{n \in \mathbb{N}} \subseteq V(T)\) such that
\begin{equation}\label{l11.33:e5}
c_1^{*} \sqsupset_{T(r)} \ldots \sqsupset_{T(r)} c_n^{*} \sqsupset_{T(r)} c_{n+1}^{*} \sqsupset_{T(r)} \ldots
\end{equation}
with \(c_1^{*} = r\). Since \(\wedge (C^{*} \cap C_{*})\) belongs to \(C^{*}\), we can find \(c_{n_0}^{*} \in C^{*}\) such that
\[
\diam B = l(c_{n_0}^{*}).
\]
Now we claim that \eqref{l11.33:e2} holds with \(v = c_{n_0}^{*}\). Indeed, if \(C\) is an arbitrary point of \(B\), then
\begin{equation}\label{l11.33:e6}
l(\wedge (C^{*} \cap C)) = d_{\mathbf{MC}_{T(r)}}(C^{*}, C) \leqslant \diam B = l(c_{n_0}^{*})
\end{equation}
holds. Write \(c = \wedge (C^{*} \cap C)\). The vertex \(c \in V(T)\) belongs to the chains \(C^{*}\) and \(C\),
\begin{equation}\label{l11.33:e7}
c \in C
\end{equation}
and
\begin{equation}\label{l11.33:e8}
c \in C^{*}
\end{equation}
and, moreover, the equality
\begin{equation}\label{l11.33:e9}
l(c) = l(\wedge (C^{*} \cap C))
\end{equation}
holds. Using \eqref{l11.33:e8} we can find an element \(c_{n_1}^{*}\) of the sequence \((c_n^{*})_{n \in \mathbb{N}}\) such that
\begin{equation}\label{l11.33:e10}
c_{n_1}^{*} = c.
\end{equation}
Since the labeling \(l\) is monotone, \eqref{l11.33:e6}, \eqref{l11.33:e9} and \eqref{l11.33:e10} imply the inequality \(n_0 \leqslant n_1\). If the equality \(n_0 = n_1\) holds, then we evidently obtain \(c_{n_0}^{*} \in C\). In the case of the strict inequality \(n_0 < n_1\), from~\eqref{l11.33:e5} it follows that
\begin{equation}\label{l11.33:e11}
c_1^{*} \sqsupset_{T(r)} \ldots \sqsupset_{T(r)} c_{n_0}^{*} \sqsupset_{T(r)} \ldots \sqsupset_{T(r)} c_{n_1}^{*}.
\end{equation}
Now Lemma~\ref{l11.27}, equality~\eqref{l11.33:e10} and relations~\eqref{l11.33:e7}, \eqref{l11.33:e11} imply
\[
\{c_1^{*}, \ldots, c_{n_0}^{*}, \ldots, c_{n_1}^{*}\} \subseteq C.
\]
The last inclusion yields the membership \(c_{n_0}^{*} \in C\). Thus, the inclusion
\begin{equation}\label{l11.33:e12}
B \subseteq \{C \in \mathbf{MC}_{T(r)} \colon C \ni c_{n_0}^{*}\}
\end{equation}
is true for the case when \(C^{*}\) is infinite and \(\diam B > 0\). Using \eqref{l11.33:e12} we see that the equality
\[
B = \{C \in \mathbf{MC}_{T(r)} \colon C \ni c_{n_0}^{*}\}
\]
holds if and only if
\begin{equation}\label{l11.33:e15}
B \supseteq \{C \in \mathbf{MC}_{T(r)} \colon C \ni c_{n_0}^{*}\}.
\end{equation}
From~\eqref{l11.33:e3}, Corollary~\ref{c5.15} and the equality \(\diam B = l(c_{n_0}^{*})\) it follows that \eqref{l11.33:e15} holds if
\begin{equation}\label{l11.33:e16}
d_{\mathbf{MC}_{T(r)}}(C, C^{*}) \leqslant l(c_{n_0}^{*})
\end{equation}
holds whenever \(C \ni c_{n_0}^{*}\). To prove the last inequality, we note that Theorem~\ref{t9.20}, Lemma~\ref{l11.27} and formula~\eqref{l11.33:e5} imply the inclusion
\begin{equation}\label{l11.33:e17}
\{c_1^*, \ldots, c_{n_0}^{*}\} \subseteq C
\end{equation}
for every \(C\) containing \(c_{n_0}^{*}\). Now~\eqref{l11.33:e16} follows from \eqref{l11.33:e17} and \eqref{e11.53}.

The case when \(|C^{*}| < \infty\) and \(\diam B > 0\) can be considered similarly if we use \eqref{l11.27:e1} instead of \eqref{l11.33:e5}.

We now turn to the case \(\diam B = 0\). First we will notice that the equality \(\diam B = 0\) implies the inequality \(|C^{*}| < \infty\). Indeed, if \(C^{*}\) is infinite, then, using the monotonicity of \(l\), we can find \(n(r^{*}) \in \mathbb{N}\) such that
\begin{equation}\label{l11.33:e13}
l(c_{n(r^{*})}^{*}) < r^{*},
\end{equation}
where \(r^{*} > 0\) is the radius of \(B\) (see~\eqref{l11.33:e3}). By \eqref{l11.33:e5}, we have
\[
c_{n(r^{*})}^{*} \sqsupset_{T(r)} c_{n(r^{*})+1}^{*}.
\]
Hence, by Remark~\ref{r10.11}, \(c_{n(r^{*})+1}^{*}\) is a direct successor of \(c_{n(r^{*})}^{*}\). According to the condition of the lemma, we have \(\delta_{T(r)}^{+}(c_{n(r^{*})}^{*}) \neq 1\). Consequently, there is \(v \in T(r)\) such that
\begin{equation}\label{l11.33:e18}
c_{n(r^{*})}^{*} \sqsupset_{T(r)} v \quad \text{and} \quad c \neq c_{n(r^{*})+1}^{*}.
\end{equation}
From~\eqref{l11.33:e5} and \eqref{l11.33:e18} it follows that the set
\[
\{c_1^{*}, \ldots, c_{n(r^{*})}^{*}, v\}
\]
is a chain of \((V(T), {\preccurlyeq}_{T(r)})\) with \(v \notin C^{*}\). Let us denote by \(C_{v}^{*}\) a maximal chain in \((V(T), {\preccurlyeq}_{T(r)})\) such that
\begin{equation}\label{l11.33:e19}
\{c_1^{*}, \ldots, c_{n(r^{*})}^{*}, v\} \subset C_{v}^{*}.
\end{equation}
From \(v \notin C^{*}\) it follows that \(C^{*} \neq C_{v}^{*}\), i.e.,
\begin{equation}\label{l11.33:e20}
d_{\mathbf{MC}_{T(r)}}(C^{*}, C_{v}^{*}) > 0
\end{equation}
holds. Now using \eqref{l11.33:e13} and \eqref{l11.33:e20} we obtain
\[
d_{\mathbf{MC}_{T(r)}}(C^{*}, C_{v}^{*}) = l(c_{n(r^{*})}^{*}) < r^{*}
\]
which together with \eqref{l11.33:e3} implies \(C_{v}^{*} \in B\). Thus, by \eqref{l11.33:e20}, the inequality \(\diam B > 0\) holds, contrary to \(\dim B = 0\).

Now let \(\dim B = 0\) hold. As was proved above, it implies \(|C^{*}| < \infty\). Hence, by Lemma~\ref{l11.27}, the elements of \(C^{*}\) can be numbered in a finite sequence \(c_{1}\), \(\ldots\), \(c_{n}\) such that
\begin{equation}\label{l11.33:e21}
c_{1} \sqsupset_{T(r)} \ldots \sqsupset_{T(r)} c_{n},
\end{equation}
where \(c_{1} = r\) and \(c_{n}\) is a leaf of \(T(r)\). Using Theorem~\ref{t9.20} and Lemma~\ref{l11.27} it is easy to prove that for every leaf \(v \in V(T)\) there is an unique \(C \in \mathbf{MC}_{T(r)}\) such that \(v \in C\). Thus, for the case \(\diam B = 0\), \(|C^{*}| < \infty\), we have
\[
B = \{C \in \mathbf{MC}_{T(r)} \colon c_n \in C\} = \{C^{*}\}.
\]

\(\ref{l11.33:s2} \Rightarrow \ref{l11.33:s1}\). Let \(v \in V(T)\) and let a set \(B_v\) be defined as
\[
B_v = \{C \in \mathbf{MC}_{T(r)} \colon C \ni v\}.
\]
We claim that \(B_v \in \BB_{\mathbf{MC}_{T(r)}}\) holds. The last statement is trivially valid if \(|V(T)| = \). Let \(|V(T)| \geqslant 2\). Suppose first that \(v\) is a leaf of \(T\). The one-point set \(\{v\}\) is a chain in \((V(T), {\preccurlyeq}_{T(r)})\). By Proposition~\ref{l11.18}, there is \(C_1 \in \mathbf{MC}_{T(r)}\) such that \(\{v\} \subseteq C_1\). Hence, \(C_1 \in B_v\) holds. From Lemma~\ref{l11.27} it follows that the elements of \(C\) can be numbered in a finite sequence \(c_1\), \(\ldots\), \(c_n\) such that
\[
c_1 \sqsupset_{T(r)} \ldots \sqsupset_{T(r)} c_n,
\]
\(c_1 = r\) and \(c_n = v\) (from \(|V(T)| \geqslant 2\) it follows that \(n \geqslant 2\)). Using Theorem~\ref{t9.20} we can prove that the implication
\[
(C \ni v) \Rightarrow (C_1 = C)
\]
is valid for every \(C \in \mathbf{MC}_{T(r)}\). Hence, the equality
\begin{equation}\label{l11.33:e22}
B_v = \{C_1\}
\end{equation}
holds. Write \(r^{*} = l(c_{n-1})\), where \(c_{n-1}\) is the unique upper cover of \(c_n = v\) in \((V(T), {\preccurlyeq}_{T(r)})\), \(c_{n-1} \sqsupset_{T(r)} c_n\). Then \(c_{n-1} \in C_1\) holds and we have the strict inequality \(r^{*} > 0\), because \(l\) is a monotone labeling on \(T(r)\) and \(l(c_n) = l(v) = 0\) (see Corollary~\ref{c11.30}).

Let us consider the open ball
\[
B_{r^{*}}(C_1) = \{C \in \mathbf{MC}_{T(r)} \colon d_{\mathbf{MC}_{T(r)}} (C_1, C) < r^{*}\}.
\]
We claim that the equality
\begin{equation}\label{l11.33:e23}
B_{r^{*}}(C_1) = B_v
\end{equation}
holds. Let \(C \in \mathbf{MC}_{T(r)}\) satisfy the inequality
\begin{equation}\label{l11.33:e24}
d_{\mathbf{MC}_{T(r)}} (C_1, C) < r^{*}.
\end{equation}
If \(C_1 \neq C\) holds, then using \eqref{e11.53} we can rewrite \eqref{l11.33:e24} as
\begin{equation}\label{l11.33:e25}
l(\wedge (C_1 \cap C)) < r^{*}.
\end{equation}
The vertex \(\wedge (C_1 \cap C)\) belongs to \(C_1\) and satisfy \eqref{l11.33:e25}. For every \(i \in \{1, \ldots, n\}\), the inequality \(l(c_i) < r^{*}\) implies the equality \(l(c_i) = 0\). Consequently,
\[
d_{\mathbf{MC}_{T(r)}} (C_1, C) = l(\wedge (C_1 \cap C)) = 0
\]
holds, contrary to \(C_1 \neq C\). Thus, \(B_{r^{*}}(C_1) = \{C_1\}\) holds. The last equality and \eqref{l11.33:e22} imply \eqref{l11.33:e23}.

Let us consider now the case when \(v\) is not a leaf of \(T\), \(l(v) > 0\). Write \(r_{*} = l(v)\). By Lemma~\ref{c6.6}, the inequality \(r_{*} > 0\) and the equalities
\begin{equation}\label{l11.33:e26}
\diam B_v = r_{*}
\end{equation}
and
\begin{equation}\label{l11.33:e27}
B_v = \overline{B}_{r_{*}}(C_1),
\end{equation}
where
\[
\overline{B}_{r_{*}}(C_1) = \{C \in \mathbf{MC}_{T(r)} \colon d_{\mathbf{MC}_{T(r)}} (C_1, C) \leqslant r_{*}\},
\]
imply that \(B_v \in \BB_{\mathbf{MC}_{T(r)}}\). Thus, it suffices to show that \eqref{l11.33:e26} and \eqref{l11.33:e27} hold.

Let \(C^2\), \(C^3\) be arbitrary points of \(B_v\). Then \(v \in C^2 \cap C^3\) holds and, consequently, we have
\begin{equation}\label{l11.33:e28}
v \succcurlyeq_{T(r)} \wedge (C^2 \cap C^3).
\end{equation}
The labeling \(l\) is monotone on \(T(r)\). Hence, inequality~\eqref{l11.33:e28} implies
\begin{equation}\label{l11.33:e29}
r_{*} = l(v) \geqslant d_{\mathbf{MC}_{T(r)}} (C^2, C^3).
\end{equation}
Since \(C^2\), \(C^3\) are arbitrary points of \(B_v\), from \eqref{l11.33:e29} it follows that
\begin{equation}\label{l11.33:e30}
r_{*} \geqslant \diam B_v.
\end{equation}

Let us prove the converse inequality. Since \(v\) is not a leaf of \(T\), from \(\delta_{T(r)}^{+}(v) \neq 1\) it follows that there are at least two different lower covers \(v_1\) and \(v_2\) of the vertex \(v\),
\begin{equation}\label{l11.33:e31}
v \sqsupset_{T(r)} v_1, \quad v \sqsupset_{T(r)} v_2.
\end{equation}
It follows directly from Definition~\ref{d8.11} (of lower covers) that
\begin{equation}\label{l11.33:e32}
v_1 \parallel_{T(r)} v_2.
\end{equation}
Let us denote by \(\widetilde{C}_i\) an element of \(\mathbf{MC}_{T(r)}\) such that \(\widetilde{C}_i \supseteq \{v, v_i\}\), \(i = 1\), \(2\). From~\eqref{l11.33:e32} it follows that \(\widetilde{C}_1 \neq \widetilde{C}_2\). Now using~\eqref{l11.33:e31}, Lemma~\ref{l11.27} and the definition of the ultrametric \(d_{\mathbf{MC}_{T(r)}}\) we obtain
\[
\wedge (\widetilde{C}_1 \cap \widetilde{C}_2) = v
\]
and
\begin{equation}\label{l11.33:e33}
d_{\mathbf{MC}_{T(r)}}(\widetilde{C}_1, \widetilde{C}_2) = l(\wedge (\widetilde{C}_1 \cap \widetilde{C}_2)) = l(v) = r_{*}.
\end{equation}
Since \(\widetilde{C}_1\), \(\widetilde{C}_2 \in B_v\), equalities \eqref{l11.33:e33} yield \(\diam B_v \geqslant r_{*}\). The last inequality and \eqref{l11.33:e30} give us equality \eqref{l11.33:e26}.

To prove equality \eqref{l11.33:e27}, we first note that \(\overline{B}_{r_{*}}(C_1)\) is the smallest ball containing \(B_v\) (see Proposition~\ref{p2.12}). Consequently, we have the inclusion \(B_v \subseteq \overline{B}_{r_{*}}(C_1)\). Thus, to complete the proof, it is enough to prove the inclusion
\begin{equation}\label{l11.33:e34}
\overline{B}_{r_{*}}(C_1) \subseteq B_v.
\end{equation}

Inclusion~\eqref{l11.33:e34} holds if and only if, for every \(C \in \overline{B}_{r_{*}}(C_1) \setminus \{C_1\}\), we have \(v \in C\). Let \(C\) be an arbitrary point of \(\overline{B}_{r_{*}}(C_1) \setminus \{C_1\}\). The point \(\wedge (C_1 \cap C)\) belongs to \(C_1 \cap C\). Using \(C \in \overline{B}_{r_{*}}(C_1) \setminus \{C_1\}\) we obtain
\begin{equation}\label{l11.33:e35}
l(\wedge (C_1 \cap C)) = d_{\mathbf{MC}_{T(r)}}(C_1, C) \leqslant r_{*} = l(v).
\end{equation}
The memberships
\[
v \in C_1 \quad \text{and} \quad \wedge (C_1 \cap C) \in C_1,
\]
the monotonicity of \(l\), and \eqref{l11.33:e35} imply the inequality
\[
\wedge (C_1 \cap C) \preccurlyeq_{T(r)} v.
\]
The last inequality and \(\wedge (C_1 \cap C) \in C\) yield \(v \in C\). Inclusion~\eqref{l11.33:e34} follows. The proof is completed.
\end{proof}

\begin{proposition}\label{p11.33}
Let \(T = T(r)\) be a locally finite rooted tree and let \(l \colon V(T) \to \RR^{+}\) be a monotone labeling on \(T(r)\). Then the following statements are equivalent:
\begin{enumerate}
\item \label{p11.33:s1} The labeled representing tree \(T_{\mathbf{MC}_{T(r)}}(l_{\mathbf{MC}_{T(r)}})\) of the compact ultrametric space \((\mathbf{MC}_{T(r)}, d_{\mathbf{MC}_{T(r)}})\) is isomorphic to the labeled tree \(T(l)\).
\item \label{p11.33:s2} For every \(u \in V(T)\) we have \(\delta_{T(r)}^{+}(u) \neq 1\).
\end{enumerate}
\end{proposition}

\begin{proof}
\(\ref{p11.33:s1} \Rightarrow \ref{p11.33:s2}\). Let \ref{p11.33:s1} hold. Then there is an isomorphism
\[
\Phi \colon V(T_{\mathbf{MC}_{T(r)}}) \to V(T)
\]
of the labeled representing tree \(T_{\mathbf{MC}_{T(r)}}(l_{\mathbf{MC}_{T(r)}})\) and the labeled tree \(T(l)\). The labelings
\[
l \colon V(T) \to \RR^{+} \quad \text{and} \quad l_{\mathbf{MC}_{T(r)}} \colon V(T_{\mathbf{MC}_{T(r)}}) \to \RR^{+}
\]
are monotone (see Example~\ref{ex11.29}). Hence, by Proposition~\ref{p11.29}, \(\Phi\) is an isomorphism of the rooted trees \(T_{\mathbf{MC}_{T(r)}}(r_{\mathbf{MC}_{T(r)}})\) and \(T(r)\). Now statement~\ref{p11.33:s2} of the proposition being proved follows from Proposition~\ref{p10.17}.

\(\ref{p11.33:s2} \Rightarrow \ref{p11.33:s1}\). Let \ref{p11.33:s2} hold. Then, by Lemma~\ref{l11.33}, the mapping
\[
\Phi \colon V(T) \to \BB_{\mathbf{MC}_{T(r)}}
\]
with
\begin{equation}\label{p11.33:e2}
\Phi(v) = \{C \in \mathbf{MC}_{T(r)} \colon v \in C\}
\end{equation}
is correctly defined and surjective. By Proposition~\ref{p11.30}, the space \((\mathbf{MC}_{T(r)}, d_{\mathbf{MC}_{T(r)}})\) is ultrametric and compact and, consequently, by Lemma~\ref{l8.2}, the equality
\[
V(T_{\mathbf{MC}_{T(r)}}) = \BB_{\mathbf{MC}_{T(r)}}
\]
holds. Hence, it suffices to show that \(\Phi\) is an isomorphism of the labeled trees \(T(l)\) and \(T_{\mathbf{MC}_{T(r)}}(l_{\mathbf{MC}_{T(r)}})\).

Analyzing of the proof of Lemma~\ref{l11.33}, we obtain the equality
\begin{equation}\label{p11.33:e1}
\diam \Phi(v) = l(v)
\end{equation}
for every \(v \in V(T)\) (see, in particular, \eqref{l11.33:e22}, when \(v\) is a leaf of \(T\) and, respectively, \eqref{l11.33:e26}, when \(v\) is not a leaf of \(T\)). Hence, it suffices to show that \(\Phi\) is an isomorphism of \(T(r)\) and \(T_{\mathbf{MC}_{T(r)}}(r_{\mathbf{MC}_{T(r)}})\).

Let us show first that the mapping \(\Phi\) is injective. Suppose contrary that there exist \(v_1\), \(v_2 \in V(T)\) such that
\begin{equation}\label{p11.33:e3}
v_1 \neq v_2 \quad \text{and} \quad \Phi(v_1) = \Phi(v_2).
\end{equation}
Let \(C_1\) and \(C_2\) belong to \(\mathbf{MC}_{T(r)}\) such that \(v_1 \in C_1\) and \(v_2 \in C_2\). From~\eqref{p11.33:e1} and \eqref{p11.33:e3} it follows that \(l(v_1) = l(v_2)\) holds. The last equality, Lemma~\ref{l11.27}, Remark~\ref{r11.28} and \(v_1 \neq v_2\) imply \(v_1 \parallel_{T(r)} v_2\). In particular, we obtain \(C_1 \neq C_2\). Consequently, the equality
\[
d_{\mathbf{MC}_{T(r)}}(C_1, C_2) = l(\wedge (C_1 \cap C_2))
\]
holds (see~\eqref{e11.53}). From \(v_1 \parallel_{T(r)} v_2\) it follows that \(v_i \notin C_1 \cap C_2\), \(i = 1\), \(2\). Hence, we have
\[
v_i \prec_{T(r)} \wedge (C_1 \cap C_2), \quad i=1, 2,
\]
that implies the inequality
\begin{equation}\label{p11.33:e5}
l(v_i) < d_{\mathbf{MC}_{T(r)}}(C_1, C_2)
\end{equation}
(see Remark~\ref{r11.28}). Using \eqref{p11.33:e1} we see that \(l(v_1) (=l(v_2))\) is the diameter of the ball \(\Phi(v_1) (=\Phi(v_2))\) which contradicts~\eqref{p11.33:e5} because \(C_1\) and \(C_2\) belong to this ball. Thus, we proved that \(\Phi\) is an injective mapping.

It was noted above that \(\Phi\) is also surjective. Hence, \(\Phi \colon V(T) \to \BB_{\mathbf{MC}_{T(r)}}\) is a bijection. By Proposition~\ref{p9.23}, the bijection \(\Phi\) is an isomorphism of the rooted tree \(T(r)\) and \(T_{\mathbf{MC}_{T(r)}}(r_{\mathbf{MC}_{T(r)}})\) if nd only if it is an order isomorphism of posets \((V(T), {\preccurlyeq}_{T(r)})\) and \((V(T_{\mathbf{MC}_{T(r)}}), {\preccurlyeq}_{T_{\mathbf{MC}_{T(r)}}})\). Lemma~\ref{l10.14} implies the equality
\[
(V(T_{\mathbf{MC}_{T(r)}}), {\preccurlyeq}_{T_{\mathbf{MC}_{T(r)}}}) = (\BB_{\mathbf{MC}_{T(r)}}, {\preccurlyeq}_{\mathbf{MC}_{T(r)}}).
\]
Thus, it suffices to prove that \(\Phi\) is an order isomorphism of the posets
\[
(V(T), {\preccurlyeq}_{T(r)}) \quad \text{and} \quad (\BB_{\mathbf{MC}_{T(r)}}, {\preccurlyeq}_{\mathbf{MC}_{T(r)}}).
\]
Using Lemma~\ref{l11.27} and the uniqueness of upper cover in \((V(T), {\preccurlyeq}_{T(r)})\) we can simply prove the validity of the implication
\[
(u \sqsubset_{T(r)} v) \Rightarrow (\Phi(u) \preccurlyeq_{\mathbf{MC}_{T(r)}} \Phi(v))
\]
for all \(u\), \(v \in V(T)\) that, together with Theorem~\ref{t9.20}, proves the validity of
\[
(p \preccurlyeq_{T(r)} q) \Rightarrow (\Phi(p) \preccurlyeq_{\mathbf{MC}_{T(r)}} \Phi(q))
\]
for all \(p\), \(q \in V(T)\). Thus, \(\Phi\) is isotone as a mapping from \((V(T), {\preccurlyeq}_{T(r)})\) to \((\BB_{\mathbf{MC}_{T(r)}}, {\preccurlyeq}_{\mathbf{MC}_{T(r)}})\). Since \(\Phi\) is bijective, the last statement implies that \(\Phi\) is strictly isotone (see formula~\eqref{d2.8:e1}). Hence, by Lemma~\ref{l6.2}, it suffices to prove the validity of implication
\begin{equation}\label{p11.33:e6}
(v_1 \parallel_{T(r)} v_2) \Rightarrow (\Phi(v_1) \parallel_{\mathbf{MC}_{T(r)}} \Phi(v_2))
\end{equation}
for all \(v_1\), \(v_2 \in V(T)\).

Let \(v_1\), \(v_2 \in V(T)\) and \(v_1 \parallel_{T(r)} v_2\) hold. Suppose that \(\Phi(v_1) \parallel_{\mathbf{MC}_{T(r)}} \Phi(v_2)\) is false. Using Proposition~\ref{p2.5} we see that \(\Phi(v_1) \parallel_{\mathbf{MC}_{T(r)}} \Phi(v_2)\) is false if and only if \(\Phi(v_1) \cap \Phi(v_2) \neq \varnothing\). Without loss of generality, we may assume that \(l(v_2) \leqslant l(v_1)\). Then the equalities \(\diam \Phi(v_1) = l(v_1)\), \(\diam \Phi(v_2) = l(v_2)\), the inclusion
\[
\BB_{\mathbf{MC}_{T(r)}} \subseteq \overline{\BB}_{\mathbf{MC}_{T(r)}}
\]
(see Corollary~\ref{c2.41}) and Proposition~\ref{p2.7} imply
\begin{equation}\label{p11.33:e7}
\Phi(v_2) \subseteq \Phi(v_1).
\end{equation}
As in the proof of injectivity of \(\Phi\), we can show that
\begin{equation}\label{p11.33:e8}
\diam \Phi(v_1) = l(v_1) < d_{\mathbf{MC}_{T(r)}} (C_1, C_2)
\end{equation}
holds whenever \(C_i \in \mathbf{MC}_{T(r)}\) and \(v_i \in C_i\), \(i =1\), \(2\). By \eqref{p11.33:e2}, we have \(C_1 \in \Phi(v_1)\) and \(C_2 \in \Phi(v_2)\). Hence, \eqref{p11.33:e8} contradicts \eqref{p11.33:e7}. Implication~\eqref{p11.33:e6} is valid.
\end{proof}

The next theorem can be considered as a main result of the paper.

\begin{theorem}\label{t7.3}
Let \(T=T(l)\) be a labeled tree. Then following statements \ref{t7.3:s1}--\ref{t7.3:s3} are equivalent:
\begin{enumerate}
\item\label{t7.3:s1} There is a nonempty compact ultrametric space \((X,d)\) such that \(T_X(l_X) \simeq T(l)\).
\item\label{t7.3:s2} There is a nonempty totally bounded ultrametric space \((Y, \rho)\) such that \(T_Y(l_Y) \simeq T(l)\).
\item\label{t7.3:s3} \(T\) is locally finite and there is \(r \in V(T)\) such that \(l\) is a monotone labeling on \(T(r)\) and for every \(u \in V(T)\) we have
\begin{equation}\label{t7.3:e1}
\delta_{T(r)}^+(u)\neq 1.
\end{equation}
\end{enumerate}
\end{theorem}

\begin{proof}
By Theorem~\ref{t11.12}, the logical equivalence \(\ref{t7.3:s1} \Leftrightarrow \ref{t7.3:s2}\) is valid.

\(\ref{t7.3:s1} \Rightarrow \ref{t7.3:s3}\). Suppose that there is a nonempty compact ultrametric space \((X, d)\) such that \(T_X(l_X) \simeq T(l)\). Then \(T\) is locally finite by Proposition~\ref{p8.4}. Let
\[
\Phi \colon V(T_X) \to V(T)
\]
be an isomorphism of labeled trees \(T_X(l_X)\) and \(T(l)\). The labeling \(l_X \colon V(T_X) \to \RR^{+}\) is monotone on \(T_X(r_X)\) (see Example~\ref{ex11.29}). Consequently, \(l\) is a monotone labeling on the rooted tree \(T = T(r)\) with \(r = \Phi(X)\). Hence, by Proposition~\ref{p11.29}, \(\Phi\) is also an isomorphism of the tooted trees \(T_X(r_X)\) and \(T(r)\), \(r = \Phi(X)\). Now \eqref{t7.3:e1} follows from formula~\eqref{p10.17:e2}.

\(\ref{t7.3:s3} \Rightarrow \ref{t7.3:s1}\). Suppose \ref{t7.3:s3} holds. Let us denote by \(r\) the vertex of \(T\) for which \(l\) is monotone on \(T(r)\) and \eqref{t7.3:e1} holds for every \(u \in V(T)\). Since \(T\) is locally finite, the set \(\mathbf{MC}_{T(r)}\) endowed with metric \(d_{\mathbf{MC}_{T(r)}}\) (see~\eqref{e11.53}) is a compact ultrametric space by Proposition~\ref{p11.30}. In addition, since we have \eqref{t7.3:e1} for every \(u \in V(T)\), Proposition~\ref{p11.33} implies
\[
T(l) \simeq T_{\mathbf{MC}_{T(r)}}(l_{\mathbf{MC}_{T(r)}}).
\]
Statement~\ref{t7.3:s1} follows.
\end{proof}

Theorem~\ref{t7.3} guarantees, in particular, the existence of Gurvich---Vyalyi representation of finite ultrametric spaces by monotone trees.

\begin{corollary}\label{c7.6}
Let \(T=T(l)\) be a finite labeled tree. Then the following statements \ref{c7.6:s1} and \ref{c7.6:s2} are equivalent.
\begin{enumerate}
\item \label{c7.6:s1} There is a nonempty finite ultrametric space \((X, d)\) with \(T_X(l_X) \simeq T(l)\).
\item \label{c7.6:s2} There is \(r \in V(T)\) such that:
\begin{enumerate}
\item \label{c7.6:s2:1} \(\delta_{T(r)}^+(u)\neq 1\) for every \(u \in V(T)\);
\item \label{c7.6:s2:2} \(l(v) < l(u)\) whenever \(v\) is a direct successor of \(u\);
\item \label{c7.6:s2:3} \(l(u) = 0\) if and only if \(u\) is a leaf of \(T(r)\).
\end{enumerate}
\end{enumerate}
\end{corollary}

\begin{remark}\label{r11.36}
If \(T = T(l)\) is a finite labeled tree satisfying statement \ref{c7.6:s2:2} of Corollary \ref{c7.6}, then \(T(l) \simeq T_X(l_X)\) is valid for unique up to isometry finite ultrametric space \((X, d)\) (see Theorem~2.7 \cite{DP2019PNUAA} and Theorem~3.6 \cite{Dov2019pNUAA}).
\end{remark}

Using Theorem~\ref{t7.3} we can simply obtain a characterization of the rooted representing trees of totally bounded ultrametric spaces.

\begin{theorem}\label{t10.16}
The following statements are equivalent for every rooted tree \(T(r)\):
\begin{enumerate}
\item \label{t10.16:s1} There is a nonempty totally bounded ultrametric space \((X, d)\) such that the rooted representing tree \(T_X(r_X)\) and \(T(r)\) are isomorphic.
\item \label{t10.16:s2} \(T(r)\) is locally finite and \(\delta_{T(r)}^{+}(u) \neq 1\) holds for every \(u \in V(T(r))\).
\end{enumerate}
\end{theorem}

\begin{proof}
The validity of implication \(\ref{t10.16:s1} \Rightarrow \ref{t10.16:s2}\) follows from Proposition~\ref{p10.17}.

\(\ref{t10.16:s2} \Rightarrow \ref{t10.16:s1}\). Let \(T = T(r)\) be a locally finite rooted tree and let \(\delta^{+}(u) \neq 1\) holds for every \(u \in V(T)\). We must show that \(T(r)\) is isomorphic to the rooted representing tree \(T_X(r_X)\) of some nonempty totally bounded ultrametric space \((X, d)\).

Suppose that there exists a monotone labeling \(l \colon V(T) \to \RR^{+}\) on \(T(r)\). Then, by Theorem~\ref{t7.3}, there is a nonempty totally bounded ultrametric space \((X, d)\) such that \(T_X(l_X) \simeq T(l)\). Let \(\Phi \colon V(T_X) \to V(T)\) be an isomorphism of \(T_X(l_X)\) and \(T(l)\). Proposition~\ref{p11.29} and Example~\ref{ex11.29} imply that \(\Phi\) is also an isomorphism of the rooted trees \(T_X(r_X)\) and \(T(r)\). Thus, it suffices to prove the existence of a monotone labeling \(l \colon V(T) \to \RR^{+}\) on \(T(r)\).

Let us define the labeling \(l \colon V(T) \to \RR^{+}\) as \(l(v) \equiv 0\) if \(T\) is empty and, in the opposite case, write
\[
l(v) = \begin{cases}
1 & \text{if } v = r,\\
0 & \text{if \(v\) is a leaf of \(T\)},\\
|V(P_v)|^{-1} & \text{othetwise},
\end{cases}
\]
where \(P_v\) is the unique path connected \(v\) and \(r\) in \(T\) and \(|V(P_v)|\) is the number of vertices of \(P_v\). Using Theorem~\ref{t9.20} it is easy to prove that any labeling so defined is monotone on \(T(r)\).
\end{proof}

Now we can completely describe the free representing trees of totally bounded ultrametric spaces. Let us start from the next example.

\begin{example}\label{ex8.8}
Let \(T\) be a tree with \(|V(T)| = 2\). Then there is no totally bounded ultrametric space \((X, d)\) such that the trees \(T\) and \(T_X\) are isomorphic. Indeed, if \((X, d)\) is a totally bounded ultrametric space for which \(|V(T_X)| = 2\), then we have \(|X| \geqslant 2\). By Proposition~\ref{p2.36}, the diametrical graph \(G_{X, d}\) is complete \(k\)-partite for an integer \(k \geqslant 2\). Thus, the inequality \(|V(T_X)| \geqslant 3\) holds whenever \(|X| \geqslant 2\).
\end{example}

\begin{theorem}\label{t9.10}
The following conditions are equivalent for every tree \(T\):
\begin{enumerate}
\item \label{t9.10:s1} There is a nonempty totally bounded ultrametric space \((X, d)\) such that the free representing tree \(T_X\) and \(T\) are isomorphic.
\item \label{t9.10:s2} \(T\) is a locally finite tree with \(|V(T)| \neq 2\) and the equation \(\delta(v) = 2\) has at most one solution \(v \in V(T)\).
\end{enumerate}
\end{theorem}

\begin{proof}
Proposition~\ref{p8.4} and Example~\ref{ex8.8} imply the validity of \(\ref{t9.10:s1} \Rightarrow \ref{t9.10:s2}\).

\(\ref{t9.10:s2} \Rightarrow \ref{t9.10:s1}\). Suppose \(T\) is a locally finite tree with \(|V(T)| \neq 2\) and the equation \(\delta_{T}(v) = 2\) has at most one solution \(v \in V(T)\). Since any two isomorphic rooted trees are also isomorphic as free trees, Theorem~\ref{t10.16} implies that statement~\ref{t9.10:s1} of Theorem~\ref{t9.10} holds if there is \(r \in V(T)\) such that
\begin{equation}\label{t9.10:e1}
\delta_{T(r)}^{+}(u) \neq 1
\end{equation}
for every \(u \in V(T(r))\).

Let us assume first that there is a unique \(v^{*} \in V(T)\) satisfying the equation \(\delta_{T}(v^{*}) = 2\). Then we claim that~\eqref{t9.10:e1} holds for every \(u \in V(T(r))\) with \(r = v^{*}\). Using~\eqref{e10.43} and \(\delta_{T}(v^{*}) = 2\) we obtain
\begin{equation}\label{t9.10:e2}
\delta_{T(r)}^{+}(v^{*}) = \delta_{T}(v^{*}) = 2 \neq 1
\end{equation}
for \(v = v^{*} = r\) and
\begin{equation}\label{t9.10:e3}
\delta_{T(r)}^{+}(v) = \delta_{T}(v)-1
\end{equation}
for \(v \neq v^{*}\) and \(v^{*} = r\). Since \(v^{*}\) is an unique solution of \(\delta_{T}(v) = 2\), equality \eqref{t9.10:e3} implies \(\delta_{T(r)}^{+}(v) \neq 1\). Thus, \eqref{t9.10:e1} is valid for every \(u \in V(T(r))\) with \(r = v^{*}\).

Let \(\delta_{T}(v) \neq 2\) hold for every \(v \in V(T)\). Condition \(|V(T)| \neq 2\) implies either \(|V(T)| = 1\) or \(|V(T)| \geqslant 3\). If \(|V(T)| = 1\), then statement~\ref{t9.10:s1} of the present theorem is evidently valid for \((X, d)\) with \(|X| = 1\). Consequently, it suffices to consider only the case \(|V(T)| \geqslant 3\).

Let \(v_1\), \(v_2\), \(v_3\) be distinct vertices of \(T\) and let \(G\) be a finite subtree of \(T\) such that \(v_i \in V(G)\), \(i = 1\), \(2\), \(3\). Using the equality
\[
\sum_{v \in V(G)} \delta_G(v) = 2(|V(G)| - 1)
\]
(see, for example, Exercise~4.1.7 in \cite{BM2008}) and the inequality \(|V(G)| \geqslant 3\) we can find a vertex \(v^{*} \in V(G)\) such that \(\delta_G(v^{*}) \geqslant 2\). Since \(V(G) \subseteq V(T)\) and \(\delta_G(v^{*}) \leqslant \delta_T(v^{*})\) hold, from \(\delta_G(v^{*}) \geqslant 2\) it follows that \(\delta_T(v^{*}) \geqslant 2\). The last inequality and \(\delta_{T}(v^{*}) \neq 2\) imply
\begin{equation}\label{t9.10:e4}
\delta_{T}(v^{*}) \geqslant 3.
\end{equation}
Inequality~\eqref{t9.10:e4}, formula~\eqref{e10.43} and the condition
\[
\delta_{T}(u) \neq 2, \quad \forall u \in V(T)
\]
yield~\eqref{t9.10:e1} for every \(u \in V(T(r))\) with \(r = v^{*}\).
\end{proof}

The next theorem is a characterization of ordered balleans \((\BB_{X}, {\preccurlyeq}_X)\) of totally bounded ultrametric spaces \((X, d)\) up to order isomorphisms.

\begin{theorem}\label{t10.17}
Let \((S, \preccurlyeq_{S})\) be a poset. Then conditions \ref{t10.17:s1} and \ref{t10.17:s2} are equivalent.
\begin{enumerate}
\item \label{t10.17:s1} There is a nonempty totally bounded ultrametric space \((X, d)\) such that the posets \((\BB_{X}, \preccurlyeq_{X})\) and \((S, \preccurlyeq_{S})\) are isomorphic.
\item \label{t10.17:s2} The poset \((S, \preccurlyeq_{S})\) has the following properties:
\begin{enumerate}
\item\label{t10.17:s2:1} \(S\) contains a largest element \(l\).
\item\label{t10.17:s2:2} If \(p \in S\) is not the largest element of \(S\), then there is a unique \(q \in S\) such that \(p \sqsubset_S q\).
\item\label{t10.17:s2:3} If \(b \in S\) is not a minimal element of \(S\), then the number of lower covers of \(b\) is finite and greater than or equal to two, i.e., there are distinct elements \(p_1\), \(\ldots\), \(p_n \in S\) such that \(n \geqslant 2\) and
\[
(p \sqsubset_S b) \Leftrightarrow (p \in \{p_1, \ldots, p_n\})
\]
is valid for every \(p \in S\).
\item\label{t10.17:s2:4} The equality \({\preccurlyeq_S} = {\sqsubseteq_{S}^{t}}\) holds.
\end{enumerate}
\end{enumerate}
\end{theorem}

\begin{proof}
\(\ref{t10.17:s1} \Rightarrow \ref{t10.17:s2}\). Let \((X, d)\) be a nonempty totally bounded ultrametric space for which \((\BB_{X}, {\preccurlyeq}_X)\) and \((S, {\preccurlyeq}_S)\) are isomorphic. We must show that \((S, {\preccurlyeq}_S)\) has properties \ref{t10.17:s2:1}--\ref{t10.17:s2:4}.

By Lemma~\ref{l10.14}, we have the equality
\begin{equation}\label{t10.17:e0}
(\BB_{X}, {\preccurlyeq}_X) = (V(T_X(r_X)), {\preccurlyeq}_{T_X(r_X)}).
\end{equation}
Consequently, the posets \((V(T_X(r_X)), {\preccurlyeq}_{T_X(r_X)})\) and \((S, {\preccurlyeq}_S)\) are isomorphic. By Theorem~\ref{t9.20} it implies that \((S, {\preccurlyeq}_S)\) has properties \ref{t10.17:s2:1}, \ref{t10.17:s2:2} and \ref{t10.17:s2:4}. Since \((\BB_{X}, {\preccurlyeq}_X)\) and \((S, {\preccurlyeq}_S)\) are isomorphic, \eqref{t10.17:e0} implies that \((S, {\preccurlyeq}_S)\) has property \ref{t10.17:s2:3} if and only if the following statement holds for every \(B^{*} \in \BB_{X}\) with \(|B^{*}| \geqslant 2\):
\begin{itemize}
\item[\((s_1)\)] There are an integer number \(n \geqslant 2\) and a set \(\{B_1, \ldots, B_n\} \subseteq \BB_{X}\) such that
\[
(B \sqsubset_{T_X(r_X)} B^*) \Leftrightarrow (B \in \{B_1, \ldots, B_n\}), \quad \forall B \in \BB_{X},
\]
where \({\sqsubset}_{T_X(r_X)}\) the upper covering relation corresponding to the partial order \({\preccurlyeq}_{T_X(r_X)}\).
\end{itemize}
From Lemma~\ref{l10.10} and the definition of out-degree of vertices of rooted trees it follows that statement \((s_1)\) is true if and only if
\begin{equation}\label{t10.17:e2}
+\infty > \delta^{+}(B^{*}) \geqslant 2
\end{equation}
holds for every \(B^{*} \in \BB_{X}\) with \(|B^{*}| \geqslant 2\). To complete the verification of \ref{t10.17:s2:3} it suffices to note that \eqref{t10.17:e2} follows from Proposition~\ref{p10.17} and equality~\eqref{p10.17:e2}.

\(\ref{t10.17:s2} \Rightarrow \ref{t10.17:s1}\). Let \((S, {\preccurlyeq}_S)\) have properties \ref{t10.17:s2:1}--\ref{t10.17:s2:4}. We want to prove the existence of nonempty totally bounded ultrametric space \((X, d)\) for which the posets \((\BB_{X}, {\preccurlyeq}_X)\) and \((S, {\preccurlyeq}_S)\) are isomorphic.

First of all, we note that condition \ref{t10.17:s2} of Theorem~\ref{t10.17} implies condition~\ref{t9.20:s2} of Theorem~\ref{t9.20}. Hence, there is a rooted tree \(T^{*}(r^{*})\) such that \((V(T^{*}(r^{*})), {\preccurlyeq}_{T^{*}(r^{*})})\) and \((S, {\preccurlyeq}_S)\) are order isomorphic. Suppose also that
\begin{itemize}
\item[\((s_2)\)] there is a nonempty totally bounded ultrametric space \((X, d)\) for which \(T^{*}(r^{*})\) and \(T_X(r_X)\) are isomorphic as rooted trees.
\end{itemize}
Let \(F \colon V(T^{*}(r^{*})) \to V(T_X(r_X))\) be an isomorphism of these rooted trees and let \(\Phi \colon S \to V(T^{*}(r^{*}))\) be an order isomorphism of \((S, {\preccurlyeq}_S)\) and \((V(T^{*}(r^{*})), {\preccurlyeq}_{T^{*}(r^{*})})\). By Proposition~\ref{p9.23}, the mapping \(F\) is also an order isomorphism of posets
\[
(V(T^{*}(r^{*})), {\preccurlyeq}_{T^{*}(r^{*})}) \quad \text{and} \quad (V(T_X(r_X)), {\preccurlyeq}_{T_X(r_X)}).
\]
Hence, the composition \(\Psi = F \circ \Phi\) is an order isomorphism of \((V(T_X(r_X)), {\preccurlyeq}_{T_X(r_X)})\) and \((S, {\preccurlyeq}_S)\). Using~\eqref{t10.17:e0} we also obtain that \(\Psi\) is an order isomorphism of \((S, {\preccurlyeq}_S)\) and \((\BB_{X}, {\preccurlyeq}_X)\). Consequently, to complete the proof it suffices to prove the validity of statement~\((s_2)\).

By Theorem~\ref{t10.16}, statement \((s_2)\) is valid if \(T^{*}(r^{*})\) is locally finite and
\begin{equation}\label{t10.17:e3}
\delta^{+}(u^{*}) \neq 1
\end{equation}
holds for every \(u^{*} \in V(T^{*}(r^{*}))\). Using Lemma~\ref{l10.10} we obtain
\begin{gather*}
\bigl(\{u^{*}, v^{*}\} \in E(T^{*}(r^{*}))\bigr) \Leftrightarrow \bigl(u^{*} \sqsubset_{T^{*}(r^{*})} v^{*} \text{ or } v^{*} \sqsubset_{T^{*}(r^{*})} u^{*}\bigr)\\
\Leftrightarrow \bigl(\Phi^{-1}(u^{*}) \sqsubset_{S} \Phi^{-1}(v^{*}) \text{ or } \Phi^{-1}(v^{*}) \sqsubset_{S} \Phi^{-1}(u^{*})\bigr).
\end{gather*}
This chain of equivalences and properties \ref{t10.17:s2:2}, \ref{t10.17:s2:3} show that \(T^{*}(r^{*})\) is locally finite. It remains to prove that \eqref{t10.17:e3} holds for every \(u^{*} \in V(T^{*}(r^{*}))\).

Relation \eqref{t10.17:e3} is trivially valid if \(|V(T^{*}(r^{*}))| = 1\). Let us consider the case when \(|V(T^{*}(r^{*}))| \geqslant 2\). Let \(u^{*} = r^{*}\). Then we have \(u^{*} = \Phi(l)\) is the largest element of \((V(T^{*}(r^{*})), {\preccurlyeq}_{T^{*}(r^{*})})\), and from Lemma~\ref{l10.10}, equality~\eqref{e10.43} and property \ref{t10.17:s2:3} it follows that
\begin{equation}\label{t10.17:e4}
\delta^{+}(u^{*}) \geqslant 2.
\end{equation}
Similarly, using \ref{t10.17:s2:2} and \ref{t10.17:s2:3} we obtain \eqref{t10.17:e4} if \(u^{*} \neq r^{*}\) and \(u^{*} = \Phi(s^{*})\), where \(s^{*}\) is not a minimal element of \((S, {\preccurlyeq}_S)\). It is clear that \eqref{t10.17:e4} implies \eqref{t10.17:e3}.

Now let \(u^{*} = \Phi(s^{*})\) hold for a minimal element \(s^{*}\) of \((S, {\preccurlyeq}_S)\). Then \(s^{*}\) has no lower covers in \((S, {\preccurlyeq}_S)\). Consequently, from \ref{t10.17:s2:2} and Lemma~\ref{l10.10} it follows the equality \(\delta(u^{*}) = 1\), that implies
\begin{equation}\label{t10.17:e5}
\delta^{+}(u^{*}) = 0
\end{equation}
because \(u^{*} \neq r^{*}\). Condition~\eqref{t10.17:e3} follows from \eqref{t10.17:e5}. Thus, \eqref{t10.17:e3} is valid for all \(u^{*} \in V(T^{*}(r^{*}))\).
\end{proof}

The following example shows that properties \ref{t10.17:s2:1}--\ref{t10.17:s2:4}, which characterized the ordering on balleans of totally bounded ultrametric spaces, are logically independent.

\begin{example}\label{ex10.16}
\((e_1)\) Let \((X, d) = (\mathbb{Q}_2, d_2)\) be the field of \(2\)-adic numbers with the metric \(d_2\) generated by norm \(|\cdot|_2\) (see Example~\ref{ex3.11}) and let \((S, {\preccurlyeq}_S) = (\BB_{X}, {\preccurlyeq}_X)\). Then \((S, {\preccurlyeq}_S)\) has the property \(\neg\ref{t10.17:s2:1} \mathbin{\&} \ref{t10.17:s2:2} \mathbin{\&} \ref{t10.17:s2:3} \mathbin{\&} \ref{t10.17:s2:4}\). Indeed, \((X, d)\) is unbounded and, consequently, \((\BB_{X}, {\preccurlyeq}_X)\) does not contain any largest element. The validity of \(\ref{t10.17:s2:2} \mathbin{\&} \ref{t10.17:s2:3} \mathbin{\&} \ref{t10.17:s2:4}\) follows from Theorem~\ref{t10.17} and the formula
\[
D(\mathbb{Q}_2) = \{2^{i} \colon i \in \mathbb{Z}\} \cup \{0\}
\]
that described the distance set of \((\mathbb{Q}_2, d_2)\) (see \eqref{e3.2}). (Note that every \(B \in \BB_{X}\) is a compact ultrametric space.)

\((e_2)\) For the poset \((S, {\preccurlyeq}_S)\) satisfying \(\ref{t10.17:s2:1} \mathbin{\&} \neg\ref{t10.17:s2:2} \mathbin{\&} \ref{t10.17:s2:3} \mathbin{\&} \ref{t10.17:s2:4}\) see Figure~\ref{fig4} below.

\begin{figure}[ht]
\begin{center}
\begin{tikzpicture}[
solid node/.style={circle,draw,inner sep=1.5,fill=black},
]
\def\dx{0.8cm}
\def\dy{1cm}
\node at (4*\dx,4*\dy) [label={right:{\((S, {\preccurlyeq_S})\)}}] {};
\node at (3*\dx,4*\dy) [solid node, label={right:{\(l\)}}] (A) {};
\node at (2*\dx,3*\dy)  [solid node] (B1) {};
\node at (4*\dx,3*\dy) [solid node] (B2) {};
\node at (3*\dx,2*\dy) [solid node, label={below:{\(s_0\)}}] (C2) {};
\node at (1*\dx,2*\dy)  [solid node] (C1) {};
\node at (5*\dx,2*\dy) [solid node] (C3) {};
\path (A) edge[-Stealth] (B1);
\path (A) edge[-Stealth] (B2);
\path (B1) edge[-Stealth] (C1);
\path (B1) edge[-Stealth] (C2);
\path (B2) edge[-Stealth] (C2);
\path (B2) edge[-Stealth] (C3);
\end{tikzpicture}
\end{center}
\caption{The upper cover of \(s_0\) is not unique in \((S, {\preccurlyeq}_S)\) (cf. Figure~\ref{fig3}).}
\label{fig4}
\end{figure}
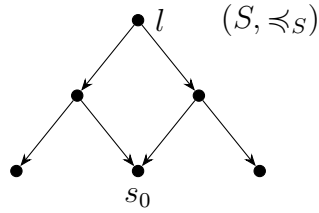

\((e_3)\) Let \((S, {\preccurlyeq}_S)\) be a subposet of \((\RR, {\leqslant}_{\RR})\) defined such that
\[
(s \in S) \Leftrightarrow (\exists n \in \mathbb{N} \colon s = -n).
\]
Then \((S, {\preccurlyeq}_S)\) has properties \ref{t10.17:s2:1}, \ref{t10.17:s2:2} and \ref{t10.17:s2:4}, but the number of lower covers of every \(s \in S\) is one. Thus, \(\ref{t10.17:s2:1} \mathbin{\&} \ref{t10.17:s2:2} \mathbin{\&} \neg\ref{t10.17:s2:3} \mathbin{\&} \ref{t10.17:s2:4}\) is valid in this case.

\((e_4)\) Let \((X, d)\) and \((Y, \rho)\) be nonempty, disjoint, totally bounded ultrametric spaces without isolated points. Write \(S = \BB_{X} \cup \BB_{Y}\) and define \({\preccurlyeq}_S\) such that
\[
{\preccurlyeq}_X = {\preccurlyeq}_S \cap (\BB_{X} \times \BB_{X}), \quad {\preccurlyeq}_Y = {\preccurlyeq}_S \cap (\BB_{Y} \times \BB_{Y})
\]
and, moreover,
\[
B_1 \preccurlyeq_S B_2
\]
holds for all \(B_1 \in \BB_{X}\) and \(B_2 \in \BB_{Y}\). Then using Theorem~\ref{t10.17} we see that \((S, {\preccurlyeq}_S)\) has properties \ref{t10.17:s2:1}--\ref{t10.17:s2:3}, but the poset \((S, {\sqsubset}_S^t)\) is not totally ordered. Hence, we have \({\preccurlyeq}_S \neq {\sqsubset}_S^t\). Thus, \(\ref{t10.17:s2:1} \mathbin{\&} \ref{t10.17:s2:2} \mathbin{\&} \ref{t10.17:s2:3} \mathbin{\&} \neg\ref{t10.17:s2:4}\) is valid (cf. \((e_1)\) in Example~\ref{ex9.21}).
\end{example}

For finite posets, Theorem~\ref{t10.17} implies the following.

\begin{corollary}\label{c12.19}
Let \((S, \preccurlyeq_{S})\) be a nonempty finite poset. Then conditions \ref{c12.19:s1} and \ref{c12.19:s2} are equivalent.
\begin{enumerate}
\item \label{c12.19:s1} There is a nonempty finite ultrametric space \((X, d)\) such that the posets \((\BB_{X}, \preccurlyeq_{X})\) and \((S, \preccurlyeq_{S})\) are isomorphic.
\item \label{c12.19:s2} The poset \((S, \preccurlyeq_{S})\) has the following properties:
\begin{enumerate}
\item\label{c12.19:s2:1} If \(p \in S\) is not the largest element of \(S\), then there is a unique \(q \in S\) such that \(p \sqsubset_S q\).
\item\label{c12.19:s2:2} If \(b \in S\) is not a minimal element of \(S\), then the number of lower covers of \(b\) is greater than or equal to two.
\end{enumerate}
\end{enumerate}
\end{corollary}

\begin{proof}
The validity of implication \(\ref{c12.19:s1} \Rightarrow \ref{c12.19:s2}\) directly follows from Theorem~\ref{t10.17}.

To prove the truth of implication \(\ref{c12.19:s2} \Rightarrow \ref{c12.19:s1}\) we note that the equality \({\preccurlyeq}_S = {\sqsubseteq}_S^t\) holds because \(S\) is finite. Moreover, the finiteness of \(S\) and condition~\ref{c12.19:s2:1} imply the existence of the largest element in \(S\). Indeed, since \(S\) is nonempty and finite, \(S\) contains a maximal element which is the largest element of \(S\) by property~\ref{c12.19:s2:1}. Consequently, condition~\ref{t10.17:s2} of Theorem~\ref{t10.17} follows from condition~\ref{c12.19:s2} of the present corollary if \(S\) is finite. Now using Theorem~\ref{t10.17} we obtain the existence of desirable finite \((X, d)\).
\end{proof}

\begin{remark}\label{r10.17}
It was noted in \cite{DPT2015} that, for every non-ultrametric, three-point metric space \((X, d)\), the poset \((\BB_{X}, {\preccurlyeq}_X)\) is order isomorphic to the poset \((S, {\preccurlyeq}_S)\) depicted by Figure~\ref{fig4}. If a tree-point metric space \((X, d)\) is ultrametric, then \((\BB_{X}, {\preccurlyeq}_X)\) is order isomorphic either to the poset \((S_1, {\preccurlyeq}_{S_1})\) or to the poset \((S_2, {\preccurlyeq}_{S_2})\) depicted in Figure~\ref{fig4.1}. Thus, Figures~\ref{fig4} and \ref{fig4.1} give us all posets which are order isomorphic to \((\BB_{X}, {\preccurlyeq}_X)\) with \(|X| = 3\).
\end{remark}

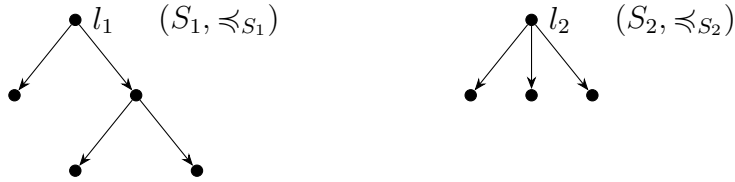
\begin{figure}[ht]
\begin{center}
\begin{tikzpicture}[
solid node/.style={circle,draw,inner sep=1.5,fill=black},
]
\def\xx{0cm}
\def\dx{0.8cm}
\def\dy{1cm}
\node at (\xx+4*\dx,4*\dy) [label={right:{\((S_1, {\preccurlyeq_{S_1}})\)}}] {};
\node at (\xx+3*\dx,4*\dy) [solid node, label={right:{\(l_1\)}}] (A) {};
\node at (\xx+2*\dx,3*\dy)  [solid node] (B1) {};
\node at (\xx+4*\dx,3*\dy) [solid node] (B2) {};
\node at (\xx+3*\dx,2*\dy) [solid node] (C2) {};
\node at (\xx+5*\dx,2*\dy) [solid node] (C3) {};
\path (A) edge[-Stealth] (B1);
\path (A) edge[-Stealth] (B2);
\path (B2) edge[-Stealth] (C2);
\path (B2) edge[-Stealth] (C3);

\def\xx{6cm}
\node at (\xx+4*\dx,4*\dy) [label={right:{\((S_2, {\preccurlyeq_{S_2}})\)}}] {};
\node at (\xx+3*\dx,4*\dy) [solid node, label={right:{\(l_2\)}}] (A) {};
\node at (\xx+2*\dx,3*\dy)  [solid node] (B1) {};
\node at (\xx+3*\dx,3*\dy) [solid node] (B2) {};
\node at (\xx+4*\dx,3*\dy) [solid node] (B3) {};
\path (A) edge[-Stealth] (B1);
\path (A) edge[-Stealth] (B2);
\path (A) edge[-Stealth] (B3);
\end{tikzpicture}
\end{center}
\caption{If \((X, d)\) is ultrametric and \(|X| = |D(X)| = 3\) holds, then \((\BB_{X}, {\preccurlyeq}_X)\) is isomorphic to \((S_1, {\preccurlyeq}_{S_1})\). For ultrametric space \((X, d)\) satisfying \(|X| = |D(X)| + 1 = 3\), the poset \((\BB_{X}, {\preccurlyeq}_X)\) is isomorphic to \((S_2, {\preccurlyeq}_{S_2})\).}
\label{fig4.1}
\end{figure}

The above remark leads to the following.

\begin{problem}\label{pr10.18}
Using the theory of directed graphs describe up to isomorphism the ordered balleans of all finite metric spaces.
\end{problem}

For the case of finite, ultrametric spaces Problem~\ref{pr10.18} is closely connected with the problem of I.~M.~Gelfand~\cite{Lem2001} that is the starting point of the present paper. It should be noted here that Problem~\ref{pr10.18} is also related to the problem of combinatorial classification of finite metric spaces via their fundamental polytopes that was proposed by A.~M.~Vershik~\cite{Ver2015AMJ} (see also \cite{DH2020EJoC} for some results in this direction).

\begin{remark}\label{r9.12}
The locally finite, infinite trees have been intensively studied by many mathematicians of the last several decades. One of the most interesting areas of such research is connected with the so-called ``Reconstruction Conjecture''. This conjecture claims that two finite graphs \(G\) and \(H\) with \(|V(G)| = |V(H)| \geqslant 3\) are isomorphic if \(G\) and \(H\) are hypomorphic, i.e., there is a bijection \(g \colon V(G) \to V(H)\) such that the induced subgraphs \(G - v\) and \(H - g(v)\) are isomorphic for every \(v \in V(G)\). Kelly \cite{Kel1957PJM} proved the Reconstruction Conjecture for finite trees. An survey of results connected with this conjecture was given by Bondy and Hemminger \cite{BH1977JGT}. A variant of the Reconstruction Conjecture for locally finite trees was proposed by Harary, Schwenk and Scott \cite{HSS1972PIM} but Bowler, Erde, Heinig Lehner and Pitz \cite{BEH+2017BLMS} found hypomorphic, locally finite trees which are not isomorphic. The reconstruction of rooted trees was considered T.~Andreae \cite{And1994JCTSB}. It seams to be interesting to translate the theorem Bowler---Erde---Heinig---Lehner---Pitz and some results related to Reconstruction Conjecture (see, for example, \cite{And1981JGT, And1994JCTSB}) on the language of ultrametric spaces.
\end{remark}

\section*{Funding}

The author was supported by the Academy of Finland, Project 359772 ``Labeled trees and totally bounded ultrametric spaces''.

\end{document}